%% file: sepdesc.tex
\documentclass[12pt,a4paper,svgnames,reqno]{amsart}
\usepackage{amssymb,stmaryrd}%,blkarray}
\usepackage[a4paper,left=2.25cm,right=2.25cm,top=3cm,bottom=3cm,headsep=1cm]{geometry}

\usepackage[mathbf]{euler} 
\usepackage{mathtools,mathdots}
\usepackage{enumerate}
\newcommand\calM{{\mathcal M}}

\newcommand\blank{{\hspace{.3mm}\rule{2mm}{0.15mm}\hspace{.3mm}}}

\newcommand\loc{\mathrm{loc}}
\newcommand\inv{\mathrm{inv}}
\usepackage{tikz}\usetikzlibrary{fit,positioning,matrix,calc,decorations.markings,angles,decorations.pathmorphing,decorations.pathreplacing,arrows.meta}%,matrix,fit,shapes.geometric}
\tikzset{mynode/.style={circle,draw=black,fill=black,inner sep=1.8pt,outer sep=0pt}}
\tikzset{edgelabel/.style={\mcol,inner sep=0pt}}
\tikzset{invlabel/.style={draw=black,text=black,circle,inner sep=0pt,minimum size=3mm}}
\tikzset{triv/.style={circle,fill=black,inner sep=0.2mm}}%for trivalent vertices to look nicer
\newcommand\tikzif[2][]{
\tikzifinpicture{#2}{\begin{tikzpicture}[#1]#2\end{tikzpicture}}% don't process options if not
}
\tikzset{math mode/.style = {execute at begin node=$, execute at end node=$}}%makes all nodes math mode
\def\dr{red!80!black} % \def\dr{red!50!black}
\def\dg{green!80!black} % \def\dg{green!50!black}
\def\db{blue!80!black} % \def\db{blue!80!black}
\tikzset{d/.style={ultra thick}}
\tikzset{dr/.style={draw=\dr,d}}
\tikzset{dg/.style={draw=\dg,d}}
\tikzset{db/.style={draw=\db,d}} 
\def\mcol{}
\tikzset{rt/.style={text=blue,execute at begin node=$\sf,execute at end node=$}}

\tikzset{arrow/.style={postaction={decorate,decoration={markings,mark = at position #1 with {\arrow{Straight Barb[line width=0.2mm,length=1.2mm]}}}}},arrow/.default=0.5}
\tikzset{invarrow/.style={postaction={decorate,decoration={markings,mark = at position #1 with {\arrowreversed{Straight Barb[line width=0.2mm,length=1.2mm]}}}}},invarrow/.default=0.5}
\newlength\myshift
\newcommand\getshift{%
\pgfmathsetlength{\myshift}{0.8mm}
}
\def\noparenx#1{%
	\ifx\relax#1
	\else
	\if)#1%
	\else
	\if(#1%
	\else
	#1%
	\fi
	\fi
	\expandafter\noparenx
	\fi
}

\newcommand\rh[5][]{
\tikzif[baseline=0,scale=1.25]{
\getshift
\draw[thick,black] (0,0) coordinate (ad) -- node[edgelabel,left,xshift=\myshift] {$#2$} ++(-60:1) coordinate (ba) -- node[edgelabel,right,xshift=-\myshift] {$#3$} ++(60:1) coordinate (cb) -- node[edgelabel,right,xshift=-\myshift] {$#4$} ++(120:1) coordinate (dc) -- node[edgelabel,left,xshift=\myshift] {$#5$} cycle;
\ifx\&#1\&\else\node[invlabel] at (0.5,0) {$\ss#1$};\fi
}}

\newcommand\uptri[4][]{
  \tikzif[baseline=0.34cm,scale=1.25]{
    \getshift
    \draw[thick,black] (-0.5,0) -- node[edgelabel] (horiz) {$\vphantom{0}\smash{#3}$} ++(0:1) -- node[edgelabel,right,xshift=-\myshift] (NE) {$#4$} ++(120:1) -- node[edgelabel,left,xshift=\myshift] (NW) {$#2$} ++(240:1) -- cycle; % redundant ending due to tikz node positioning bug with cycle
    \ifx\&#1\&\else\node[invlabel] at (0,0.33) {$\ss#1$};\fi
  }}
\newcommand\downtri[4][]{
  \tikzif[baseline=-0.54cm,scale=1.25]{
    \getshift
    \begin{scope}[scale=-1]
      \draw[thick,black] (-0.5,0) -- node[edgelabel] (horiz) {$\vphantom{0}\smash{#3}$} ++(0:1) -- node[edgelabel,left,xshift=\myshift] (NE) {$#4$} ++(120:1) -- node[edgelabel,right,xshift=-\myshift] (NW) {$#2$} ++(240:1) -- cycle; % redundant ending due to tikz node positioning bug with cycle
      \ifx\&#1\&\else\node[invlabel] at (0,0.33) {$\ss#1$};\fi
    \end{scope}
  }}

\usepackage{hyperref}
\allowdisplaybreaks[1]
\renewcommand\ss{\scriptstyle}
\newcommand\sss{\scriptscriptstyle}
\newcommand\Uq{{\mathcal U}_q}
\newcommand\gqg{\Uq(\mathfrak{g}[z^\pm])}

\newcommand\Aqg{\Uq(\mathfrak{gl}_{d+2}[z^\pm])}

\DeclareMathOperator{\sign}{sign}
\DeclareMathOperator{\sort}{sort}
\DeclareMathOperator{\wt}{wt}
\DeclareMathOperator{\fug}{fug}
\DeclareMathOperator{\desc}{desc}
\DeclareMathOperator{\overlap}{over}
\DeclareMathOperator{\hull}{hull}
\DeclareMathOperator{\supp}{supp}
\DeclareMathOperator{\parity}{parity}

\newcommand\ZZ{{\mathbb Z}}
\newcommand\NN{{\mathbb N}}
\newcommand\QQ{{\mathbb Q}}
\newcommand\CC{{\mathbb C}}
\newcommand\PP{{\mathbb P}}

\newcommand\RR{{\mathbb R}}
\newcommand\Otimes\bigotimes
\theoremstyle{plain}
\newtheorem{thm}{Theorem}
\newtheorem{thmd}{Theorem}
\newtheorem{prop}{Proposition}
\newtheorem{lem}{Lemma}

\theoremstyle{remark}
\newtheorem{rmk}{Remark}
\newtheorem{ex}{Example}
\newcommand\defn[1]{\textbf{#1}}
\newcommand\into\hookrightarrow
\newcommand\onto\twoheadrightarrow
\newcommand\otno\twoheadleftarrow
\newcommand\dom\backslash
\newcommand\tensor\otimes
\newcommand\Tensor\bigotimes
\newcommand\iso\cong
\newcommand\lie[1]{{\mathfrak{#1}}}
\tikzset{gauged/.style={rectangle,rounded corners=2mm,draw,inner sep=0.5mm,minimum size=4mm}}
\tikzset{framed/.style={rectangle,draw,inner sep=0.5mm,minimum size=4mm}}
\tikzset{script math mode/.style = {execute at begin node=$\ss , execute at end node=$}}%makes all nodes math mode
\newcommand\actson\circlearrowright

\newcommand\junk[1]{}

\newcommand\Sinf{{\mathcal S}_\infty}
\newcommand\der{\partial}
\newcommand\Sc{{\mathfrak S}}
\newcommand\Gr{{\mathfrak G}}
\newcommand\rem[2][]{}
\title[Schubert puzzles and integrability III]{Schubert puzzles and integrability III: \\ separated descents}
\author{Allen Knutson}
\address{Allen Knutson, Cornell University, Ithaca, New York}
\email{allenk@math.cornell.edu}
\thanks{AK was supported by NSF grant 1953948.}
\author{Paul Zinn-Justin}
\address{Paul Zinn-Justin, School of Mathematics and Statistics, The University of Melbourne, 
Victoria 3010, Australia}
\email{pzinn@unimelb.edu.au}
\thanks{PZJ was supported by ARC grants DP190102897 and DP210103081.
  Computerized checks of the results of this paper were performed
  with the help of Macaulay 2 \cite{M2,artic82};
  all examples of this paper can be found implemented at
  \url{https://www.unimelb-macaulay2.cloud.edu.au/\#tutorial-sepdesc}.
}
\date{\today}
\begin{document}

\begin{abstract}
  In paper I of this series we gave positive formul\ae\ for expanding
  the product $\Sc^\pi \Sc^\rho$ of two Schubert polynomials, in the
  case that both $\pi,\rho$ had shared descent set of size $\leq 3$.
  Here we introduce and give positive formul\ae\ for two new classes of
  Schubert product problems: {\em separated descent} in which
  $\pi$'s first descent occurs at (or after) $\rho$'s last,
  and {\em almost separated descent} in which
  $\pi$'s first two descents occur at (or after) $\rho$'s last two
  respectively. In both cases our puzzle formul\ae\ extend to
  $K$-theory (multiplying Grothendieck polynomials), and in the
  separated descent case, to equivariant $K$-theory. The two formul\ae\
  arise (via quantum integrability) from fusion of minuscule
  quantized loop algebra representations in types $A$, $D$ respectively.
\end{abstract}

%\vspace*{-0.75cm}
\maketitle
%\vspace*{-0.75cm}

{\small
  \tableofcontents
}

\section{Introduction}

Our references for the classic material in
\S\ref{ssec:schubgroth}-\ref{ssec:descents} are \cite{AndersonFulton}, and
for the well-studied\footnote{It seems astonishing that these are not
  yet in textbooks, other than a passing reference in
  \cite{MS-book}.} Grothendieck polynomials, \cite{Kirillov}.

\subsection{Schubert and Grothendieck polynomials}
\label{ssec:schubgroth}
Let $\Sinf := \bigcup_{n\ge 1} \mathcal S_n$ be the infinite symmetric group,
where $\mathcal S_n\subset \mathcal S_{n+1}$ in the natural way.
Given $\sigma\in\Sinf$, we define the \defn{Schubert polynomial}
$\Sc^\sigma\in R := \ZZ[x_1,x_2,\ldots]$ inductively by
\begin{align*}
\Sc^\sigma  := \der_i \Sc^{\sigma s_i}\qquad \text{for }\sigma(i)<\sigma(i+1)
\qquad\qquad\qquad\qquad
\Sc^{n\ldots 21}  := \prod_{i=1}^n x_i^{n-i}
\end{align*}
where $s_i$ is the elementary transposition $i\leftrightarrow i+1$, $\der_i$ is the corresponding \defn{divided difference operator}
\[
\der_i f := \frac{f-f|_{x_i\leftrightarrow x_{i+1}}}{x_i-x_{i+1}}\qquad f\in R
\]
and $n\ldots 21$ is the longest element in $\mathcal S_n$. This $\Sc^\sigma$
is a homogeneous polynomial, whose degree is
the \defn{inversion number} $\ell(\sigma)$ of $\sigma$ defined by
\[
\ell(\sigma) := \#\{i<j: \sigma_i>\sigma_j\}\ .
\]
The $\{\Sc^\sigma,\ \sigma\in \Sinf\}$ form a $\ZZ$-basis of $R$,
and this basis is compatible with each subring
$\bigcap_{i \in C} \ker(\partial_i)\leq R$ defined by a choice of
``ascent set'' $C \subseteq \NN_+$, i.e.
$$ \bigcap_{i \in C} \ker(\partial_i) \quad\text{ has $\ZZ$-basis }\quad
\{\Sc^\sigma\colon \sigma \in \Sinf,\ \sigma(i)<\sigma(i+1) \ \forall i\in C\}
\ .
$$

Define the \defn{Schubert structure constants}
$c^{\pi\rho}_\sigma \in \ZZ$ by
\[
\Sc^\pi \Sc^\rho =\sum_{\sigma\in\Sinf} c^{\pi\rho}_\sigma \Sc^\sigma\ .
\]
It is well-known that $c^{\pi\rho}_\sigma\in \ZZ_{\ge 0}$ from
geometric considerations \cite[Chapter 19.3]{AndersonFulton}.
This prompts the search for a manifestly positive combinatorial formula for
the $c^{\pi\rho}_\sigma$.

The compatibility of the basis with the subrings has the consequence
$$ \pi(i) < \pi(i+1), \rho(i)<\rho(i+1), c^{\pi\rho}_\sigma \neq 0
\quad\implies\quad
\sigma(i)<\sigma(i+1). $$

The Schubert polynomial $\Sc^\sigma$ is the lowest degree homogeneous
part of $\Gr^\sigma(1-x_1,1-x_2,\ldots)$, where the
inhomogeneous \defn{Grothendieck polynomials} $\Gr^\sigma$ are defined by
the recursion
\begin{align*}
\Gr^\sigma := \bar\der_i \Gr^{\sigma s_i}\qquad \text{for }\sigma(i)<\sigma(i+1)
\qquad\qquad
\Gr^{n\ldots 21} := \prod_{i=1}^n (1-x_i)^{n-i}
\end{align*}
and the \defn{Demazure operator} $\bar\der_i$ is defined by
\[
\bar\der_i f := \frac{x_{i+1}f-x_if|_{x_i\leftrightarrow x_{i+1}}}{x_{i+1}-x_i}\qquad f\in R
\ .
\]

There are ``double'' versions (in that they use two sets of variables)
of both the Schubert and Grothendieck polynomials, based on the same
operators but with modified initial conditions:
\begin{align*}
\Sc^{n\ldots 21} = \prod_{i,j\ge1,\, i+j\le n} (x_i-y_j)
\qquad\qquad
\Gr^{n\ldots 21} = \prod_{i,j\ge1,\, i+j\le n} (1-x_i/y_j)
\ .
\end{align*}

\subsection{Descents and overlap}\label{ssec:descents}
The problem of computing the $c^{\pi\rho}_\sigma$ has a long history;
see \cite[Chapter 1]{AndersonFulton}, \cite{KnutsonICM}.
In particular, if $\sigma$ has (at most) a single descent at say $k$,
then $\Sc^\sigma$ is a {\em Schur polynomial}\/ in $k$ variables,
and if all of $\pi,\rho,\sigma$ do then the corresponding $c^{\pi\rho}_\sigma$
are Littlewood--Richardson coefficients \cite[Chapter 10.6]{AndersonFulton}.

More generally, past work suggests that the complexity of the product rule increases
with $\desc(\pi,\rho)=\#(D(\pi)\cup D(\rho))$ where
\[
  D(\sigma) := \{ i\in \ZZ_{> 0}: \sigma_i>\sigma_{i+1} \}% \qquad \sigma\in \Sinf
\]
is the descent set of $\sigma\in \Sinf$.
In particular, in \cite{BKPT,artic71}, the cases $\desc(\pi,\rho)=2,3$
were treated, giving ``puzzle rules'' for $c^{\pi\rho}_\sigma$ (only
nonequivariantly in the second case), and in \cite{artic80} a mildly
nonpositive (and nonequivariant) puzzle rule was given for
$\desc(\pi,\rho)=4$.

In this paper we measure this complexity in a different way.
Given $\pi,\rho\in\Sinf$, we introduce the \defn{overlap}\/ of (the
descent sets of) $\pi$ and $\rho$ to be
\[
  O(\pi,\rho) := ((D(\pi)\cup D(\rho))\cap [\min D(\pi),\max D(\rho)])
  % o(\pi,\rho)=\# ((D(\pi)\cup D(\rho))\cap [\min D(\pi),\max D(\rho)])
\]
and its cardinality $\overlap(\pi,\rho) := \# O(\pi,\rho)$.
Note that $\overlap(\pi,\rho)\le \desc(\pi,\rho)$.

In what follows we provide two combinatorial formul\ae\ for  $c^{\pi\rho}_\sigma$ when $\overlap(\pi,\rho)=1,2$ 
(only the first of which extends to the equivariant setting).
Note that $\overlap(\pi,\rho)$ and $\overlap(\rho,\pi)$ are in general different;
since $c^{\pi\rho}_\sigma=c^{\rho\pi}_\sigma$, we have the freedom to interchange $\pi$ and $\rho$.

\newcommand\Al{{\mathcal A}}%alphabet
\subsection{From strings to permutations, and back}
Before we formulate our main results, we need a different labeling of Schubert polynomials.
Given a totally ordered set $\Al$, to a string $\lambda\in \Al^n$ of $n$ letters in $\Al$ we associate a permutation $f_\Al(\lambda)\in \mathcal S_n$
as follows.
First consider the \defn{standardization}\/ $\tilde\lambda$ of $\lambda$ which is obtained by replacing letters of $\lambda$ with
$\{1,\ldots,n\}$ in such a way that $\lambda_i\le \lambda_j \Leftrightarrow \tilde\lambda_i<\tilde\lambda_j$ for all $i<j$.
View $\tilde\lambda$ as a permutation and define $f_\Al(\lambda)=\tilde\lambda^{-1}$.
For example, if $\Al=\{0,1,2\}$ and $\lambda=0201$, then $\tilde\lambda=1423$ and $f_\Al(\lambda)=1342$.

%View $f_A$ as a map from $\bigsqcup_{n\ge0} A^n$ to $\Sinf$. Note that 
This $f_\Al:\Al^n\to \mathcal S_n$ is not injective as soon as $\# \Al>1$,
%(even at fixed $n$), 
e.g., any single-letter string is mapped to the identity permutation. It is however surjective for $\Al=\ZZ_{\ge 0}$,
and in fact any permutation $\sigma\in\mathcal S_n$ has a unique preimage in $\Al^n$ when $\# \Al = \# D(\sigma)+1$ (this being the minimal value of $\#\Al$).
With the standard choice $\Al=\{0,\ldots,d\}$, write $\omega=0^{p_0}\ldots d^{p_d}\in \Al^n$, where $p_i=n_{i+1}-n_i$
and $D(\sigma)=\{ n_1<\cdots<n_d\}$, $n_0=0$, $n_{d+1}=n$.
Then $\lambda$ defined by $\lambda_{\sigma(i)}=\omega_i$, $i=1,\ldots,n$, 
satisfies $f_\Al(\lambda)=\sigma$. For example,
\[
\begin{tikzpicture}[scale=.5]
\node at (-1,1) {$\sigma$}; \foreach \s [count=\x] in {1,3,6,2,5,4,7} \node at (\x,1) {$\s$};
\node at (-1,0) {$\omega$}; \foreach \s [count=\x] in {0,0,0,1,1,2,2} \node at (\x,0) {$\s$};
%\foreach \s [count=\x] in {1,3,6,2,5,4,7} \draw[-latex,thick] (\x,-0.8) -- (\s,-2.2);
\foreach \s [count=\x] in {1,3,6,2,5,4,7} \draw[-latex,thick] (\x,-0.8) .. controls (0.85*\x+0.15*\s,-1.5) and (0.85*\s+0.15*\x,-1.5) .. (\s,-2.2);
\node at (-1,-3) {$\lambda$}; \foreach \s [count=\x] in {0,1,0,2,1,0,2} \node at (\x,-3) {$\s$};
\draw (3.5,1.5) -- (3.5,-0.5);
\draw (5.5,1.5) -- (5.5,-0.5);
\end{tikzpicture}
\]
% \[
% \begin{array}{ccccc|cc|cc}
% w && 1 & 3 & 6 & 2 & 5 & 4 & 7 
% \\
% \omega && 0 & 0 &0 & 1 & 1 & 2 & 2
% \end{array}
% \quad
% \Rightarrow
% \quad
% \lambda = 0 1 0 2 1 0 2
% \]
\junk{PZJ: should we define the content of $\lambda$ to be $\omega$?
  AK: Do you mean, define $\omega$ to be the content of $\lambda$?
  PZJ: it's just that I never defined the concept of "content". an
  obvious way to define it is to say that it's same as the "sort"
  operation that takes $\lambda$ to $\omega$}
Write $\omega = \sort(\lambda)$ in this case and call $\omega$ the \defn{content}
of $\lambda$. The notion of ``content'' more usually refers to the
multiplicities of the symbols used, but this approach is of
equivalent utility; obviously $\lambda,\lambda'$ have
the same content exactly when $\sort(\lambda) = \sort(\lambda')$.

We can encode $\sigma$ with more letters in the alphabet by adding gratuitous nondescents to the set
$N=\{n_1,\ldots,n_d\}$. For example, given two permutations $\pi,\rho\in\Sinf$, it is natural to use
the set $N=D(\pi)\cup D(\rho)$, leading to strings $\lambda$ and $\mu$ (for $\pi$ and $\rho$) with the same content.
This is the choice that's (implicitly) made in \cite{artic71,artic80}, but not in the present work.

\subsection{Separated-descent puzzles}
Assume given $\pi,\rho\in\mathcal S_n$ with $\overlap(\pi,\rho)=1$. Writing $\# D(\rho)=k+1$ and $\# D(\pi)=d-k$, we choose the alphabets
%\begin{align}
\begin{equation} \label{eq:AB}
\begin{split}
\Al_1&=\{\ \blank<k+1<\cdots<d\}
\\
\Al_2&=\{0<\cdots<k<\blank\ \}
\end{split}
\end{equation}
% \end{align}
(the $\blank\ $ letter is a ``blank''\footnote{The ``blank'' edge labels will not be drawn in puzzles; we hope this does not cause confusion.}
to describe $\pi$ as a string $\lambda\in \Al_1^n$ and $\rho$ as a string $\mu\in \Al_2^n$.

Note that if $O(\pi,\rho)=\{r\}$, 
then there are exactly $r$ blanks in $\lambda$ because $\min D(\pi)=r$, and $n-r$ blanks in $\mu$ because $\max D(\rho)=r$.

Here is a practical test for when $(\pi,\rho)$ form a
separated-descent pair, and what strings $\lambda,\mu$ to associate to them.
\begin{itemize}
\item Invert $\pi$ and $\rho$ to form the initial $\lambda,\mu$.
\item Find a $k \in [n]$ such that the values $1\ldots k$ occur in order
  in $\lambda$, and $k+1\ldots n$ occur in order in $\mu$.
  There may be multiple $k$ for which this is satisfied.
  (If there are none, we are not in the separated-descent case.)
  Erase $1\ldots k$ from $\lambda$ and $k+1\ldots n$ from $\mu$,
  leaving blanks $\blank\,$ in their places.
\item If more numbers in the resulting strings occur in order ($3$ left of
  $4$ left of $5$, say), they can safely be replaced by a single letter
  (all $3$s). The Kogan cases \cite{Kogan} occur when all of $\lambda$'s
  remaining letters are identifiable, or all of $\mu$'s.
\end{itemize}
These will be the resulting strings $\lambda,\mu$ in the theorem below.

We need a third alphabet obtained by removing the blank in $\Al_1\cup \Al_2$:
\begin{equation}\label{eq:C}
\Al_3=\{0<\cdots<d\}
\ .
\end{equation}

\begin{thm}\label{thm:sepdesc}
Given strings $\lambda\in \Al_1^n$ and $\mu\in \Al_2^n$ in the alphabets \eqref{eq:AB}, such that
the combined number of blanks in $\lambda$ and $\mu$ is $n$,
write $\pi=f_{\Al_1}(\lambda)$ and $\rho=f_{\Al_2}(\mu)$; then for any $\sigma\in\mathcal S_n$,
$c^{\pi\rho}_\sigma$ is the number of ``puzzles'' with boundaries $\lambda,\mu,\nu$ such that $\sigma=f_{\Al_3}(\nu)$
(with $\Al_3$ given by \eqref{eq:C}).
A puzzle means here a size $n$ equilateral triangle subdivided into size $1$ triangles, with labels on edges of the latter,
following the patterns
\begin{align*}
& \uptri {} i i \qquad \downtri {} i i 
\qquad %&
 \uptri i i {} \qquad \downtri i i {}
\qquad \qquad\qquad
& \uptri j {ij} i\qquad \downtri j {ij} i \qquad i<j\ .
\end{align*}
For any such puzzle, $\nu\in \Al_3^n$, and its content is the concatenation
of $\lambda$'s with $\mu$'s with the blanks removed.

\noindent
\begin{minipage}{0.58\linewidth}  \hskip 1em
To compute structure constants for Grothendieck polynomials, allow extra ``$K$-triangles'', as at right.
Then
$c^{\pi\rho}_\sigma=(-1)^{\ell(\pi)+\ell(\rho)-\ell(\sigma)} \#\{\text{such puzzles}\}$,
realizing the (alternating signed) ``$K$-positivity''
guaranteed in \cite{Brion-KPos}.
Alternatively, one should consider that each $K$-triangle contributes
a factor of $-1$ (the contribution of a full puzzle being the product
of the contributions of its pieces), insofar as {\em every one}\/ of
the puzzles has $\ell(\sigma)-\ell(\pi)-\ell(\rho)$ $K$-tiles.
\end{minipage}
\hfill
\begin{minipage}{0.35\linewidth}
\begin{align*}
&\uptri {i} {ij} {j}\qquad i\le k<j
\\[3mm]
&\downtri{i}{ij}{j} \qquad i<j\le k\text{ or } k<i<j
\end{align*}
\end{minipage}

\noindent
\begin{minipage}{0.8\linewidth}
\hskip 1em
To compute structure constants for double Schubert polynomials, allow an
``equivariant rhombus'' with all blank edges, as at right.
Each such rhombus contributes a factor of $y_j-y_i$ where
$n-i$ is the distance to the NE side and $j$ is the distance to the NW side.
(Practically speaking, draw lines SW and SE from the rhombus, exiting
at positions $j$ and $i$.)
\end{minipage}
\hfill
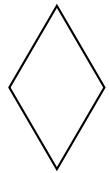
\begin{minipage}{0.1\linewidth}
  \vskip-.2cm
  $ \rh{}{}{}{} $
\end{minipage}

\noindent
\begin{minipage}{0.58\linewidth}  \hskip 1em
  Finally, to compute structure constants for double Grothendieck
  polynomials, allow both equivariant rhombi and $K$-triangles;
  an equivariant rhombus contributes $1-y_j/y_i$, $K$-triangles still
  contribute $-1$ except the up-pointing $K$-triangles contribute an
  extra $y_j/y_i$, and so do rhombi of the form at right.
This latter rule manifestly realizes the ``$K_T$-positivity'' predicted in \cite{AGM-Kpos}.
\end{minipage}
\quad
\begin{minipage}{0.35\linewidth}
  \vskip-.3cm
\[
\tikz[baseline=0,scale=1.25]{\rh{}{j}{j}{}\draw (0,0) -- node {$j$} (1,0);}\quad j>k
\qquad
\tikz[baseline=0,scale=1.25]{\rh{i}{}{}{i}\draw (0,0) -- node {$i$} (1,0);}\quad i\le k
\]
\end{minipage}
\end{thm}

Later, for better visualization of puzzles, we'll connect using colored paths
those edges of a given triangle with the same label.

\begin{ex}\label{ex:sepdesc}
Let $\pi=1362547$ and $\rho=7321456$:
\[
\begin{tikzpicture}[xscale=.75,yscale=.5]
\node at (-1,4) {$\pi$}; \foreach \s [count=\x] in {1,3,6,2,5,4,7} \node at (\x,4) {$\s$};
\node at (-1,3) {$\omega_1$}; \foreach \s [count=\x] in {\blank,\blank,\blank,3,3,4,4} \node at (\x,3) {$\vphantom{0}\s$};
\node at (-1,2) {$\omega_3$}; \foreach \s [count=\x] in {0,1,2,3,3,4,4} \node at (\x,2) {$\s$};
\node at (-1,1) {$\omega_2$}; \foreach \s [count=\x] in {0,1,2,\blank,\blank,\blank,\blank} \node at (\x,1) {$\vphantom{0}\s$};
\node at (-1,0) {$\rho$}; \foreach \s [count=\x] in {7,3,2,1,4,5,6} \node at (\x,0) {$\s$};
%\foreach \s [count=\x] in {1,3,6,2,5,4,7} \draw[-latex,thick] (\x,-0.8) -- (\s,-2.2);
%\foreach \s [count=\x] in {1,3,6,2,5,4,7} \draw[-latex,thick] (\x,-0.8) .. controls (0.85*\x+0.15*\s,-1.5) and (0.85*\s+0.15*\x,-1.5) .. (\s,-2.2);
%\node at (-1,-3) {$\lambda$}; \foreach \s [count=\x] in {0,1,0,2,1,0,2} \node at (\x,-3) {$\s$};
\draw (3.5,4.5) -- (3.5,-0.5);
\draw (5.5,4.5) -- (5.5,1.5);
\draw (1.5,2.5) -- (1.5,-0.5);
\draw (2.5,2.5) -- (2.5,-0.5);
\end{tikzpicture}
\]
where the $\omega_i$ are the contents on the three sides of the puzzles.
We read off $\lambda=\blank3\blank43\blank4$ and $\mu=\blank21\blank\blank\blank0$.
Here are the puzzles computing the product $\Sc^\pi \Sc^\rho$
of the single Schubert polynomials:
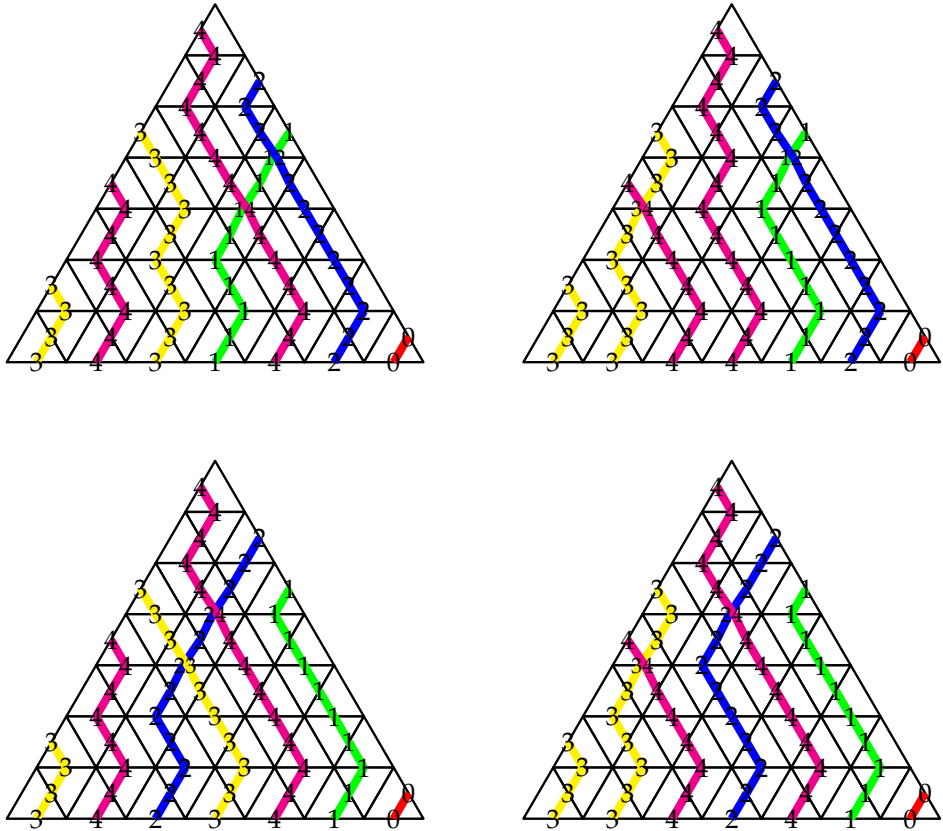
\begin{center}
\tikzset{every picture/.style={scale=0.5,every node/.style={scale=0.5}}}
\input ex1.tex
\end{center}
We've connected edges of a given triangle with the same label using
colored paths for better visualization.  We conclude from the puzzle
calculation that
\[
\Sc^{1362547}\Sc^{7321456} = \Sc^{7461325}+\Sc^{7561234}+\Sc^{7631425}+\Sc^{7641235}+\cdots
%\Sc^{1362547}\Sc^{7312456} = \Sc^{7451326}+\Sc^{7531426}+\Sc^{7541236}+\Sc^{7621435}+\cdots
\]
where ``$\cdots$'' stand for a sum of Schubert polynomials labeled by permutations of size $>7$. In order to get the full sum, one needs to go to $n=10$,
i.e., consider $\lambda=\blank3\blank43\blank4444$ and $\mu=2\blank1\blank\blank\blank0\blank\blank\blank$, resulting in 24 puzzles.
\end{ex}

\begin{rmk}\label{rmk:breakties}
  There is no obligation to pick the minimal alphabets to encode $\pi$
  and $\rho$: the theorem is valid as stated for any strings
  satisfying the constraint on the number of blanks. (Moreover, the
  theorem follows straightforwardly from the case where the alphabets
  are {\em maximal}, i.e., where no number is repeated, in that the
  pullback of a Schubert class from a partial flag manifold to a
  less-partial one is again a Schubert class. This independence is not
  true of the ``motivic Segre classes'', cf Lemma~\ref{lem:pullback}.)
  In particular, we can also treat the case $\overlap(\pi,\rho)=0$ by
  e.g.\ adding $\max D(\rho)$ to $\pi$'s alphabet, or $\min D(\pi)$ to
  $\rho$'s. 
\end{rmk}

\begin{rmk}\label{rmk:sepdescduality}
  There is an evident duality on these puzzles, flipping them left-right
  while taking $i \mapsto d-i$. This is of course a shadow of
  Grassmannian duality $(F_0 \leq F_1 \leq \ldots \leq F_d)
  \mapsto (F_d^\perp \leq F_{d-1}^\perp \leq \ldots \leq F_0^\perp)$.
\end{rmk}

\begin{ex}\label{ex:k01}
  Let $\pi=2431$, $\rho=2134$. Note $D(\pi)=\{2,3\}$, $D(\rho)=\{1\}$,
  $O(\pi,\rho)=\varnothing$, so we need to add one more nondescent to
  define our alphabets.  There are two choices ($k=0,1$) leading to
  $\lambda=3\blank21$, $3\blank2\blank$ and $\mu=\blank0\blank\blank$, $10\blank\blank$ respectively.
  For either choice, we draw a row of corresponding puzzles for double
  Grothendieck polynomials:
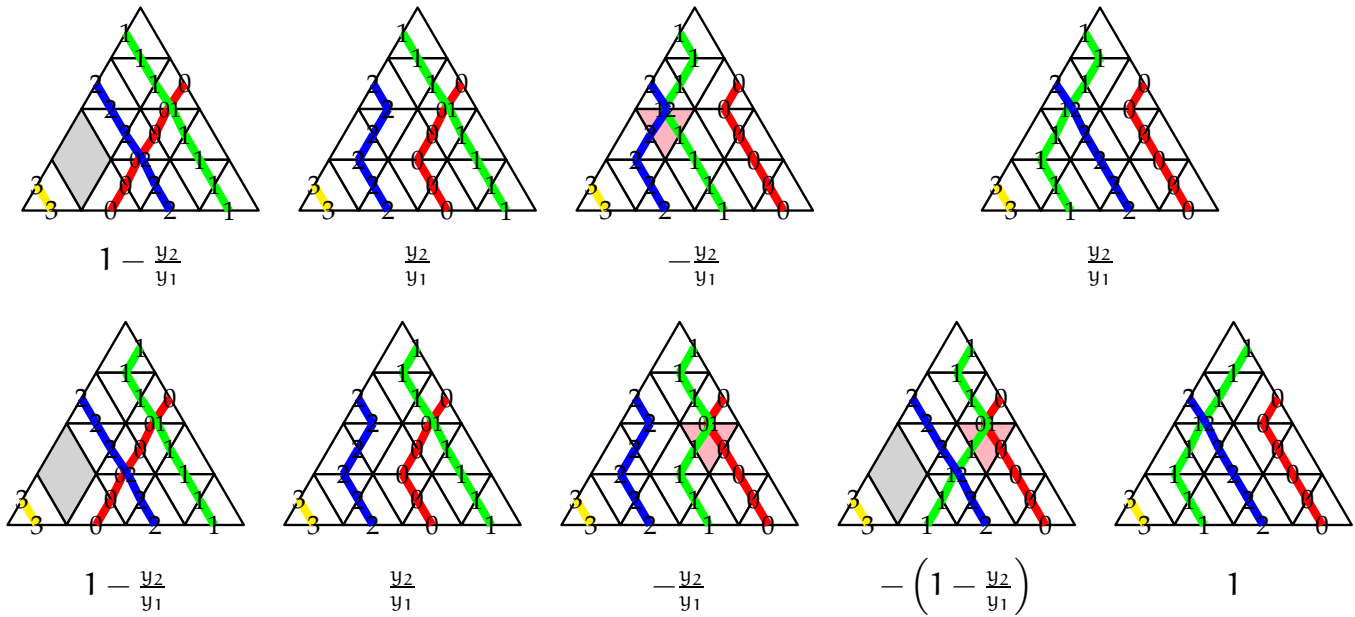
\begin{gather*}
  \tikzset{every picture/.style={scale=0.5,every node/.style={scale=0.5}}}
  \input ex2a.tex
  \\
  \tikzset{every picture/.style={scale=0.5,every node/.style={scale=0.5}}}
  { \hspace{-1cm}
    \input ex2b.tex }
\end{gather*}
(equivariant rhombi are grey, and $K$-triangles pink).

\junk{ \rem[gray]{it used to be that equivariant rhombi and K-triangles had
  inversion charges on them, but in {\it CotangentSchubert}\/ they're
  colored.  revert? AK: here the inversion charges are just $0,\pm 1$ right? Yes.
  In which case color is enough. The triangle will be too crowded with
  the inversion charge written in too.} }
Despite the different number of puzzles, we obtain the same structure constants.
For an instance of the subtlety of their computation, note the
\tikz[baseline=0,scale=.75]{
\path[draw=green,line width=.108798608162487cm] (.75,.43) -- (.5,0);
\path[draw=green,line width=.108798608162487cm] (.5,0) -- (.75,-.43);
\rh{}{j}{j}{}
\draw (0,0) -- node {$j$} (1,0);}
with $j=1$ and in position $(1,2)$ of each of the rightmost puzzles.
In the first row ($k=0$), because $j>k$ that rhombus contributes
a $y_2/y_1$ factor, but in the second row ($k=1$) because $j\not >k$
it does not.

% In this case,
One can check from the explicit expression of the
Laurent polynomials
\begin{align*}
\Gr^{2431}&=(1-x_1/y_1)(1-x_2/y_1)(1-x_3/y_1)(1-x_1x_2/(y_2y_3))
\\
\Gr^{2134}&=1-x_1/y_1
\\
\Gr^{3421}&=(1-x_1/y_1)(1-x_2/y_1)(1-x_3/y_1)(1-x_1/y_2)(1-x_2/y_2)
\\
\Gr^{4231}&=(1-x_1/y_1)(1-x_2/y_1)(1-x_3/y_1)(1-x_1/y_2)(1-x_1/y_3)
\\
\Gr^{4321}&=(1-x_1/y_1)(1-x_2/y_1)(1-x_3/y_1)(1-x_1/y_2)(1-x_2/y_2)(1-x_1/y_3)
\end{align*}
that the following identity holds:
\[
  \Gr^{2431}\Gr^{2134} = \left(1-\frac{y_2}{y_1}\right) \Gr^{2431} + \frac{y_2}{y_1}(\Gr^{3421}+\Gr^{4231}-\Gr^{4321})\ .
\]
\end{ex}
(In particular, unlike in Example~\ref{ex:sepdesc}, these puzzles happen to already be large
enough to compute the stable expansion; there aren't additional terms
from $\mathcal S_n, n\geq 5$.)

It can happen that multiple puzzle rules cover the same problem, giving
different formul\ae. For the two versions above of the computation,
the $\desc(\pi,\rho)=4$ rule from \cite{artic80} and $\desc(\pi,\rho)=3$ rule from \cite{artic71}
(respectively) would serve. However, those do not give $K_T$-positive rules,
as we enjoy here.

% \rem[gray]{note that both of these examples have alternative puzzles -- $d=4$ and $d=3$ respectively -- but neither give positive rules (latter being $K_T$)}

\subsection{Almost separated-descent puzzles}\label{ssec:almostsepdesc}
Now assume $\overlap(\pi,\rho)=2$, which means that
$O(\pi,\rho)=\{r,s\}$ with $r=\min D(\pi)$ and $s=\max D(\rho)$.
Let $N=D(\pi)\cup\{s\}$ and $N'=D(\rho)\cup \{r\}$ (i.e., if $s\not\in D(\pi)$, we add a gratuitous nondescent at $s$, and similarly
for $D(\rho)$). If $\#N=d-k+1$ and $\#N'=k+1$, we choose the alphabets
\begin{equation} \label{eq:AB2}
\begin{split}
  \Al_1&=\{0<\cdots<k<\blank\}
  \\
  \Al_2&=\{\blank<k<\cdots<d\}
\end{split}
\end{equation}
to describe $\rho$ as a string $\lambda\in \Al_1^n$ and $\pi$ as a string
$\mu\in \Al_2^n$.
The letter $k$ occurs exactly $s-r$ times in
both $\lambda$ and $\mu$, and that the total number of blanks in
$\lambda$ and $\mu$ is $n-(s-r)$. {\em (Warning:} note the switch of $\pi$ and $\rho$
-- these puzzles will
have small numbers on the NW side and large numbers on the NE side,
contrary to the situation in separated-descent. So even though
separated-descent Schubert problems will be calculable with
almost-separated-descent puzzles, it will not be easy to compare/biject the
two resulting kinds of puzzles; compare examples
\ref{ex:sepdesc} and \ref{ex:ex1redux}.)

Here is a practical test for when $(\pi,\rho)$ form an
almost-separated-descent pair, and what strings to associate to them.
\begin{itemize}
\item Invert $\rho$ and $\pi$ to form the initial $\lambda,\mu$.
\item Find a pair $j \leq k \in [n]$ such that the values $1\ldots j$
  occur in order
  in $\mu$, that $k+1\ldots n$ occur in order in $\lambda$,
  and $j+1\ldots k$ occur in order in both.
  (In separated-descent we have $j=k$, and the third condition is empty.)
  There may be multiple pairs $j\leq k$ for which this is satisfied.
  (If there are none, we are not in the almost-separated-descent case.)
  Erase $k+1\ldots n$ from $\lambda$ and $1\ldots j$ from $\mu$,
  {\bf which is backwards from separated-descent,}
  leaving blanks $\blank$ in their place.
\item If some numbers in the resulting strings occur in order ($3$ left of
  $4$ left of $5$, say), they can safely be replaced by a single letter
  (all $3$s, say). For example, this was required to be the case for $[j+1,k]$,
  which we do indeed replace with all $k$s.
\end{itemize}
These will be the resulting strings $\lambda,\mu$ in the theorem below.

\newcommand\da\searrow
\newcommand\ua\nearrow

We also need the third alphabet
\begin{equation}\label{eq:C2}
  \Al_3=\{\da0<\cdots<\da k-1 < odd <\ua k+1<\cdots<\ua d \}
\ .
\end{equation}
\junk{PZJ: there's really no reason to leave out $\da k$ or $\ua k$ --
  but not both (cf 10s or 21s at the bottom for $d=2$).  though here
  it doesn't lead to solving a more general problem -- it's equivalent
  to a relabeling. mention? AK: is it {\em really} about $k$ not $\blank$ ?}

\begin{thm}\label{thm:almostsepdesc}
Given strings $\lambda\in \Al_1^n$ and $\mu\in \Al_2^n$ in the alphabets \eqref{eq:AB2}, such that
$k$ occurs with the same multiplicity $m$ in $\lambda$ and $\mu$, and
the combined number of blanks in $\lambda$ and $\mu$ is $n-m$,
write $\rho=f_{\Al_1}(\lambda)$ and $\pi=f_{\Al_2}(\mu)$; then for any $\sigma\in\mathcal S_n$,
$c^{\pi\rho}_\sigma$ is the number of ``puzzles'' with boundaries $\lambda,\mu,\nu$ such that $\sigma=f_{\Al_3}(\nu)$
(with $\Al_3$ given by \eqref{eq:C2}).

The puzzle labels on the diagonal sides of pieces are by
subsets $X \subseteq \{0,\ldots,d\}$. On the horizontal
sides they are $\ua i$ or $\da i$ with $i$ a single number,
or the words ``$even$'' or ``$odd$''. By ``$i<X$'' we mean each $x\in X$
has $i<x$, and ``$X<j$'' similarly. The puzzle pieces are
$$
 \begin{matrix}
   \uptri {iX} {\da i} {X} %\quad   \downtri {iX} {\da i} {X} 
&\qquad
&  \uptri {X} {\ua j} {Xj} %\quad   \downtri {X} {\ua j} {Xj} 
&\qquad
&  \uptri {\varnothing} {even} {\varnothing} 
%&   \downtri {\varnothing} {even} {\varnothing}
&   \qquad \uptri {i} {odd} {i} 
%&  \downtri {i} {odd} {i}
  \\ \\
 i<X &&  X<j 
\end{matrix}
$$
and their $180^\circ$ rotations with arrows inverted.
\junk{AK: Sadly, the $180^\circ$ rotation of $\ua$ isn't $\da$ ... should we make
  them just $\uparrow$ and $\downarrow$? that wouldn't help, it's the direction of the arrow that needs reversing,
see correction above}
\junk{
\begin{align*}
&  \uptri {iX} {\da i} {X} \qquad   \downtri {iX} {\da i} {X} \qquad i<X
\\
&  \uptri {X} {\ua j} {Xj} \qquad   \downtri {X} {\ua j} {Xj} \qquad X<j
\\
&  \uptri {\varnothing} {even} {\varnothing} \qquad   \downtri {\varnothing} {even} {\varnothing}
\\
&  \uptri {i} {odd} {i} \qquad  \downtri {i} {odd} {i}
\end{align*}
} %end junk
For any such puzzle, $\nu\in \Al_3^n$. Its content is obtained by combining
the contents of $\lambda$ and $\mu$, deleting all blanks and all occurrences
of $k$, replacing each $i<k$ by $\da i$ and each $j>k$ by $\ua j$, and
inserting $m$ copies of $odd$.

To compute structure constants for (single) Grothendieck polynomials,
allow the following pieces (which include the ones above):
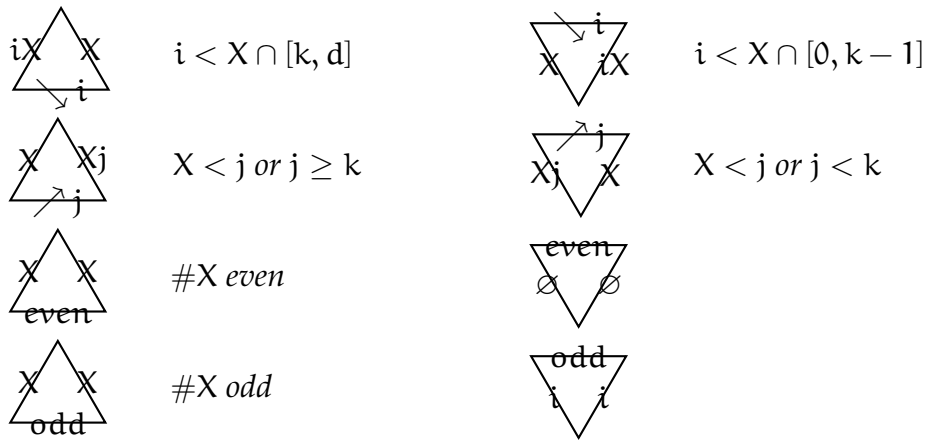
\begin{align*}
&  \uptri {iX} {\da i} {X} \qquad i< X \cap [k,d] && \downtri {iX} {\da i} {X} \qquad i < X \cap [0,k-1]
\\
&  \uptri {X} {\ua j} {Xj} \qquad X<j \text{ or } j\ge k && \downtri {X} {\ua j} {Xj} \qquad X<j \text{ or } j< k
\\
&  \uptri {X} {even} {X} \qquad \#X \text{ even} &&   \downtri {\varnothing} {even} {\varnothing}
\\
&  \uptri {X} {odd} {X} \qquad \#X \text{ odd} && \downtri {i} {odd} {i}\ .
\end{align*}
Then
$c^{\pi\rho}_\sigma=(-1)^{\ell(\pi)+\ell(\rho)-\ell(\sigma)} \#\{\text{such puzzles}\}$.
\end{thm}
Once again, the sign can be distributed over the triangles, according to the following rule: define the {\em inversion charge}\/ $\inv$ of a triangle
to be
\begin{align}\notag
\inv(\uptri {iX} {\da i} {X})&=\inv(\downtri {iX} {\da i} {X})=\#\{x\in X : x<i\}
\\\label{eq:inv}
\inv(\uptri {X} {\ua j} {Xj})&=\inv(\downtri {X} {\ua j} {Xj})=\#\{x\in X: x>j\}
\\\notag
\inv(\uptri {X} {\sss \parity(\# X)} {X})&=\inv(\downtri {X} {\sss \parity(\# X)\atop} {X})=\lfloor \# X/2\rfloor
\ .
\end{align}
Then $\ell(\sigma)-\ell(\pi)-\ell(\rho)$ equals the sum of inversion charges of the triangles of every puzzle
that contributes to $c^{\pi\rho}_\sigma$; so that assigning to each triangle $(-1)^{\inv}$ produces
the desired overall sign.

\junk{there's 3 notations for the last letter of the alphabet: $a,m,d$. rationalize}

We emphasize, with regret, that we {\em do not} have a version of Theorem~\ref{thm:almostsepdesc} for double Schubert polynomials (much less
double Grothendieck), nor do we expect one.
\junk{ This is related to our failure to obtain an equivariant puzzle rule
  for $3$-step Schubert calculus in \cite{artic71}, and we comment on
  it in \S\ref{ssec:almostsepdescR}.}

\begin{ex}
% If $\pi=243165$ and $\rho=432156$, then $O(\pi,\rho)=\{2,3\}$:
% \[
% \begin{tikzpicture}[xscale=1.05,yscale=.5]
% \node at (-1,4) {$\pi$}; \foreach \s [count=\x] in {2,4,3,1,6,5} \node at (\x,4) {$\s$};
% \node at (-1,3) {$\omega_1$}; \foreach \s [count=\x] in {\blank,\blank,2,3,3,4} \node at (\x,3) {$\vphantom{0}\s$};
% \node at (-1,2) {$\omega_3$}; \foreach \s [count=\x] in {\llap{$\da$} 0,\llap{$\da$} 1,odd,\llap{$\ua$} 3,\llap{$\ua$} 3,\llap{$\ua$}4} \node at (\x,2) {$\s$};
% \node at (-1,1) {$\omega_2$}; \foreach \s [count=\x] in {0,1,2,\blank,\blank,\blank} \node at (\x,1) {$\vphantom{0}\s$};
% \node at (-1,0) {$\rho$}; \foreach \s [count=\x] in {4,3,2,1,5,6} \node at (\x,0) {$\s$};
% %\foreach \s [count=\x] in {1,3,6,2,5,4,7} \draw[-latex,thick] (\x,-0.8) -- (\s,-2.2);
% %\foreach \s [count=\x] in {1,3,6,2,5,4,7} \draw[-latex,thick] (\x,-0.8) .. controls (0.85*\x+0.15*\s,-1.5) and (0.85*\s+0.15*\x,-1.5) .. (\s,-2.2);
% %\node at (-1,-3) {$\lambda$}; \foreach \s [count=\x] in {0,1,0,2,1,0,2} \node at (\x,-3) {$\s$};
% \draw (3.5,4.5) -- (3.5,-0.5);
% \draw (5.5,4.5) -- (5.5,1.5);
% \draw (1.5,2.5) -- (1.5,-0.5);
% \draw (2.5,4.5) -- (2.5,-0.5);
% \end{tikzpicture}
% \]
% from which we derive $\lambda=3\blank2\blank43$ and $\mu=\blank210\blank\blank$. The puzzles for Schubert polynomials are:
% \begin{center}
% \tikzset{every picture/.style={scale=0.5,every node/.style={scale=0.5}}}
% \input ex3.tex
% \end{center}
% \rem{blah blah}
If $\pi=2543167$ and $\rho=4132567$, then $O(\pi,\rho)=\{2,3\}$:
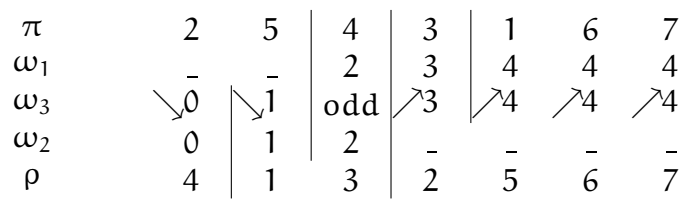
\[
\begin{tikzpicture}[xscale=1.05,yscale=.5]
\node at (-1,4) {$\pi$}; \foreach \s [count=\x] in {2,5,4,3,1,6,7} \node at (\x,4) {$\s$};
\node at (-1,3) {$\omega_1$}; \foreach \s [count=\x] in {\blank,\blank,2,3,4,4,4} \node at (\x,3) {$\vphantom{0}\s$};
\node at (-1,2) {$\omega_3$}; \foreach \s [count=\x] in {\llap{$\da$} 0,\llap{$\da$} 1,odd,\llap{$\ua$} 3,\llap{$\ua$} 4,\llap{$\ua$}4,\llap{$\ua$}4} \node at (\x,2) {$\s$};
\node at (-1,1) {$\omega_2$}; \foreach \s [count=\x] in {0,1,2,\blank,\blank,\blank,\blank} \node at (\x,1) {$\vphantom{0}\s$};
\node at (-1,0) {$\rho$}; \foreach \s [count=\x] in {4,1,3,2,5,6,7} \node at (\x,0) {$\s$};
%\foreach \s [count=\x] in {1,3,6,2,5,4,7} \draw[-latex,thick] (\x,-0.8) -- (\s,-2.2);
%\foreach \s [count=\x] in {1,3,6,2,5,4,7} \draw[-latex,thick] (\x,-0.8) .. controls (0.85*\x+0.15*\s,-1.5) and (0.85*\s+0.15*\x,-1.5) .. (\s,-2.2);
%\node at (-1,-3) {$\lambda$}; \foreach \s [count=\x] in {0,1,0,2,1,0,2} \node at (\x,-3) {$\s$};
\draw (3.5,4.5) -- (3.5,-0.5);
\draw (4.5,4.5) -- (4.5,1.5);
\draw (1.5,2.5) -- (1.5,-0.5);
\draw (2.5,4.5) -- (2.5,0.5);
\end{tikzpicture}
\]
from which we derive $\lambda=1\blank20\blank\blank\blank$ and $\mu=4\blank32\blank44$.
The seven puzzles computing the product of the Schubert polynomials are these:
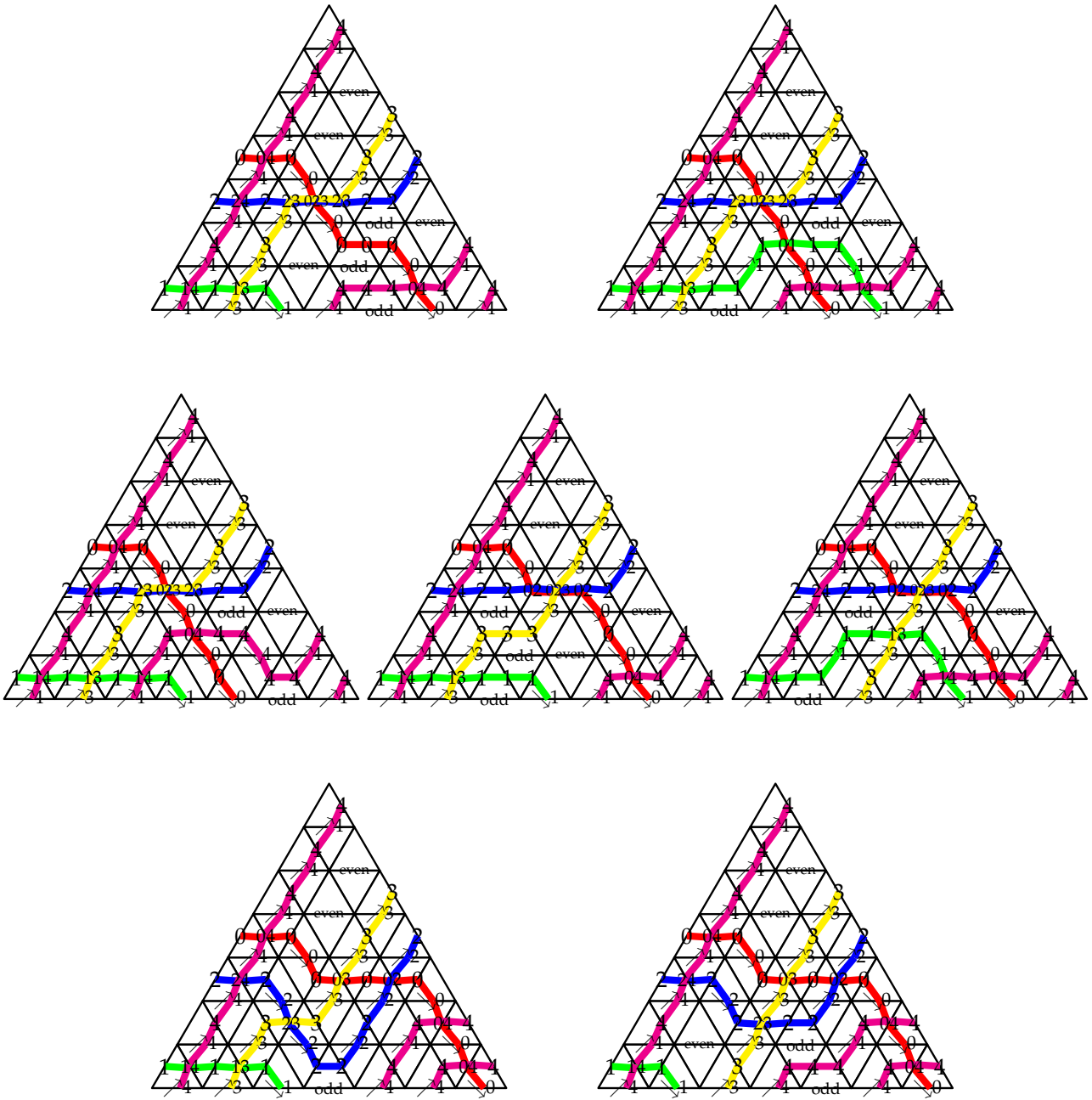
\begin{center}
\tikzset{every picture/.style={scale=0.5,every node/.style={scale=0.5}}}
\input ex3.tex
\end{center}
This leads to the following identity:
\[
\Sc^{2543167}\Sc^{4132567}
\ =\ 
\Sc^{6352147}+\Sc^{5632147}+\Sc^{5462137}+\Sc^{6432157}+\Sc^{6523147}+\Sc^{7342156}+\Sc^{7253146}
\ .
\]
\end{ex}

\begin{rmk}\label{rmk:almostsepdescduality}
  There is an evident duality on the $H^*$ puzzles that compute
  almost-separated descent; it flips them left-right while taking
  $i \mapsto d-i$, and reversing the arrow.
  Again, this is a shadow of Grassmannian duality.
  However, this is not a symmetry of the $K$-theoretic puzzle pieces
  computing almost-separated descent (see \S \ref{ssec:almostvs2});
  indeed, one can use it to obtain a second, distinct, rule for
  $K$-theoretic almost-separated descent Schubert problems.
\end{rmk}

% \rem{maybe a remark that $o=3$ should be the $E$ series? or not}

\subsection{Segre motivic classes}\label{ssec:motivic}
In the geometric framework of \cite{artic71}, (double) Schubert and
Grothendieck polynomials represent classes in the appropriate
cohomology rings of flag varieties.  However, the quantum
integrability pointed to a $q$-deformation of Schubert classes, and in
\cite{artic80} we showed that there exist puzzle formul\ae\ for
multiplying the resulting {\em Segre motivic classes}.

\junk{In {\em cotangent Schubert calculus} \rem{ref}, an important
  generalisation is to consider so-called of Schubert cells; these
  classes can be viewed (roughly) as a one-parameter deformation of
  Schubert classes.}
In fact, our Theorems~\ref{thm:sepdesc} and \ref{thm:almostsepdesc}
both follow from more general theorems which can be loosely stated as
\begin{thm}\label{thm:loose}
  Drop from Theorems~\ref{thm:sepdesc} and \ref{thm:almostsepdesc}
  the conditions like $i<j$ or $i\leq k<j$ in the definition of puzzle piece,
  thereby allowing more general pieces. Continue to require, as in those
  theorems, that one's permutations $\pi$, $\rho$ have $\overlap(\pi,\rho) = 1,2$
  (for Theorems~\ref{thm:sepdesc},~\ref{thm:almostsepdesc} respectively).

  Then for suitable ``fugacities'' on the pieces, one obtains puzzle rules
  for multiplying the equivariant Segre motivic classes of the
  corresponding Schubert cells $X^\pi_\circ,X^\rho_\circ$ in appropriate
  partial flag varieties.
\end{thm}
\junk{PZJ: very vague. AK: how's the above now. PZJ: OK}

To recover Theorems~\ref{thm:sepdesc} and \ref{thm:almostsepdesc} will
require taking $q\to 0$, which will work well in separated-descent but
will introduce infinities in almost-separated-descent, unless we first
give up $T$-equivariance. As long as we stick to finite $q$, though,
we get an equivariant rule even for the $\overlap(\pi,\rho)=2$
almost-separated-descent case.

\junk{AK: I was going to say ``has divergences'' but decided that was
  too physicsy}

\rem[gray]{AK: maybe a comment about the dual basis for $K$-theory? PZJ:
  haven't checked but should work, not sure it's worth mentioning.}

\subsection{Plan of the paper}
The details of Theorem~\ref{thm:loose} will be made more precise in
what follows: the necessary setup will be described in \S\ref{sec:setup},
see in particular Theorem~\ref{thm:motivic} (and an application
to the computation of Euler characteristics of triple intersections,
Theorem~\ref{thm:euler}).
We will then realize
this setup in two cases: separated descents in \S\ref{sec:sepdesc} and
almost-separated descents in \S\ref{sec:almostsepdesc}.
In each case, we will show how taking the quantum parameter $q$ to $0$ in
Theorem~\ref{thm:loose} reproduces Theorems~\ref{thm:sepdesc}
(see \S\ref{ssec:sepdescq0} and \S\ref{ssec:sepdescq0T}) and \ref{thm:almostsepdesc}
(see \S\ref{ssec:almostsepdescq0}), respectively.

\junk{both consequences of: main theorem 3 about motivic Chern classes. though in intro we may not give details
  of thm 3 cause requires R-matrix etc not to mention the whole geometric setup}

\subsection{Relations to prior work}
The two special cases of Theorem~\ref{thm:sepdesc} that fall within
the ``Schur times Schubert'' (i.e. $\#D(\pi)=1$ or $\#D(\rho)=1$) problem
were first handled in \cite{Kogan} using a bijection on pipe dreams.
In \cite{KYong} the $\#D(\pi)=1$ case was given a streamlined proof
using ``truncation'', and extended to $K$-theory; that technique was
then rediscovered in \cite{Assaf}.  However, the truncation approach
doesn't seem to produce an equivariant formula.

After we announced Theorem~\ref{thm:sepdesc}, an alternate formula
(in nonequivariant $H^*$ only) counting certain tableaux was given in \cite{Huang}.
Huang and Gao have since\footnote{personal communication} found a
correspondence between Huang's tableaux and our separated-descent puzzles.

We include another approach to solving %(nonequivariant)
separated-descent problems positively,
making use of the ring endomorphism from \cite[\S 4.3]{BergeronSottile}.
Motivated by an embedding $Flags(\CC^n) \into Flags(\CC^{n+1})$,
those authors define a ring endomorphism
$$
  \Psi_i\colon \ZZ[x_1,x_2,\ldots] \to \ZZ[x_1,x_2,\ldots],
                                          \qquad\qquad
  x_j \mapsto
                \begin{cases}
                  x_j &\text{if }j<i\ , \\
                  0 &\text{if }i=j \ , \\
                  x_{j-1} &\text{if }j>i \ .
                \end{cases}
$$
It is easy to prove that
\begin{alignat*}{3}
  \Psi_i(S_\pi) &= S_\pi &&\qquad\text{if $\pi$'s last descent is $<i$} \\
  \Psi_i(S_{I_1\oplus\rho}) = \Psi_1(S_{I_1\oplus\rho}) &=
   S_\rho &&\qquad\text{if $\rho$'s first descent is $\geq i$}                  
\end{alignat*}
(where $\oplus$ is to be interpreted on the permutation matrices).
Hence, for $\pi,\rho\in S_n$ and $\pi$'s last descent $\leq k \leq $
$\rho$'s first descent,
$$ S_\pi S_\rho 
= (\Psi_{k+1})^n \big(S_\pi\big)\ (\Psi_{k+1})^n\big(S_{I_n\oplus\rho}\big)
= (\Psi_{k+1})^n\big(S_\pi S_{I_n\oplus\rho}\big)
= (\Psi_{k+1})^n\big(S_{\pi\oplus\rho}\big)
$$
where the latter can be computed positively (in nonequivariant $H^*$) from the
Schubert-positive formula for $\Psi_i(S_\sigma)$ given in 
\cite[Thm.~4.2.4(i)]{BergeronSottile}.

\section{Setup of proofs}\label{sec:setup}
\subsection{Tensor calculus}\label{ssec:tensor}
We use the same tensor calculus as in \cite{artic71,artic80}. Starting
with a quantized loop algebra\footnote{Note that this $\lie g$ is not
  the Lie algebra of the group $GL(n)$ acting on our flag manifolds.
  Rather, as explained in \cite{artic80}, (the cotangent bundles of)
  our flag manifolds are quiver varieties for the Dynkin diagram of $\lie g$
  for particular choices of dimension vectors.}
$\gqg$, we consider three families of
representations $V_a(z)$, where $a=1,2,3$ and $z$ is a formal parameter;
an integer $\alpha\in\ZZ$ is also specified. We then require the existence of intertwiners
\begin{equation}\label{eq:RUD}
\begin{array}{rrcl}
 {\check R}_{a,b}(z',z''): &V_a(z')\otimes V_b(z'') &\to& V_b(z'')\otimes V_a(z') \\
  U(z): &V_1(q^\alpha z)\otimes V_2(q^{-\alpha}z) &\to& V_3(z) \\
  D(z): &V_3(z) &\to& V_2(q^{-\alpha}z)\otimes V_1(q^\alpha z) 
\end{array}
\end{equation}
which are necessarily unique up to normalization (to be specified in each case below).
%and can be made to depend on $z=z''/z'$.
% Later (in \S\ref{ssec:genthm}) we will also need to pick certain subspaces $V^A_a(z) \leq V_a(z)$,
% and ordered bases thereof. \rem{PZJ: I'm not sure why we need this sentence
%   here. AK: what {\em is} a puzzle rule? I'd say it's the puzzle pieces,
%   fugacities, and the ``single-number sector''. That is to say, I think the
%   $V_a^A(z)$ subspaces are on the same level of importance as the rest of
%   the data required here.
%   I like being able to say in \S\ref{ssec:sepdescdata},
%    \S\ref{ssec:almostsepdescdata} ``The data from \S\ref{ssec:tensor}''
%    when introducing all the data.
% I agree that to state e.g. (\ref{eq:Rfac}) we
%   don't yet need to mention these subspaces, nor properties (a)-(d).}
The three intertwiners are related by the factorization property
\begin{equation}\label{eq:Rfac}
  \check R_{1,2}(q^\alpha z,q^{-\alpha}z) = D(z) U(z)
\ .\end{equation}

Combining these basic intertwiners allows one to build many intertwiners
acting on tensor products of the form $\bigotimes_i V_{a_i}(z_i)$.
Because for generic $z_i$ these tensor products are irreducible 
\cite{Chari-braid}, Schur's lemma leads to various identities
satisfied by these intertwiners, among which are the {\em Yang--Baxter}\/
equation for the $\check R$ matrices, and other similar relations
involving $\check R_{a,b}$, $U$ and $D$.  These identities are best
described diagrammatically, and we shall not repeat them here,
referring to \cite[Prop.~5]{artic80} for details.

\subsection{Puzzles}
We are particularly interested in $\check R_{1,2}$, for which we will use the graphical depiction
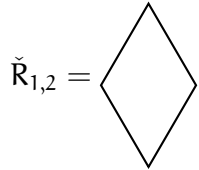
\[
\check R_{1,2} = \rh{}{}{}{}
\ .
\]
Similarly, the $U$ and $D$ matrices will be drawn as
$U=\uptri{}{}{}$, $D=\downtri{}{}{}$. Time flows downwards on our diagrams, that is, the source (resp.\ the target)
of a linear map is represented by the boundary edges at the top (resp.\ bottom) of the picture.

\junk{AK: presumably the next sentence should be ``Consider a tensor
  $\bf P$'' ? And why mention $D$, as it's not used. PZJ: fixed}
Consider a tensor $\mathbf P$ built out of $\check R_{1,2}$ and $U$, forming
an equilateral triangle of size $n$, e.g., at $n=4$,
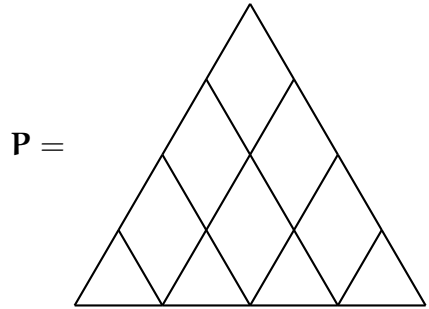
\[
\mathbf P=
\begin{tikzpicture}[math mode,nodes={edgelabel},x={(-0.577cm,-1cm)},y={(0.577cm,-1cm)},baseline=(current  bounding  box.center)]
\draw[thick] (0,0) -- node[pos=0.5] {} ++(0,1); \draw[thick] (0,0) -- node[pos=0.5] {} ++(1,0);
\draw[thick] (0,1) -- node[pos=0.5] {} ++(0,1); \draw[thick] (0,1) -- node[pos=0.5] {} ++(1,0); 
\draw[thick] (0,2) -- node[pos=0.5] {} ++(0,1); \draw[thick] (0,2) -- node[pos=0.5] {} ++(1,0); 
\draw[thick] (0,3) -- node[pos=0.5] {} ++(0,1); \draw[thick] (0,3) -- node[pos=0.5] {} ++(1,0); \draw[thick] (0+1,3) -- node {} ++(-1,1); 
\draw[thick] (1,0) -- node[pos=0.5] {} ++(0,1); \draw[thick] (1,0) -- node[pos=0.5] {} ++(1,0);
\draw[thick] (1,1) -- node[pos=0.5] {} ++(0,1); \draw[thick] (1,1) -- node[pos=0.5] {} ++(1,0);
\draw[thick] (1,2) -- node[pos=0.5] {} ++(0,1); \draw[thick] (1,2) -- node[pos=0.5] {} ++(1,0); \draw[thick] (1+1,2) -- node {} ++(-1,1); 
\draw[thick] (2,0) -- node[pos=0.5] {} ++(0,1); \draw[thick] (2,0) -- node[pos=0.5] {} ++(1,0);
\draw[thick] (2,1) -- node[pos=0.5] {} ++(0,1); \draw[thick] (2,1) -- node[pos=0.5] {} ++(1,0); \draw[thick] (2+1,1) -- node {} ++(-1,1); 
\draw[thick] (3,0) -- node[pos=0.5] {} ++(0,1); \draw[thick] (3,0) -- node[pos=0.5] {} ++(1,0); \draw[thick] (3+1,0) -- node {} ++(-1,1); 
\end{tikzpicture}
\]
%It is
intertwining $V_1(q^\alpha z_1)\otimes\cdots\otimes V_1(q^\alpha z_n)\otimes V_2(q^{-\alpha} z_1)\otimes\cdots\otimes V_2(q^{-\alpha} z_n) \to V_3(z_1)\otimes\cdots\otimes V_3(z_n)$.

If we now fix appropriate bases of the various spaces $V_a(z)$, we can
expand $\mathbf P$ by picking a basis vector for each edge of the
diagram; we represent this by marking the edges with the labels of the
chosen basis elements, e.g., if the labels are $0, 1, 10$, then one of
the summands of $\mathbf P$ is pictured as
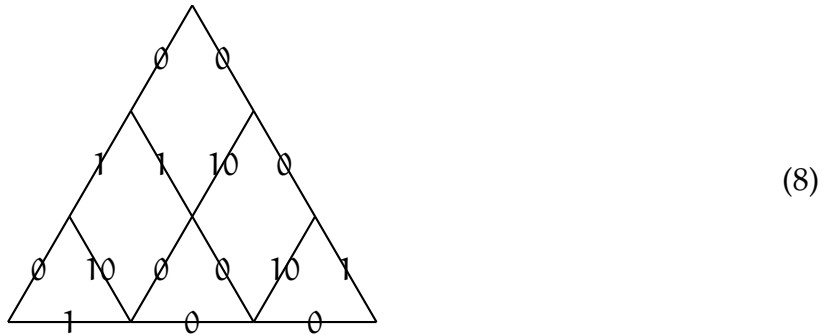
\junk{AK: shouldn't the example include an eqvt piece?  PZJ: oops,
  this should be an equivariant puzzle, fixed.  this issue reappears
  later in the proof of prop 1 and 2 where I draw what are really
  equivariant puzzles as nonequivariant ones: in a sense it's an abuse
  of notation to draw an equivariant puzzle as nonequivariant even if
  one can fill in the horizontal edges.  I tried explaining below}
\begin{equation}\label{eq:exa}
\def\thescale{1.4}
\def\posa{0.5}\def\posb{0.5}
\begin{tikzpicture}[baseline=(current  bounding  box.center),math mode,nodes={edgelabel},x={(-0.577cm,-1cm)},y={(0.577cm,-1cm)},scale=\thescale]\useasboundingbox (0,0) -- (3+0.5,0) -- (0,3+0.5);
\draw[thick] (0,0) -- node[pos=\posa] {0} ++(0,1); \draw[thick] (0,0) -- node[pos=\posb] {0} ++(1,0); %\draw[thick] (0+1,0) -- node {0} ++(-1,1); 
\draw[thick] (0,1) -- node[pos=\posa] {0} ++(0,1); \draw[thick] (0,1) -- node[pos=\posb] {10} ++(1,0); %\draw[thick] (0+1,1) -- node {0} ++(-1,1); 
\draw[thick] (0,2) -- node[pos=\posa] {1} ++(0,1); \draw[thick] (0,2) -- node[pos=\posb] {10} ++(1,0); \draw[thick] (0+1,2) -- node {0} ++(-1,1); 
\draw[thick] (1,0) -- node[pos=\posa] {1} ++(0,1); \draw[thick] (1,0) -- node[pos=\posb] {1} ++(1,0); %\draw[thick] (1+1,0) -- node {10} ++(-1,1); 
\draw[thick] (1,1) -- node[pos=\posa] {0} ++(0,1); \draw[thick] (1,1) -- node[pos=\posb] {0} ++(1,0); \draw[thick] (1+1,1) -- node {0} ++(-1,1); 
\draw[thick] (2,0) -- node[pos=\posa] {10} ++(0,1); \draw[thick] (2,0) -- node[pos=\posb] {0} ++(1,0); \draw[thick] (2+1,0) -- node {1} ++(-1,1); 
\end{tikzpicture}
\junk{
\qquad
+
\qquad
\def\thescale{1.4}
\def\posa{0.5}\def\posb{0.5}
\begin{tikzpicture}[baseline=(current  bounding  box.center),math mode,nodes={edgelabel},x={(-0.577cm,-1cm)},y={(0.577cm,-1cm)},scale=\thescale]\useasboundingbox (0,0) -- (3+0.5,0) -- (0,3+0.5);
\draw[thick] (0,0) -- node[pos=\posa] {0} ++(0,1); \draw[thick] (0,0) -- node[pos=\posb] {0} ++(1,0); %\draw[thick] (0+1,0) -- node {0} ++(-1,1); 
\draw[thick] (0,1) -- node[pos=\posa] {0} ++(0,1); \draw[thick] (0,1) -- node[pos=\posb] {0} ++(1,0); %\draw[thick] (0+1,1) -- node {0} ++(-1,1); 
\draw[thick] (0,2) -- node[pos=\posa] {1} ++(0,1); \draw[thick] (0,2) -- node[pos=\posb] {10} ++(1,0); \draw[thick] (0+1,2) -- node {0} ++(-1,1); 
\draw[thick] (1,0) -- node[pos=\posa] {0} ++(0,1); \draw[thick] (1,0) -- node[pos=\posb] {1} ++(1,0); %\draw[thick] (1+1,0) -- node {10} ++(-1,1); 
\draw[thick] (1,1) -- node[pos=\posa] {1} ++(0,1); \draw[thick] (1,1) -- node[pos=\posb] {10} ++(1,0); \draw[thick] (1+1,1) -- node {0} ++(-1,1); 
\draw[thick] (2,0) -- node[pos=\posa] {10} ++(0,1); \draw[thick] (2,0) -- node[pos=\posb] {0} ++(1,0); \draw[thick] (2+1,0) -- node {1} ++(-1,1); 
\end{tikzpicture}
}
\end{equation}
Such a picture is called an (equivariant) \defn{puzzle}. The \defn{fugacity}\/ $\fug(P)$
of a puzzle $P$ is the product of matrix entries of the individual tensors
on the diagram.  By convention, we add the requirement that puzzles
have nonzero fugacity. %; equivalently, the puzzle must be made of
%\defn{puzzle pieces} derived from the nonzero matrix entries in the
%intertwiners.\footnote{If one follows this definition to the letter,
%  one observes from the factorization (\ref{eq:Rfac}) that some
%  rhombus puzzle pieces arise from gluing of two triangular puzzle pieces.
%  \rem{PZJ: I don't understand this remark. AK: say we glue two triangular
%    puzzle pieces together making a rhombus. From this we learn that some
%    $R$-matrix entry is $\neq 0$. Following the sentence, we derive a
%    rhomboidal puzzle ``piece''. See the problem?}}
We can then state that any matrix element of the tensor
$\mathbf P$ can be written as
\[
\tikz[scale=1.8,baseline=0.5cm]{\uptri{\lambda}{\nu}{\mu}}
:=
\left< e_\nu^*, \mathbf P\,  (e_\lambda\otimes e_\mu)\right>
=\sum_{\substack{\text{puzzles $P$}\\\text{with sides $\lambda,\mu,\nu$}}} \fug(P)
\]
where $\lambda$, $\mu$, $\nu$ are three strings of length $n$ in the labels,
$e_\nu^*=\bigotimes_{k=1}^n e_{3,\nu_k}^*$,
$e_\lambda=\bigotimes_{k=1}^n e_{1,\lambda_k}$,
$e_\mu=\bigotimes_{k=1}^n e_{2,\mu_k}$,
and the $e_a$s and $e_a^*$s are basis and dual basis elements of the $V_a$s.

In what follows, we require our representations $V_a(z)$
to only have one-dimensional weight spaces\footnote{This ``quasiminuscule'' condition was violated for $\desc(\pi,\rho)=4$ in \cite{artic80},
but no such issue will arise in the present work.}, and our bases to consist of weight vectors.

If we specialize the parameters $z_1,\ldots,z_n$ to be equal, then according to factorization property 
\eqref{eq:Rfac},
matrix elements of the tensor $\mathbf P$ are now expressed as sums over {\em nonequivariant}\/ puzzles, e.g.,
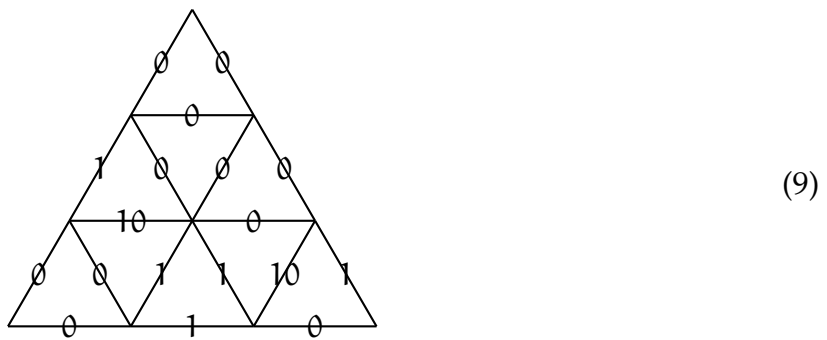
\begin{equation}\label{eq:exb}
\def\thescale{1.4}
\def\posa{0.5}\def\posb{0.5}
\begin{tikzpicture}[baseline=(current  bounding  box.center),math mode,nodes={edgelabel},x={(-0.577cm,-1cm)},y={(0.577cm,-1cm)},scale=\thescale]\useasboundingbox (0,0) -- (3+0.5,0) -- (0,3+0.5);
\draw[thick] (0,0) -- node[pos=\posa] {0} ++(0,1); \draw[thick] (0,0) -- node[pos=\posb] {0} ++(1,0); \draw[thick] (0+1,0) -- node {0} ++(-1,1); 
\draw[thick] (0,1) -- node[pos=\posa] {0} ++(0,1); \draw[thick] (0,1) -- node[pos=\posb] {0} ++(1,0); \draw[thick] (0+1,1) -- node {0} ++(-1,1); 
\draw[thick] (0,2) -- node[pos=\posa] {1} ++(0,1); \draw[thick] (0,2) -- node[pos=\posb] {10} ++(1,0); \draw[thick] (0+1,2) -- node {0} ++(-1,1); 
\draw[thick] (1,0) -- node[pos=\posa] {0} ++(0,1); \draw[thick] (1,0) -- node[pos=\posb] {1} ++(1,0); \draw[thick] (1+1,0) -- node {10} ++(-1,1); 
\draw[thick] (1,1) -- node[pos=\posa] {1} ++(0,1); \draw[thick] (1,1) -- node[pos=\posb] {1} ++(1,0); \draw[thick] (1+1,1) -- node {1} ++(-1,1); 
\draw[thick] (2,0) -- node[pos=\posa] {0} ++(0,1); \draw[thick] (2,0) -- node[pos=\posb] {0} ++(1,0); \draw[thick] (2+1,0) -- node {0} ++(-1,1); 
\end{tikzpicture}
\end{equation}

Even for equivariant puzzles, it is customary to draw and label the horizontal diagonal of rhombi if the resulting pair of triangles corresponds to
nonzero entries of $U$ and $D$ (such a label will be unique if it exists since weight spaces are one-dimensional, so no information is added);
e.g., for the Example~\eqref{eq:exa} above, one would draw
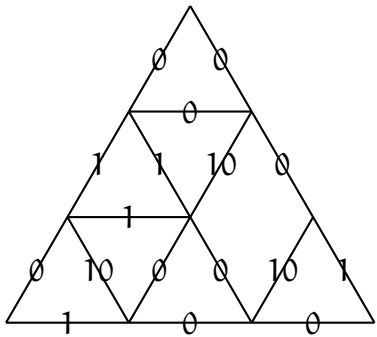
\[
\def\thescale{1.4}
\def\posa{0.5}\def\posb{0.5}
\begin{tikzpicture}[math mode,nodes={edgelabel},x={(-0.577cm,-1cm)},y={(0.577cm,-1cm)},scale=\thescale]\useasboundingbox (0,0) -- (3+0.5,0) -- (0,3+0.5);
\draw[thick] (0,0) -- node[pos=\posa] {0} ++(0,1); \draw[thick] (0,0) -- node[pos=\posb] {0} ++(1,0); \draw[thick] (0+1,0) -- node {0} ++(-1,1); 
\draw[thick] (0,1) -- node[pos=\posa] {0} ++(0,1); \draw[thick] (0,1) -- node[pos=\posb] {10} ++(1,0); % \draw[thick] (0+1,1) -- node {0} ++(-1,1); 
\draw[thick] (0,2) -- node[pos=\posa] {1} ++(0,1); \draw[thick] (0,2) -- node[pos=\posb] {10} ++(1,0); \draw[thick] (0+1,2) -- node {0} ++(-1,1); 
\draw[thick] (1,0) -- node[pos=\posa] {1} ++(0,1); \draw[thick] (1,0) -- node[pos=\posb] {1} ++(1,0); \draw[thick] (1+1,0) -- node {1} ++(-1,1); 
\draw[thick] (1,1) -- node[pos=\posa] {0} ++(0,1); \draw[thick] (1,1) -- node[pos=\posb] {0} ++(1,0); \draw[thick] (1+1,1) -- node {0} ++(-1,1); 
\draw[thick] (2,0) -- node[pos=\posa] {10} ++(0,1); \draw[thick] (2,0) -- node[pos=\posb] {0} ++(1,0); \draw[thick] (2+1,0) -- node {1} ++(-1,1); 
\end{tikzpicture}
\]
With this convention, equivariant puzzles are now made of
two types of \defn{puzzle pieces}: triangles, and rhombi that cannot be bisected in the way described above.
These remaining rhombi are the ``equivariant'' rhombi,
i.e., the ones whose contribution vanishes as one specializes the parameters
$z_1,\ldots,z_n$ to be equal. The drawback to bisecting a ``nonequivariant
rhombus'' is that its fugacity may not be the product of the
nonequivariant fugacities of the two triangles;
see e.g. the discussion just after Example~\ref{ex:k01}.

\subsection{The general theorem}\label{ssec:genthm}
Fix a Cartan subalgebra $\lie h \leq \lie g$.  (We don't use $\lie t$,
as we reserve $T$ for the group that acts geometrically,
on the partial flag manifold.)
The $\lie h$-weights in a representation $V_a(z)$ are independent of $z$,
and their convex hull is called the representation's \defn{weight polytope}.
To a face $F$ of the weight polytope, we associate a \defn{face subspace}
of $V_a(z)$, the direct sum of the weight spaces with weight in $F$.

For each $a=1,2,3$, pick a face subspace $V^A_a(z)$ of $V_a(z)$.
An easy lemma based on the $\lie h$-equivariance of $\check R_{a,a}(z',z'')$ 
shows that it
% $\check R_{a,a}(z',z'')$
sends $V_a^A(z')\otimes V_a^A(z'')$ to $V_a^A(z'')\otimes V_a^A(z')$.

In what follows we fix an \emph{ordered} basis (with conditions on the
order to come later, cf~Proposition~\ref{prop:weaktostrong}) of weight vectors of each $V_a^A(z)$,
with label set $L_a$.  The weights of $\bigotimes_{i=1}^n V^A_a(z_i)$
are naturally indexed by weakly increasing strings $\omega_a\in L_a^n$;
let us denote the corresponding weight spaces $[\bigotimes_{i=1}^n V^A_a(z_i)]_{\omega_a}$.

Pick three such weakly increasing strings $\omega_a$, $a=1,2,3$
whose corresponding weights satisfy $\omega_3 = \omega_1 + \omega_2$.
We will make use of the following crucial properties:
\begin{enumerate}[(a)]
\item $\mathbf P$ maps
  % $V^A_1(q^\alpha z_1)\otimes\cdots\otimes V^A_1(q^\alpha z_n)\otimes V^A_2(q^{-\alpha} z_1)\otimes\cdots\otimes V^A_2(q^{-\alpha} z_n)$ to $V^A_3(z_1)\otimes\cdots\otimes V^A_3(z_n)$.
  $[\Tensor_{i=1}^n V^A_1(q^\alpha z_i)]_{\omega_1}
  \tensor [\Tensor_{i=1}^n V^A_2(q^{-\alpha} z_i)]_{\omega_2}$
  to $[\Tensor_{i=1}^n V^A_3(z_i)]_{\omega_3}$. \\
  This follows straightforwardly
  from two facts: $\mathbf P$ is weight-preserving, and, if $V^F \leq V$ is
  a face subspace (where the face is given by maximizing dot product against
  some coweight $\eta$) then $(V^F)^{\tensor n} \leq V^{\tensor n}$ is again
  a face subspace (for the same $\eta$).
\item Inversely,
  $e^*_{\omega_3} \mathbf P|_{[\Tensor_{i=1}^n V^A_1(q^\alpha z_i)]_{\omega_1}
    \tensor [\Tensor_{i=1}^n V^A_2(q^{-\alpha} z_i)]_{\omega_2}}
  = e_{\omega_1}^*\otimes e_{\omega_2}^*$.
\end{enumerate}

\junk{AK: %I'm confused about (a). It says, if the NW and NE labels are from
  %the right subsets, then so too are the S labels. But that's not true in
  %sep-desc; if we don't have a full $n$ many $\blank$ labels on NW \& NE
  %together, then we'll have some forbidden $ij$ labels on the S side.
  %In any case, 
  there should probably be an ``In other words, puzzles
  such that ... are ...'' paragraph here, spelling out the details
  then to be checked individually for sep-desc and almost-sep-desc}

In puzzle terms, (a) says that if the NW, NE sides of a puzzle have
{\em contents} $\omega_1,\omega_2$, then the S side must have
{\em content} $\omega_3$, whereas (b) says that if the S side is {\em actually}
$\omega_3$ (the weakly increasing string)
then the NW, NE sides must {\em actually} be $\omega_1,\omega_2$.

\junk{
\rem{AK: at the moment the next paragraph seems just all wrong}

These are implied ??? by the following extra structure. Pick coweights
$\eta_i, i=1,2,3$ whose values on the weight polytopes of $V_i(z)$ are
maximized on the chosen faces. If $\eta_3 = \eta_1+\eta_2$, then (a), (b) follow
by weight considerations: the subspaces in (a),(b) are the face subspaces of 
$\bigotimes_{i=1}^n V^A_1(q^\alpha z_i)\otimes
 \bigotimes_{i=1}^n V^A_2(q^{-\alpha} z_i)$ and
$\bigotimes_{i=1}^n V^A_3(z_i)$ on which $\eta_3$ is maximized.

\rem{AK: The subtlety is that the while $\eta_i$ $(i=1,2)$ is constant on 
  $\bigotimes_{i=1}^n V^A_1(q^\alpha z_i)$, the pair $(\eta_1,\eta_2)$
  isn't, nor is its sum $\eta_3$. Worse, the maximum value of $\eta_3$ may
  be strictly less than the sum of the two maxima.
  Incidentally, $P$ is equivariant w.r.t. the Levi centralizing
  $\eta_1,\eta_2$, can we do anything with that?}
}

We now fix the normalization of $\check R_{a,a}(z',z'')$ by requiring that it send the highest weight vector of $V_a^A\otimes V_a^A$ to itself.
(We'll see later that this matches the natural geometric normalization of $R$-matrices.)
Denote by $\check R^A_{a,a}(z',z'')$ the restriction of
$\check R_{a,a}(z',z'')$ to $V_a^A(z')\otimes V_a^A(z'')$; and $W = \mathcal S_n$.

To a permutation $\sigma\in W$ is associated an ``$\check R$-matrix''
$\check R^A_{\sigma,a}$ from $V^A_a(z_1)\otimes\cdots\otimes V^A_a(z_n)$ to 
$V^A_a(z_{\sigma^{-1}(1)})\otimes\cdots\otimes V^A_a(z_{\sigma^{-1}(n)})$; explicitly, if $\sigma=s_{i_1}\ldots s_{i_\ell}$ where the $s_i$
are elementary transpositions, then $\check R^A_{\sigma,a}$ is the ordered product of $\check R$-matrices 
$\check R^A_{a,a}(z_{s_{i_{\ell}}\ldots s_{i_{k+1}}(i_k)},z_{s_{i_{\ell}}\ldots s_{i_{k+1}}(i_k+1)})$
acting on the $i_k^{\rm th}$ and $(i_k+1)^{\rm th}$ spaces of the tensor product.
% Diagramtically, we draw it as
% \[
% \check R^A_{\sigma,a}=
% \begin{tikzpicture}[baseline=-3pt,scale=0.9]
% \draw (0,-1) rectangle (5,1); \node at (2.5,0) {$\sigma$};
% \foreach\x/\t/\u in {1/z_1/z_{\bar\sigma(1)},2/z_2/z_{\bar\sigma(2)},3/\cdots/\cdots,4/z_n/z_{\bar\sigma(n)}}
% {
% \draw[d,invarrow=0.5] (\x,1) -- node[right,pos=0.5] {$\ss\t$} ++(0,1);
% \draw[d,invarrow=0.5] (\x,-2) -- node[right=-1mm,pos=0.4] {$\ss\u$} ++(0,1);
% }
% %\draw[decorate,decoration=brace] (4,-2.3) -- node[below] {$\m\omega$} (1,-2.3);
% %\draw[decorate,decoration=brace] (1,2.3) -- node[above] {$\m\lambda$} (4,2.3);
% \end{tikzpicture}
% \]
% \rem{actually, we may never need the diagram?}

We consider the matrix elements %
\begin{equation}\label{eq:defS}
S^\lambda_a|_{\sigma} := \left< e_\omega^*, \check R^A_{\sigma,a} e_\lambda\right>
\end{equation}
where $\lambda\in L_a^n$,
and $\omega=\sort(\lambda)$.

One of the key results of \cite{artic71} is:
\begin{thm}[{\cite[Thm.~5]{artic71}}]\label{thm:puzzle}
In the setup above, 
conditions (a) and (b) imply,
given strings $\lambda\in L_1^n$ and $\mu\in L_2^n$ with $\sort(\lambda)=\omega_1$ and $\sort(\mu)=\omega_2$,
for each $\sigma\in W$,
the following puzzle identity in $\QQ(z_1,\ldots,z_n,q)$:
\begin{equation}\label{eq:puzzle}
%    \sum_\nu \left( \sum \Bigg\{fug(P)\ :
%      \ \text{puzzles $P$ with boundary }\tikz[scale=1.8,baseline=0.5cm]{\uptri{\lambda}{\nu}{\mu}}\Bigg\} \right)
\sum_{\nu\in L_3^n} \tikz[scale=1.8,baseline=0.5cm]{\uptri{\lambda}{\nu}{\mu}}    \ S^\nu_3|_{\sigma}
\quad =\quad
S^\lambda_1|_{\sigma}\  S^\mu_2|_\sigma\ .
  \end{equation}
\end{thm}

We will need to require two more conditions in order to interpret the puzzle
identity above in the context of Schubert calculus:
\begin{enumerate}[(a)]\setcounter{enumi}{2}
\item (weak version) The successive differences of the weights in each
  chosen face should form a type $A$ root subsystem of $\lie{g}$. 
  \\[2pt]
  (strong version) $\check R^A_{a,a}(z',z'')$, in the chosen basis, coincides
  with the type $A$ $\check R$-matrix 
\begin{equation}\label{eq:Rsingle}
\check R^A(z',z'')_{ij}^{ml}=
\frac{1}{\textstyle 1-q^2 z''/z'}
\begin{cases}
 \ \, 1-q^2 z''/z'& i=j=m=l
\\
 (1-z''/z')q & i=l\ne j=m
\\
 \ \, 1-q^2 & i=m<j=l
\\
 (1-q^2)z''/z' & i=m>j=l
\\
0 & \text{else}
\end{cases}\ .
\junk{
=
\begin{cases}
1& i=j=m=l
\\
\frac{\textstyle q(1-z''/z')}{\textstyle 1-q^2 z''/z'} & i=l\ne j=m
\\
\frac{\textstyle 1-q^2}{\textstyle 1-q^2 z''/z'} & i=m<j=l
\\
\frac{\textstyle (1-q^2)z''/z'}{\textstyle 1-q^2 z''/z'} & i=m>j=l
\\
0 & \text{else}
\end{cases}
}
\end{equation}\\
In Proposition~\ref{prop:weaktostrong} we will show that under a
natural assumption on the ordering of the basis, the weak version
implies the strong version.

\item Say $\omega_a$ contains $p_{a,i}$ times the $i^{\rm th}$ label
  of $L_a$.  Let $G=GL_n(\CC)$, $B_-< G$ be the lower triangular
  matrices, $P_a\ge B_-$ the parabolic subgroups with Levi factors
  $\prod_i GL_{p_{a,i}}(\CC)$, and $\mathcal F_a = P_a \backslash G$
  the corresponding flag varieties. \\
  Then we require %The second condition can then be expressed as:\\
  $P_3\le P_1\cap P_2$, i.e., that $\mathcal F_3$ is a refinement of
  $\mathcal F_1$ and $\mathcal F_2$.
  \junk{is refinement the correct word? AK: yes}
\end{enumerate}
\junk{maybe number conditions, rather than bullets. AK: you've lettered them}

Write $W_a = W\cap P_a$. There is a $W$-equivariant bijection 
between strings of $n$ labels in $L_a$ with content $\omega_a$ and
cosets $W_a\backslash W$, which sends $\sigma\in W_a\backslash W$ to
$\lambda=(\omega_{a,\sigma(i)})_{i=1,\ldots,n}\in L_a^n$; we identify
strings and cosets via this bijection in what follows.

We are now in a position to state our most general
Theorem.  We consider the map $p=p_1\times p_2$ from $\mathcal F_3$
to $\mathcal F_1\times \mathcal F_2$ that to a flag
$F\in \mathcal F_3$ associates the pair of subflags obtained from $F$
by keeping only the parts that match the dimension vector of
$\mathcal F_1$, $\mathcal F_2$.

The Cartan torus $T=(\CC^\times)^n \leq G$ acts on each $\mathcal F_a$.
Consider the equivariant $K$-theory rings $K_{T}(\mathcal F_a)[q^\pm]$.\footnote{
%define $\hat T = T\times \CC^\times$ where the extra $\CC^\times$ acts trivially.
  In \cite{artic80}, $q$ appeared as the equivariant parameter
  associated to an extra $\CC^\times$ scaling the fiber of the
  cotangent bundles of the $\mathcal F_a$; in the present paper, which
  contains little geometry, we shall not make use of this interpretation.}
They are modules over $K_{T}(pt)[q^\pm]\cong\ZZ[z_1^\pm,\ldots,z_n^\pm,q^\pm]$.
The various $R$-matrices have poles, so we need to localize:
we choose the extended base ring $\mathcal R$ by adding to $K_{T}(pt)[q^\pm]$
the inverses of $1-q^{2k}z_i/z_j$, $k\ne0$, $1\le i,j\le n$. The identity \eqref{eq:puzzle} then takes value in $\mathcal R$.
Tensoring with $\mathcal R$ will be denoted by the superscript $\loc$.
\rem[gray]{note that the $z_i$ here are the unshifted parameters (no $q^{h/3}$ or whatever) so localisation as above is fine, I think}

According to \cite[Lemma~2]{artic80}, $S^\lambda_a|_{\sigma}$ only
depends on the class of $\sigma$ in $W_a\dom W$; we define
$S^\lambda_a \in K_{T}^{\loc}(\mathcal F_a)[q^\pm]$ by the property
that its restriction to each fixed point $\sigma\in W_{a}\dom W$ is
given by $S^\lambda_a|_\sigma$. It is known (see \cite[\S 2]{artic80})
that $q^{\ell(\lambda)}S^\lambda_a$ can be identified with the
(equivariant) motivic Segre class associated to the Schubert cell $X^\lambda_\circ$
indexed by $\lambda$ inside $\mathcal F_a$. In what follows we ignore
this power of $q$ (essentially a choice of convention)
and just call $S^\lambda_a$ the motivic Segre class.

As an immediate corollary of Theorem~\ref{thm:puzzle}, we have
the following statement,
which is a more precise version of Theorem~\ref{thm:loose} that was advertised in the introduction:
%\begin{cor}\label{cor:motivic}
\setcounter{thmd}{2}
\begin{thmd}\label{thm:motivic}
In the setup above,
%the ``product'' of two motivic Segre classes $p^*(S^\lambda\otimes S^\mu)$ in $K_{\hat T}^{\loc}(T^*(P_-\backslash G))$ is given
conditions (a)--(d) imply that motivic Segre classes
satisfy the following puzzle identity in $K_{T}^{\loc}(\mathcal F_3)[q^\pm]$:
\begin{equation}\label{eq:main}
p^*(S^\lambda_1\otimes S^\mu_2)=
%    \sum_\nu \left( \sum \Bigg\{fug(P)\ :
%      \ \text{puzzles $P$ with boundary }\tikz[scale=1.8,baseline=0.5cm]{\uptri{\lambda}{\nu}{\mu}}\Bigg\} \right)
\sum_\nu \tikz[scale=1.8,baseline=0.5cm]{\uptri{\lambda}{\nu}{\mu}}
    \ S^\nu_3\ .
  \end{equation}
  % \end{cor}
\end{thmd}
\begin{proof}
  By definition $p=(p_1\times p_2)\circ\Delta$, with obvious
  notations, so $p^*(S^\lambda_1\otimes S^\mu_2) = p_1^*(S^\lambda_1) p_2^*(S^\mu_2)$.
  Furthermore, $p$ is also compatible with restriction to fixed
  points, in the sense that $p\circ i_3 = (i_1\times i_2)\circ p$,
  where $i_a: W\to \mathcal F_a$ is the inclusion of fixed points.
  The r.h.s.\ of \eqref{eq:puzzle} is therefore
  $p^*(S^\lambda_1\otimes S^\mu_2)|_\sigma$.
  We conclude by using injectivity of the $i_a^*$.
\end{proof}

In what follows, we shall call such puzzles \defn{generic} puzzles
(corresponding to a generic value of $q$), to differentiate them from
ordinary puzzles which correspond to the limit $q\to0$.

%\rem{AK: now that this isn't a rmk the reference to it has broken}

%\begin{rmk}\label{rmk:MSpullback}
Theorem~\ref{thm:motivic}, in contrast to \cite[Thm.~1]{artic80}
which it generalizes, is {\em not}\/ a product rule for motivic
Segre classes of Schubert cells; %in the sense that the factors $p_1^*(S_1^\lambda)$
% and $p_2^*(S_2^\mu)$ in the l.h.s.\ are not themselves
rather, we have the following lemma:

\begin{lem}\label{lem:pullback}
  The pullback under $p_1: \mathcal F_3\to \mathcal F_1$ of the Segre
  motivic class of a Schubert cell is the Segre motivic class of its
  preimage; equivalently,
  \[
    p_1^*(S_1^\lambda)
    = \sum_{\mu \in W_{3}\backslash W:\ p_1(\mu)=\lambda} q^{\ell(\mu)-\ell(\lambda)} S_3^\mu\ .
  \]
\end{lem}

\begin{proof}
  Since $P_3\subseteq P_1$, the projection
  $p_1:P_3\backslash G\to P_1\backslash G$ is a locally trivial
  fibration with smooth fiber $P_3\backslash P_1$. In particular,
  $p_1$ is smooth and $T$-equivariant. It maps each Schubert cell
  $X^\mu_\circ\subset\mathcal F_3$ onto
  $X^{p_1(\mu)}_\circ\subset\mathcal F_1$, since these cells are
  orbits of the same Borel subgroup and $p_1$ maps their fixed points
  according to $W_3\backslash W\to W_1\backslash W$. Thus
  \[
    p_1^{-1}(X^\lambda_\circ)
    =\bigsqcup_{\substack{\mu\in W_3\backslash W\\p_1(\mu)=\lambda}}
      X^\mu_\circ\ .
  \]

  Motivic Segre classes commute with pullback under smooth
  equivariant maps. Indeed, \cite[Thm.~6.1]{MNS} gives the
  Verdier--Riemann--Roch formula for motivic Chern classes, whose
  relative cotangent factor cancels between the numerator and
  denominator in the Segre class. Consequently, the pullback of
  the geometric motivic Segre class of $X^\lambda_\circ$ is the
  class of the displayed preimage.

  By the convention stated before Theorem~\ref{thm:motivic}, the
  geometric motivic Segre class of $X^\nu_\circ\subset\mathcal F_a$
  is $q^{\ell(\nu)}S_a^\nu$. Additivity over the above decomposition
  therefore gives
  \[
    p_1^*\bigl(q^{\ell(\lambda)}S_1^\lambda\bigr)
    =\sum_{\substack{\mu\in W_3\backslash W\\p_1(\mu)=\lambda}}
      q^{\ell(\mu)}S_3^\mu\ .
  \]
  Dividing by $q^{\ell(\lambda)}$ proves the stated formula.
\end{proof}
%\end{rmk}

We give an application of Theorem~\ref{thm:motivic} and Lemma~\ref{lem:pullback} in the case of nonequivariant cohomology.
Define $\overleftarrow \nu := (\nu_n,\ldots,\nu_2,\nu_1)$,
and the \defn{$H$-fugacity} of a nonequivariant puzzle to be its
fugacity at the specialization $q=-1$.
\begin{thm}\label{thm:euler}
  In the same setup as Theorem~\ref{thm:motivic}, the Euler characteristic of
  $Y := g\, p_1^{-1}(X^\lambda_\circ)\cap g' p_2^{-1}(X^\mu_\circ)\cap g'' X^\nu_\circ$
  (where $g,g',g''$ are general elements of $GL(n)$) is
  $(-1)^{\dim Y}$ times the sum of $H$-fugacities of nonequivariant
  puzzles with sides $\lambda$, $\mu$, $\overleftarrow \nu$.
\end{thm}

In the two applications that we
have in mind, these $H$-fugacities will be equal to $1$, so that
$\chi(Y)$ is $(-1)^{\dim Y}$ times the number of such puzzles.

We omit the proof of Theorem~\ref{thm:euler}, which is identical to that of \cite[Thm.~3]{artic80}, the only new ingredient being
the interpretation of the pullbacks of the $S_a^\lambda$ in Lemma~\ref{lem:pullback}.

\junk{there's a slightly annoying issue with having ``blank'' labels in
  that I often draw partially filled puzzles (where the unfilled edges
  are supposed to be summed over), but there's no way to distinguish
  unfilled from blank. AK: not sure this is the place to apologize for that.
  PZJ: correct, comment to myself only}

\junk{AK: there should be a subsection here that proves the type A
  $R$-matrix property (c) under the condition that the shortest $w$ that
  moves the type A face to a top face, takes the vertex order to the
  vertex order at the top}

\junk{BTW puzzles with multinumbers at bottom also fit in this framework}

\subsection{Flag-type faces}\label{ssec:flagfaces}
Fix a generic (real) coweight $\vec c$. With it, we can choose a
Borel subalgebra $\lie b \geq \lie h$, as containing the roots that
pair positively with $\vec c$, at which point we can call $\vec c$
\defn{dominant}. We also can put a total order on the
vertices of any weight polytope (ordered by their pairing with $\vec c$).

The weight polytope of a $\lie{g}$-irrep with high weight $\lambda$ is
combinatorially determined (i.e. its face poset, but not its edge lengths)
by the set of simple roots perpendicular to $\lambda$.
Let $\Delta_1$ denote $\lie g$'s simple roots (the vertices of
$\lie g$'s Dynkin diagram), let $P_\lambda \geq B$ be the standard
parabolic generated by $B$ and the negative roots
$-\Delta_1 \cap \lambda^\perp$, and let $\supp(\lambda)$ (for
``support'') denote $\Delta_1 \setminus \lambda^\perp$, considered as
a subdiagram of $\lie g$'s Dynkin diagram.

We collect several results, surely well-known to experts, about the
combinatorics of the weight polytope $\hull(W\cdot \lambda)$. This is
mostly to fix notation, and we will only use the third statement.

\begin{prop}\label{prop:faces}
  \begin{enumerate}
  \item The vertices of $\hull(W\cdot \lambda)$ are in correspondence
    with $W/W_P$. 
  \item Each subdiagram $S$ of $\Delta_1$, having the property
    that every connected component of $S$ meets $\supp(\lambda)$,
    induces a $\#S$-dimensional face $\hull(W_S \cdot \lambda)$
    containing the basepoint $\lambda$.
  \item Each face $F$ of $\hull(W\cdot \lambda)$ is of the form
    $w \cdot W_S W_P/W_P$ for a unique such $S$,
    and there is a unique shortest such $w$.
  \end{enumerate}
\end{prop}

Say a face of a weight polytope is \defn{of flag type} if it is, on its own,
a standard simplex. In this case $S$ is automatically a type $A$ subdiagram,
and, the $w$ from Proposition~\ref{prop:faces}(3) is the unique one
that corresponds the vertices of $F$ with those of $W_S W_P/W_P$ in an
order-preserving manner. (The terminology arises from the Nakajima quiver
variety perspective. To each weight $\mu$ in the polytope, the $\mu$
weight space arises as $K(\calM)$ for $\calM$ a certain quiver variety
\cite{Nakaj-quiv3}, and $\mu$ lies on a face of flag type iff its $\calM$
is the cotangent bundle of a partial flag variety.)

\junk{

\begin{proof}
  \begin{enumerate}
  \item The $W$-stabilizer of a point $\lambda$ in the reflection
    representation is generated by the simple reflections $r_\alpha$
    for $\alpha \perp \lambda$, hence, is $W_P$.
  \item When $\lambda$ is regular (i.e. $W_P=1$ and $supp(\lambda)=\Delta_1$),
    the corner of $hull(W\cdot\lambda)$ nearby the basepoint is an
    orthant with edges in correspondence with $\Delta_1$. Hence the
    faces are in correspondence with subdiagrams of $\Delta_1$.

    When $\lambda$ is not regular, we can study its polytope as a
    degenerate limit of the regular case. 
  \end{enumerate}
\end{proof}
}

The following result should have a purely representation-theoretic proof,
and in particular seems certain to hold for non-simply-laced groups.
However, for the two examples in this paper it suffices to consider
the simply-laced case (types $A,D$ specifically). Within this case,
one can compute using quiver varieties.

\begin{prop}\label{prop:weaktostrong}
  Assume that $\lie{g}$ is simply laced, and that condition (c) holds
  in the weak version. Assume that the ordering on the basis of $V_a(z)$
  is induced by the pairing of its weights with a dominant coweight.
  Then condition (c) holds in the strong version (equation (\ref{eq:Rsingle})).
\end{prop}

\begin{proof}
  We recapitulate the proof from \cite[Prop.~11,12]{artic80}, which for
  the most part just quotes literature about Nakajima quiver varieties.
  The principal things one needs to know about such a variety
  (for this proof) is that it depends on a graph 
  (or ``simply laced Dynkin diagram'') with vertex set $I$, on two integer 
  vectors $w,v: I \to \NN$, and on a real vector $\theta: I \to \RR$;
  also, on this variety $\calM_I(w,v,\theta)$ there is an action of
  the ``flavor group'' $\prod_I GL(w^i)$.
  Write $\theta > \vec 0$ if every $\theta^i > 0$.
  \begin{enumerate}
  \item When $\lie g$ is simply laced, and its representations $V,W$ are
    tensor products of evaluation representations of $\gqg$, %$U_q(\lie g[z^\pm])$,
    then by \cite{Nakaj-quiv3} we can compute the weight spaces of $V,W$ as
    the $K$-groups of type $\lie g$ quiver varieties with $\theta > \vec 0$.
    More specifically, the weight given to $K(\calM_I(v,w,\theta))$ is
    $$ \sum_{i\in I} w^i \omega_i - \sum_{i\in I} v^i \alpha_i $$
    where $(\alpha_i,\omega_i)_{i\in I}$ denote the simple roots and
    fundamental weights of $\lie g$. 
  \item Using \cite{Ok-K}, one can compute the $\check R$-matrix from
    the equivariant geometry of those quiver varieties. This geometric construction of the
    $\check R$-matrix was shown in \cite{Negut-Kstab} to coincide with its algebraic definition in types $ADE$,
    from which we derive the explicit formul\ae of this paper, including \eqref{eq:Rsingle} in type $A$.
  \item   It turns out to be obvious from the full definition (which we will
  not recapitulate) that if $v^i=0$ for some vertex $i\in I$, deleting
  that vertex gives an isomorphic quiver variety.
  \end{enumerate}
  
  As such, it suffices to show that the type $\lie g$ quiver varieties
  computing the weight spaces in the face subspaces $V_a(z)$ are
  isomorphic to the expected type $A$ quiver varieties.

  \begin{enumerate}
    \setcounter{enumi}{3} 
  \item  \cite{Naka-Weyl} Let $\pi$ be in the Weyl group of $\lie g$,
    and define $v'$ by
    $$ \pi \cdot \left(\sum_i w^i\omega_i - \sum_i v^i\alpha_i \right)
    \ = \ \sum_i w^i\omega_i - \sum_i v'^i\alpha_i\ . $$
    Meanwhile, define $\pi\cdot\theta$ by identifying $\theta$ with
    $\sum_i \theta^i \omega_i$ (a regular dominant real weight, when
    $\theta > \vec 0$). Then the varieties
    $\calM_I(w,v',\pi\cdot\theta)$, $\calM_I(w,v,\theta)$
    are $\prod_{i\in I} GL(w^i)$-equivariantly isomorphic.
    \junk{To calculate $v'$ from $v$ when $\pi = r_{\vec\alpha_i}$, one replaces
      the $\vec\alpha_i$ label $v^i$ by the sum of all adjacent labels {\em
        including} labels on framed vertices, minus the original label $v^i$.}
  \item \cite[Lemma 3.2.3(i)]{Ginz-naka}
    If $\theta_1,\theta_2$ are two all-positive choices of real vector,
    then $\calM_I(w,v,\theta_1),\calM_I(w,v,\theta_2)$ are again
    equivariantly isomorphic.
  \end{enumerate}

  Since the weak version of (c) is assumed to hold,
  \begin{enumerate}
    \setcounter{enumi}{5} 
  \item the face $F$ of the weight polytope is a simplex,
  \item $F$ has vertices $u \cdot W_S W_P/W_P$ where $S$ is a type $A$
    subdiagram, and
  \item the shortest-length such $u$ corresponds the vertices of
    $W_S W_P/W_P$ with those of $F$ in an order-preserving manner.
  \item (One can also infer that $w^i = 0$ on each $i\in S$ except
    for one end of $S$.)
  \end{enumerate}

  One of the equivalent conditions for $u$ to be shortest in its
  $W_S\backslash W_G$ coset is that $u^{-1} \cdot \alpha_s$ is a positive root
  for each $s\in S$. 
  
  At this point our goal is to show that for $\lambda$ a weight on $F$,
  the corresponding quiver variety $\calM_I(w,v,\theta > \vec 0)$ is isomorphic
  to a type $A$ quiver variety with $\theta > \vec 0$.
  By (7), we know that $u^{-1} \cdot \lambda$ lies on a ``top'' face
  with vertices $W_S W_P/W_P$, for $S$ a type $A$ subdiagram.
  By (4), we know
  $$ \calM_I(w,v,\theta) \iso \calM_I(w,v',u^{-1}\cdot \theta) $$
  where $v'$ is supported on the subdiagram $S$. By (3), we can safely
  delete the remaining vertices. So finally, we want to check that
  $u^{-1}\cdot \theta$ is positive on $S$, i.e. that
  $$ 0 < \langle u^{-1}\cdot\theta, \alpha_s \rangle
  = \langle \theta, u\cdot\alpha_s \rangle $$
  but this inequality follows from the fact that $\theta$ is dominant and
  $u r_s > u\ \forall s\in S$. 
\end{proof}

\subsection{Schubert classes}%The limit $q\to 0$}
\label{ssec:q0}
Geometrically, Schubert and Grothendieck polynomials arise as polynomial representatives of \defn{Schubert classes},
and we'll use the same notation $\Sc$ and $\Gr$ for $H^*$- and $K$-classes of Schubert varieties, respectively.
All our classes are equivariant unless otherwise stated, i.e., related to the ``double'' versions of the polynomials.
It is well-known that Schubert classes have the same structure constants $c^{\pi\rho}_\sigma$ as the corresponding polynomials.
Either by explicit computation (as was done in \cite[\S 3.5]{artic80}) or from general principles, one can show that
motivic Segre classes and $K$-theoretic Schubert classes are closely related; namely,
\begin{equation}\label{eq:StoSc}
\Gr^\lambda = \left(\lim_{q\to 0} q^{-\ell(\lambda)} S^\lambda\right)^\vee
\end{equation}
where $\vee$ is the duality map that takes classes of vector bundles to classes of their duals.

Taking the limit $q\to0$ is done in two steps.

Firstly, one needs to {\em twist}\/ the $R$-matrices in order to absorb into them the powers of $q$ that appear
in \eqref{eq:StoSc}. The twist typically takes the form $\Omega=q^{\frac{1}{2}B}$
of the exponential of
some skew-symmetric form $B$ in the weights. At the level of puzzles,
$B$ has the simple meaning that it computes the inversion charge
of puzzle pieces: namely, define the inversion charge of a triangle to be one half
of the form $B$ applied to the weights of two successive sides in counterclockwise
order (by weight conservation, this definition does not depend on which two sides).
Similarly, define the inversion charge of a rhombus to be one half the sum of $B$ applied
to the top two edges and $B$ applied to the bottom two edges, counterclockwise.
Then the effect of the twist is simply to multiply the fugacity of each puzzle
piece by $q^{\inv}$.

It is important to extend slightly what we mean by ``weight'' at this stage. Each space $V_a(z)$ can be decomposed into weight spaces
by diagonalizing the action of the Cartan subalgebra $\lie{h}$.
All our $R$-matrices commute with the $\lie{h}$-action, and therefore preserve weight spaces.
%This Cartan action does not however generate a maximal commutative subalgebra of operators acting on tensor products of $V_a(z)$
%and commuting with the $R$-matrices (and $U$, $D$); \rem{is that sentence correct?}
It is convenient to enlarge this weight space as follows:
we assign an additional weight $y_a$ to any vector in $V_a(z)$, $a=1,2$, and extend by usual additivity to tensor products,
the existence of $U$ and $D$ implying that vectors in $V_3(z)$ must have weight $y_1+y_2$.

The skew-symmetric form $B$ then acts on this extended $(\dim\lie{h}+2)$-dimensional weight space.

% Define the operator
% \[
% \Omega_n = q^{\frac{1}{2}\sum_{1\le k< \ell\le n}B(h^{(k)},h^{(\ell)})}
% \]
% acting on $V_a^{\otimes n}$, $a=1,2,3$, where $h^{(k)}$ is the weight at the
% $k^{\rm th}$ factor of the tensor product.
Specifically,
let $\tilde S^\lambda_a$ be defined just like $S^\lambda_a$ by \eqref{eq:defS}, but with the $R$-matrix $\check R^A_{a,a}$ replaced with its 
twisted version $\Omega \check R^A_{a,a} \Omega^{-1}$. The
required property for $B$ is:
\begin{lem}[{\cite[\S3.5]{artic80}}]\label{lem:twist}
% \[
% \Omega_n e_\lambda = q^{\ell(\lambda)+D_a/2} e_\lambda
% \]
% for $e_\lambda=\bigotimes_{k=1}^n e_{a,\lambda_k}\in (V^A_a)^{\otimes n}$, $a=1,2,3$, and $D_a=\dim\mathcal F_a$.
If for any two vectors $e_{a,i}$ and $e_{a,j}$ of the ordered basis of $V^A_a$, $a=1,2,3$ one has
\[
B(\wt(e_{a,i}),\wt(e_{a,j}))=\sign(i-j)=\begin{cases}-1&i<j\\0&i=j\\1&i>j\end{cases}
\]
then $\tilde S^\lambda_a$ satisfies
\[
\tilde S^\lambda_a = q^{-\ell(\lambda)} S^\lambda_a\ .
\]
\end{lem}
If Lemma~\ref{lem:twist} applies, then
$\tilde S^\lambda$ has a limit as $q\to0$ which is related to $\Sc^\lambda$
by the simple duality of \eqref{eq:StoSc}.

By applying the twist to puzzles (i.e., to $\check R_{1,2}$, $U$, $D$),
one has a puzzle formula for the $\tilde S^\lambda$, and therefore,
at this stage, a would-be puzzle formula for $K$-theoretic Schubert classes by taking $q\to0$, except for the fact that
nothing guarantees that the fugacities of individual puzzles, and of puzzle pieces, remain finite in this limit.

Therefore, secondly, one needs to renormalize the weight vectors
(i.e., conjugate $\check R_{1,2}$, $U$ and $D$ by diagonal matrices
with powers of $q$ down the diagonal) in order to render the fugacities of all puzzle pieces finite as $q\to0$.
% \rem[gray]{I mean the non single number weight vectors -- a slightly
%   ambiguous terminalogy in the current setting, which should be
%   avoided, all NW/NE labels looking ``single number''}

Note that even in the first step, there is some freedom in the choice
of $B$ -- it is not entirely determined by the requirement that it
absorb the powers of $q$ in \eqref{eq:StoSc}. Ultimately, both steps
form an integer programming problem, so a computer can determine whether
this procedure
\begin{itemize}
\item succeeds (it will for separate descents, see
\S\ref{ssec:sepdescq0}-\ref{ssec:sepdescq0T}),
or 
\item fails entirely, as happens e.g. for 4-step
flag varieties, or
\item  fails for $\check R_{1,2}$ but works for $U$ and
$D$, so that one still has a nonequivariant rule, as was the case for
3-step flag varieties in \cite{artic71}, and likewise for almost
separated descents, as we shall see in \S\ref{ssec:almostsepdescq0}.
\end{itemize}

The same procedure can in principle be repeated for Schubert polynomials (i.e., Schubert classes in $H^*$ or $H^*_T$).
However, it is more convenient to take the limit from $K$-theory to cohomology which corresponds to substituting $y_i\mapsto 1-y_i$
and expanding at first nontrivial order in the fugacities. If all triangles have nonnegative inversion charge (as is the case in
Theorems~\ref{thm:sepdesc} and \ref{thm:almostsepdesc}), a further simplification occurs in this limit: by a simple inversion count, one sees
that triangles with nonzero inversion charge cannot occur in cohomological puzzles.

\rem[gray]{AK: are we clearly at a stage where all
  inversion charges are nonnegative? nope. this is a tricky question\dots Pretty sure $q\to0$ limit exists $\Leftrightarrow$
there are no triangles with negative inversion charge, but don't know of a simple proof}

\newcommand\Alt{\mathrm{Alt}}
\section{Separated descents}\label{sec:sepdesc}
\subsection{The data from \S\ref{ssec:tensor}}
\label{ssec:sepdescdata}
The algebra is $\Aqg$,
\junk{PZJ: you really mean $gl$, not $sl$? AK: I put $gl$ instead of $sl$
  more by reflex than anything else -- do you think it matters?
  I'd rather work in $(x_i)$ co\"ordinates than $(\alpha_i)$ co\"ordinates.
  PZJ: OK}
the representations $V_{1,2,3}$ are $\CC^{d+2}(z)$, $\CC^{d+2}(z)$, and the $q$-deformation of the exterior square
$\Alt^2 \CC^{d+2}(z)$ respectively, and the exponent $\alpha$ is $1$.
The scaling of the intertwiners will be fixed (after Proposition~\ref{prop:trivpuzzlea}) by \S\ref{ssec:genthm}'s condition (c).
\junk{PZJ: where? BTW, the single color R-matrices already needed
  normalizing for the main theorem. AK: yeah, where should we do this.
  Presumably the type $A$ normalization stuff should
  be happening at the beginning of \S\ref{ssec:flagfaces}.
  PZJ: it's already in 2.3}

Put the usual co\"ordinates on the usual (diagonal) Cartan
$\lie h \leq \lie {gl}_{d+2}$, for ease of computation (though we will
renumber below).
The weights of $V_{1,2}$ are then $\{x_i\}_{i\in [1,d+2]}$,
and of $V_3$ are $\{x_i+x_j\}_{i,j\in [1,d+2], i\neq j}$.
Choose the positive Weyl chamber using $\vec c = (d+2,d+1,\ldots,1)$. 
The $V^A_i$ faces (as required in \S\ref{ssec:genthm}),
with their vertices ordered according to $\cdot \vec c$,
are determined by maximizing dot product with the following coweights:
\begin{itemize}
\item $\eta_1 = (0^{k+1},1,1^{d-k})$,
  so $V_1^A$ has weights $\{x_i\colon i \in [k+2,d+2] \}$
\item $\eta_2 = (1^{k+1},1, 0^{d-k})$,
  so $V_2^A$ has weights $\{x_i\colon i \in [1,k+2] \}$
\item $\eta_3 = (1^{k+1},2, 1^{d-k})$,
  so $V_3^A$ has weights $\{x_i + x_{k+2}\colon i \in [1,d+2],\ i\neq {k+2} \}$
\end{itemize}

With these explicit descriptions, it is easy to check condition (c) in
the weak form, which then implies the strong form by
Proposition~\ref{prop:weaktostrong}.

In our estimation, the labeling system giving the nicest puzzles
(in a visual, not mathematically precise, sense) includes a blank label.
Renumber the $d+2$ Cartan co\"ordinates $0<\cdots<k<\blank<k+1<\cdots<d$,
at which point the weights become
$$
\begin{array}{ccc}
V_1,V_2  &\qquad & V_3  \\
\left\{ x_\blank\, \} \sqcup \{x_i\colon i\in [0,d]\right\} 
&&
\left\{ x_i + x_j \colon i,j \in \{\blank\} \sqcup [0,d],\ i<j \right\}
\end{array}
$$
which for $V_1,V_2$ match the internal labels on the diagonal edges and for
$V_3$ match those on the horizontal edges of the separated-descent puzzles, and
$$
\begin{array}{ccccc}
  V_1^A  & \qquad & V_2^A  &\qquad & V_3^A  \\
  \{x_\blank\, \} \sqcup \{x_i\colon i\in [k+1,d]\} &&
  \{x_i\colon i\in [0,k]\} \sqcup \{x_\blank\, \} &&
  \{x_i + x_{\blank} \colon i \in [0,d]\ (\text{i.e. }i\neq \blank)\}
\end{array}
$$
which are exactly the subsets of labels that we see on the NW, NE, and S sides.
(The correspondence between weights and labels will be somewhat trickier
in almost-separated-descent puzzles.)

Consider now a triple $\omega_1 + \omega_2 = \omega_3$ of weights, where
$\omega_i$ is the sum of $n$ weights from $V_i^A$.
Then the coefficient of $x_\blank$ in $\omega_3$ is $n$, 
and if we write $\omega_3 = \sum_{i=0}^k c_i x_i + \sum_{i=k+1}^d c_i x_i + nx_\blank$,
necessarily $\omega_1 = \sum_{i=k+1}^d c_i x_i + (n - \sum_{i=k+1}^d c_i) x_\blank$
and $\omega_2 = \sum_{i=0}^k c_i x_i + (n - \sum_{i=0}^k c_i) x_\blank$.
To verify property (d) we compute the three standard parabolics, each of
which is a group of block-upper-triangular matrices.
\begin{align*}
  \text{$P_1$'s blocks } 
&= \left( \sum\nolimits_{i=0}^k c_i, c_{k+1}, \ldots, c_d\right)\\
  \text{$P_2$'s blocks } 
&= \left( c_0, \ldots, c_k,  \sum\nolimits_{i=k+1}^d c_i\right)\\
  \text{$P_3$'s blocks } 
  &= \left( c_0, \ldots, c_k, c_{k+1}, \ldots, c_d \right)
      \qquad \text{so }P_3 = P_1\cap P_2\ .
\end{align*}

% end junk

\subsection{Proof of property (b)} %\label{ssec:sepdescId}
Given a string $\lambda$ in the symbols $\{0,\ldots,d\}$ define
\junk{
  $\lambda_>$ (resp.\ $\lambda_\ge$, $\lambda_\le$) to be the string
  obtained from $\lambda$ by replacing every digit $\le k$ (resp.\ $<k$,
  $> k$) with blanks.
}
$$
\begin{array}{rcccl}
  &  \lambda_> 
  & 
  & \leq k & \\
  & \lambda_\geq &:= \text{ the string obtained from $\lambda$ by replacing every digit }& <k & \text{ with blanks.}\\
  & \lambda_\leq && >k
\end{array}
$$

\begin{prop}\label{prop:trivpuzzlea}
  There is a unique puzzle
  \tikz[scale=1.8,baseline=0.5cm]{\uptri{\lambda}{\omega}{\mu}} with a
  weakly increasing string $\omega$ at the bottom, and $\lambda$ (resp.\ 
  $\mu$) with content $\omega_>$ (resp.\ $\omega_\le$);
  it has $\lambda=\omega_>$ and $\mu=\omega_\le$, and
  the labels on diagonal edges are constant along each diagonal
  (NW/SE or NE/SW).
\end{prop}

\begin{proof}
Induction on the size of the puzzle. Consider the leftmost bottom label $i$. Let's first treat the
``generic'' case where $i\le k$. Then we know that the path starting at that bottom edge
must exit the puzzle on the NE side, and it can only do NW and NE steps; therefore it must exit
at the leftmost edge of the NE side, e.g., if $i=0$,
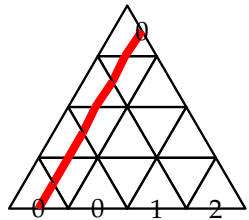
\begin{center}
\begin{tikzpicture}[scale=0.5,every node/.style={scale=0.5},execute at begin picture={\bgroup\tikzset{every path/.style={}}\clip (-2.48,-.432012) rectangle ++(4.96,4.46413);\egroup},x={(1.55426771761071cm,0cm)},y={(0cm,-1.55426394417233cm)},baseline=(current  bounding  box.center),every path/.style={draw=black,fill=none},line join=round]
\begin{scope}[every path/.append style={fill=white,line width=.0310853166178533cm}]
\path (-.5,.866025403784439) -- (0,0) -- (.5,.866025403784439) -- cycle;
\path (-.5,.866025403784439) -- (.5,.866025403784439) -- (0,1.73205080756888) -- cycle;
\path[draw=red,line width=.108798608162487cm] (.25,.433012701892219) -- (-.04,.866025403784439);
\path[draw=red,line width=.108798608162487cm] (-.25,1.29903810567666) -- (-.04,.866025403784439);
%\path[opacity=0,line width=.108798608162487cm] (.04,.866025403784439) -- (-.25,.433012701892219);
%\path[opacity=0,line width=.108798608162487cm] (.04,.866025403784439) -- (.25,1.29903810567666);
\node[black,,scale=1.66461870488604] at (.25,.433012701892219) {0};
%\node[black,,scale=1.66461870488604] at (-.25,.433012701892219) {2};
%\node[black,,scale=1.38718062209258] at (0,.866025403784439) {02};
\path (0,1.73205080756888) -- (.5,.866025403784439) -- (1,1.73205080756888) -- cycle;
\path (0,1.73205080756888) -- (1,1.73205080756888) -- (.5,2.59807621135332) -- cycle;
%\path[draw=red,line width=.108798608162487cm] (.75,1.29903810567666) -- (.46,1.73205080756888);
%\path[draw=red,line width=.108798608162487cm] (.25,2.1650635094611) -- (.46,1.73205080756888);
%\path[opacity=0,line width=.108798608162487cm] (.54,1.73205080756888) -- (.25,1.29903810567666);
%\path[opacity=0,line width=.108798608162487cm] (.54,1.73205080756888) -- (.75,2.1650635094611);
%\node[black,,scale=1.66461870488604] at (.75,1.29903810567666) {0};
%\node[black,,scale=1.66461870488604] at (.25,1.29903810567666) {2};
%\node[black,,scale=1.38718062209258] at (.5,1.73205080756888) {02};
\path (.5,2.59807621135332) -- (1,1.73205080756888) -- (1.5,2.59807621135332) -- cycle;
\path (.5,2.59807621135332) -- (1.5,2.59807621135332) -- (1,3.46410161513775) -- cycle;
%\path[opacity=0,line width=.108798608162487cm] (1,2.59807621135332) -- (.75,2.1650635094611);
%\path[opacity=0,line width=.108798608162487cm] (1,2.59807621135332) -- (1.25,3.03108891324554);
%\node[black,,scale=1.66461870488604] at (.75,2.1650635094611) {2};
%\node[black,,scale=1.66461870488604] at (1,2.59807621135332) {2};
\path (1,3.46410161513775) -- (1.5,2.59807621135332) -- (2,3.46410161513775) -- cycle;
%\path[opacity=0,line width=.108798608162487cm] (1.5,3.46410161513775) -- (1.25,3.03108891324554);
%\node[black,,scale=1.66461870488604] at (1.25,3.03108891324554) {2};
\node[black,,scale=1.66461870488604] at (1.5,3.46410161513775) {2};
\path (-1,1.73205080756888) -- (-.5,.866025403784439) -- (0,1.73205080756888) -- cycle;
\path (-1,1.73205080756888) -- (0,1.73205080756888) -- (-.5,2.59807621135332) -- cycle;
\path[draw=red,line width=.108798608162487cm] (-.25,1.29903810567666) -- (-.54,1.73205080756888);
\path[draw=red,line width=.108798608162487cm] (-.75,2.1650635094611) -- (-.54,1.73205080756888);
%\path[opacity=0,line width=.108798608162487cm] (-.46,1.73205080756888) -- (-.75,1.29903810567666);
%\path[opacity=0,line width=.108798608162487cm] (-.46,1.73205080756888) -- (-.25,2.1650635094611);
%\node[black,,scale=1.66461870488604] at (-.25,1.29903810567666) {0};
%\node[black,,scale=1.66461870488604] at (-.75,1.29903810567666) {1};
%\node[black,,scale=1.38718062209258] at (-.5,1.73205080756888) {01};
\path (-.5,2.59807621135332) -- (0,1.73205080756888) -- (.5,2.59807621135332) -- cycle;
\path (-.5,2.59807621135332) -- (.5,2.59807621135332) -- (0,3.46410161513775) -- cycle;
%\path[draw=red,line width=.108798608162487cm] (.25,2.1650635094611) -- (-.04,2.59807621135332);
%\path[draw=red,line width=.108798608162487cm] (-.25,3.03108891324554) -- (-.04,2.59807621135332);
%\path[opacity=0,line width=.108798608162487cm] (.04,2.59807621135332) -- (-.25,2.1650635094611);
%\path[opacity=0,line width=.108798608162487cm] (.04,2.59807621135332) -- (.25,3.03108891324554);
%\node[black,,scale=1.66461870488604] at (.25,2.1650635094611) {0};
%\node[black,,scale=1.66461870488604] at (-.25,2.1650635094611) {1};
%\node[black,,scale=1.38718062209258] at (0,2.59807621135332) {01};
\path (0,3.46410161513775) -- (.5,2.59807621135332) -- (1,3.46410161513775) -- cycle;
%\path[opacity=0,line width=.108798608162487cm] (.5,3.46410161513775) -- (.25,3.03108891324554);
%\node[black,,scale=1.66461870488604] at (.25,3.03108891324554) {1};
\node[black,,scale=1.66461870488604] at (.5,3.46410161513775) {1};
\path (-1.5,2.59807621135332) -- (-1,1.73205080756888) -- (-.5,2.59807621135332) -- cycle;
\path (-1.5,2.59807621135332) -- (-.5,2.59807621135332) -- (-1,3.46410161513775) -- cycle;
\path[draw=red,line width=.108798608162487cm] (-.75,2.1650635094611) -- (-1,2.59807621135332);
\path[draw=red,line width=.108798608162487cm] (-1.25,3.03108891324554) -- (-1,2.59807621135332);
%\node[black,,scale=1.66461870488604] at (-.75,2.1650635094611) {0};
%\node[black,,scale=1.66461870488604] at (-1,2.59807621135332) {0};
\path (-1,3.46410161513775) -- (-.5,2.59807621135332) -- (0,3.46410161513775) -- cycle;
%\path[draw=red,line width=.108798608162487cm] (-.25,3.03108891324554) -- (-.5,3.46410161513775);
%\node[black,,scale=1.66461870488604] at (-.25,3.03108891324554) {0};
\node[black,,scale=1.66461870488604] at (-.5,3.46410161513775) {0};
\path (-2,3.46410161513775) -- (-1.5,2.59807621135332) -- (-1,3.46410161513775) -- cycle;
\path[draw=red,line width=.108798608162487cm] (-1.25,3.03108891324554) -- (-1.5,3.46410161513775);
%\node[black,,scale=1.66461870488604] at (-1.25,3.03108891324554) {0};
\node[black,,scale=1.66461870488604] at (-1.5,3.46410161513775) {0};
\end{scope}
\end{tikzpicture}
\end{center}

In other words, that first SW/NE diagonal must consist of a \uptri{}{0}{0}
at the SW end followed as one goes Northeast by \rh{0}{j}{0}{j} (where $j$ is some other label, possibly blank, that's not drawn on the picture).
Then we apply the induction hypothesis to the puzzle with the
completed diagonal removed, noting that the content of the new NW side is
that of the old one minus a blank.

If we iterate this process, we'll eventually reach the case where
$i>k$, i.e., the NE side is entirely made of blanks. We then repeat
the same argument but using the rightmost bottom label $i$ (which, by
monotonicity of $\omega$, is also $>k$): it must go to the NW side and
its only endpoint is the rightmost edge of the NW side. We finally
obtain
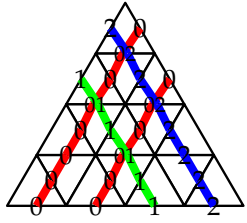
\begin{center}
\begin{tikzpicture}[scale=0.5,every node/.style={scale=0.5},execute at begin picture={\bgroup\tikzset{every path/.style={}}\clip (-2.48,-.432012) rectangle ++(4.96,4.46413);\egroup},x={(1.55426771761071cm,0cm)},y={(0cm,-1.55426394417233cm)},baseline=(current  bounding  box.center),every path/.style={draw=black,fill=none},line join=round]
\begin{scope}[every path/.append style={fill=white,line width=.0310853166178533cm}]
\path (-.5,.866025403784439) -- (0,0) -- (.5,.866025403784439) -- cycle;
\path (-.5,.866025403784439) -- (.5,.866025403784439) -- (0,1.73205080756888) -- cycle;
\path[draw=red,line width=.108798608162487cm] (.25,.433012701892219) -- (-.04,.866025403784439);
\path[draw=red,line width=.108798608162487cm] (-.25,1.29903810567666) -- (-.04,.866025403784439);
\path[draw=blue,line width=.108798608162487cm] (.04,.866025403784439) -- (-.25,.433012701892219);
\path[draw=blue,line width=.108798608162487cm] (.04,.866025403784439) -- (.25,1.29903810567666);
\node[black,,scale=1.66461870488604] at (.25,.433012701892219) {0};
\node[black,,scale=1.66461870488604] at (-.25,.433012701892219) {2};
\node[black,,scale=1.38718062209258] at (0,.866025403784439) {02};
\path (0,1.73205080756888) -- (.5,.866025403784439) -- (1,1.73205080756888) -- cycle;
\path (0,1.73205080756888) -- (1,1.73205080756888) -- (.5,2.59807621135332) -- cycle;
\path[draw=red,line width=.108798608162487cm] (.75,1.29903810567666) -- (.46,1.73205080756888);
\path[draw=red,line width=.108798608162487cm] (.25,2.1650635094611) -- (.46,1.73205080756888);
\path[draw=blue,line width=.108798608162487cm] (.54,1.73205080756888) -- (.25,1.29903810567666);
\path[draw=blue,line width=.108798608162487cm] (.54,1.73205080756888) -- (.75,2.1650635094611);
\node[black,,scale=1.66461870488604] at (.75,1.29903810567666) {0};
\node[black,,scale=1.66461870488604] at (.25,1.29903810567666) {2};
\node[black,,scale=1.38718062209258] at (.5,1.73205080756888) {02};
\path (.5,2.59807621135332) -- (1,1.73205080756888) -- (1.5,2.59807621135332) -- cycle;
\path (.5,2.59807621135332) -- (1.5,2.59807621135332) -- (1,3.46410161513775) -- cycle;
\path[draw=blue,line width=.108798608162487cm] (1,2.59807621135332) -- (.75,2.1650635094611);
\path[draw=blue,line width=.108798608162487cm] (1,2.59807621135332) -- (1.25,3.03108891324554);
\node[black,,scale=1.66461870488604] at (.75,2.1650635094611) {2};
\node[black,,scale=1.66461870488604] at (1,2.59807621135332) {2};
\path (1,3.46410161513775) -- (1.5,2.59807621135332) -- (2,3.46410161513775) -- cycle;
\path[draw=blue,line width=.108798608162487cm] (1.5,3.46410161513775) -- (1.25,3.03108891324554);
\node[black,,scale=1.66461870488604] at (1.25,3.03108891324554) {2};
\node[black,,scale=1.66461870488604] at (1.5,3.46410161513775) {2};
\path (-1,1.73205080756888) -- (-.5,.866025403784439) -- (0,1.73205080756888) -- cycle;
\path (-1,1.73205080756888) -- (0,1.73205080756888) -- (-.5,2.59807621135332) -- cycle;
\path[draw=red,line width=.108798608162487cm] (-.25,1.29903810567666) -- (-.54,1.73205080756888);
\path[draw=red,line width=.108798608162487cm] (-.75,2.1650635094611) -- (-.54,1.73205080756888);
\path[draw=green,line width=.108798608162487cm] (-.46,1.73205080756888) -- (-.75,1.29903810567666);
\path[draw=green,line width=.108798608162487cm] (-.46,1.73205080756888) -- (-.25,2.1650635094611);
\node[black,,scale=1.66461870488604] at (-.25,1.29903810567666) {0};
\node[black,,scale=1.66461870488604] at (-.75,1.29903810567666) {1};
\node[black,,scale=1.38718062209258] at (-.5,1.73205080756888) {01};
\path (-.5,2.59807621135332) -- (0,1.73205080756888) -- (.5,2.59807621135332) -- cycle;
\path (-.5,2.59807621135332) -- (.5,2.59807621135332) -- (0,3.46410161513775) -- cycle;
\path[draw=red,line width=.108798608162487cm] (.25,2.1650635094611) -- (-.04,2.59807621135332);
\path[draw=red,line width=.108798608162487cm] (-.25,3.03108891324554) -- (-.04,2.59807621135332);
\path[draw=green,line width=.108798608162487cm] (.04,2.59807621135332) -- (-.25,2.1650635094611);
\path[draw=green,line width=.108798608162487cm] (.04,2.59807621135332) -- (.25,3.03108891324554);
\node[black,,scale=1.66461870488604] at (.25,2.1650635094611) {0};
\node[black,,scale=1.66461870488604] at (-.25,2.1650635094611) {1};
\node[black,,scale=1.38718062209258] at (0,2.59807621135332) {01};
\path (0,3.46410161513775) -- (.5,2.59807621135332) -- (1,3.46410161513775) -- cycle;
\path[draw=green,line width=.108798608162487cm] (.5,3.46410161513775) -- (.25,3.03108891324554);
\node[black,,scale=1.66461870488604] at (.25,3.03108891324554) {1};
\node[black,,scale=1.66461870488604] at (.5,3.46410161513775) {1};
\path (-1.5,2.59807621135332) -- (-1,1.73205080756888) -- (-.5,2.59807621135332) -- cycle;
\path (-1.5,2.59807621135332) -- (-.5,2.59807621135332) -- (-1,3.46410161513775) -- cycle;
\path[draw=red,line width=.108798608162487cm] (-.75,2.1650635094611) -- (-1,2.59807621135332);
\path[draw=red,line width=.108798608162487cm] (-1.25,3.03108891324554) -- (-1,2.59807621135332);
\node[black,,scale=1.66461870488604] at (-.75,2.1650635094611) {0};
\node[black,,scale=1.66461870488604] at (-1,2.59807621135332) {0};
\path (-1,3.46410161513775) -- (-.5,2.59807621135332) -- (0,3.46410161513775) -- cycle;
\path[draw=red,line width=.108798608162487cm] (-.25,3.03108891324554) -- (-.5,3.46410161513775);
\node[black,,scale=1.66461870488604] at (-.25,3.03108891324554) {0};
\node[black,,scale=1.66461870488604] at (-.5,3.46410161513775) {0};
\path (-2,3.46410161513775) -- (-1.5,2.59807621135332) -- (-1,3.46410161513775) -- cycle;
\path[draw=red,line width=.108798608162487cm] (-1.25,3.03108891324554) -- (-1.5,3.46410161513775);
\node[black,,scale=1.66461870488604] at (-1.25,3.03108891324554) {0};
\node[black,,scale=1.66461870488604] at (-1.5,3.46410161513775) {0};
\end{scope}
\end{tikzpicture}
\end{center}
\end{proof}

Note that nothing we have discussed so far depended on the normalizations of $\check R_{1,2}(z)$, $U(z)$
(which are the building blocks for puzzles).
We will need to fix them now in order to satisfy property (b):
indeed Proposition~\ref{prop:trivpuzzlea} says that
\[
e^*_{\omega} \mathbf P|_{[\Tensor_{i=1}^n V^A_1(q^\alpha z_i)]_{\omega_>}
\tensor [\Tensor_{i=1}^n V^A_2(q^{-\alpha} z_i)]_{\omega_\le}} = C\,e_{\omega_>}^*\otimes e_{\omega_\le}^*
\]
where $C$ is the fugacity of the unique puzzle of the Proposition.

In the present context of separated descents, all our matrices are of
type $A$, which means $\check R_{1,2}(z)$ coincides with
\eqref{eq:Rsingle} up to normalization. We fix the latter by
specifying that $\check R_{1,2}(z)^{ml}_{ij}=1$ for $i=l\ne j=m$,
i.e., with the convenient parametrization $z=q^{-2}z'/z''$,
\[
\check R_{1,2}(z)^{ml}_{ij}=
\rh{i}{j}{l}{m}
=
\begin{cases}
\frac{q(1-z)}{1-q^2z}& i=j=m=l
\\
1 & i=l\ne j=m
\\
-q \frac{(1-q^2)z}{1-q^2z} & i=m<j=l
\\
-q^{-1}\frac{1-q^2}{1-q^2z} & i=m>j=l
\\
0 & \text{else}
\end{cases}\ .
\]

Similarly, starting from factorization property \eqref{eq:Rfac}, we find that $U(z)$ is given by
\[
U(z)^{ij}_a =\uptri{i}{a}{j}=
\begin{cases}
0&a\ne ij\\
1& i>j\\
-q&i<j
\end{cases}
\]
up to normalization, which we fix according to the above formula.
\junk{this assumes we've defined all labels and their ordering,
  including the ones on S side}

Then $C=1$, and property (b) is satisfied with $\omega_1=\omega_>$,
$\omega_2=\omega_\le$, $\omega_3=\omega$. The corresponding weights are of the form that is discussed at the end of \S\ref{ssec:sepdescdata}.

We also include $D(z)$ for reference, since it appears in nonequivariant puzzles:
\[
  D(z)^a_{ij} = \downtri{j}{a}{i} =
  \begin{cases}
     0     & a\ne ij\\
     1     & i<j\\
    -q^{-1} & i>j
  \end{cases}\ .
\]

At this stage, we've got the setup of \S\ref{sec:setup} working for separated descents; this means that
Theorem~\ref{thm:motivic} applies here, providing a puzzle formula for the product of the pullbacks of two motivic Segre classes of
Schubert cells of partial flag varieties $\mathcal F_1$ and $\mathcal F_2$ (where the dimensions of $\mathcal F_1$ are less of equal
to those of $\mathcal F_2$) to their common refinement $\mathcal F_3$,
where the fugacities of the puzzle pieces are given by the entries of $\check R$, $U$, $D$ right above.

These generic equivariant puzzles have the simple interpretation that
they are colored {\em lattice paths}, where the lattice is triangular
and the paths go Southwest or Southeast with the only constraint that
they cannot share edges (in particular, they are allowed to cross at
horizontal edges). Nonequivariant generic puzzles are the subset of
them in which no two lines of the same color cross, and no horizontal
edge is empty.

\rem[gray]{ AK: where do you want to state the $\chi$ rule?
  End of \S\ref{ssec:motivic}? PZJ: actually I put it in 2.3 cause it requires the geometric setup that isn't in the intro.
  I did mention it in the plan in 1.7.
  One thing to point out here is that
  the things being intersected are slightly different in sep-desc
  and almost-sep-desc, even if the permutations are the same, and it'd be
  nice if we could do an example where the two numbers are different.
  }

\begin{ex}
  % puzzle("_2_2","10__","2120",Equivariant=>false,Paths=>true)
  There are three nonequivariant puzzles with sides $\lambda=\blank2\blank2$,
  $\mu=10\blank\blank$, $\nu=2120$:
  \begin{center}
    \tikzset{every picture/.style={scale=0.5,every node/.style={scale=0.5}}}
    \input exchi.tex
  \end{center}
  Note that these puzzles contain triangles that are not allowed by Theorem~\ref{thm:sepdesc}, even
  in $K$-theory. As an application of Theorem~\ref{thm:euler}, we compute
  \[
    \chi\left(g\,p_1^{-1}(X^{\blank2\blank2}_\circ)\cap g' p_2^{-1}(X^{10\blank\blank}_\circ)\cap
      g'' X^{0212}_\circ\right)=3\ .
  \]
  Indeed, write the three general flags as $P_i\in\ell_i\subset H_i\subset\PP^3$,
  and denote the open triple intersection by $Y$. Its closed incidence
  conditions are
  \[
    L\cap\ell_1\ne\varnothing,\qquad p\in L\cap H_2,\qquad
    L\cap\ell_3\ne\varnothing\ .
  \]
  Since $\ell_1,\ell_3$ are disjoint,
  a line $L$ meeting both is determined by its two intersection points
  $a$ and $b$,
  thus parametrized by
  $S=\ell_1\times\ell_3\cong\PP^1\times\PP^1$.
  Put $B=\ell_3\cap H_2$. The open-cell conditions
  $L\not\subset H_2$ and $p\notin\ell_3$ together amount to $b\ne B$:
  if $b=B$ and $L\not\subset H_2$, then $p=B$.
  For $b\ne B$, the point $p=L\cap H_2$ is uniquely determined
  and does not lie on $\ell_3$.

  Thus $Y$ is identified with $S$ minus the five fibers
  \[
    a=P_1,\quad a=\ell_1\cap H_3,\qquad
    b=\ell_3\cap H_1,\quad b=P_3,\quad b=B,
  \]
  and the curve
  \[
    C=\{(a,b)\in S:\,L\cap\ell_2\ne\varnothing\}\ .
  \]
  Here the first four fibers impose the remaining open-cell conditions
  $P_1\notin L$, $L\not\subset H_3$, $L\not\subset H_1$, and
  $P_3\notin L$, respectively; removing $C$ imposes $p\notin\ell_2$,
  since $\ell_2\subset H_2$.
  Since $\ell_2$ is disjoint from $\ell_1$ and $\ell_3$, for each
  $a\in\ell_1$ the plane containing $a$ and $\ell_2$ meets $\ell_3$
  at the unique corresponding point $b$. Projection onto the first
  factor therefore identifies $C$ with $\ell_1\cong\PP^1$.
  The two fibers with fixed $a$ meet the three fibers with fixed $b$
  in six points, and $C$ meets the five fibers in five further points.
  For general flags these eleven intersections are distinct, with no
  triple intersections. In particular, $\dim Y=2$, which fixes
  the sign in Theorem~\ref{thm:euler}, and inclusion/exclusion gives
  the desired Euler characteristic $\chi(Y)=2\times 2-6\times 2+11=3$.
  % \[
  %   \chi(Y)=2\times2-6\times2+11=3\ .
  % \]
\end{ex}

\newcommand\php{\phantom{+}}
\subsection{The $B$-matrix}
With a view towards the $q\to0$ limit, we now introduce the $B$-matrix.
It acts on a $(d+4)$-dimensional space, generated by
the usual Cartan generators $x_i$, $i\in \{0<\cdots<k<\blank<k+1<\cdots<d\}$, as well as $y_1,y_2$. 
%In fact, in this section, only of the $y_i$ is required, say $y_2$.

We define $B$ to be the skew-symmetric form satisfying \rem[gray]{the
  one in the notes is fine, but this one is more symmetric. in either
  case, note that $B$ is rank one less than max}
\begin{align*}
B(x_i,x_j)&=-1\qquad\rlap{$i<j,\ i,j\ne\blank$}
\\
B(x_\blank,x_j)&=\php 0
&
B(x_i,y_1)&=\php 0
&
B(x_i,y_2)&=\php 0
\\
B(x_\blank,y_1)&=-1
&
B(x_\blank,y_2)&=+1
&
B(y_1,y_2)&=-1\ .
\end{align*}

We now check that $B$ satisfies the hypothesis of Lemma~\ref{lem:twist}.
We compute $B(\wt(e_{a,i}),\wt(e_{a,j}))$ case by case:
\begin{itemize}
\item $a=1$: the weights are
$x_i+y_1$ where $i\in \{\blank<k+1<\cdots<d\}$, and one finds
\[
B(x_i+y_1,x_j+y_1)=\begin{cases}
B(x_i,x_j)=\sign(i-j)&i,j\ne\blank\\
B(y_1,x_\blank)=1 & i\ne\blank,\ j=\blank\\
B(x_\blank,y_1)=-1 & i=\blank,\ j\ne\blank\\
0 & i=j=\blank
\end{cases}\ .
\]
\item $a=2$: the weights are
$x_i+y_2$ where $i\in \{0<\cdots<k<\blank\}$, and similarly
\[
B(x_i+y_2,x_j+y_2)=\begin{cases}
B(x_i,x_j)=\sign(i-j)&i,j\ne\blank\\
B(y_2,x_\blank)=-1 & i\ne\blank,\ j=\blank\\
B(x_\blank,y_2)=1 & i=\blank,\ j\ne\blank\\
0 & i=j=\blank
\end{cases}\ .
\]
\item $a=3$: the weights are
$x_i+x_\blank+y_1+y_2$ where $i\in\{0<\cdots<k<k+1<\cdots<d\}$, and one finds 
\[
B(x_i+x_\blank+y_1+y_2,x_j+x_\blank+y_1+y_2)=
B(x_i,x_j)=\sign(i-j)\ .
\]
\end{itemize}

We are now in a position to take the limit $q\to0$.

\subsection{The limit \texorpdfstring{$q\to0$}{q -> 0}: nonequivariant puzzles}\label{ssec:sepdescq0}
For pedagogical reasons we perform the limit $q\to0$ twice, first on $U$ and $D$
only, then on $\check R_{1,2}$.

In order to twist with $\Omega = q^{\frac{1}{2}B}$,
we compute the inversion charges of every triangle; we list up-pointing
triangles, only, since inversion charge is invariant under $180^\circ$ rotation:
\[
\begin{split}
\inv(\uptri{}{i}{i})&=\inv(\uptri{i}{i}{})=\inv(\uptri{j}{ij}{i})=0
\\
\inv(\uptri{i}{ij}{j})&=1
\end{split}
\qquad i<j\ .
\]

The twisted intertwiners take the form
\begin{align*}
\tilde U(z)^{ij}_a &=\uptri{i}{a}{j}=
\begin{cases}
0&a\ne ij\\
1& i>j\\
-q^2&i<j,\ i\ne\blank\text{ and }j\ne\blank\\
-q& i<j,\ i=\blank\text{ or }j=\blank
\end{cases}
\\
\tilde D(z)^a_{ij} &=\downtri{j}{a}{i}=
\begin{cases}
0&a\ne ij\\
1& i<j\\
-1&i>j,\ i\ne\blank\text{ and }j\ne\blank\\
-q^{-1}&i>j,\ i=\blank\text{ or }j=\blank
\end{cases}\ .
\end{align*}

Now perform the following change of basis:
\begin{align}\notag
&\text{in $V_1$:}\quad  e'_{1,i}=-q^{-1} e_{1,i}\text{ for }i\le k
\\\label{eq:sepdescconj}
&\text{in $V_2$:}\quad e'_{2,j}=-q^{-1} e_{2,j}\text{ for }j\ge k+1
\\\notag
&\text{in $V_3$:}\quad e'_{3,ij}=-q^{-1} e_{3,ij}\text{ when }i,j\le k\text{ or }i,j\ge k+1
\end{align}
all other basis vectors remaining unchanged.
Note that none of the labels above occur on the boundary of puzzles, so the fugacity of the puzzle is unaffected by such transformations.

We find:
\begin{align*}
\tilde U'(z)^{ij}_a &=\uptri{i}{a}{j}=
\begin{cases}
0&a\ne ij\\
1& i>j\text{ or } i=\blank\text{ or }j=\blank
\\
-1& i<\blank<j
\\
-q^2& i<j<\blank\text{ or }\blank<i<j
\end{cases}
\\
\tilde D'(z)^a_{ij} &=\downtri{j}{a}{i}=
\begin{cases}
0&a\ne ij\\
1& i<j\text{ or } i=\blank\text{ or }j=\blank\\
-1& i>j>\blank\text{ or } \blank>i>j\\
-q^2 & i>\blank>j
\end{cases}\ .
\end{align*}

At this stage, we can safely take the limit $q\to0$, resulting in the triangles (including
$K$-triangles) of Theorem~\ref{thm:sepdesc}.

\subsection{The limit \texorpdfstring{$q\to0$}{q -> 0}: equivariant puzzles}
\label{ssec:sepdescq0T}
We now repeat the procedure for $\check R_{1,2}$. Here are the inversion charges
of the equivariant rhombi:
\[
  \inv\rh{}{}{}{} = -1 \qquad\qquad\qquad \inv{\rh{i}{i}{i}{i}}=1\qquad i\ne\blank\ .
\]
Twisting the $R$-matrix results in
\[
\tilde R_{1,2}(z)^{ml}_{ij}=
\rh{i}{j}{l}{m}
=
\begin{cases}
\frac{q(1-z)}{1-q^2z}\begin{cases}q^{-1}&i=\blank\\q&\text{else}\end{cases}& i=j=m=l
\\
\begin{cases} 1 & i<j\text{ or }i=\blank\text{ or }j=\blank
\\
q^2 &\text{else}
\end{cases}
&  i=l\ne j=m
\\
\frac{1-q^2}{1-q^2z}
\begin{cases}
-qz & i<j,\ i=\blank\text{ or }j=\blank\\
-q^2z & i<j,\text{ else}\\
-q^{-1} & i>j,\ i=\blank\text{ or }j=\blank\\
-1&i>j,\text{ else}
\end{cases}
&i=m\ne j=l
\\
0 & \text{else}
\end{cases}\ .
\]

We then perform the change of basis above \eqref{eq:sepdescconj}, and find the final form:
\[
\tilde R'_{1,2}(z)^{ml}_{ij}=
\rh{i}{j}{l}{m}
=
\begin{cases}
\frac{1-z}{1-q^2z}\begin{cases}1&i=\blank\\ q^2&\text{else}
\end{cases}
& i=j=m=l \\
\begin{cases} 1 & i<j\text{ or }i=\blank\text{ or }j=\blank
\\
q^2 &\text{else}
\end{cases}
& i=l\ne j=m
\\
\frac{1-q^2}{1-q^2z}\begin{cases}
-q^2z& i<j<\blank\text{ or }\blank<i<j
\\
-z& i<\blank<j
\\
-1 & j<i<\blank\text{ or }\blank<j<i
\\
-q^2 & j<\blank<i
\\
z&i=\blank<j\text{ or }i<j=\blank
\\
1&j=\blank<i\text{ or }j<i=\blank
\end{cases}
 & i=m\ne j=l
\\
0 & \text{else}
\end{cases}\ .
\]
At $q=0$,
\[
\rh{i}{j}{l}{m}
=
\begin{cases}
1-z& i=j=m=l=\blank
\\
1 & i=l\ne j=m,\ i<j\text{ or }i=\blank\text{ or }j=\blank
\\
-z & i=m<\blank<j=l
\\
-1 & j=l<i=m<\blank\text{ or }\blank<j=l<i=m
\\
z&j=l,i=m,\ i=\blank<j\text{ or }i<j=\blank
\\
1&j=l,i=m,\ j=\blank<i\text{ or }j<i=\blank
\\
0 & \text{else}
\end{cases}
\]
which coincides with the equivariant fugacities of
Theorem~\ref{thm:sepdesc}.

\section{Almost-separated descents}\label{sec:almostsepdesc}
\subsection{The data from \S\ref{ssec:tensor}}
\label{ssec:almostsepdescdata}
% We specify the data from \S\ref{ssec:tensor}. 
The algebra is $\Uq(\lie{so}_{2(d+2)}[z^\pm])$, which has three
minuscule representations $\CC^{2(d+2)}, spin_+, spin_-,$ corresponding to
the tail and the two antlers of the Dynkin diagram $D_{d+2}$.
We use the standard Cartan subalgebra $\lie h := \oplus^{d+2} \lie{so}_2
\leq \lie g = \lie{so}_{2(d+2)}$, naming its co\"ordinates $[0,d]\sqcup \{\blank\,\}$.

Our three representations and their weights are
$$
  \def\arraystretch{1.2}
\begin{array}{c|c|lc}
i & V_i & \text{its weights} \\ \hline  
1 & spin_+ &
\left\{ \frac{1}{2}(\pm 1,\pm 1,\ldots,\pm 1) \right\} 
 & \text{ with evenly many $-$}\\
2 & (spin_-)^* &
\left\{ \frac{1}{2}(\pm 1,\pm 1,\ldots,\pm 1) \right\} 
 &\text{ with oddly many $+$}\\
3 & \CC^{2(d+2)} & \{\pm x_i\colon i\in [0,d] \sqcup \{\blank\} \} 
\end{array}
$$
so when we add a weight of $spin_+$ to a weight of $(spin_-)^*$ we get
an integer vector, whose total is an odd integer (hence has a chance to
be a weight of $\CC^{2(d+2)}$). Take $\alpha = d$. 
For notational convenience let
$\vec 1 := (1,\ldots,1)$ denote the all-$1$s vector.

To specify a Borel subalgebra (or a positive Weyl chamber) we
indicate\footnote{%
  This is {\em slightly} too much information -- it specifies a $B/C$ Weyl
  chamber, and those glue together in pairs to make type $D$ Weyl chambers.
}
which of the weights $\{+x_i,-x_i\}_{i\in [0,d]\sqcup \{\blank\,\}}$ are positive,
and then, the order on the positive ones. As the answer is somewhat
unintuitive we put off specifying it until later, when it will be
more uniquely determined. \junk{; for right now order them
  $x_0,x_1,x_2,\ldots,x_d,x_\blank\,$, which will soon come in groups
  $(x_0,\ldots,x_{k-1})$, $(x_k)$, $(x_{k+1},\ldots,x_d)$, $(x_\blank)$. }
  
% Choose the positive Weyl chamber using $\vec c = ???$.
The $V^A_i$ faces (as required in \S\ref{ssec:genthm})
% with their vertices ordered according to $\cdot \vec c$,
are determined by maximizing dot product with the following coweights:
\junk{Check order, and probably go through weak (c) in detail}

\junk{ {\em Can} get $\eta_3 = \eta_1 + \eta_2$:
\begin{itemize}
\item $\eta_1 = (+a^{k}, +d^{d-k}, -a)$, 
  so $V_1^A$ has weights 
  $\{\vec 1/2\} \sqcup \{\vec 1/2 - (x_i+x_\blank)\}\colon i \neq \blank \}$
\item $\eta_2 = (-b^{k}, -c^{d-k}, -c)$,
  so $V_2^A$ has weights $\{-\vec 1/2 + x_i\colon i \in [k,d] \sqcup \{\blank\,\}\}$
  (i.e. $i\neq \blank$)
\item $\eta_3 = ((a-b)^{k}, (d-c)^{d-k}, -a-c)$,
  so $V_3^A$ has weights $\{-x_i \colon i<k \} \sqcup \{+x_i \colon i\geq k \}$
\end{itemize}
$$ d>a>0 \qquad b>c>0  \qquad b-a = d-c \geq |a+c| \qquad c+b-a = d > 2c+a 
\qquad b > 2a+c 
$$
$$ (a,b,c,d) = (1,4,1,4) $$
} % end junk

\junk{ These are the good ones we should use in the final version
  \begin{itemize}
  \item $\eta_1 = (\php 1^{k}, \php 4^{d-k+1}, -1)$, 
    so $V_1^A$ has weights 
    $\{\vec 1/2\} \sqcup \{\vec 1/2 - (x_i+x_\blank)\}\colon i \in [0,k-1] \}$
  \item $\eta_2 = (-4^{k}, -1^{d-k+1}, -1)$,
  so $V_2^A$ has weights $\{-\vec 1/2 +x_i \colon i\in [k,d] \sqcup \{\blank\,\} \}$
  % (i.e. $i\neq \blank$)
  \item $\eta_3 = (-3^{k}, \php 3^{d-k+1}, \underline{-2})$,
    so $V_3^A$ has weights $\{-x_i \colon i<k\} \sqcup \{+x_i \colon i\geq k \}$
    % \sqcup \{\varepsilon x_\blank\}
  \end{itemize}
  Again we have $\eta_3 = \eta_1+\eta_2$, though we won't use this
  for anything.

  Below are the bad ones that make $k$ not appear on the S edge
} % end junk

\junk{AK: need to make clear the order used on the basis, when
  writing these $\eta_i$. Is it the $0$ that could be changed?}

\begin{itemize}
\item[] $\phantom{\eta_1 = (}x_{<k}\ \ \ \ x_k\ \ \ \ \ x_{>k}\ \ \ \ \ x_\blank$
\item $\eta_1 = (\php 1^{k}, \php 1, \php 3^{d-k}, -1)$,  so $V_1^A$ has weights 
  $\{\vec 1/2 - (x_i+x_\blank)\colon i \in [0,k] \} \sqcup \{\vec 1/2\}$
\item $\eta_2 = (-3^{k}, -1, -1^{d-k}, -1)$, so $V_2^A$ has weights
  $\{-\vec 1/2 +x_i \colon i\in [k,d] \sqcup \{\blank\,\} \}$
  % (i.e. $i\neq \blank$)
\item $\eta_3 = (-2^{k}, \php 0, \php 2^{d-k}, {-2} )$,
  so $V_3^A$ has weights
  $\{-x_i \colon i<k \} \sqcup \{+x_i \colon i> k\} \sqcup \{-x_\blank\,\}$
\end{itemize}

\junk{
$$ 0<a<b \qquad c > d > 0 \qquad 0 < c-a = b-d = a+f \geq |a-d| $$
$$ 2b = c+3d \qquad c = 2a+d \qquad a=d=1, c=3, b=3 $$
\begin{itemize}
\item $\eta_1 = (\php a^{k}, \php b^{d-k}, \php a, -a)$,  so $V_1^A$ has weights 
  $\{\vec 1/2 - (x_i+x_\blank)\colon i \in [0,k] \} \sqcup \{\vec 1/2\}$
\item $\eta_2 = (-c^{k}, -d^{d-k}, -d, -d)$, so $V_2^A$ has weights
  $\{-\vec 1/2 +x_i \colon i\in [k,d] \sqcup \{\blank\,\} \}$
  % (i.e. $i\neq \blank$)
\item $\eta_3 = ((a-c)^{k}, (b-d)^{d-k}, a-d, -a-d)$, so $V_3^A$ %has weights
  $\{-x_i \colon i<k \} \sqcup \{+x_i \colon i> k\} \sqcup \{-x_\blank\,\}$
\end{itemize}
} %end junk

%While in \S\ref{ssec:sepdescdata} we were able to choose the functionals
%$\eta_{1,2,3}$ to get $\eta_3 = \eta_1 + \eta_2$ (though not for any
%particular reason), here that leads to incompatible inequalities.
We have again managed that $\eta_3 = \eta_1 + \eta_2$
(though not for any useful reason we could come up with).
If we give that up, the $0$ can in fact be
changed to $+2$ or $-2$, enlarging the third face, but doing so
doesn't get us any extra Schubert calculus in the end.

The way that we {\em draw} the weights of $V_{1,2,3}$ as edge labels is
slightly complicated, and was optimized to have the nicest-looking puzzles.
\begin{itemize}
\item On $/$ edges with weight $+\vec 1/2 - \sum_{i\in R} x_i$, we draw
  the set $R$ (except for $\blank$). \\ Since we know $\#R$ to be even, we can
  infer whether $\blank$ is or isn't in $R$ even though it isn't drawn.
  To get the right order, we need   $ -x_0 < -x_1 < \ldots < -x_k < +x_\blank$. 
\item On $\backslash$ edges with weight $-\vec 1/2 + \sum_{i\in S} x_i$, we draw
  the set $S$ (except for $\blank$). \\ Since we know $\#S$ to be odd, we can
  infer whether $\blank$ is or isn't in $S$ even though it isn't drawn.
  To get the right order, we need $ +x_\blank < +x_k < +x_{k+1} < \ldots < +x_d $.
\item On $-$ edges with weight $+x_j$ we draw $\nearrow j$
  (or $even$ if $j=\blank$), 
  and with weight $-x_j$ we draw $\searrow j$ (or $odd$ if $j=\blank$).
  To get the right order, we need \\
  $ -x_0 < -x_1 < \ldots < -x_{k-1} < -x_\blank < +x_{k+1} < \ldots < +x_d $.
\end{itemize}

Combining these conditions, we get a consistent set of inequalities
$$
\begin{array}{rl}
  x_0 > x_1 > \ldots > x_{k-1} &  \\ 
  & > x_k > \pm x_\blank \\
  x_d > x_{d-1} > \ldots > x_{k+1}  & \\
\end{array}\ .
$$
One of the many ways to achieve this is to take all the $x_i$ positive
(including $x_\blank$) and order them $0,\ldots,k-1, d,\ldots,k, \blank\,$
decreasing.

Had we not made the $x_\blank$ label blank, it would need to appear on
every edge on the NW boundary of a puzzle.

In a triangular puzzle piece we will have
$ \left( \vec 1/2 - \sum_{i\in R} x_i \right)
+ \left( - \vec 1/2 + \sum_{i\in S} x_i \right) = \pm x_j, $
so $R \cap S = \min(R,S) \subsetneq \max(R,S) = R\cup S$ and the
difference is by the one element $j$.

We check (c), (d), having proven (a) in general.
Let $\vec g := \vec 1/2 - x_\blank$ for short.
Using our strange order $0,\ldots,k-1, d,\ldots,k, \blank\,$ on
co\"ordinates, we have (omitting empty ranges and terms with indices
outside $[0,d]$)
$$
\begin{array}{rl}
  \text{$V_1^A$'s weights: } &
\vec g + x_\blank,\quad (\vec g-x_i)_{i=k,\ldots,0} \\
  \text{Their differences: } &
x_\blank+x_k,\quad (x_{i-1}-x_i)_{i=k,\ldots,1}
\\
  \text{$V_2^A$'s weights: } &
(-\vec 1/2+x_i)_{i=d,\ldots,k},\quad -\vec 1/2+x_\blank \\
  \text{Their differences: } &
(x_i-x_{i-1})_{i=d,\ldots,k+1},\quad x_k-x_\blank
\\
  \text{$V_3^A$'s weights: } &
(x_i)_{i=d,\ldots,k+1},\quad -x_\blank,\quad (-x_i)_{i=k-1,\ldots,0} \\
  \text{Their differences: } &
(x_i-x_{i-1})_{i=d,\ldots,k+2},\quad x_{k+1}+x_\blank,\\
&x_{k-1}-x_\blank,\quad (x_{i-1}-x_i)_{i=k-1,\ldots,1}
\end{array}
$$ 
and in each case the differences form a type $A$ root subsystem, as needed
for the weak version of (c).

\junk{ What if we used the order $0\ldots k-1, \blank, k\ldots d$? We'd get
$$
\begin{array}{rl}
  \text{$V_1^A$'s weights: } &
\vec g + x_\blank,\quad \vec g - x_k,\quad \ldots,\quad \vec g - x_d, 
\quad \vec g - x_{k-1}, \quad \ldots,\quad \vec g - x_0 \\
  \text{Their differences: } &
x_\blank + x_k, x_{k-1} - x_k, \ldots, x_d - x_{d+1}, 
x_{k-1} - x_d, x_{k-2} - x_{k-1}, \ldots, x_0 - x_1 
\\
  \text{$V_2^A$'s weights: } &
-\vec 1/2 + x_d,\quad -\vec 1/2 + x_{d-1}, \quad \ldots \quad, 
-\vec 1/2 + x_k, \quad -\vec 1/2 + x_\blank \\
  \text{Their differences: } &
x_d - x_{d-1}, x_{d-1} - x_{d-2}, \ldots, x_{k+1} - x_k, x_k - x_\blank 
\\
  \text{$V_3^A$'s weights: } &
+x_d, \quad \ldots, \quad +x_k, \quad -x_{k-1}, \quad -x_{k-2}, 
\quad\ldots, \quad -x_1, \quad -x_0  \\
  \text{Their differences: } &
x_d-x_{d-1}, \ldots, x_{k+1}-x_k, x_k+x_{k-1}, x_{k-2}-x_{k-1}, \ldots, x_0-x_1 
\end{array}
$$ 
} % end junk

To see (d), consider three weights $\omega_1 + \omega_2 = \omega_3$ with
$\omega_i$ a sum of $n$ weights from $V_i^A$. Ignoring the $\pm \vec 1/2$
summands (of which there will obviously be $n$ in $\omega_1$ canceling
$-n$ in $\omega_2$) we get $\omega_1$ is $-mx_\blank$ minus a sum of $m$ $x_{i\leq k}$s
(for some $m\leq n$), and $\omega_2$ is $m' x_\blank$ plus a sum of $n-m'$
many $x_{j\geq k}$s (for some $m'\leq m$), totaling $\omega_1 + \omega_2$
which is then a sum of $n-(m-m')$ many $\pm x_i$ minus $(m-m')x_\blank$.
To verify property (d) we compute the three standard parabolics,
each of which is a group of block-upper-triangular matrices.

$$ \begin{array}{cclcrc}
  \text{$P_1$'s blocks } 
  &=& \big( c_0, \ldots, c_{k-1}, &c_k,&  \sum\nolimits_{i=k+1}^d c_i\big)\\
  \text{$P_2$'s blocks }
  &=& \big( \sum\nolimits_{i=0}^{k-1} c_i, &c_k,& c_{k+1}, \ldots, c_d\big)\\
  \text{$P_3$'s blocks } 
  &=& \big( c_0, \ldots, c_{k-1}, &c_k,&  c_{k+1}, \ldots, c_d \big)
      &\qquad \text{so }P_3 = P_1\cap P_2\ .
\end{array}
$$

\subsection{The intertwiners}\label{ssec:almostsepdescR}
We now describe the intertwiners $\check R_{1,2}$, $U$ and $D$, which are the building blocks of our puzzles.
Our reference for this section is \cite{Okado-BD}.

In order to help with the conversion to the unusual labeling of weights
of \S\ref{ssec:almostsepdescdata}, we introduce the bijection
\begin{equation}\label{eq:mapping}
  w(0,\ldots,k-1,d,\ldots,k,\blank) := (1,\ldots,k,k+1,\ldots,d,d+1,d+2)\ .
\end{equation}
Write $\tilde x_j = x_{w^{-1}(j)}$.
Then the simple roots of $\lie{g}=\lie{d_{d+2}}$ are $\alpha_j=\tilde x_j-\tilde x_{j+1}$, $j=1,\ldots,d$, and $\alpha_\pm=\tilde x_{d+1}\pm \tilde x_{d+2}$.
$spin_\epsilon$ is the fundamental module with highest weight $\omega_\epsilon=\frac{1}{2}(\tilde x_1+\cdots+\tilde x_{d+1}+\epsilon \tilde x_{d+2})$, the other
fundamental weights are $\omega_j=\tilde x_1+\cdots+\tilde x_{j}$, $j=1,\ldots,d$. We also introduce the notation
$W_j$ for the irreducible module with highest weight $\tilde x_1+\cdots+\tilde x_{d+2-j}$, that is $\omega_{d+2-j}$ if $1<j<d+2$, 
$\omega_++\omega_-$ if $j=1$,
$2\omega_+$ if $j=0$.

%\rem{AK: just below, you had $V_2$ without ${}^*$. I assume that was a slip}

Recall that $V_1=spin_+$ and $V_2=(spin_-)^*$.
According to \cite[Eq.~(4.2)]{Okado-BD},
one has the $\Uq(\lie{g})$-module\footnote{%
  Remember that the irreducibility we cited \cite{Chari-braid} for
  is about the $\Uq(\lie{g}[z^\pm])$-module structure, not
  the $\Uq(\lie{g})$-module structure. }
decomposition \rem[gray]{$n=d+2$, $k=n-1=d+1$, $x=1/z$ from the notes}
\[
  spin_+ \otimes spin_-^*
  = W_{d+1\ \text{mod}\ 2} \oplus  \cdots \oplus W_{d-1} \oplus W_{d+1}\ .
\]
In particular $W_{d+1}\cong V_3$ as $\Uq(\lie g)$-modules.

The $R$-matrix from $spin_+\otimes spin_-^*$ to $spin_-^*\otimes spin_+$ is given in terms of operators $P_j$ which
are $\Uq(\lie{g})$%$\Uq(\lie{so}_{2(d+2)})$
-intertwiners
implementing the channels $spin_+\otimes spin_-^*\to W_j\to spin_-^*\otimes spin_+$: \cite[\S 5]{Okado-BD}
\begin{equation}\label{eq:defRD}
\check R_{1,2}(z)=c(z)^{-1} \sum_{\substack{j=0\\j\equiv d+1\\(\text{mod }2)}}^{d+1} \rho_j(z) P_j\ .
\end{equation}
Here we have introduced an extra normalization factor $c(z)$ which will be fixed below. The functions $\rho_j(z)$ are given by
\[
  \rho_j(z)=
  \begin{cases}
    \displaystyle
    \ \ \prod_{i=1}^{j/2} (q^{2i-1}-q^{-2i+1}z)
    \prod_{i=j/2+1}^{(d+1)/2} (q^{2i-1}z-q^{-2i+1})
    &j\text{ even, }d\text{ odd}
    \\
    \displaystyle
    \prod_{i=1}^{(j-1)/2} (q^{2i}-q^{-2i}z)
    \ \ \ \ \prod_{i=(j+1)/2}^{d/2} (q^{2i}z-q^{-2i})
    &j\text{ odd, }d\text{ even}
  \end{cases}\ .
\]
Note that all $\rho_j(z)$ with $j<d+1$ have a factor of $q^{2i-1}z- q^{-2i+1}$ for $i=(d+1)/2$ ($d$ odd),
resp.\ $q^{2i}z-q^{-2i}$ for $i=d/2$ ($d$ even), but $\rho_{d+1}(z)$ doesn't.
This implies that $\check R_{1,2}(z=q^{-2d})$ is proportional to $P_{d+1}$ and therefore the factorization \eqref{eq:Rfac} occurs at $\alpha=d$.
Because $\check R_{1,2}(z=q^{-2d})$ is a $\Uq(\lie{g}[z^\pm])$-intertwiner, its image is $\Uq(\lie{g}[z^\pm])$-invariant and isomorphic
to $V_3(z)$ as a $\Uq(\lie{g}[z^\pm])$-module (there is some arbitrariness in shifting $z\mapsto az$ which is fixed by this statement);
and therefore we can choose $U$, $D$ to be $\Uq(\lie{g}[z^\pm])$-intertwiners themselves, as in \eqref{eq:RUD}.

In fact, one has the following expression for $U$ and $D$, derivable from the explicit expression of the $P_j$ given in 
\cite[Prop.~5.1]{Okado-BD}:
\begin{align}
\notag
 \big< \epsilon \tilde x_i, U (a_1,\ldots,a_{i-1},\epsilon/2,a_{i},\ldots,a_{d+1})\otimes
&\\\label{eq:defUD}
(-a_1,\ldots,-a_{i-1},\epsilon/2,-a_{i},\ldots,-a_{d+1}) \big>&= c_U (-q)^M
% \left< \epsilon \tilde x_i, U (a_1,\ldots,a_{i-1},\epsilon/2,a_{i},\ldots,a_{d+1})\otimes(-a_1,\ldots,-a_{i-1},\epsilon/2,-a_{i},\ldots,-a_{d+1}) \right>&=
%(-q)^{\sum_{\substack{1\le j\le d+1\\a_j=-1/2}}(d+1-j)}&
\\
\notag
\big< 
(a_1,\ldots,a_{i-1},\epsilon/2,a_{i},\ldots,a_{d+1})\otimes
&\\\label{eq:defDD}
(-a_1,\ldots,-a_{i-1},\epsilon/2,-a_{i},\ldots,-a_{d+1}), 
D \epsilon \tilde x_i
  \big>&= c_D (-q)^{-M}
%(-q)^{-\sum_{\substack{1\le j\le d+1\\a_j=-1/2}}(d+1-j)}&
%\end{align}
\\\notag
M := {\sum_{\substack{1\le j\le d+1\\a_j=-1/2}}(d+1-j)}
&
\end{align}
where vectors are (provisionally) described by their weights,
$a\in \{+1/2,-1/2\}^{d+1}$ and $\epsilon\in\{+1,-1\}$. We shall convert to the labeling of \S\ref{ssec:almostsepdescdata} below.
The constants $c_U$ and $c_D$ will also be fixed below.

\begin{rmk}
  As soon as $d\ge 3$, the $R$-matrix \eqref{eq:defRD} has more than 2
  terms in its decomposition. This seems somehow related to the lack
  of an equivariant rule for Schubert classes, as already pointed out
  in \cite[\S1.3]{artic71} in the context of 3-step Schubert calculus.
\end{rmk}

\subsection{Proof of property (b)}
\begin{prop}\label{prop:trivpuzzled}
  There is a unique puzzle
  \tikz[scale=1.8,baseline=0.5cm]{\uptri{\lambda}{\omega}{\mu}} with a
  weakly increasing string $\omega$ at the bottom, and $\lambda$ (resp.\ 
  $\mu$) with content $\omega_\le$ (resp.\ $\omega_\ge$);
  it has $\lambda=\omega_\le$ and $\mu=\omega_\ge$, and
  its labels on diagonal edges are constant along each diagonal
  (NW/SE or NE/SW). %Its fugacity is $1$.
\end{prop}
\begin{proof}
Labels form paths that go E, NE or SE. In particular, the path ending on the leftmost bottom edge must come from the leftmost NW edge;
and then inductively, because NW edges can only carry a single label, every path ending on the bottom edge must go straight SE. A similar
reasoning holds for paths starting on the bottom edge; at this stage we have the following configuration:
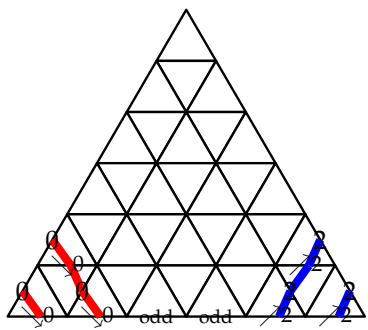
\begin{center}
\begin{tikzpicture}[scale=0.5,every node/.style={scale=0.5},execute at begin picture={\bgroup\tikzset{every path/.style={}}\clip (-3.72,-.637138) rectangle ++(7.44,6.58376);\egroup},x={(1.56597774923292cm,0cm)},y={(0cm,-1.56597907993384cm)},baseline=(current  bounding  box.center),every path/.style={draw=black,fill=none},line join=round]
\begin{scope}[every path/.append style={fill=white,line width=.0313195682916704cm}]
\path (-.5,.866025403784439) -- (0,0) -- (.5,.866025403784439) -- cycle;
\path (-.5,.866025403784439) -- (.5,.866025403784439) -- (0,1.73205080756888) -- cycle;
%\node[black,,scale=1.04822680126184] at (0,.866025403784439) {even};
\path (0,1.73205080756888) -- (.5,.866025403784439) -- (1,1.73205080756888) -- cycle;
\path (0,1.73205080756888) -- (1,1.73205080756888) -- (.5,2.59807621135332) -- cycle;
%\node[black,,scale=1.04822680126184] at (.5,1.73205080756888) {even};
\path (.5,2.59807621135332) -- (1,1.73205080756888) -- (1.5,2.59807621135332) -- cycle;
\path (.5,2.59807621135332) -- (1.5,2.59807621135332) -- (1,3.46410161513775) -- cycle;
\path[opacity=0,line width=.109618489020846cm] (1.25,2.1650635094611) -- (1.06,2.59807621135332);
\path[opacity=0,line width=.109618489020846cm] (.75,3.03108891324553) -- (1.06,2.59807621135332);
%\node[black,,scale=1.67716288201895] at (1.25,2.1650635094611) {1};
%\node[black,,scale=1.39763409073845] at (1,2.59807621135332) {\(\nearrow\)1};
\path (1,3.46410161513775) -- (1.5,2.59807621135332) -- (2,3.46410161513775) -- cycle;
\path (1,3.46410161513775) -- (2,3.46410161513775) -- (1.5,4.33012701892219) -- cycle;
\path[opacity=0,line width=.109618489020846cm] (1.75,3.03108891324554) -- (1.56,3.46410161513775);
\path[opacity=0,line width=.109618489020846cm] (1.25,3.89711431702997) -- (1.56,3.46410161513775);
%\node[black,,scale=1.67716288201895] at (1.75,3.03108891324554) {1};
%\node[black,,scale=1.39763409073845] at (1.5,3.46410161513775) {\(\nearrow\)1};
\path (1.5,4.33012701892219) -- (2,3.46410161513775) -- (2.5,4.33012701892219) -- cycle;
\path (1.5,4.33012701892219) -- (2.5,4.33012701892219) -- (2,5.19615242270663) -- cycle;
\path[draw=blue,line width=.109618489020846cm] (2.25,3.89711431702997) -- (2.06,4.33012701892219);
\path[draw=blue,line width=.109618489020846cm] (1.75,4.76313972081441) -- (2.06,4.33012701892219);
\node[black,,scale=1.67716288201895] at (2.25,3.89711431702997) {2};
\node[black,,scale=1.39763409073845] at (2,4.33012701892219) {\(\nearrow\)2};
\path (2,5.19615242270663) -- (2.5,4.33012701892219) -- (3,5.19615242270663) -- cycle;
\path[draw=blue,line width=.109618489020846cm] (2.75,4.76313972081441) -- (2.56,5.19615242270663);
\node[black,,scale=1.67716288201895] at (2.75,4.76313972081441) {2};
\node[black,,scale=1.39763409073845] at (2.5,5.19615242270663) {\(\nearrow\)2};
\path (-1,1.73205080756888) -- (-.5,.866025403784439) -- (0,1.73205080756888) -- cycle;
\path (-1,1.73205080756888) -- (0,1.73205080756888) -- (-.5,2.59807621135332) -- cycle;
%\node[black,,scale=1.04822680126184] at (-.5,1.73205080756888) {even};
\path (-.5,2.59807621135332) -- (0,1.73205080756888) -- (.5,2.59807621135332) -- cycle;
\path (-.5,2.59807621135332) -- (.5,2.59807621135332) -- (0,3.46410161513775) -- cycle;
%\node[black,,scale=1.04822680126184] at (0,2.59807621135332) {even};
\path (0,3.46410161513775) -- (.5,2.59807621135332) -- (1,3.46410161513775) -- cycle;
\path (0,3.46410161513775) -- (1,3.46410161513775) -- (.5,4.33012701892219) -- cycle;
\path[opacity=0,line width=.109618489020846cm] (.75,3.03108891324553) -- (.56,3.46410161513775);
\path[opacity=0,line width=.109618489020846cm] (.25,3.89711431702997) -- (.56,3.46410161513775);
%\node[black,,scale=1.67716288201895] at (.75,3.03108891324553) {1};
%\node[black,,scale=1.39763409073845] at (.5,3.46410161513775) {\(\nearrow\)1};
\path (.5,4.33012701892219) -- (1,3.46410161513775) -- (1.5,4.33012701892219) -- cycle;
\path (.5,4.33012701892219) -- (1.5,4.33012701892219) -- (1,5.19615242270663) -- cycle;
\path[opacity=0,line width=.109618489020846cm] (1.25,3.89711431702997) -- (1.06,4.33012701892219);
\path[opacity=0,line width=.109618489020846cm] (.75,4.76313972081441) -- (1.06,4.33012701892219);
%\node[black,,scale=1.67716288201895] at (1.25,3.89711431702997) {1};
%\node[black,,scale=1.39763409073845] at (1,4.33012701892219) {\(\nearrow\)1};
\path (1,5.19615242270663) -- (1.5,4.33012701892219) -- (2,5.19615242270663) -- cycle;
\path[draw=blue,line width=.109618489020846cm] (1.75,4.76313972081441) -- (1.56,5.19615242270663);
\node[black,,scale=1.67716288201895] at (1.75,4.76313972081441) {2};
\node[black,,scale=1.39763409073845] at (1.5,5.19615242270663) {\(\nearrow\)2};
\path (-1.5,2.59807621135332) -- (-1,1.73205080756888) -- (-.5,2.59807621135332) -- cycle;
\path (-1.5,2.59807621135332) -- (-.5,2.59807621135332) -- (-1,3.46410161513775) -- cycle;
\path[opacity=0,line width=.109618489020846cm] (-.94,2.59807621135332) -- (-1.25,2.1650635094611);
\path[opacity=0,line width=.109618489020846cm] (-.94,2.59807621135332) -- (-.75,3.03108891324553);
%\node[black,,scale=1.67716288201895] at (-1.25,2.1650635094611) {1};
%\node[black,,scale=1.39763409073845] at (-1,2.59807621135332) {\(\searrow\)1};
\path (-1,3.46410161513775) -- (-.5,2.59807621135332) -- (0,3.46410161513775) -- cycle;
\path (-1,3.46410161513775) -- (0,3.46410161513775) -- (-.5,4.33012701892219) -- cycle;
\path[opacity=0,line width=.109618489020846cm] (-.44,3.46410161513775) -- (-.75,3.03108891324553);
\path[opacity=0,line width=.109618489020846cm] (-.44,3.46410161513775) -- (-.25,3.89711431702997);
%\node[black,,scale=1.67716288201895] at (-.75,3.03108891324553) {1};
%\node[black,,scale=1.39763409073845] at (-.5,3.46410161513775) {\(\searrow\)1};
\path (-.5,4.33012701892219) -- (0,3.46410161513775) -- (.5,4.33012701892219) -- cycle;
\path (-.5,4.33012701892219) -- (.5,4.33012701892219) -- (0,5.19615242270663) -- cycle;
\path[opacity=0,line width=.109618489020846cm] (.25,3.89711431702997) -- (-.25,3.89711431702997);
\path[opacity=0,line width=.109618489020846cm] (-.25,4.76313972081441) -- (.25,4.76313972081441);
%\node[black,,scale=1.67716288201895] at (.25,3.89711431702997) {1};
%\node[black,,scale=1.67716288201895] at (-.25,3.89711431702997) {1};
%\node[black,,scale=1.1979727824661] at (0,4.33012701892219) {odd};
\path (0,5.19615242270663) -- (.5,4.33012701892219) -- (1,5.19615242270663) -- cycle;
\path[opacity=0,line width=.109618489020846cm] (.75,4.76313972081441) -- (.25,4.76313972081441);
%\node[black,,scale=1.67716288201895] at (.75,4.76313972081441) {1};
%\node[black,,scale=1.67716288201895] at (.25,4.76313972081441) {1};
\node[black,,scale=1.1979727824661] at (.5,5.19615242270663) {odd};
\path (-2,3.46410161513775) -- (-1.5,2.59807621135332) -- (-1,3.46410161513775) -- cycle;
\path (-2,3.46410161513775) -- (-1,3.46410161513775) -- (-1.5,4.33012701892219) -- cycle;
\path[opacity=0,line width=.109618489020846cm] (-1.44,3.46410161513775) -- (-1.75,3.03108891324554);
\path[opacity=0,line width=.109618489020846cm] (-1.44,3.46410161513775) -- (-1.25,3.89711431702997);
%\node[black,,scale=1.67716288201895] at (-1.75,3.03108891324554) {1};
%\node[black,,scale=1.39763409073845] at (-1.5,3.46410161513775) {\(\searrow\)1};
\path (-1.5,4.33012701892219) -- (-1,3.46410161513775) -- (-.5,4.33012701892219) -- cycle;
\path (-1.5,4.33012701892219) -- (-.5,4.33012701892219) -- (-1,5.19615242270663) -- cycle;
\path[opacity=0,line width=.109618489020846cm] (-.94,4.33012701892219) -- (-1.25,3.89711431702997);
\path[opacity=0,line width=.109618489020846cm] (-.94,4.33012701892219) -- (-.75,4.76313972081441);
%\node[black,,scale=1.67716288201895] at (-1.25,3.89711431702997) {1};
%\node[black,,scale=1.39763409073845] at (-1,4.33012701892219) {\(\searrow\)1};
\path (-1,5.19615242270663) -- (-.5,4.33012701892219) -- (0,5.19615242270663) -- cycle;
\path[opacity=0,line width=.109618489020846cm] (-.25,4.76313972081441) -- (-.75,4.76313972081441);
%\node[black,,scale=1.67716288201895] at (-.25,4.76313972081441) {1};
%\node[black,,scale=1.67716288201895] at (-.75,4.76313972081441) {1};
\node[black,,scale=1.1979727824661] at (-.5,5.19615242270663) {odd};
\path (-2.5,4.33012701892219) -- (-2,3.46410161513775) -- (-1.5,4.33012701892219) -- cycle;
\path (-2.5,4.33012701892219) -- (-1.5,4.33012701892219) -- (-2,5.19615242270663) -- cycle;
\path[draw=red,line width=.109618489020846cm] (-1.94,4.33012701892219) -- (-2.25,3.89711431702997);
\path[draw=red,line width=.109618489020846cm] (-1.94,4.33012701892219) -- (-1.75,4.76313972081441);
\node[black,,scale=1.67716288201895] at (-2.25,3.89711431702997) {0};
\node[black,,scale=1.39763409073845] at (-2,4.33012701892219) {\(\searrow\)0};
\path (-2,5.19615242270663) -- (-1.5,4.33012701892219) -- (-1,5.19615242270663) -- cycle;
\path[draw=red,line width=.109618489020846cm] (-1.44,5.19615242270663) -- (-1.75,4.76313972081441);
\node[black,,scale=1.67716288201895] at (-1.75,4.76313972081441) {0};
\node[black,,scale=1.39763409073845] at (-1.5,5.19615242270663) {\(\searrow\)0};
\path (-3,5.19615242270663) -- (-2.5,4.33012701892219) -- (-2,5.19615242270663) -- cycle;
\path[draw=red,line width=.109618489020846cm] (-2.44,5.19615242270663) -- (-2.75,4.76313972081441);
\node[black,,scale=1.67716288201895] at (-2.75,4.76313972081441) {0};
\node[black,,scale=1.39763409073845] at (-2.5,5.19615242270663) {\(\searrow\)0};
\end{scope}
\end{tikzpicture}
\end{center}
Next, there must be an odd number of labels on the sides of the remaining
bottom triangles. Because of the content of $\lambda$ and $\mu$, only
labels $k$ can be used; this means that this odd set of labels is the
singleton $k$. There is then a unique way to complete this into a path
labeled $k$; and we can iterate the process, resulting in the unique puzzle
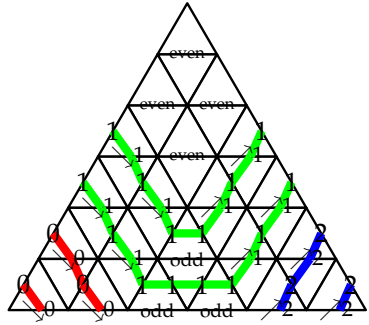
\begin{center}
\begin{tikzpicture}[scale=0.5,every node/.style={scale=0.5},execute at begin picture={\bgroup\tikzset{every path/.style={}}\clip (-3.72,-.637138) rectangle ++(7.44,6.58376);\egroup},x={(1.56597774923292cm,0cm)},y={(0cm,-1.56597907993384cm)},baseline=(current  bounding  box.center),every path/.style={draw=black,fill=none},line join=round]
\begin{scope}[every path/.append style={fill=white,line width=.0313195682916704cm}]
\path (-.5,.866025403784439) -- (0,0) -- (.5,.866025403784439) -- cycle;
\path (-.5,.866025403784439) -- (.5,.866025403784439) -- (0,1.73205080756888) -- cycle;
\node[black,,scale=1.04822680126184] at (0,.866025403784439) {even};
\path (0,1.73205080756888) -- (.5,.866025403784439) -- (1,1.73205080756888) -- cycle;
\path (0,1.73205080756888) -- (1,1.73205080756888) -- (.5,2.59807621135332) -- cycle;
\node[black,,scale=1.04822680126184] at (.5,1.73205080756888) {even};
\path (.5,2.59807621135332) -- (1,1.73205080756888) -- (1.5,2.59807621135332) -- cycle;
\path (.5,2.59807621135332) -- (1.5,2.59807621135332) -- (1,3.46410161513775) -- cycle;
\path[draw=green,line width=.109618489020846cm] (1.25,2.1650635094611) -- (1.06,2.59807621135332);
\path[draw=green,line width=.109618489020846cm] (.75,3.03108891324553) -- (1.06,2.59807621135332);
\node[black,,scale=1.67716288201895] at (1.25,2.1650635094611) {1};
\node[black,,scale=1.39763409073845] at (1,2.59807621135332) {\(\nearrow\)1};
\path (1,3.46410161513775) -- (1.5,2.59807621135332) -- (2,3.46410161513775) -- cycle;
\path (1,3.46410161513775) -- (2,3.46410161513775) -- (1.5,4.33012701892219) -- cycle;
\path[draw=green,line width=.109618489020846cm] (1.75,3.03108891324554) -- (1.56,3.46410161513775);
\path[draw=green,line width=.109618489020846cm] (1.25,3.89711431702997) -- (1.56,3.46410161513775);
\node[black,,scale=1.67716288201895] at (1.75,3.03108891324554) {1};
\node[black,,scale=1.39763409073845] at (1.5,3.46410161513775) {\(\nearrow\)1};
\path (1.5,4.33012701892219) -- (2,3.46410161513775) -- (2.5,4.33012701892219) -- cycle;
\path (1.5,4.33012701892219) -- (2.5,4.33012701892219) -- (2,5.19615242270663) -- cycle;
\path[draw=blue,line width=.109618489020846cm] (2.25,3.89711431702997) -- (2.06,4.33012701892219);
\path[draw=blue,line width=.109618489020846cm] (1.75,4.76313972081441) -- (2.06,4.33012701892219);
\node[black,,scale=1.67716288201895] at (2.25,3.89711431702997) {2};
\node[black,,scale=1.39763409073845] at (2,4.33012701892219) {\(\nearrow\)2};
\path (2,5.19615242270663) -- (2.5,4.33012701892219) -- (3,5.19615242270663) -- cycle;
\path[draw=blue,line width=.109618489020846cm] (2.75,4.76313972081441) -- (2.56,5.19615242270663);
\node[black,,scale=1.67716288201895] at (2.75,4.76313972081441) {2};
\node[black,,scale=1.39763409073845] at (2.5,5.19615242270663) {\(\nearrow\)2};
\path (-1,1.73205080756888) -- (-.5,.866025403784439) -- (0,1.73205080756888) -- cycle;
\path (-1,1.73205080756888) -- (0,1.73205080756888) -- (-.5,2.59807621135332) -- cycle;
\node[black,,scale=1.04822680126184] at (-.5,1.73205080756888) {even};
\path (-.5,2.59807621135332) -- (0,1.73205080756888) -- (.5,2.59807621135332) -- cycle;
\path (-.5,2.59807621135332) -- (.5,2.59807621135332) -- (0,3.46410161513775) -- cycle;
\node[black,,scale=1.04822680126184] at (0,2.59807621135332) {even};
\path (0,3.46410161513775) -- (.5,2.59807621135332) -- (1,3.46410161513775) -- cycle;
\path (0,3.46410161513775) -- (1,3.46410161513775) -- (.5,4.33012701892219) -- cycle;
\path[draw=green,line width=.109618489020846cm] (.75,3.03108891324553) -- (.56,3.46410161513775);
\path[draw=green,line width=.109618489020846cm] (.25,3.89711431702997) -- (.56,3.46410161513775);
\node[black,,scale=1.67716288201895] at (.75,3.03108891324553) {1};
\node[black,,scale=1.39763409073845] at (.5,3.46410161513775) {\(\nearrow\)1};
\path (.5,4.33012701892219) -- (1,3.46410161513775) -- (1.5,4.33012701892219) -- cycle;
\path (.5,4.33012701892219) -- (1.5,4.33012701892219) -- (1,5.19615242270663) -- cycle;
\path[draw=green,line width=.109618489020846cm] (1.25,3.89711431702997) -- (1.06,4.33012701892219);
\path[draw=green,line width=.109618489020846cm] (.75,4.76313972081441) -- (1.06,4.33012701892219);
\node[black,,scale=1.67716288201895] at (1.25,3.89711431702997) {1};
\node[black,,scale=1.39763409073845] at (1,4.33012701892219) {\(\nearrow\)1};
\path (1,5.19615242270663) -- (1.5,4.33012701892219) -- (2,5.19615242270663) -- cycle;
\path[draw=blue,line width=.109618489020846cm] (1.75,4.76313972081441) -- (1.56,5.19615242270663);
\node[black,,scale=1.67716288201895] at (1.75,4.76313972081441) {2};
\node[black,,scale=1.39763409073845] at (1.5,5.19615242270663) {\(\nearrow\)2};
\path (-1.5,2.59807621135332) -- (-1,1.73205080756888) -- (-.5,2.59807621135332) -- cycle;
\path (-1.5,2.59807621135332) -- (-.5,2.59807621135332) -- (-1,3.46410161513775) -- cycle;
\path[draw=green,line width=.109618489020846cm] (-.94,2.59807621135332) -- (-1.25,2.1650635094611);
\path[draw=green,line width=.109618489020846cm] (-.94,2.59807621135332) -- (-.75,3.03108891324553);
\node[black,,scale=1.67716288201895] at (-1.25,2.1650635094611) {1};
\node[black,,scale=1.39763409073845] at (-1,2.59807621135332) {\(\searrow\)1};
\path (-1,3.46410161513775) -- (-.5,2.59807621135332) -- (0,3.46410161513775) -- cycle;
\path (-1,3.46410161513775) -- (0,3.46410161513775) -- (-.5,4.33012701892219) -- cycle;
\path[draw=green,line width=.109618489020846cm] (-.44,3.46410161513775) -- (-.75,3.03108891324553);
\path[draw=green,line width=.109618489020846cm] (-.44,3.46410161513775) -- (-.25,3.89711431702997);
\node[black,,scale=1.67716288201895] at (-.75,3.03108891324553) {1};
\node[black,,scale=1.39763409073845] at (-.5,3.46410161513775) {\(\searrow\)1};
\path (-.5,4.33012701892219) -- (0,3.46410161513775) -- (.5,4.33012701892219) -- cycle;
\path (-.5,4.33012701892219) -- (.5,4.33012701892219) -- (0,5.19615242270663) -- cycle;
\path[draw=green,line width=.109618489020846cm] (.25,3.89711431702997) -- (-.25,3.89711431702997);
\path[draw=green,line width=.109618489020846cm] (-.25,4.76313972081441) -- (.25,4.76313972081441);
\node[black,,scale=1.67716288201895] at (.25,3.89711431702997) {1};
\node[black,,scale=1.67716288201895] at (-.25,3.89711431702997) {1};
\node[black,,scale=1.1979727824661] at (0,4.33012701892219) {odd};
\path (0,5.19615242270663) -- (.5,4.33012701892219) -- (1,5.19615242270663) -- cycle;
\path[draw=green,line width=.109618489020846cm] (.75,4.76313972081441) -- (.25,4.76313972081441);
\node[black,,scale=1.67716288201895] at (.75,4.76313972081441) {1};
\node[black,,scale=1.67716288201895] at (.25,4.76313972081441) {1};
\node[black,,scale=1.1979727824661] at (.5,5.19615242270663) {odd};
\path (-2,3.46410161513775) -- (-1.5,2.59807621135332) -- (-1,3.46410161513775) -- cycle;
\path (-2,3.46410161513775) -- (-1,3.46410161513775) -- (-1.5,4.33012701892219) -- cycle;
\path[draw=green,line width=.109618489020846cm] (-1.44,3.46410161513775) -- (-1.75,3.03108891324554);
\path[draw=green,line width=.109618489020846cm] (-1.44,3.46410161513775) -- (-1.25,3.89711431702997);
\node[black,,scale=1.67716288201895] at (-1.75,3.03108891324554) {1};
\node[black,,scale=1.39763409073845] at (-1.5,3.46410161513775) {\(\searrow\)1};
\path (-1.5,4.33012701892219) -- (-1,3.46410161513775) -- (-.5,4.33012701892219) -- cycle;
\path (-1.5,4.33012701892219) -- (-.5,4.33012701892219) -- (-1,5.19615242270663) -- cycle;
\path[draw=green,line width=.109618489020846cm] (-.94,4.33012701892219) -- (-1.25,3.89711431702997);
\path[draw=green,line width=.109618489020846cm] (-.94,4.33012701892219) -- (-.75,4.76313972081441);
\node[black,,scale=1.67716288201895] at (-1.25,3.89711431702997) {1};
\node[black,,scale=1.39763409073845] at (-1,4.33012701892219) {\(\searrow\)1};
\path (-1,5.19615242270663) -- (-.5,4.33012701892219) -- (0,5.19615242270663) -- cycle;
\path[draw=green,line width=.109618489020846cm] (-.25,4.76313972081441) -- (-.75,4.76313972081441);
\node[black,,scale=1.67716288201895] at (-.25,4.76313972081441) {1};
\node[black,,scale=1.67716288201895] at (-.75,4.76313972081441) {1};
\node[black,,scale=1.1979727824661] at (-.5,5.19615242270663) {odd};
\path (-2.5,4.33012701892219) -- (-2,3.46410161513775) -- (-1.5,4.33012701892219) -- cycle;
\path (-2.5,4.33012701892219) -- (-1.5,4.33012701892219) -- (-2,5.19615242270663) -- cycle;
\path[draw=red,line width=.109618489020846cm] (-1.94,4.33012701892219) -- (-2.25,3.89711431702997);
\path[draw=red,line width=.109618489020846cm] (-1.94,4.33012701892219) -- (-1.75,4.76313972081441);
\node[black,,scale=1.67716288201895] at (-2.25,3.89711431702997) {0};
\node[black,,scale=1.39763409073845] at (-2,4.33012701892219) {\(\searrow\)0};
\path (-2,5.19615242270663) -- (-1.5,4.33012701892219) -- (-1,5.19615242270663) -- cycle;
\path[draw=red,line width=.109618489020846cm] (-1.44,5.19615242270663) -- (-1.75,4.76313972081441);
\node[black,,scale=1.67716288201895] at (-1.75,4.76313972081441) {0};
\node[black,,scale=1.39763409073845] at (-1.5,5.19615242270663) {\(\searrow\)0};
\path (-3,5.19615242270663) -- (-2.5,4.33012701892219) -- (-2,5.19615242270663) -- cycle;
\path[draw=red,line width=.109618489020846cm] (-2.44,5.19615242270663) -- (-2.75,4.76313972081441);
\node[black,,scale=1.67716288201895] at (-2.75,4.76313972081441) {0};
\node[black,,scale=1.39763409073845] at (-2.5,5.19615242270663) {\(\searrow\)0};
\end{scope}
\end{tikzpicture}
\end{center}
\end{proof}
We now fix the normalizing constant $c(z)$ in \eqref{eq:defRD} by requiring that the puzzle in the Proposition above have a fugacity of $1$.
We have the following
\begin{lem}
Define
\[
  c(z) :=
  \begin{cases}
    \displaystyle
    \prod_{i=1}^{(d+1)/2} (q^{2i-2}z-q^{-2i+2})&d\text{ odd}
    \\
    \displaystyle
    \ \ \prod_{i=1}^{d/2}\ \ (q^{2i-1}z-q^{-2i+1})&d\text{ even}
  \end{cases}
\]
as well as $c_U=c_D=1$.
Then the fugacity of every rhombus and triangle contributing to the puzzle of Proposition~\ref{prop:trivpuzzled} is $1$.
\end{lem}
\begin{proof}
Let us pick for example the blank rhombus (the same reasoning applies to all other rhombi in the puzzle, resulting in the same fugacity).
Denote $a_{j,d+1}=\left<x_\blank\otimes x_\blank, P_j x_\blank\otimes x_\blank\right>$ where we emphasize the dependence on $d$.
According to \cite[Eq.~(5.13)]{Okado-BD}, the following recurrence relation holds:
\begin{gather*}
a_{j,\ell}=f(-j+1)a_{j-1,\ell-1}+f(j+1)a_{j+1,\ell-1}
\\
a_{j,\ell}=0\qquad j<0\text{ or }j>\ell
\\
f(j)=\begin{cases}
1&j=0\\
(-1)^{j+1}/(q^j+q^{-j})&j>0
\\
(-1)^{j}/(q^j+q^{-j})&j<0
\end{cases}\ .
\end{gather*}
This recurrence can be easily solved; writing $[m]=q^{m}-q^{-m}$, one has
\[
a_{j,\ell}
=(-1)^{\lfloor j/2\rfloor}
\frac{\prod_{i=1}^{\lceil \ell/2\rceil}[2i-1]}
{\prod_{i=1}^{(\ell-j)/2}[2i]\prod_{i=\lceil (\ell+1)/2\rceil}^{(\ell+j)/2}[2i]}
\begin{cases}
\frac{[2j]}{[j]}&j>0
\\
1&j=0
\end{cases}
\qquad
j\equiv \ell\pmod 2\ .
\]
Finally, consider
\[
  c(z)-\sum_{\substack{j=0\\j\equiv d+1\\(\text{mod }2)}}^{d+1} \rho_j(z) \ a_{j,d+1}
\]
where $c(z)$ is given as in the Lemma. This is a polynomial of degree at most $\lceil d/2\rceil$ in $z$, and its evaluation
at $z=q^{-2d+4i}$, $i=0,\ldots,\lceil d/2\rceil$, is easily seen to be zero. Therefore it is zero and
\[
\left<x_\blank\otimes x_\blank, \ \check R_{1,2}(z) x_\blank\otimes x_\blank\right>=1\ .
\]

The analysis is simpler for the bottom row of up-pointing triangles. There are three types of triangles, which we convert
to the notation of \eqref{eq:defUD}:
\begin{itemize}
\item \uptri{i}{\da i}{} where $i<k$, with fugacity $\left<-x_i,U(\frac{1}{2}\vec 1-x_i-x_{\blank})\otimes(-\frac{1}{2}\vec1+x_{\blank})\right>$.
\item \uptri{k}{odd}{k} with fugacity $\left<-x_{\blank},U(\frac{1}{2}\vec 1-x_{k}-x_{\blank})\otimes(-\frac{1}{2}\vec1+x_{k})\right>$.
\item \uptri{}{\ua j}{j} where $j>k$, with fugacity $\left<x_{j},U(\frac{1}{2}\vec 1)\otimes(-\frac{1}{2}\vec1+x_{j})\right>$.
\end{itemize}
Paying attention to the ordering of the labels, one checks in each case from \eqref{eq:defUD} that the matrix entry is $1$ provided
$c_U=1$. Imposing $\check R_{1,2}(z=q^{-2d})=DU$ also fixes $c_D=1$.
\end{proof}

This concludes the proof of all the required properties of \S\ref{sec:setup}, which means Theorem~\ref{thm:motivic} provides us
with a puzzle formula for the product of the pullbacks of two (equivariant) motivic Segre classes 
in $\mathcal F_1$ and $\mathcal F_2$ (where the flag dimensions satisfy the ``almost separated descent'' condition, i.e., the overlap
of dimensions of $\mathcal F_1$ and $\mathcal F_2$ is at most $2$) to their common refinement $\mathcal F_3$.
We have not provided the explicit fugacities of the rhombi (the entries of $\check R_{1,2}$),
though they can be extracted from \cite[\S 5]{Okado-BD}. In what follows,
we only ever consider the nonequivariant case of the theorem, which requires the knowledge of the entries of $U$ and $D$ only,
given in \eqref{eq:defUD} and \eqref{eq:defDD} respectively.

Such nonequivariant generic puzzles can be described as follows: they are colored lattice paths going East, NorthEast or SouthEast,
with only two constraints: paths of the same color cannot touch, and in a given triangle at most one path can deviate from the
horizontal.
\junk{reformulate nonequivariant puzzles as paths going E, NE, SE,
such that in a given triangle at most one path can deviate from horizontal}

At $q=-1$, the fugacities of all triangles become $1$. Thus the number
of nonequivariant generic puzzles with sides $\lambda,\mu,\overleftarrow\nu$
is $(-1)^{\dim Y}\chi(Y)$ for the triple intersection $Y$ in
Theorem~\ref{thm:euler}.

\begin{ex}Consider $\pi=15342$ and $\rho=21435$, with $O(\pi,\rho)=\{2,3\}$:
\[
\begin{tikzpicture}[xscale=1.05,yscale=.5]
\node at (-1,4) {$\pi$}; \foreach \s [count=\x] in {1,5,3,4,2} \node at (\x,4) {$\s$};
\node at (-1,3) {$\omega_A$}; \foreach \s [count=\x] in {\blank,\blank,2,3,4} \node at (\x,3) {$\vphantom{0}\s$};
\node at (-1,2) {$\omega_C$}; \foreach \s [count=\x] in {\llap{$\da$} 0,\llap{$\da$} 1,odd,\llap{$\ua$} 3,\llap{$\ua$} 4} \node at (\x,2) {$\s$};
\node at (-1,1) {$\omega_B$}; \foreach \s [count=\x] in {0,1,2,\blank,\blank} \node at (\x,1) {$\vphantom{0}\s$};
\node at (-1,0) {$\rho$}; \foreach \s [count=\x] in {2,1,4,3,5} \node at (\x,0) {$\s$};
%\foreach \s [count=\x] in {1,3,6,2,5,4,7} \draw[-latex,thick] (\x,-0.8) -- (\s,-2.2);
%\foreach \s [count=\x] in {1,3,6,2,5,4,7} \draw[-latex,thick] (\x,-0.8) .. controls (0.85*\x+0.15*\s,-1.5) and (0.85*\s+0.15*\x,-1.5) .. (\s,-2.2);
%\node at (-1,-3) {$\lambda$}; \foreach \s [count=\x] in {0,1,0,2,1,0,2} \node at (\x,-3) {$\s$};
\draw (3.5,3.5) -- (3.5,-0.5);
\draw (4.5,4.5) -- (4.5,1.5);
\draw (1.5,2.5) -- (1.5,-0.5);
\draw (2.5,4.5) -- (2.5,0.5);
\end{tikzpicture}
\]
Let us also choose $\sigma=13254$, so the corresponding strings are
$\lambda=10\blank2\blank$, $\mu=\blank423\blank$, $\overleftarrow\nu=\ua\!3\ua\!4\da\!1\,odd\da\!0$;
there are 3 nonequivariant generic puzzles:
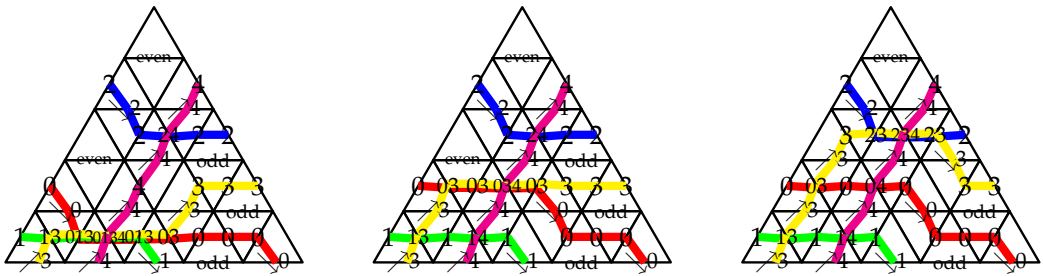
\begin{center}
\tikzset{every picture/.style={scale=0.5,every node/.style={scale=0.5}}}
%"ex5.tex" << concatenate(tex\ (puzzle("10_2_","_423_",o23,Paths=>true,Equivariant=>false))) << close
\input ex5.tex
\end{center}
One can check that the triple intersection of the rotated preimage of the Schubert cell $X^\pi_\circ$ in $Gr(1,2,3;5)$ (resp.\ $X^\rho_\circ$ in $Gr(2,3,4;5)$)
and of the Schubert cell $X^\sigma_\circ$ in the full flag variety of $\CC^5$ is a $\mathbb P^1$ minus 5 points, which has Euler characteristic $-3$,
as predicted by Theorem~\ref{thm:euler}.
\end{ex}

\subsection{The $B$-matrix}\label{ssec:almostspecdescB}
Let $B$ be the skew-symmetric form with
\begin{align*}
B(x_i,x_j)&=\sign(i-j)
&
B(x_i,x_\blank)&=0
\\
B(x_i,y_1)&=\frac{d+1}{2}-i
&
B(x_\blank,y_1)&=\frac{1}{2}
\\
B(x_i,y_2)&=i-\frac{d-1}{2}
&
B(x_\blank,y_2)&=-\frac{1}{2}
\\
B(y_1,y_2)&=-\frac{d}{2}
\end{align*}
where $i,j=0,\ldots,d$.
\rem[gray]{note that $B(y_1,y_2)$ is just an overall scaling of the
  puzzle fugacities.  it doesn't affect the type $A$ inversion
  calculation, but is needed in making sure $H$-triangles have
  inversion charge zero (among which the normalization condition of
  Prop (c)). same remark would apply in sep desc BTW}

Once again, we need to check Lemma~\ref{lem:twist}. To help with the
calculation, we first calculate
\[
\begin{split}
B(\vec 1/2+y_1,x_i)&=-1/2
\\
B(-\vec 1/2+y_2,x_i)&=-1/2
\end{split}
\qquad  i=0,\ldots,d\ .
\]

We then compute $B(\wt(e_{a,i}),\wt(e_{a,j}))$ case by case:
\begin{itemize}
\item $a=1$: the weights are $\vec 1/2-(x_i+x_\blank)+y_1$ for $i=0<\cdots<k$ and $\vec 1/2+y_1$ for $i=\blank$ ($>$all):
\begin{align*}
&B(\vec 1/2-(x_i+x_\blank)+y_1,\vec 1/2-(x_j+x_\blank)+y_1)\\
&\qquad=-B(\vec 1/2-x_\blank+y_1,x_j)-B(x_i,\vec 1/2-x_\blank+y_1)+B(x_i,x_j)
\\
&\qquad=\sign(i-j)
\\
&B(\vec 1/2-(x_i+x_\blank)+y_1,\vec 1/2+y_1)=-1/2-1/2=-1\ .
\end{align*}
\item $a=2$: the weights are $-\vec 1/2+x_i+y_2$ for $i=\blank<k<\cdots<d$.
\begin{align*}
&B(-\vec 1/2+x_\blank+y_2,-\vec 1/2+x_i+y_2)=-1/2-1/2=-1
\\
&B(-\vec 1/2+x_i+y_2,-\vec 1/2+x_j+y_2)\\
&\qquad=B(-\vec 1/2+y_2,x_j)+B(x_i,-\vec 1/2+y_2)+B(x_i,x_j)
\\
&\qquad=\sign(i-j)\qquad i,j\ne\blank\ .
\end{align*}
\item $a=3$: the weights are $-x_i+y_1+y_2$, $i=0<\cdots<k-1$, $-x_\blank+y_1+y_2$ for $i=odd$ (which for ordering purposes is like $k$), 
and $+x_i+y_1+y_2$ for $i=k+1<\cdots<d$.
\begin{align*}
B(-x_i+y_1+y_2,-x_j+y_1+y_2)&=\sign(i-j)
\\
B(-x_i+y_1+y_2,-x_\blank+y_1+y_2)&=-1
\\
B(x_i+y_1+y_2,-x_\blank+y_1+y_2)&=+1
\\
B(-x_i+y_1+y_2,x_j+y_1+y_2)&=1-2=-1    & i<k<j
\\
B(x_i+y_1+y_2,x_j+y_1+y_2)&=\sign(i-j)\ .
\end{align*}
\end{itemize}

\subsection{The limit \texorpdfstring{$q\to 0$}{q -> 0}}\label{ssec:almostsepdescq0}
\rem[gray]{see {\tt Rmatrix-Dn.m2}, {\tt Dpuzzles2.m2}.}
We are now ready to perform the limit $q\to0$. We start from the expressions \eqref{eq:defUD} and \eqref{eq:defDD}
for $U$ and $D$. Define for convenience
\[
r(X)=\sum_{j\in X}(d+1-w(j))
\]
where $w$ was defined in \eqref{eq:mapping}. More explicitly
\begin{align*}
r(X)&=r_<(X)+r_>(X)
&\qquad\qquad
r_<(X)&=\sum_{a\in X,\, a<k} (d-a)
&\qquad\qquad
r_>(X)&=\sum_{a\in X,\, a> k}(a-k)\ .
\end{align*}
This allows us to rephrase the matrix entries of $U$ and $D$ in terms
of the ``unintuitive'' indexing of \S\ref{ssec:almostsepdescdata}.

We focus on $U$ first.
Since all entries are powers of $-q$, we only write those powers below:
\begin{align*}
\log_{-q} U^{iX,X}_{\da i} = \log_{-q}  \uptri {iX} {\da i} {X} &=r(X)+\#\{a\in X: w(a)>w(i)\}
\\
  &= r(X)+\begin{cases} \#\{a\in X:a>i\}& i<k \\ \#\{a\in X: k\le a<i\} & i\ge k 
  \end{cases}
  \\
\log_{-q} U^{X,Xj}_{\ua j} = \log_{-q}  \uptri {X} {\ua j} {Xj} &= r(X)+\#\{a\in X: w(a)>w(j)\}
\\
  &= r(X)+\begin{cases} \#\{a\in X:a>j\}& j<k \\ \#\{a\in X: k\le a<j\} & j\ge k 
  \end{cases}
\\
\log_{-q} U^{X,X}_{even}=\log_{-q}  \uptri {X} {even} {X} &= r(X)
\\
\log_{-q} U^{X,X}_{odd}=\log_{-q}  \uptri {X} {odd} {X} &= r(X)\ .
\end{align*}

Next we are supposed to apply the twist, and conjugate the matrices. For reasons which will become clear below, it is more
convenient to apply those two (commuting) operations in the reverse order.
With $X_{\ge k}:=\{a\in X:\,a\ge k\}$, the change of basis is:
\begin{align}\notag
&\text{in $V_1$:}\quad e'_{1,X}=(-q)^{-r_>(X)+\binom{\#X_{\ge k}}{2}-\lfloor(\#X)^2/4\rfloor}e_{1,X}
\\\label{eq:almostsepdescconj}
&\text{in $V_2$:}\quad e'_{2,X}=(-q)^{-r_<(X)-\binom{\#X_{\ge k}}{2}+\lfloor(\#X-1)^2/4\rfloor}e_{2,X}
\\\notag
&\text{in $V_3$:}\quad e'_{3,\ua i}=(-q)^{-r_<(i)}e_{3,\ua i}\text{ and }e'_{3,\da i}=(-q)^{-r_>(i)}e_{3,\da i}\ .
\end{align}
The vectors labeled $even$ and $odd$ are unchanged.
Note that none of the basis vectors in $V_a^A$ are affected by such a transformation.

Under this change of basis, the exponent of each entry of $U$ gains the
two input exponents minus the output exponent. Substitution gives
\begin{align*}
\log_{-q} U'{}^{iX,X}_{\da i} = \log_{-q}  \uptri {iX} {\da i} {X} &=
\sum_{a\in X:\, a<i} (-1)^{[a<k]}
\\
&  = \begin{cases} -\#\{a\in X: a<i\} & i<k \\ -\#\{a\in X: a<k\} + \#\{a\in X: k\le a<i\} & i\ge k
  \end{cases}
  \\
\log_{-q} U'{}^{X,Xj}_{\ua j} = \log_{-q}  \uptri {X} {\ua j} {Xj} &= \# \{a\in X:\, a>j\} \ (-1)^{[j\ge k]}
\\
\log_{-q} U'{}^{X,X}_{even}=\log_{-q}  \uptri {X} {even} {X} &= - \#X / 2
\\
\log_{-q} U'{}^{X,X}_{odd}=\log_{-q}  \uptri {X} {odd} {X} &= - (\#X-1)/2
\end{align*}
where $[\text{true}] := 1$, $[\text{false}] := 0$.

Finally, we apply the twist by computing the inversion charges of triangles; one has from \S\ref{ssec:almostspecdescB}
\[
\frac{1}{2}B(e_{2,Y},e_{1,X})=\frac{1}{2}\sum_{x\in X,y\in Y} \sign(x-y) + \frac{1}{8}((-1)^{\# X}+(-1)^{\# Y}) + \frac{1}{4}(\# X+\# Y-1)
\]
so that
\begin{align*}
\inv(\uptri {X} {\sss \parity(\# X)} {X}) &= \frac{1}{4} (-1)^{\#X} +\frac{1}{2}\# X -\frac{1}{4} = \lfloor \# X/2 \rfloor
\\
\inv(\uptri {iX} {\da i} {X}) &= \frac{1}{2}\sum_{y\in X} \sign(i-y) + \frac{1}{2} \#X = \# \{ y\in X : y<i \}
\\
\inv(\uptri {X} {\ua j} {Xj}) &= \frac{1}{2}\sum_{x\in X} \sign(x-j) + \frac{1}{2} \#X = \# \{ x\in X : x>j \}\ .
\end{align*}
This matches with what was announced in \eqref{eq:inv}.

After twisting, the entries of $U$ will acquire an extra power of $q$, which we write as $(-1)\times(-q)$,
and set aside the $(-1)^{\inv}$, resulting in:
$$
\begin{aligned}
\tilde U'{}^{iX,X}_{\da i} = \uptri {iX} {\da i} {X} &=(-1)^\inv (-q)^{2\# \{a\in X: k\le a<i\}}
\\
\tilde U'{}^{X,Xj}_{\ua j} = \uptri {X} {\ua j} {Xj} &= (-1)^\inv (-q)^{2\# \{a\in X: a>j\} [j<k]}
\end{aligned}
\qquad\qquad\qquad
\begin{aligned}
\tilde U'{}^{X,X}_{even}=\uptri {X} {even} {X} &= (-1)^\inv
\\
\tilde U'{}^{X,X}_{odd}=\uptri {X} {odd} {X} &= (-1)^\inv
\end{aligned}
$$
which clearly leads at $q\to0$ to the rule as stated in Theorem~\ref{thm:almostsepdesc}.

We must perform the same calculation for $D$. Comparing the entries \eqref{eq:defDD} of $D$ to those \eqref{eq:defUD} of $U$,
we note that the powers of $-q$ are opposite. Furthermore, changes of basis affect $U$ and $D$ in opposite ways, so the same fact
holds for the modified entries $U'$ and $D'$. On the other hand, the twist of a down pointing triangle
is the same as that of its 180 degree rotated/arrow reverted version, and it affects $U$ and $D$ identically; so we only have
to redo the last step of the computation for $D$, and we obtain
$$
\begin{aligned}
\tilde D'{}_{X,iX}^{\da i} =   \downtri {iX} {\da i} {X} &=(-1)^\inv (-q)^{2\# \{a\in X: a<\min(i,k)\}}
\\
\tilde D'{}_{Xj,X}^{\ua j} = \downtri {X} {\ua j} {Xj} &= (-1)^\inv (-q)^{2\# \{a\in X: a>j\} [j\ge k]}
\end{aligned}
\qquad\quad
\begin{aligned}
\tilde D'{}_{X,X}^{even}=\downtri {X} {even} {X} &= (-1)^\inv (-q)^{\# X}
\\
\tilde D'{}_{X,X}^{odd}=\downtri {X} {odd} {X} &= (-1)^\inv (-q)^{\# X-1}\ .
\end{aligned}
$$
Again, we recover at $q\to0$ the $K$-pieces of Theorem~\ref{thm:almostsepdesc}.

\junk{AK: I don't really know this story well. Can we include something
  that says, if the tensor product has only two components then here's
  how to achieve equivariant? Which wouldn't prove that $>2$ means
  no chance, but, would at least say why we're not trying to do it for
  $3$-step or almost-separated descent. PZJ: I don't really know of such a statement, but I added a remark at the end of 4.2,
  feel free to modify}

\section{Comparison of puzzle rules}
% \rem[gray]{maybe all discussions can be put in an extra section ``more examples and discussion''
%   which should include: how $A_n$ generalizes $A_2$; how $D_n$ generalizes $D_4$ (and doesn't -- no equiv);
%   comparison between various rules (vs 3,4-step; sep desc vs almost sep desc). some of the examples can be nondegenerate too.}
%
%\rem{some samples in comment. 
%puzzle("1_20","_2_3",Generic=>false,Equivariant=>false,Ktheory=>true,Paths=>true)
%puzzle("1320","0213",Generic=>false,Equivariant=>false,Ktheory=>true)
%}

\subsection{Separated-descent vs. Grassmannian puzzles}
Grassmannian puzzles \cite{KT} are based on the algebra $\lie{a}_2$
(this fact was known as early as \cite{artic46}, though it was only
properly explained in \cite{artic71}); separated descent puzzles are
based on $\lie{a}_{d+1}$. Furthermore, the separated descent condition
$\overlap(\pi,\rho)\le 1$ is implied by the more restrictive Grassmannian
condition $\desc(\pi,\rho)\le 1$.  Therefore one expects a relation
between the two rules.

Indeed, there is a simple bijection that converts Knutson--Tao puzzles to separated descent puzzles (with $k=0$, $d=1$): replace edge labels as follows
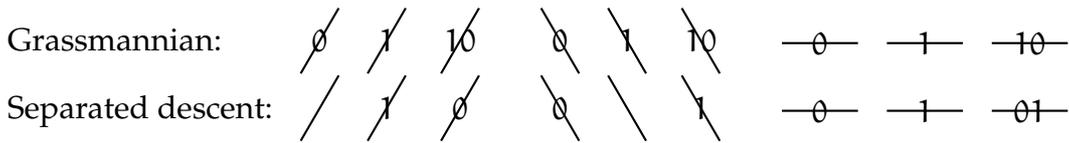
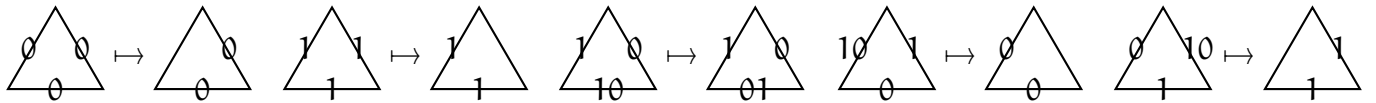
\[
\begin{array}{lccccccccccc}
\text{Grassmannian:}
&
\tikz[baseline=0.3cm]{\draw[thick] (0,0) -- node{$0$} (60:1);}
&
\tikz[baseline=0.3cm]{\draw[thick] (0,0) -- node{$1$} (60:1);} 
&
\tikz[baseline=0.3cm]{\draw[thick] (0,0) -- node{$10$} (60:1);}
&
&
\tikz[baseline=0.3cm]{\draw[thick] (0,0) -- node{$0$} (120:1);}
&
\tikz[baseline=0.3cm]{\draw[thick] (0,0) -- node{$1$} (120:1);}
&
\tikz[baseline=0.3cm]{\draw[thick] (0,0) -- node{$10$} (120:1);}
&
&
\tikz[baseline=-1mm]{\draw[thick] (0,0) -- node{$0$} (0:1);}
&
\tikz[baseline=-1mm]{\draw[thick] (0,0) -- node{$1$} (0:1);}
&
\tikz[baseline=-1mm]{\draw[thick] (0,0) -- node{$10$} (0:1);}
\\
\text{Separated descent:}
&
\tikz[baseline=0.3cm]{\draw[thick] (0,0) -- node{} (60:1);}
&
\tikz[baseline=0.3cm]{\draw[thick] (0,0) -- node{$1$} (60:1);}
&
\tikz[baseline=0.3cm]{\draw[thick] (0,0) -- node{$0$} (60:1);}
&
&
\tikz[baseline=0.3cm]{\draw[thick] (0,0) -- node{$0$} (120:1);}
&
\tikz[baseline=0.3cm]{\draw[thick] (0,0) -- node{} (120:1);}
&
\tikz[baseline=0.3cm]{\draw[thick] (0,0) -- node{$1$} (120:1);}
&
&
\tikz[baseline=-1mm]{\draw[thick] (0,0) -- node{$0$} (0:1);}
&
\tikz[baseline=-1mm]{\draw[thick] (0,0) -- node{$1$} (0:1);}
&
\tikz[baseline=-1mm]{\draw[thick] (0,0) -- node{$01$} (0:1);}
\end{array}
\]

% \begin{align*}
% \tikz[baseline=0.3cm]{\draw[thick] (0,0) -- node{$0$} (60:1);} &\mapsto \tikz[baseline=0.3cm]{\draw[thick] (0,0) -- node{} (60:1);}
% &
% \tikz[baseline=0.3cm]{\draw[thick] (0,0) -- node{$1$} (60:1);} &\mapsto \tikz[baseline=0.3cm]{\draw[thick] (0,0) -- node{$1$} (60:1);}
% &
% \tikz[baseline=0.3cm]{\draw[thick] (0,0) -- node{$10$} (60:1);} &\mapsto \tikz[baseline=0.3cm]{\draw[thick] (0,0) -- node{$0$} (60:1);}
% \\
% \tikz[baseline=0.3cm]{\draw[thick] (0,0) -- node{$0$} (120:1);} &\mapsto \tikz[baseline=0.3cm]{\draw[thick] (0,0) -- node{$0$} (120:1);}
% &
% \tikz[baseline=0.3cm]{\draw[thick] (0,0) -- node{$1$} (120:1);} &\mapsto \tikz[baseline=0.3cm]{\draw[thick] (0,0) -- node{$$} (120:1);}
% &
% \tikz[baseline=0.3cm]{\draw[thick] (0,0) -- node{$10$} (120:1);} &\mapsto \tikz[baseline=0.3cm]{\draw[thick] (0,0) -- node{$1$} (120:1);}
% \\
% \tikz[baseline=-1mm]{\draw[thick] (0,0) -- node{$0$} (0:1);} &\mapsto \tikz[baseline=-1mm]{\draw[thick] (0,0) -- node{$0$} (0:1);}
% &
% \tikz[baseline=-1mm]{\draw[thick] (0,0) -- node{$1$} (0:1);} &\mapsto \tikz[baseline=-1mm]{\draw[thick] (0,0) -- node{$1$} (0:1);}
% &
% \tikz[baseline=-1mm]{\draw[thick] (0,0) -- node{$10$} (0:1);} &\mapsto \tikz[baseline=-1mm]{\draw[thick] (0,0) -- node{$01$} (0:1);}
% \end{align*}
The triangles match
\begin{gather*}
\uptri{0}{0}{0}\mapsto \uptri{}{0}{0}
\quad
\uptri{1}{1}{1}\mapsto \uptri{1}{1}{}
\\[2mm]
\uptri{1}{10}{0}\mapsto \uptri{1}{01}{0}
\quad
\uptri{10}{0}{1}\mapsto \uptri{0}{0}{}
\quad
\uptri{0}{1}{10}\mapsto \uptri{}{1}{1}
\end{gather*}
and similarly for down-pointing triangles;
the $K$-triangle \cite{Vakil} becomes
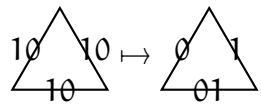
\[
\uptri{10}{10}{10}\mapsto \uptri{0}{01}{1}
\]
while the generic puzzle rule also allows for the corresponding down-pointing triangle;
and the equivariant rhombus becomes
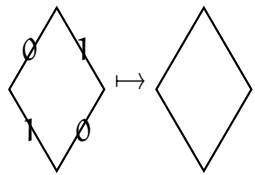
\[
\rh{1}{0}{1}{0}
\mapsto
\rh{}{}{}{}
\]
while the generic puzzle rule as stated in \cite[\S4.1]{artic80} also allows
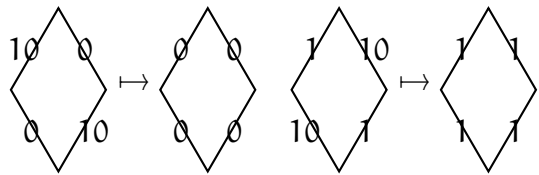
\[
\rh{0}{10}{0}{10}
\mapsto
\rh{0}{0}{0}{0}
\quad
\rh{10}{1}{10}{1}
\mapsto
\rh{1}{1}{1}{1}\ .
\]
One can check that all fugacities match the various versions of the rule.

The original $0,10,1$ labeling makes evident the $\ZZ_3$-symmetry of the
rule computing the coefficients $\int_{Gr(k;n)} S^\lambda S^\mu S^\nu$
(although the question itself enjoys $\mathcal S_3$-symmetry). 
The new $0,1,2,01,02,12$ labeling is more natural in the sense that it
displays the conservation laws (the continuity of the colored pipes)
more explicitly.

\junk{comment that these are not the same lines as in \cite{artic46}}

\subsection{Almost separated descent vs. 2-step puzzles}
\label{ssec:almostvs2}
In a similar vein, 2-step puzzles \cite{BKPT} are based on $\lie d_4$ \cite{artic71}, almost separated descent puzzles
are based on $\lie d_{d+2}$, and 
the almost separated descent condition $\overlap(\pi,\rho)\le 2$ is implied by the 2-step condition $\desc(\pi,\rho)\le 2$.

We can once again convert 2-step puzzle labels (with $k=1$, $d=2$) using the following dictionary:
\[
\begin{array}{lccccccccc}
\text{2-step:}
&
\tikz[baseline=0.3cm]{\draw[thick] (0,0) -- node{$0$} (60:1);}
&
\tikz[baseline=0.3cm]{\draw[thick] (0,0) -- node{$1$} (60:1);} 
&
\tikz[baseline=0.3cm]{\draw[thick] (0,0) -- node{$2$} (60:1);}
&
\tikz[baseline=0.3cm]{\draw[thick] (0,0) -- node{$10$} (60:1);}
&
\tikz[baseline=0.3cm]{\draw[thick] (0,0) -- node{$20$} (60:1);} 
&
\tikz[baseline=0.3cm]{\draw[thick] (0,0) -- node{$21$} (60:1);}
&
\tikz[baseline=0.3cm]{\draw[thick] (0,0) -- node{$2(10)$} (60:1);}
&
\tikz[baseline=0.3cm]{\draw[thick] (0,0) -- node{$(21)0$} (60:1);} 
\\
\text{Almost separated descent:}
&
\tikz[baseline=0.3cm]{\draw[thick] (0,0) -- node{$0$} (60:1);}
&
\tikz[baseline=0.3cm]{\draw[thick] (0,0) -- node{$1$} (60:1);}
&
\tikz[baseline=0.3cm]{\draw[thick] (0,0) -- node{} (60:1);}
&
\tikz[baseline=0.3cm]{\draw[thick] (0,0) -- node{$01$} (60:1);}
&
\tikz[baseline=0.3cm]{\draw[thick] (0,0) -- node{$02$} (60:1);}
&
\tikz[baseline=0.3cm]{\draw[thick] (0,0) -- node{$2$} (60:1);}
&
\tikz[baseline=0.3cm]{\draw[thick] (0,0) -- node{$12$} (60:1);}
&
\tikz[baseline=0.3cm]{\draw[thick] (0,0) -- node{$012$} (60:1);}
\\[4mm]
\text{2-step:}
&
\tikz[baseline=0.3cm]{\draw[thick] (0,0) -- node{$0$} (120:1);}
&
\tikz[baseline=0.3cm]{\draw[thick] (0,0) -- node{$1$} (120:1);} 
&
\tikz[baseline=0.3cm]{\draw[thick] (0,0) -- node{$2$} (120:1);}
&
\tikz[baseline=0.3cm]{\draw[thick] (0,0) -- node{$10$} (120:1);}
&
\tikz[baseline=0.3cm]{\draw[thick] (0,0) -- node{$20$} (120:1);} 
&
\tikz[baseline=0.3cm]{\draw[thick] (0,0) -- node{$21$} (120:1);}
&
\tikz[baseline=0.3cm]{\draw[thick] (0,0) -- node{$2(10)$} (120:1);}
&
\tikz[baseline=0.3cm]{\draw[thick] (0,0) -- node{$(21)0$} (120:1);} 
\\
\text{Almost separated descent:}
&
\tikz[baseline=0.3cm]{\draw[thick] (0,0) -- node{} (120:1);}
&
\tikz[baseline=0.3cm]{\draw[thick] (0,0) -- node{$1$} (120:1);}
&
\tikz[baseline=0.3cm]{\draw[thick] (0,0) -- node{$2$} (120:1);}
&
\tikz[baseline=0.3cm]{\draw[thick] (0,0) -- node{$0$} (120:1);}
&
\tikz[baseline=0.3cm]{\draw[thick] (0,0) -- node{$02$} (120:1);}
&
\tikz[baseline=0.3cm]{\draw[thick] (0,0) -- node{$12$} (120:1);}
&
\tikz[baseline=0.3cm]{\draw[thick] (0,0) -- node{$012$} (120:1);}
&
\tikz[baseline=0.3cm]{\draw[thick] (0,0) -- node{$01$} (120:1);}
\\[4mm]
\text{2-step:}
&
\tikz[baseline=-1mm]{\draw[thick] (0,0) -- node{$0$} (0:1);}
&
\tikz[baseline=-1mm]{\draw[thick] (0,0) -- node{$1$} (0:1);} 
&
\tikz[baseline=-1mm]{\draw[thick] (0,0) -- node{$2$} (0:1);}
&
\tikz[baseline=-1mm]{\draw[thick] (0,0) -- node{$10$} (0:1);}
&
\tikz[baseline=-1mm]{\draw[thick] (0,0) -- node{$20$} (0:1);} 
&
\tikz[baseline=-1mm]{\draw[thick] (0,0) -- node{$21$} (0:1);}
&
\tikz[baseline=-1mm]{\draw[thick] (0,0) -- node{$2(10)$} (0:1);}
&
\tikz[baseline=-1mm]{\draw[thick] (0,0) -- node{$(21)0$} (0:1);} 
\\
\text{Almost separated descent:}
&
\tikz[baseline=-1mm]{\draw[thick] (0,0) -- node{$\da0$} (0:1);}
&
\tikz[baseline=-1mm]{\draw[thick] (0,0) -- node{$odd$} (0:1);}
&
\tikz[baseline=-1mm]{\draw[thick] (0,0) -- node{$\ua2$} (0:1);}
&
\tikz[baseline=-1mm]{\draw[thick] (0,0) -- node{$\da1$} (0:1);}
&
\tikz[baseline=-1mm]{\draw[thick] (0,0) -- node{$even$} (0:1);}
&
\tikz[baseline=-1mm]{\draw[thick] (0,0) -- node{$\ua1$} (0:1);}
&
\tikz[baseline=-1mm]{\draw[thick] (0,0) -- node{$\ua0$} (0:1);}
&
\tikz[baseline=-1mm]{\draw[thick] (0,0) -- node{$\da2$} (0:1);}
\end{array}
\]

\begin{ex}The example $c^{0102,0201}_{0210}$ of \cite[\S1.2]{artic71} becomes
\begin{center}
\tikzset{every picture/.style={scale=0.5,every node/.style={scale=0.5}}}
%first puzzle("010_","_2_1",Equivariant=>false,Paths=>true)
\input ex2stepa.tex
\end{center}
\end{ex}
We now list the $K$-pieces according to Theorem~\ref{thm:almostsepdesc}:
\begin{gather*}
\uptri{2}{\ua 1}{12}\qquad\uptri{02}{even}{02}\qquad\uptri{01}{\da1}{0}
\\
\uptri{01}{even}{01}\quad\uptri{02}{\da2}{0}\quad\uptri{012}{\da1}{02}
\\
\uptri{02}{\ua1}{012}\quad\downtri{2}{\ua0}{02}\quad\uptri{12}{even}{12}
\qquad
\downtri{12}{\da2}{1}\quad\uptri{012}{odd}{012}\quad\downtri{1}{\ua0}{01}
\\
\downtri{12}{\ua0}{012}\ .
\end{gather*}
If one compares with the list in \cite[Thm.~2]{artic71}, one finds that the two sets of pieces are related
by the duality that takes a triangle to its mirror image with labels inverted according to $i\mapsto 2-i$.
%This can be cured by noting that the almost separated descent setup also admits duality; in other words,
%almost separated descent puzzles with given boundaries must come in equal numbers as puzzles
%with boundaries obtained by mirror image of the original boundaries and substitution $i\mapsto d-i$.
As already noted in \S\ref{ssec:almostsepdesc}, there is also a duality of almost separated descent puzzles (mirror image
combined with $i\mapsto d-i$); it generalizes the duality of 2-step puzzles, in the sense that
duality commutes with the bijection of labels described above.
It is therefore the dual $K$-theoretic almost-separated descent puzzle rule which generalizes the $d=2$ rule
that is stated in \cite[Thm.~2]{artic71}.

\subsection{Separated descent vs. almost separated puzzles}
\rem[gray]{comparison sep desc / almost sep desc: sep desc are the ones for which there are no $k$ paths going from NW to NE (in which case
might as well renumber those starting on NE for easier comparison --  in $K$-theory,
rule depends explicitly on $k$, but we're saved by the odd asymmetry of the rule: we only check $<k$ vs $\ge k$, so $k$ and
$k+1$ seem to play the same role. weird). redo ex 1}
The almost separated descent rule applies to any pairs of permutations
$\pi$ and $\rho$ with $\overlap(\pi,\rho)\le 2$ (adding gratuitous nondescents if $\overlap(\pi,\rho)<2$),
so it looks like it supersedes the separated descent rule which only covers the cases $\overlap(\pi,\rho)\le 1$.
A few comments should be made:
\begin{itemize}
\item Although this statement is strictly true at the level of motivic Segre classes, once one takes the limit $q\to0$,
only the separated descent rule allows preserving equivariance (and therefore, a rule for {\em double}\/ Schubert polynomials).
\item As already mentioned in \S\ref{ssec:almostsepdesc}, because the
  two rules have the strings $\lambda$ and $\mu$ switched between
  Northwest and Northeast sides, there is little hope of a bijection
  between them -- also, almost-separated descent puzzles tend to look
  significantly more complicated.
\end{itemize}

We illustrate the last point below. 
But first we must decide, when $\overlap(\pi,\rho)=1$,
where to add the gratuitous nondescent: in principle there are multiple choices,
anywhere from $\rho$'s second to last descent to $\pi$'s second descent. However, it is more natural to
formally set it equal to their common descent: nothing prevents, in the statement of
Theorem~\ref{thm:almostsepdesc}, from having no occurrence of the letter $k$
(i.e., no paths crossing from the NW side to the NE side);
and since the rule does not differentiate between the various labels $\ge k$, it is simpler to reindex
$i>k \mapsto i-1$, without any change to the puzzle rule. In this case, the strings match exactly
the ones from the separated descent rule, making comparison easy. (A further advantage of this choice is that
it works even for generic puzzles, since the underlying partial flag varieties are the same.)

\begin{ex}\label{ex:ex1redux}
%puzzle("_21___0","_3_43_4",Paths=>true,Generic=>false,Equivariant=>false)
We redo Example~\ref{ex:sepdesc} using the almost separated descent puzzle rule:
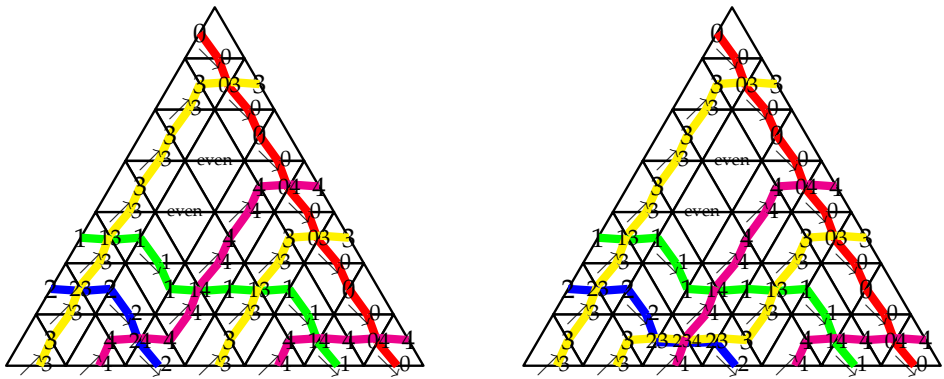
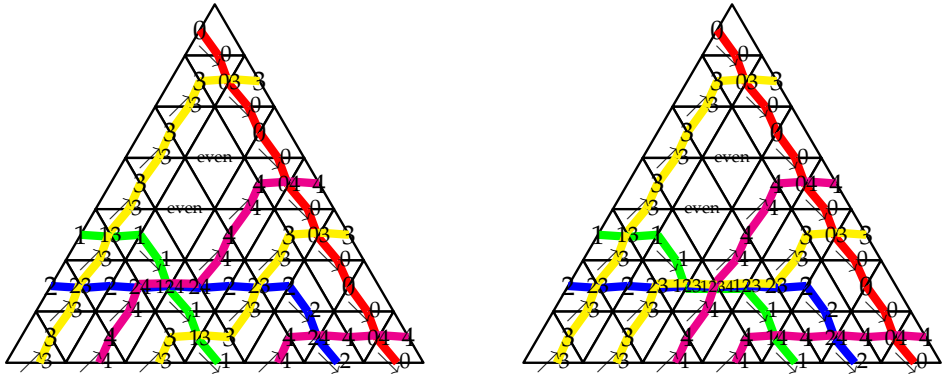
\begin{center}
\tikzset{every picture/.style={scale=0.5,every node/.style={scale=0.5}}}
\input{ex1b.tex}
\end{center}
\end{ex}

\subsection{Separated-descent vs. Grassmannian puzzles with 10s at the bottom}
In \cite[Thm.~2]{artic73}, it was pointed out that Grassmannian puzzles can be generalized
to compute products of pull-backs of Schubert classes from two different Grassmannians to a 2-step flag variety,
on condition that one allow 10s on the bottom side of the puzzles (i.e., the bottom alphabet is $0<10<1$).
The result is only stated in equivariant cohomology, though it works equally well in equivariant $K$-theory,
and even more generally for motivic Segre classes; such puzzle rules fit in the framework of \S\ref{sec:setup},
and we skip the proof, which is yet another simple variation on the existing results.

Let $\pi$ be a (Grassmannian) permutation with single descent at $j$,
and $\rho$ be a (Grassmannian) permutation with single descent at $k$,
$j<k$.
Clearly $\overlap(\pi,\rho)=2$ -- in fact, $\desc(\pi,\rho)=2$ so we can of course use
2-step puzzles (or almost-separated-descent puzzles) to solve this problem -- but more interestingly,
$\overlap(\rho,\pi)=0$, which means this problem is also amenable to a separated-descent solution. 
We let the interested reader try to figure out possible bijections between these various puzzles
(with the warning that there is a switch of $\pi$ and $\rho$ between some of them).

\begin{ex}We use the same permutations $\pi=213$, $\rho=231$ as in \cite[Ex.~3]{artic73} except we consider
equivariant $K$-theoretic Schubert classes. Here are the puzzles with 10s at the bottom:
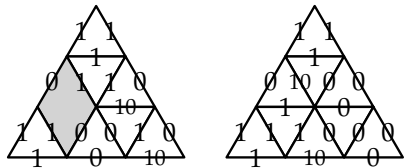
\begin{center}
\tikzset{every picture/.style={scale=0.5,every node/.style={scale=0.5}}}
%"ex10a.tex" << concatenate(tex\ (puzzle("101","100",Generic=>false,Paths=>false))) << close
\input{ex10a.tex}
\end{center}
as well as the corresponding 2-step puzzles:
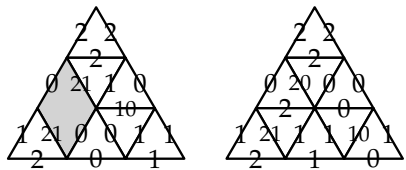
\begin{center}
\tikzset{every picture/.style={scale=0.5,every node/.style={scale=0.5}}}
%"ex10b.tex" << concatenate reverse(tex\ (puzzle("102","201",Generic=>false,Paths=>false))) << close
\input{ex10b.tex}
\end{center}
In each case, the two puzzles have fugacity $1-y_2/y_1$ and $y_2/y_1$, respectively.

After switching $\pi$ and $\rho$, we have two choices of separation of descents, leading respectively to
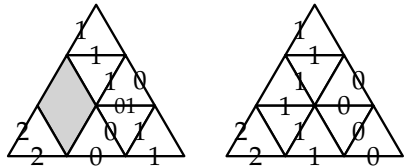
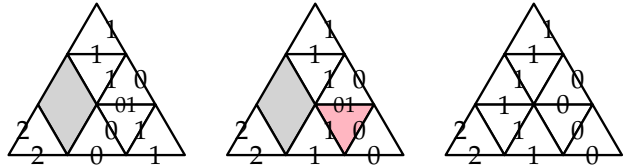
\begin{center}
\tikzset{every picture/.style={scale=0.5,every node/.style={scale=0.5}}}
%"ex10c.tex" << concatenate(tex\ (puzzle("2_1","_0_",Generic=>false,Paths=>false))) << close
\input{ex10c.tex}
\end{center}
and
\begin{center}
\tikzset{every picture/.style={scale=0.5,every node/.style={scale=0.5}}}
%"ex10d.tex" << concatenate (tex\ (puzzle("2__","10_",Generic=>false,Ktheory=>true,Paths=>false))) << close
\input{ex10d.tex}
\end{center}
Here are the corresponding 2-step puzzles:
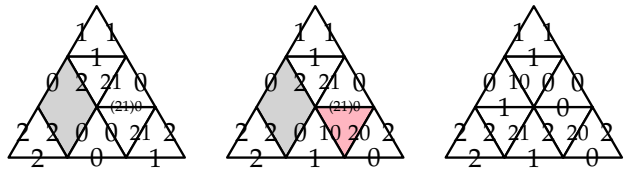
\begin{center}
\tikzset{every picture/.style={scale=0.5,every node/.style={scale=0.5}}}
%"ex10e.tex" << concatenate (tex\ (puzzle("201","102",Generic=>false,Ktheory=>true,Paths=>false)))_{1,2,0} << close
\input{ex10e.tex}
\end{center}
\end{ex}
In the last two sets of puzzles, the fugacities are $1-y_2/y_1$, $-(1-y_2/y_1)$, $1$, respectively.

\subsection{Almost-separated-descent vs. 2-step puzzles
  with 10s/21s at the bottom}

% Finally, there is yet another simple variation of 2-step puzzles
There is a $2$-step version of \cite[Thm.~2]{artic73}, in which we allow
either $10$s or $21$s (but not both) on the bottom side. We provide
the version with $10$s at the bottom (the version with $21$s follows
the same pattern, and can also be obtained by duality).

Let $\pi$ (resp.\ $\rho$) be a permutation with descent set $D(\pi)=\{j,k\}$ (resp.\ $D(\rho)=\{j',k\}$),
with $j>j'$.
We encode them with strings in $\{0,1,2\}$ (though their contents are different). Then one has a product
rule using 2-step puzzles
(as defined in the various papers \cite{BKPT,artic71,artic80} depending on the cohomology theory and the choice
of classes), except the bottom permutation uses the alphabet $0<10<1<2$.

Alternatively, since $\overlap(\pi,\rho)=2$, one can use almost-separated-descents (for equivariant
motivic Segre classes;
for Schubert classes, only nonequivariantly).
Since there is no need to swap $\pi$ and $\rho$ in this case, a bijection might seem possible (though
still nonobvious).

\begin{ex}
Let us consider $\pi=1432$ and $\rho=2143$, and {\em generic}\/ nonequivariant puzzles. In the modified
2-step puzzle rule, one finds:
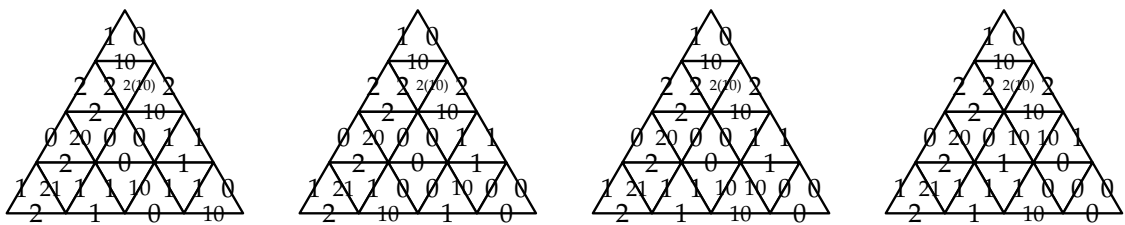
\begin{center}
\tikzset{every picture/.style={scale=0.5,every node/.style={scale=0.5}}}
%"ex21a.tex" << concatenate(tex\ (puzzle("1021","0210",Generic=>true,Equivariant=>false)))<<close
\input{ex21a.tex}
\end{center}
where the first two puzzles are ordinary nonequivariant puzzles.
If instead we use almost-separated-descents, one has:
\begin{center}
%"ex21b.tex" << concatenate(tex\ ((puzzle("10_2","_32_",Generic=>true,Equivariant=>false))_{2,0,1,3}))<<close
\tikzset{every picture/.style={scale=0.5,every node/.style={scale=0.5}}}
\input{ex21b.tex}
\end{center}
\rem[gray]{I can't be bothered computing their $K$-fugacities}
\end{ex}

\vfill\eject

% to fix MR issues
\gdef\MRshorten#1 #2MRend{#1}%
\gdef\MRfirsttwo#1#2{\if#1M%
MR\else MR#1#2\fi}
\def\MRfix#1{\MRshorten\MRfirsttwo#1 MRend}
\renewcommand\MR[1]{\relax\ifhmode\unskip\spacefactor3000 \space\fi
\MRhref{\MRfix{#1}}{{\scriptsize \MRfix{#1}}}}
\renewcommand{\MRhref}[2]{%
\href{http://www.ams.org/mathscinet-getitem?mr=#1}{#2}}
\bibliographystyle{amsalphahyper}
\bibliography{biblio}

\end{document}

%% file: ex2a.tex
\hskip-1cm
\setlength{\arraycolsep}{-3pt} % AK put this in
\begin{array}{cccc}
\begin{tikzpicture}[execute at begin picture={\bgroup\tikzset{every path/.style={}}\clip (-2.48,-.432012) rectangle ++(4.96,4.46413);\egroup},x={(1.55426771761071cm,0cm)},y={(0cm,-1.55426394417233cm)},baseline=(current  bounding  box.center),every path/.style={draw=black,fill=none},line join=round]
\begin{scope}[every path/.append style={fill=white,line width=.0310853166178533cm}]
\path (-.5,.866025403784439) -- (0,0) -- (.5,.866025403784439) -- cycle;
\path (-.5,.866025403784439) -- (.5,.866025403784439) -- (0,1.73205080756888) -- cycle;
\path[draw=green,line width=.108798608162487cm] (0,.866025403784439) -- (-.25,.433012701892219);
\path[draw=green,line width=.108798608162487cm] (0,.866025403784439) -- (.25,1.29903810567666);
\node[black,,scale=1.66461870488604] at (-.25,.433012701892219) {1};
\node[black,,scale=1.66461870488604] at (0,.866025403784439) {1};
\path (0,1.73205080756888) -- (.5,.866025403784439) -- (1,1.73205080756888) -- cycle;
\path (0,1.73205080756888) -- (1,1.73205080756888) -- (.5,2.59807621135332) -- cycle;
\path[draw=red,line width=.108798608162487cm] (.75,1.29903810567666) -- (.46,1.73205080756888);
\path[draw=red,line width=.108798608162487cm] (.25,2.1650635094611) -- (.46,1.73205080756888);
\path[draw=green,line width=.108798608162487cm] (.54,1.73205080756888) -- (.25,1.29903810567666);
\path[draw=green,line width=.108798608162487cm] (.54,1.73205080756888) -- (.75,2.1650635094611);
\node[black,,scale=1.66461870488604] at (.75,1.29903810567666) {0};
\node[black,,scale=1.66461870488604] at (.25,1.29903810567666) {1};
\node[black,,scale=1.38718062209258] at (.5,1.73205080756888) {01};
\path (.5,2.59807621135332) -- (1,1.73205080756888) -- (1.5,2.59807621135332) -- cycle;
\path (.5,2.59807621135332) -- (1.5,2.59807621135332) -- (1,3.46410161513775) -- cycle;
\path[draw=green,line width=.108798608162487cm] (1,2.59807621135332) -- (.75,2.1650635094611);
\path[draw=green,line width=.108798608162487cm] (1,2.59807621135332) -- (1.25,3.03108891324554);
\node[black,,scale=1.66461870488604] at (.75,2.1650635094611) {1};
\node[black,,scale=1.66461870488604] at (1,2.59807621135332) {1};
\path (1,3.46410161513775) -- (1.5,2.59807621135332) -- (2,3.46410161513775) -- cycle;
\path[draw=green,line width=.108798608162487cm] (1.5,3.46410161513775) -- (1.25,3.03108891324554);
\node[black,,scale=1.66461870488604] at (1.25,3.03108891324554) {1};
\node[black,,scale=1.66461870488604] at (1.5,3.46410161513775) {1};
\path (-1,1.73205080756888) -- (-.5,.866025403784439) -- (0,1.73205080756888) -- cycle;
\path (-1,1.73205080756888) -- (0,1.73205080756888) -- (-.5,2.59807621135332) -- cycle;
\path[draw=blue,line width=.108798608162487cm] (-.5,1.73205080756888) -- (-.75,1.29903810567666);
\path[draw=blue,line width=.108798608162487cm] (-.5,1.73205080756888) -- (-.25,2.1650635094611);
\node[black,,scale=1.66461870488604] at (-.75,1.29903810567666) {2};
\node[black,,scale=1.66461870488604] at (-.5,1.73205080756888) {2};
\path (-.5,2.59807621135332) -- (0,1.73205080756888) -- (.5,2.59807621135332) -- cycle;
\path (-.5,2.59807621135332) -- (.5,2.59807621135332) -- (0,3.46410161513775) -- cycle;
\path[draw=red,line width=.108798608162487cm] (.25,2.1650635094611) -- (-.04,2.59807621135332);
\path[draw=red,line width=.108798608162487cm] (-.25,3.03108891324554) -- (-.04,2.59807621135332);
\path[draw=blue,line width=.108798608162487cm] (.04,2.59807621135332) -- (-.25,2.1650635094611);
\path[draw=blue,line width=.108798608162487cm] (.04,2.59807621135332) -- (.25,3.03108891324554);
\node[black,,scale=1.66461870488604] at (.25,2.1650635094611) {0};
\node[black,,scale=1.66461870488604] at (-.25,2.1650635094611) {2};
\node[black,,scale=1.38718062209258] at (0,2.59807621135332) {02};
\path (0,3.46410161513775) -- (.5,2.59807621135332) -- (1,3.46410161513775) -- cycle;
\path[draw=blue,line width=.108798608162487cm] (.5,3.46410161513775) -- (.25,3.03108891324554);
\node[black,,scale=1.66461870488604] at (.25,3.03108891324554) {2};
\node[black,,scale=1.66461870488604] at (.5,3.46410161513775) {2};
\path[fill=LightGray] (-1.5,2.59807621135332) -- (-1,1.73205080756888) -- (-.5,2.59807621135332) -- (-1,3.46410161513775) -- cycle;
\path (-1,3.46410161513775) -- (-.5,2.59807621135332) -- (0,3.46410161513775) -- cycle;
\path[draw=red,line width=.108798608162487cm] (-.25,3.03108891324554) -- (-.5,3.46410161513775);
\node[black,,scale=1.66461870488604] at (-.25,3.03108891324554) {0};
\node[black,,scale=1.66461870488604] at (-.5,3.46410161513775) {0};
\path (-2,3.46410161513775) -- (-1.5,2.59807621135332) -- (-1,3.46410161513775) -- cycle;
\path[draw=yellow,line width=.108798608162487cm] (-1.5,3.46410161513775) -- (-1.75,3.03108891324554);
\node[black,,scale=1.66461870488604] at (-1.75,3.03108891324554) {3};
\node[black,,scale=1.66461870488604] at (-1.5,3.46410161513775) {3};
\end{scope}
\end{tikzpicture}&\begin{tikzpicture}[execute at begin picture={\bgroup\tikzset{every path/.style={}}\clip (-2.48,-.432012) rectangle ++(4.96,4.46413);\egroup},x={(1.55426771761071cm,0cm)},y={(0cm,-1.55426394417233cm)},baseline=(current  bounding  box.center),every path/.style={draw=black,fill=none},line join=round]
\begin{scope}[every path/.append style={fill=white,line width=.0310853166178533cm}]
\path (-.5,.866025403784439) -- (0,0) -- (.5,.866025403784439) -- cycle;
\path (-.5,.866025403784439) -- (.5,.866025403784439) -- (0,1.73205080756888) -- cycle;
\path[draw=green,line width=.108798608162487cm] (0,.866025403784439) -- (-.25,.433012701892219);
\path[draw=green,line width=.108798608162487cm] (0,.866025403784439) -- (.25,1.29903810567666);
\node[black,,scale=1.66461870488604] at (-.25,.433012701892219) {1};
\node[black,,scale=1.66461870488604] at (0,.866025403784439) {1};
\path (0,1.73205080756888) -- (.5,.866025403784439) -- (1,1.73205080756888) -- cycle;
\path (0,1.73205080756888) -- (1,1.73205080756888) -- (.5,2.59807621135332) -- cycle;
\path[draw=red,line width=.108798608162487cm] (.75,1.29903810567666) -- (.46,1.73205080756888);
\path[draw=red,line width=.108798608162487cm] (.25,2.1650635094611) -- (.46,1.73205080756888);
\path[draw=green,line width=.108798608162487cm] (.54,1.73205080756888) -- (.25,1.29903810567666);
\path[draw=green,line width=.108798608162487cm] (.54,1.73205080756888) -- (.75,2.1650635094611);
\node[black,,scale=1.66461870488604] at (.75,1.29903810567666) {0};
\node[black,,scale=1.66461870488604] at (.25,1.29903810567666) {1};
\node[black,,scale=1.38718062209258] at (.5,1.73205080756888) {01};
\path (.5,2.59807621135332) -- (1,1.73205080756888) -- (1.5,2.59807621135332) -- cycle;
\path (.5,2.59807621135332) -- (1.5,2.59807621135332) -- (1,3.46410161513775) -- cycle;
\path[draw=green,line width=.108798608162487cm] (1,2.59807621135332) -- (.75,2.1650635094611);
\path[draw=green,line width=.108798608162487cm] (1,2.59807621135332) -- (1.25,3.03108891324554);
\node[black,,scale=1.66461870488604] at (.75,2.1650635094611) {1};
\node[black,,scale=1.66461870488604] at (1,2.59807621135332) {1};
\path (1,3.46410161513775) -- (1.5,2.59807621135332) -- (2,3.46410161513775) -- cycle;
\path[draw=green,line width=.108798608162487cm] (1.5,3.46410161513775) -- (1.25,3.03108891324554);
\node[black,,scale=1.66461870488604] at (1.25,3.03108891324554) {1};
\node[black,,scale=1.66461870488604] at (1.5,3.46410161513775) {1};
\path (-1,1.73205080756888) -- (-.5,.866025403784439) -- (0,1.73205080756888) -- cycle;
\path (-1,1.73205080756888) -- (0,1.73205080756888) -- (-.5,2.59807621135332) -- cycle;
\path[draw=blue,line width=.108798608162487cm] (-.5,1.73205080756888) -- (-.75,1.29903810567666);
\path[draw=blue,line width=.108798608162487cm] (-.75,2.1650635094611) -- (-.5,1.73205080756888);
\node[black,,scale=1.66461870488604] at (-.75,1.29903810567666) {2};
\node[black,,scale=1.66461870488604] at (-.5,1.73205080756888) {2};
\path (-.5,2.59807621135332) -- (0,1.73205080756888) -- (.5,2.59807621135332) -- cycle;
\path (-.5,2.59807621135332) -- (.5,2.59807621135332) -- (0,3.46410161513775) -- cycle;
\path[draw=red,line width=.108798608162487cm] (.25,2.1650635094611) -- (0,2.59807621135332);
\path[draw=red,line width=.108798608162487cm] (0,2.59807621135332) -- (.25,3.03108891324554);
\node[black,,scale=1.66461870488604] at (.25,2.1650635094611) {0};
\node[black,,scale=1.66461870488604] at (0,2.59807621135332) {0};
\path (0,3.46410161513775) -- (.5,2.59807621135332) -- (1,3.46410161513775) -- cycle;
\path[draw=red,line width=.108798608162487cm] (.5,3.46410161513775) -- (.25,3.03108891324554);
\node[black,,scale=1.66461870488604] at (.25,3.03108891324554) {0};
\node[black,,scale=1.66461870488604] at (.5,3.46410161513775) {0};
\path (-1.5,2.59807621135332) -- (-1,1.73205080756888) -- (-.5,2.59807621135332) -- cycle;
\path (-1.5,2.59807621135332) -- (-.5,2.59807621135332) -- (-1,3.46410161513775) -- cycle;
\path[draw=blue,line width=.108798608162487cm] (-.75,2.1650635094611) -- (-1,2.59807621135332);
\path[draw=blue,line width=.108798608162487cm] (-1,2.59807621135332) -- (-.75,3.03108891324553);
\node[black,,scale=1.66461870488604] at (-.75,2.1650635094611) {2};
\node[black,,scale=1.66461870488604] at (-1,2.59807621135332) {2};
\path (-1,3.46410161513775) -- (-.5,2.59807621135332) -- (0,3.46410161513775) -- cycle;
\path[draw=blue,line width=.108798608162487cm] (-.5,3.46410161513775) -- (-.75,3.03108891324553);
\node[black,,scale=1.66461870488604] at (-.75,3.03108891324553) {2};
\node[black,,scale=1.66461870488604] at (-.5,3.46410161513775) {2};
\path (-2,3.46410161513775) -- (-1.5,2.59807621135332) -- (-1,3.46410161513775) -- cycle;
\path[draw=yellow,line width=.108798608162487cm] (-1.5,3.46410161513775) -- (-1.75,3.03108891324554);
\node[black,,scale=1.66461870488604] at (-1.75,3.03108891324554) {3};
\node[black,,scale=1.66461870488604] at (-1.5,3.46410161513775) {3};
\end{scope}
\end{tikzpicture}&\begin{tikzpicture}[execute at begin picture={\bgroup\tikzset{every path/.style={}}\clip (-2.48,-.432012) rectangle ++(4.96,4.46413);\egroup},x={(1.55426771761071cm,0cm)},y={(0cm,-1.55426394417233cm)},baseline=(current  bounding  box.center),every path/.style={draw=black,fill=none},line join=round]
\begin{scope}[every path/.append style={fill=white,line width=.0310853166178533cm}]
\path (-.5,.866025403784439) -- (0,0) -- (.5,.866025403784439) -- cycle;
\path (-.5,.866025403784439) -- (.5,.866025403784439) -- (0,1.73205080756888) -- cycle;
\path[draw=green,line width=.108798608162487cm] (0,.866025403784439) -- (-.25,.433012701892219);
\path[draw=green,line width=.108798608162487cm] (-.25,1.29903810567666) -- (0,.866025403784439);
\node[black,,scale=1.66461870488604] at (-.25,.433012701892219) {1};
\node[black,,scale=1.66461870488604] at (0,.866025403784439) {1};
\path (0,1.73205080756888) -- (.5,.866025403784439) -- (1,1.73205080756888) -- cycle;
\path (0,1.73205080756888) -- (1,1.73205080756888) -- (.5,2.59807621135332) -- cycle;
\path[draw=red,line width=.108798608162487cm] (.75,1.29903810567666) -- (.5,1.73205080756888);
\path[draw=red,line width=.108798608162487cm] (.5,1.73205080756888) -- (.75,2.1650635094611);
\node[black,,scale=1.66461870488604] at (.75,1.29903810567666) {0};
\node[black,,scale=1.66461870488604] at (.5,1.73205080756888) {0};
\path (.5,2.59807621135332) -- (1,1.73205080756888) -- (1.5,2.59807621135332) -- cycle;
\path (.5,2.59807621135332) -- (1.5,2.59807621135332) -- (1,3.46410161513775) -- cycle;
\path[draw=red,line width=.108798608162487cm] (1,2.59807621135332) -- (.75,2.1650635094611);
\path[draw=red,line width=.108798608162487cm] (1,2.59807621135332) -- (1.25,3.03108891324554);
\node[black,,scale=1.66461870488604] at (.75,2.1650635094611) {0};
\node[black,,scale=1.66461870488604] at (1,2.59807621135332) {0};
\path (1,3.46410161513775) -- (1.5,2.59807621135332) -- (2,3.46410161513775) -- cycle;
\path[draw=red,line width=.108798608162487cm] (1.5,3.46410161513775) -- (1.25,3.03108891324554);
\node[black,,scale=1.66461870488604] at (1.25,3.03108891324554) {0};
\node[black,,scale=1.66461870488604] at (1.5,3.46410161513775) {0};
\path (-1,1.73205080756888) -- (-.5,.866025403784439) -- (0,1.73205080756888) -- cycle;
\path[fill=LightPink] (-1,1.73205080756888) -- (0,1.73205080756888) -- (-.5,2.59807621135332) -- cycle;
\path[draw=green,line width=.108798608162487cm] (-.25,1.29903810567666) -- (-.54,1.73205080756888);
\path[draw=green,line width=.108798608162487cm] (-.54,1.73205080756888) -- (-.25,2.1650635094611);
\path[draw=blue,line width=.108798608162487cm] (-.46,1.73205080756888) -- (-.75,1.29903810567666);
\path[draw=blue,line width=.108798608162487cm] (-.75,2.1650635094611) -- (-.46,1.73205080756888);
\node[black,,scale=1.66461870488604] at (-.25,1.29903810567666) {1};
\node[black,,scale=1.66461870488604] at (-.75,1.29903810567666) {2};
\node[black,,scale=1.38718062209258] at (-.5,1.73205080756888) {12};
\path (-.5,2.59807621135332) -- (0,1.73205080756888) -- (.5,2.59807621135332) -- cycle;
\path (-.5,2.59807621135332) -- (.5,2.59807621135332) -- (0,3.46410161513775) -- cycle;
\path[draw=green,line width=.108798608162487cm] (0,2.59807621135332) -- (-.25,2.1650635094611);
\path[draw=green,line width=.108798608162487cm] (0,2.59807621135332) -- (.25,3.03108891324554);
\node[black,,scale=1.66461870488604] at (-.25,2.1650635094611) {1};
\node[black,,scale=1.66461870488604] at (0,2.59807621135332) {1};
\path (0,3.46410161513775) -- (.5,2.59807621135332) -- (1,3.46410161513775) -- cycle;
\path[draw=green,line width=.108798608162487cm] (.5,3.46410161513775) -- (.25,3.03108891324554);
\node[black,,scale=1.66461870488604] at (.25,3.03108891324554) {1};
\node[black,,scale=1.66461870488604] at (.5,3.46410161513775) {1};
\path (-1.5,2.59807621135332) -- (-1,1.73205080756888) -- (-.5,2.59807621135332) -- cycle;
\path (-1.5,2.59807621135332) -- (-.5,2.59807621135332) -- (-1,3.46410161513775) -- cycle;
\path[draw=blue,line width=.108798608162487cm] (-.75,2.1650635094611) -- (-1,2.59807621135332);
\path[draw=blue,line width=.108798608162487cm] (-1,2.59807621135332) -- (-.75,3.03108891324553);
\node[black,,scale=1.66461870488604] at (-.75,2.1650635094611) {2};
\node[black,,scale=1.66461870488604] at (-1,2.59807621135332) {2};
\path (-1,3.46410161513775) -- (-.5,2.59807621135332) -- (0,3.46410161513775) -- cycle;
\path[draw=blue,line width=.108798608162487cm] (-.5,3.46410161513775) -- (-.75,3.03108891324553);
\node[black,,scale=1.66461870488604] at (-.75,3.03108891324553) {2};
\node[black,,scale=1.66461870488604] at (-.5,3.46410161513775) {2};
\path (-2,3.46410161513775) -- (-1.5,2.59807621135332) -- (-1,3.46410161513775) -- cycle;
\path[draw=yellow,line width=.108798608162487cm] (-1.5,3.46410161513775) -- (-1.75,3.03108891324554);
\node[black,,scale=1.66461870488604] at (-1.75,3.03108891324554) {3};
\node[black,,scale=1.66461870488604] at (-1.5,3.46410161513775) {3};
\end{scope}
\end{tikzpicture} &\qquad\qquad
\begin{tikzpicture}[execute at begin picture={\bgroup\tikzset{every path/.style={}}\clip (-2.48,-.432012) rectangle ++(4.96,4.46413);\egroup},x={(1.55426771761071cm,0cm)},y={(0cm,-1.55426394417233cm)},baseline=(current  bounding  box.center),every path/.style={draw=black,fill=none},line join=round]
\begin{scope}[every path/.append style={fill=white,line width=.0310853166178533cm}]
\path (-.5,.866025403784439) -- (0,0) -- (.5,.866025403784439) -- cycle;
\path (-.5,.866025403784439) -- (.5,.866025403784439) -- (0,1.73205080756888) -- cycle;
\path[draw=green,line width=.108798608162487cm] (0,.866025403784439) -- (-.25,.433012701892219);
\path[draw=green,line width=.108798608162487cm] (-.25,1.29903810567666) -- (0,.866025403784439);
\node[black,,scale=1.66461870488604] at (-.25,.433012701892219) {1};
\node[black,,scale=1.66461870488604] at (0,.866025403784439) {1};
\path (0,1.73205080756888) -- (.5,.866025403784439) -- (1,1.73205080756888) -- cycle;
\path (0,1.73205080756888) -- (1,1.73205080756888) -- (.5,2.59807621135332) -- cycle;
\path[draw=red,line width=.108798608162487cm] (.75,1.29903810567666) -- (.5,1.73205080756888);
\path[draw=red,line width=.108798608162487cm] (.5,1.73205080756888) -- (.75,2.1650635094611);
\node[black,,scale=1.66461870488604] at (.75,1.29903810567666) {0};
\node[black,,scale=1.66461870488604] at (.5,1.73205080756888) {0};
\path (.5,2.59807621135332) -- (1,1.73205080756888) -- (1.5,2.59807621135332) -- cycle;
\path (.5,2.59807621135332) -- (1.5,2.59807621135332) -- (1,3.46410161513775) -- cycle;
\path[draw=red,line width=.108798608162487cm] (1,2.59807621135332) -- (.75,2.1650635094611);
\path[draw=red,line width=.108798608162487cm] (1,2.59807621135332) -- (1.25,3.03108891324554);
\node[black,,scale=1.66461870488604] at (.75,2.1650635094611) {0};
\node[black,,scale=1.66461870488604] at (1,2.59807621135332) {0};
\path (1,3.46410161513775) -- (1.5,2.59807621135332) -- (2,3.46410161513775) -- cycle;
\path[draw=red,line width=.108798608162487cm] (1.5,3.46410161513775) -- (1.25,3.03108891324554);
\node[black,,scale=1.66461870488604] at (1.25,3.03108891324554) {0};
\node[black,,scale=1.66461870488604] at (1.5,3.46410161513775) {0};
\path (-1,1.73205080756888) -- (-.5,.866025403784439) -- (0,1.73205080756888) -- cycle;
\path (-1,1.73205080756888) -- (0,1.73205080756888) -- (-.5,2.59807621135332) -- cycle;
\path[draw=green,line width=.108798608162487cm] (-.25,1.29903810567666) -- (-.54,1.73205080756888);
\path[draw=green,line width=.108798608162487cm] (-.75,2.1650635094611) -- (-.54,1.73205080756888);
\path[draw=blue,line width=.108798608162487cm] (-.46,1.73205080756888) -- (-.75,1.29903810567666);
\path[draw=blue,line width=.108798608162487cm] (-.46,1.73205080756888) -- (-.25,2.1650635094611);
\node[black,,scale=1.66461870488604] at (-.25,1.29903810567666) {1};
\node[black,,scale=1.66461870488604] at (-.75,1.29903810567666) {2};
\node[black,,scale=1.38718062209258] at (-.5,1.73205080756888) {12};
\path (-.5,2.59807621135332) -- (0,1.73205080756888) -- (.5,2.59807621135332) -- cycle;
\path (-.5,2.59807621135332) -- (.5,2.59807621135332) -- (0,3.46410161513775) -- cycle;
\path[draw=blue,line width=.108798608162487cm] (0,2.59807621135332) -- (-.25,2.1650635094611);
\path[draw=blue,line width=.108798608162487cm] (0,2.59807621135332) -- (.25,3.03108891324554);
\node[black,,scale=1.66461870488604] at (-.25,2.1650635094611) {2};
\node[black,,scale=1.66461870488604] at (0,2.59807621135332) {2};
\path (0,3.46410161513775) -- (.5,2.59807621135332) -- (1,3.46410161513775) -- cycle;
\path[draw=blue,line width=.108798608162487cm] (.5,3.46410161513775) -- (.25,3.03108891324554);
\node[black,,scale=1.66461870488604] at (.25,3.03108891324554) {2};
\node[black,,scale=1.66461870488604] at (.5,3.46410161513775) {2};
\path (-1.5,2.59807621135332) -- (-1,1.73205080756888) -- (-.5,2.59807621135332) -- cycle;
\path (-1.5,2.59807621135332) -- (-.5,2.59807621135332) -- (-1,3.46410161513775) -- cycle;
\path[draw=green,line width=.108798608162487cm] (-.75,2.1650635094611) -- (-1,2.59807621135332);
\path[draw=green,line width=.108798608162487cm] (-1,2.59807621135332) -- (-.75,3.03108891324553);
\node[black,,scale=1.66461870488604] at (-.75,2.1650635094611) {1};
\node[black,,scale=1.66461870488604] at (-1,2.59807621135332) {1};
\path (-1,3.46410161513775) -- (-.5,2.59807621135332) -- (0,3.46410161513775) -- cycle;
\path[draw=green,line width=.108798608162487cm] (-.5,3.46410161513775) -- (-.75,3.03108891324553);
\node[black,,scale=1.66461870488604] at (-.75,3.03108891324553) {1};
\node[black,,scale=1.66461870488604] at (-.5,3.46410161513775) {1};
\path (-2,3.46410161513775) -- (-1.5,2.59807621135332) -- (-1,3.46410161513775) -- cycle;
\path[draw=yellow,line width=.108798608162487cm] (-1.5,3.46410161513775) -- (-1.75,3.03108891324554);
\node[black,,scale=1.66461870488604] at (-1.75,3.03108891324554) {3};
\node[black,,scale=1.66461870488604] at (-1.5,3.46410161513775) {3};
\end{scope}
\end{tikzpicture}
\qquad \\
1-\frac{y_2}{y_1} & \frac{y_2}{y_1} & -\frac{y_2}{y_1} & \qquad \frac{y_2}{y_1}
\end{array}
%%% Local Variables:
%%% TeX-master: "sepdesc"
%%% End:

%% file: ex2b.tex
\setlength{\arraycolsep}{-3pt} % AK put this in
\begin{array}{ccccc}
\begin{tikzpicture}[execute at begin picture={\bgroup\tikzset{every path/.style={}}\clip (-2.48,-.432012) rectangle ++(4.96,4.46413);\egroup},x={(1.55426771761071cm,0cm)},y={(0cm,-1.55426394417233cm)},baseline=(current  bounding  box.center),every path/.style={draw=black,fill=none},line join=round]
\begin{scope}[every path/.append style={fill=white,line width=.0310853166178533cm}]
\path (-.5,.866025403784439) -- (0,0) -- (.5,.866025403784439) -- cycle;
\path (-.5,.866025403784439) -- (.5,.866025403784439) -- (0,1.73205080756888) -- cycle;
\path[draw=green,line width=.108798608162487cm] (.25,.433012701892219) -- (0,.866025403784439);
\path[draw=green,line width=.108798608162487cm] (0,.866025403784439) -- (.25,1.29903810567666);
\node[black,,scale=1.66461870488604] at (.25,.433012701892219) {1};
\node[black,,scale=1.66461870488604] at (0,.866025403784439) {1};
\path (0,1.73205080756888) -- (.5,.866025403784439) -- (1,1.73205080756888) -- cycle;
\path (0,1.73205080756888) -- (1,1.73205080756888) -- (.5,2.59807621135332) -- cycle;
\path[draw=red,line width=.108798608162487cm] (.75,1.29903810567666) -- (.46,1.73205080756888);
\path[draw=red,line width=.108798608162487cm] (.25,2.1650635094611) -- (.46,1.73205080756888);
\path[draw=green,line width=.108798608162487cm] (.54,1.73205080756888) -- (.25,1.29903810567666);
\path[draw=green,line width=.108798608162487cm] (.54,1.73205080756888) -- (.75,2.1650635094611);
\node[black,,scale=1.66461870488604] at (.75,1.29903810567666) {0};
\node[black,,scale=1.66461870488604] at (.25,1.29903810567666) {1};
\node[black,,scale=1.38718062209258] at (.5,1.73205080756888) {01};
\path (.5,2.59807621135332) -- (1,1.73205080756888) -- (1.5,2.59807621135332) -- cycle;
\path (.5,2.59807621135332) -- (1.5,2.59807621135332) -- (1,3.46410161513775) -- cycle;
\path[draw=green,line width=.108798608162487cm] (1,2.59807621135332) -- (.75,2.1650635094611);
\path[draw=green,line width=.108798608162487cm] (1,2.59807621135332) -- (1.25,3.03108891324554);
\node[black,,scale=1.66461870488604] at (.75,2.1650635094611) {1};
\node[black,,scale=1.66461870488604] at (1,2.59807621135332) {1};
\path (1,3.46410161513775) -- (1.5,2.59807621135332) -- (2,3.46410161513775) -- cycle;
\path[draw=green,line width=.108798608162487cm] (1.5,3.46410161513775) -- (1.25,3.03108891324554);
\node[black,,scale=1.66461870488604] at (1.25,3.03108891324554) {1};
\node[black,,scale=1.66461870488604] at (1.5,3.46410161513775) {1};
\path (-1,1.73205080756888) -- (-.5,.866025403784439) -- (0,1.73205080756888) -- cycle;
\path (-1,1.73205080756888) -- (0,1.73205080756888) -- (-.5,2.59807621135332) -- cycle;
\path[draw=blue,line width=.108798608162487cm] (-.5,1.73205080756888) -- (-.75,1.29903810567666);
\path[draw=blue,line width=.108798608162487cm] (-.5,1.73205080756888) -- (-.25,2.1650635094611);
\node[black,,scale=1.66461870488604] at (-.75,1.29903810567666) {2};
\node[black,,scale=1.66461870488604] at (-.5,1.73205080756888) {2};
\path (-.5,2.59807621135332) -- (0,1.73205080756888) -- (.5,2.59807621135332) -- cycle;
\path (-.5,2.59807621135332) -- (.5,2.59807621135332) -- (0,3.46410161513775) -- cycle;
\path[draw=red,line width=.108798608162487cm] (.25,2.1650635094611) -- (-.04,2.59807621135332);
\path[draw=red,line width=.108798608162487cm] (-.25,3.03108891324554) -- (-.04,2.59807621135332);
\path[draw=blue,line width=.108798608162487cm] (.04,2.59807621135332) -- (-.25,2.1650635094611);
\path[draw=blue,line width=.108798608162487cm] (.04,2.59807621135332) -- (.25,3.03108891324554);
\node[black,,scale=1.66461870488604] at (.25,2.1650635094611) {0};
\node[black,,scale=1.66461870488604] at (-.25,2.1650635094611) {2};
\node[black,,scale=1.38718062209258] at (0,2.59807621135332) {02};
\path (0,3.46410161513775) -- (.5,2.59807621135332) -- (1,3.46410161513775) -- cycle;
\path[draw=blue,line width=.108798608162487cm] (.5,3.46410161513775) -- (.25,3.03108891324554);
\node[black,,scale=1.66461870488604] at (.25,3.03108891324554) {2};
\node[black,,scale=1.66461870488604] at (.5,3.46410161513775) {2};
\path[fill=LightGray] (-1.5,2.59807621135332) -- (-1,1.73205080756888) -- (-.5,2.59807621135332) -- (-1,3.46410161513775) -- cycle;
\path (-1,3.46410161513775) -- (-.5,2.59807621135332) -- (0,3.46410161513775) -- cycle;
\path[draw=red,line width=.108798608162487cm] (-.25,3.03108891324554) -- (-.5,3.46410161513775);
\node[black,,scale=1.66461870488604] at (-.25,3.03108891324554) {0};
\node[black,,scale=1.66461870488604] at (-.5,3.46410161513775) {0};
\path (-2,3.46410161513775) -- (-1.5,2.59807621135332) -- (-1,3.46410161513775) -- cycle;
\path[draw=yellow,line width=.108798608162487cm] (-1.5,3.46410161513775) -- (-1.75,3.03108891324554);
\node[black,,scale=1.66461870488604] at (-1.75,3.03108891324554) {3};
\node[black,,scale=1.66461870488604] at (-1.5,3.46410161513775) {3};
\end{scope}
\end{tikzpicture}&\begin{tikzpicture}[execute at begin picture={\bgroup\tikzset{every path/.style={}}\clip (-2.48,-.432012) rectangle ++(4.96,4.46413);\egroup},x={(1.55426771761071cm,0cm)},y={(0cm,-1.55426394417233cm)},baseline=(current  bounding  box.center),every path/.style={draw=black,fill=none},line join=round]
\begin{scope}[every path/.append style={fill=white,line width=.0310853166178533cm}]
\path (-.5,.866025403784439) -- (0,0) -- (.5,.866025403784439) -- cycle;
\path (-.5,.866025403784439) -- (.5,.866025403784439) -- (0,1.73205080756888) -- cycle;
\path[draw=green,line width=.108798608162487cm] (.25,.433012701892219) -- (0,.866025403784439);
\path[draw=green,line width=.108798608162487cm] (0,.866025403784439) -- (.25,1.29903810567666);
\node[black,,scale=1.66461870488604] at (.25,.433012701892219) {1};
\node[black,,scale=1.66461870488604] at (0,.866025403784439) {1};
\path (0,1.73205080756888) -- (.5,.866025403784439) -- (1,1.73205080756888) -- cycle;
\path (0,1.73205080756888) -- (1,1.73205080756888) -- (.5,2.59807621135332) -- cycle;
\path[draw=red,line width=.108798608162487cm] (.75,1.29903810567666) -- (.46,1.73205080756888);
\path[draw=red,line width=.108798608162487cm] (.25,2.1650635094611) -- (.46,1.73205080756888);
\path[draw=green,line width=.108798608162487cm] (.54,1.73205080756888) -- (.25,1.29903810567666);
\path[draw=green,line width=.108798608162487cm] (.54,1.73205080756888) -- (.75,2.1650635094611);
\node[black,,scale=1.66461870488604] at (.75,1.29903810567666) {0};
\node[black,,scale=1.66461870488604] at (.25,1.29903810567666) {1};
\node[black,,scale=1.38718062209258] at (.5,1.73205080756888) {01};
\path (.5,2.59807621135332) -- (1,1.73205080756888) -- (1.5,2.59807621135332) -- cycle;
\path (.5,2.59807621135332) -- (1.5,2.59807621135332) -- (1,3.46410161513775) -- cycle;
\path[draw=green,line width=.108798608162487cm] (1,2.59807621135332) -- (.75,2.1650635094611);
\path[draw=green,line width=.108798608162487cm] (1,2.59807621135332) -- (1.25,3.03108891324554);
\node[black,,scale=1.66461870488604] at (.75,2.1650635094611) {1};
\node[black,,scale=1.66461870488604] at (1,2.59807621135332) {1};
\path (1,3.46410161513775) -- (1.5,2.59807621135332) -- (2,3.46410161513775) -- cycle;
\path[draw=green,line width=.108798608162487cm] (1.5,3.46410161513775) -- (1.25,3.03108891324554);
\node[black,,scale=1.66461870488604] at (1.25,3.03108891324554) {1};
\node[black,,scale=1.66461870488604] at (1.5,3.46410161513775) {1};
\path (-1,1.73205080756888) -- (-.5,.866025403784439) -- (0,1.73205080756888) -- cycle;
\path (-1,1.73205080756888) -- (0,1.73205080756888) -- (-.5,2.59807621135332) -- cycle;
\path[draw=blue,line width=.108798608162487cm] (-.5,1.73205080756888) -- (-.75,1.29903810567666);
\path[draw=blue,line width=.108798608162487cm] (-.75,2.1650635094611) -- (-.5,1.73205080756888);
\node[black,,scale=1.66461870488604] at (-.75,1.29903810567666) {2};
\node[black,,scale=1.66461870488604] at (-.5,1.73205080756888) {2};
\path (-.5,2.59807621135332) -- (0,1.73205080756888) -- (.5,2.59807621135332) -- cycle;
\path (-.5,2.59807621135332) -- (.5,2.59807621135332) -- (0,3.46410161513775) -- cycle;
\path[draw=red,line width=.108798608162487cm] (.25,2.1650635094611) -- (0,2.59807621135332);
\path[draw=red,line width=.108798608162487cm] (0,2.59807621135332) -- (.25,3.03108891324554);
\node[black,,scale=1.66461870488604] at (.25,2.1650635094611) {0};
\node[black,,scale=1.66461870488604] at (0,2.59807621135332) {0};
\path (0,3.46410161513775) -- (.5,2.59807621135332) -- (1,3.46410161513775) -- cycle;
\path[draw=red,line width=.108798608162487cm] (.5,3.46410161513775) -- (.25,3.03108891324554);
\node[black,,scale=1.66461870488604] at (.25,3.03108891324554) {0};
\node[black,,scale=1.66461870488604] at (.5,3.46410161513775) {0};
\path (-1.5,2.59807621135332) -- (-1,1.73205080756888) -- (-.5,2.59807621135332) -- cycle;
\path (-1.5,2.59807621135332) -- (-.5,2.59807621135332) -- (-1,3.46410161513775) -- cycle;
\path[draw=blue,line width=.108798608162487cm] (-.75,2.1650635094611) -- (-1,2.59807621135332);
\path[draw=blue,line width=.108798608162487cm] (-1,2.59807621135332) -- (-.75,3.03108891324553);
\node[black,,scale=1.66461870488604] at (-.75,2.1650635094611) {2};
\node[black,,scale=1.66461870488604] at (-1,2.59807621135332) {2};
\path (-1,3.46410161513775) -- (-.5,2.59807621135332) -- (0,3.46410161513775) -- cycle;
\path[draw=blue,line width=.108798608162487cm] (-.5,3.46410161513775) -- (-.75,3.03108891324553);
\node[black,,scale=1.66461870488604] at (-.75,3.03108891324553) {2};
\node[black,,scale=1.66461870488604] at (-.5,3.46410161513775) {2};
\path (-2,3.46410161513775) -- (-1.5,2.59807621135332) -- (-1,3.46410161513775) -- cycle;
\path[draw=yellow,line width=.108798608162487cm] (-1.5,3.46410161513775) -- (-1.75,3.03108891324554);
\node[black,,scale=1.66461870488604] at (-1.75,3.03108891324554) {3};
\node[black,,scale=1.66461870488604] at (-1.5,3.46410161513775) {3};
\end{scope}
\end{tikzpicture}&\begin{tikzpicture}[execute at begin picture={\bgroup\tikzset{every path/.style={}}\clip (-2.48,-.432012) rectangle ++(4.96,4.46413);\egroup},x={(1.55426771761071cm,0cm)},y={(0cm,-1.55426394417233cm)},baseline=(current  bounding  box.center),every path/.style={draw=black,fill=none},line join=round]
\begin{scope}[every path/.append style={fill=white,line width=.0310853166178533cm}]
\path (-.5,.866025403784439) -- (0,0) -- (.5,.866025403784439) -- cycle;
\path (-.5,.866025403784439) -- (.5,.866025403784439) -- (0,1.73205080756888) -- cycle;
\path[draw=green,line width=.108798608162487cm] (.25,.433012701892219) -- (0,.866025403784439);
\path[draw=green,line width=.108798608162487cm] (0,.866025403784439) -- (.25,1.29903810567666);
\node[black,,scale=1.66461870488604] at (.25,.433012701892219) {1};
\node[black,,scale=1.66461870488604] at (0,.866025403784439) {1};
\path (0,1.73205080756888) -- (.5,.866025403784439) -- (1,1.73205080756888) -- cycle;
\path[fill=LightPink] (0,1.73205080756888) -- (1,1.73205080756888) -- (.5,2.59807621135332) -- cycle;
\path[draw=red,line width=.108798608162487cm] (.75,1.29903810567666) -- (.46,1.73205080756888);
\path[draw=red,line width=.108798608162487cm] (.46,1.73205080756888) -- (.75,2.1650635094611);
\path[draw=green,line width=.108798608162487cm] (.54,1.73205080756888) -- (.25,1.29903810567666);
\path[draw=green,line width=.108798608162487cm] (.25,2.1650635094611) -- (.54,1.73205080756888);
\node[black,,scale=1.66461870488604] at (.75,1.29903810567666) {0};
\node[black,,scale=1.66461870488604] at (.25,1.29903810567666) {1};
\node[black,,scale=1.38718062209258] at (.5,1.73205080756888) {01};
\path (.5,2.59807621135332) -- (1,1.73205080756888) -- (1.5,2.59807621135332) -- cycle;
\path (.5,2.59807621135332) -- (1.5,2.59807621135332) -- (1,3.46410161513775) -- cycle;
\path[draw=red,line width=.108798608162487cm] (1,2.59807621135332) -- (.75,2.1650635094611);
\path[draw=red,line width=.108798608162487cm] (1,2.59807621135332) -- (1.25,3.03108891324554);
\node[black,,scale=1.66461870488604] at (.75,2.1650635094611) {0};
\node[black,,scale=1.66461870488604] at (1,2.59807621135332) {0};
\path (1,3.46410161513775) -- (1.5,2.59807621135332) -- (2,3.46410161513775) -- cycle;
\path[draw=red,line width=.108798608162487cm] (1.5,3.46410161513775) -- (1.25,3.03108891324554);
\node[black,,scale=1.66461870488604] at (1.25,3.03108891324554) {0};
\node[black,,scale=1.66461870488604] at (1.5,3.46410161513775) {0};
\path (-1,1.73205080756888) -- (-.5,.866025403784439) -- (0,1.73205080756888) -- cycle;
\path (-1,1.73205080756888) -- (0,1.73205080756888) -- (-.5,2.59807621135332) -- cycle;
\path[draw=blue,line width=.108798608162487cm] (-.5,1.73205080756888) -- (-.75,1.29903810567666);
\path[draw=blue,line width=.108798608162487cm] (-.75,2.1650635094611) -- (-.5,1.73205080756888);
\node[black,,scale=1.66461870488604] at (-.75,1.29903810567666) {2};
\node[black,,scale=1.66461870488604] at (-.5,1.73205080756888) {2};
\path (-.5,2.59807621135332) -- (0,1.73205080756888) -- (.5,2.59807621135332) -- cycle;
\path (-.5,2.59807621135332) -- (.5,2.59807621135332) -- (0,3.46410161513775) -- cycle;
\path[draw=green,line width=.108798608162487cm] (.25,2.1650635094611) -- (0,2.59807621135332);
\path[draw=green,line width=.108798608162487cm] (0,2.59807621135332) -- (.25,3.03108891324554);
\node[black,,scale=1.66461870488604] at (.25,2.1650635094611) {1};
\node[black,,scale=1.66461870488604] at (0,2.59807621135332) {1};
\path (0,3.46410161513775) -- (.5,2.59807621135332) -- (1,3.46410161513775) -- cycle;
\path[draw=green,line width=.108798608162487cm] (.5,3.46410161513775) -- (.25,3.03108891324554);
\node[black,,scale=1.66461870488604] at (.25,3.03108891324554) {1};
\node[black,,scale=1.66461870488604] at (.5,3.46410161513775) {1};
\path (-1.5,2.59807621135332) -- (-1,1.73205080756888) -- (-.5,2.59807621135332) -- cycle;
\path (-1.5,2.59807621135332) -- (-.5,2.59807621135332) -- (-1,3.46410161513775) -- cycle;
\path[draw=blue,line width=.108798608162487cm] (-.75,2.1650635094611) -- (-1,2.59807621135332);
\path[draw=blue,line width=.108798608162487cm] (-1,2.59807621135332) -- (-.75,3.03108891324553);
\node[black,,scale=1.66461870488604] at (-.75,2.1650635094611) {2};
\node[black,,scale=1.66461870488604] at (-1,2.59807621135332) {2};
\path (-1,3.46410161513775) -- (-.5,2.59807621135332) -- (0,3.46410161513775) -- cycle;
\path[draw=blue,line width=.108798608162487cm] (-.5,3.46410161513775) -- (-.75,3.03108891324553);
\node[black,,scale=1.66461870488604] at (-.75,3.03108891324553) {2};
\node[black,,scale=1.66461870488604] at (-.5,3.46410161513775) {2};
\path (-2,3.46410161513775) -- (-1.5,2.59807621135332) -- (-1,3.46410161513775) -- cycle;
\path[draw=yellow,line width=.108798608162487cm] (-1.5,3.46410161513775) -- (-1.75,3.03108891324554);
\node[black,,scale=1.66461870488604] at (-1.75,3.03108891324554) {3};
\node[black,,scale=1.66461870488604] at (-1.5,3.46410161513775) {3};
\end{scope}
\end{tikzpicture}&\begin{tikzpicture}[execute at begin picture={\bgroup\tikzset{every path/.style={}}\clip (-2.48,-.432012) rectangle ++(4.96,4.46413);\egroup},x={(1.55426771761071cm,0cm)},y={(0cm,-1.55426394417233cm)},baseline=(current  bounding  box.center),every path/.style={draw=black,fill=none},line join=round]
\begin{scope}[every path/.append style={fill=white,line width=.0310853166178533cm}]
\path (-.5,.866025403784439) -- (0,0) -- (.5,.866025403784439) -- cycle;
\path (-.5,.866025403784439) -- (.5,.866025403784439) -- (0,1.73205080756888) -- cycle;
\path[draw=green,line width=.108798608162487cm] (.25,.433012701892219) -- (0,.866025403784439);
\path[draw=green,line width=.108798608162487cm] (0,.866025403784439) -- (.25,1.29903810567666);
\node[black,,scale=1.66461870488604] at (.25,.433012701892219) {1};
\node[black,,scale=1.66461870488604] at (0,.866025403784439) {1};
\path (0,1.73205080756888) -- (.5,.866025403784439) -- (1,1.73205080756888) -- cycle;
\path[fill=LightPink] (0,1.73205080756888) -- (1,1.73205080756888) -- (.5,2.59807621135332) -- cycle;
\path[draw=red,line width=.108798608162487cm] (.75,1.29903810567666) -- (.46,1.73205080756888);
\path[draw=red,line width=.108798608162487cm] (.46,1.73205080756888) -- (.75,2.1650635094611);
\path[draw=green,line width=.108798608162487cm] (.54,1.73205080756888) -- (.25,1.29903810567666);
\path[draw=green,line width=.108798608162487cm] (.25,2.1650635094611) -- (.54,1.73205080756888);
\node[black,,scale=1.66461870488604] at (.75,1.29903810567666) {0};
\node[black,,scale=1.66461870488604] at (.25,1.29903810567666) {1};
\node[black,,scale=1.38718062209258] at (.5,1.73205080756888) {01};
\path (.5,2.59807621135332) -- (1,1.73205080756888) -- (1.5,2.59807621135332) -- cycle;
\path (.5,2.59807621135332) -- (1.5,2.59807621135332) -- (1,3.46410161513775) -- cycle;
\path[draw=red,line width=.108798608162487cm] (1,2.59807621135332) -- (.75,2.1650635094611);
\path[draw=red,line width=.108798608162487cm] (1,2.59807621135332) -- (1.25,3.03108891324554);
\node[black,,scale=1.66461870488604] at (.75,2.1650635094611) {0};
\node[black,,scale=1.66461870488604] at (1,2.59807621135332) {0};
\path (1,3.46410161513775) -- (1.5,2.59807621135332) -- (2,3.46410161513775) -- cycle;
\path[draw=red,line width=.108798608162487cm] (1.5,3.46410161513775) -- (1.25,3.03108891324554);
\node[black,,scale=1.66461870488604] at (1.25,3.03108891324554) {0};
\node[black,,scale=1.66461870488604] at (1.5,3.46410161513775) {0};
\path (-1,1.73205080756888) -- (-.5,.866025403784439) -- (0,1.73205080756888) -- cycle;
\path (-1,1.73205080756888) -- (0,1.73205080756888) -- (-.5,2.59807621135332) -- cycle;
\path[draw=blue,line width=.108798608162487cm] (-.5,1.73205080756888) -- (-.75,1.29903810567666);
\path[draw=blue,line width=.108798608162487cm] (-.5,1.73205080756888) -- (-.25,2.1650635094611);
\node[black,,scale=1.66461870488604] at (-.75,1.29903810567666) {2};
\node[black,,scale=1.66461870488604] at (-.5,1.73205080756888) {2};
\path (-.5,2.59807621135332) -- (0,1.73205080756888) -- (.5,2.59807621135332) -- cycle;
\path (-.5,2.59807621135332) -- (.5,2.59807621135332) -- (0,3.46410161513775) -- cycle;
\path[draw=green,line width=.108798608162487cm] (.25,2.1650635094611) -- (-.04,2.59807621135332);
\path[draw=green,line width=.108798608162487cm] (-.25,3.03108891324554) -- (-.04,2.59807621135332);
\path[draw=blue,line width=.108798608162487cm] (.04,2.59807621135332) -- (-.25,2.1650635094611);
\path[draw=blue,line width=.108798608162487cm] (.04,2.59807621135332) -- (.25,3.03108891324554);
\node[black,,scale=1.66461870488604] at (.25,2.1650635094611) {1};
\node[black,,scale=1.66461870488604] at (-.25,2.1650635094611) {2};
\node[black,,scale=1.38718062209258] at (0,2.59807621135332) {12};
\path (0,3.46410161513775) -- (.5,2.59807621135332) -- (1,3.46410161513775) -- cycle;
\path[draw=blue,line width=.108798608162487cm] (.5,3.46410161513775) -- (.25,3.03108891324554);
\node[black,,scale=1.66461870488604] at (.25,3.03108891324554) {2};
\node[black,,scale=1.66461870488604] at (.5,3.46410161513775) {2};
\path[fill=LightGray] (-1.5,2.59807621135332) -- (-1,1.73205080756888) -- (-.5,2.59807621135332) -- (-1,3.46410161513775) -- cycle;
\path (-1,3.46410161513775) -- (-.5,2.59807621135332) -- (0,3.46410161513775) -- cycle;
\path[draw=green,line width=.108798608162487cm] (-.25,3.03108891324554) -- (-.5,3.46410161513775);
\node[black,,scale=1.66461870488604] at (-.25,3.03108891324554) {1};
\node[black,,scale=1.66461870488604] at (-.5,3.46410161513775) {1};
\path (-2,3.46410161513775) -- (-1.5,2.59807621135332) -- (-1,3.46410161513775) -- cycle;
\path[draw=yellow,line width=.108798608162487cm] (-1.5,3.46410161513775) -- (-1.75,3.03108891324554);
\node[black,,scale=1.66461870488604] at (-1.75,3.03108891324554) {3};
\node[black,,scale=1.66461870488604] at (-1.5,3.46410161513775) {3};
\end{scope}
\end{tikzpicture}&\begin{tikzpicture}[execute at begin picture={\bgroup\tikzset{every path/.style={}}\clip (-2.48,-.432012) rectangle ++(4.96,4.46413);\egroup},x={(1.55426771761071cm,0cm)},y={(0cm,-1.55426394417233cm)},baseline=(current  bounding  box.center),every path/.style={draw=black,fill=none},line join=round]
\begin{scope}[every path/.append style={fill=white,line width=.0310853166178533cm}]
\path (-.5,.866025403784439) -- (0,0) -- (.5,.866025403784439) -- cycle;
\path (-.5,.866025403784439) -- (.5,.866025403784439) -- (0,1.73205080756888) -- cycle;
\path[draw=green,line width=.108798608162487cm] (.25,.433012701892219) -- (0,.866025403784439);
\path[draw=green,line width=.108798608162487cm] (-.25,1.29903810567666) -- (0,.866025403784439);
\node[black,,scale=1.66461870488604] at (.25,.433012701892219) {1};
\node[black,,scale=1.66461870488604] at (0,.866025403784439) {1};
\path (0,1.73205080756888) -- (.5,.866025403784439) -- (1,1.73205080756888) -- cycle;
\path (0,1.73205080756888) -- (1,1.73205080756888) -- (.5,2.59807621135332) -- cycle;
\path[draw=red,line width=.108798608162487cm] (.75,1.29903810567666) -- (.5,1.73205080756888);
\path[draw=red,line width=.108798608162487cm] (.5,1.73205080756888) -- (.75,2.1650635094611);
\node[black,,scale=1.66461870488604] at (.75,1.29903810567666) {0};
\node[black,,scale=1.66461870488604] at (.5,1.73205080756888) {0};
\path (.5,2.59807621135332) -- (1,1.73205080756888) -- (1.5,2.59807621135332) -- cycle;
\path (.5,2.59807621135332) -- (1.5,2.59807621135332) -- (1,3.46410161513775) -- cycle;
\path[draw=red,line width=.108798608162487cm] (1,2.59807621135332) -- (.75,2.1650635094611);
\path[draw=red,line width=.108798608162487cm] (1,2.59807621135332) -- (1.25,3.03108891324554);
\node[black,,scale=1.66461870488604] at (.75,2.1650635094611) {0};
\node[black,,scale=1.66461870488604] at (1,2.59807621135332) {0};
\path (1,3.46410161513775) -- (1.5,2.59807621135332) -- (2,3.46410161513775) -- cycle;
\path[draw=red,line width=.108798608162487cm] (1.5,3.46410161513775) -- (1.25,3.03108891324554);
\node[black,,scale=1.66461870488604] at (1.25,3.03108891324554) {0};
\node[black,,scale=1.66461870488604] at (1.5,3.46410161513775) {0};
\path (-1,1.73205080756888) -- (-.5,.866025403784439) -- (0,1.73205080756888) -- cycle;
\path (-1,1.73205080756888) -- (0,1.73205080756888) -- (-.5,2.59807621135332) -- cycle;
\path[draw=green,line width=.108798608162487cm] (-.25,1.29903810567666) -- (-.54,1.73205080756888);
\path[draw=green,line width=.108798608162487cm] (-.75,2.1650635094611) -- (-.54,1.73205080756888);
\path[draw=blue,line width=.108798608162487cm] (-.46,1.73205080756888) -- (-.75,1.29903810567666);
\path[draw=blue,line width=.108798608162487cm] (-.46,1.73205080756888) -- (-.25,2.1650635094611);
\node[black,,scale=1.66461870488604] at (-.25,1.29903810567666) {1};
\node[black,,scale=1.66461870488604] at (-.75,1.29903810567666) {2};
\node[black,,scale=1.38718062209258] at (-.5,1.73205080756888) {12};
\path (-.5,2.59807621135332) -- (0,1.73205080756888) -- (.5,2.59807621135332) -- cycle;
\path (-.5,2.59807621135332) -- (.5,2.59807621135332) -- (0,3.46410161513775) -- cycle;
\path[draw=blue,line width=.108798608162487cm] (0,2.59807621135332) -- (-.25,2.1650635094611);
\path[draw=blue,line width=.108798608162487cm] (0,2.59807621135332) -- (.25,3.03108891324554);
\node[black,,scale=1.66461870488604] at (-.25,2.1650635094611) {2};
\node[black,,scale=1.66461870488604] at (0,2.59807621135332) {2};
\path (0,3.46410161513775) -- (.5,2.59807621135332) -- (1,3.46410161513775) -- cycle;
\path[draw=blue,line width=.108798608162487cm] (.5,3.46410161513775) -- (.25,3.03108891324554);
\node[black,,scale=1.66461870488604] at (.25,3.03108891324554) {2};
\node[black,,scale=1.66461870488604] at (.5,3.46410161513775) {2};
\path (-1.5,2.59807621135332) -- (-1,1.73205080756888) -- (-.5,2.59807621135332) -- cycle;
\path (-1.5,2.59807621135332) -- (-.5,2.59807621135332) -- (-1,3.46410161513775) -- cycle;
\path[draw=green,line width=.108798608162487cm] (-.75,2.1650635094611) -- (-1,2.59807621135332);
\path[draw=green,line width=.108798608162487cm] (-1,2.59807621135332) -- (-.75,3.03108891324553);
\node[black,,scale=1.66461870488604] at (-.75,2.1650635094611) {1};
\node[black,,scale=1.66461870488604] at (-1,2.59807621135332) {1};
\path (-1,3.46410161513775) -- (-.5,2.59807621135332) -- (0,3.46410161513775) -- cycle;
\path[draw=green,line width=.108798608162487cm] (-.5,3.46410161513775) -- (-.75,3.03108891324553);
\node[black,,scale=1.66461870488604] at (-.75,3.03108891324553) {1};
\node[black,,scale=1.66461870488604] at (-.5,3.46410161513775) {1};
\path (-2,3.46410161513775) -- (-1.5,2.59807621135332) -- (-1,3.46410161513775) -- cycle;
\path[draw=yellow,line width=.108798608162487cm] (-1.5,3.46410161513775) -- (-1.75,3.03108891324554);
\node[black,,scale=1.66461870488604] at (-1.75,3.03108891324554) {3};
\node[black,,scale=1.66461870488604] at (-1.5,3.46410161513775) {3};
\end{scope}
\end{tikzpicture}\\
1-\frac{y_2}{y_1}&\frac{y_2}{y_1}&-\frac{y_2}{y_1}&-\left(1-\frac{y_2}{y_1}\right)&1
%\frac{z_{2}-z_{1}}{z_{2}}&\frac{z_{1}}{z_{2}}&\frac{-\left(z_{2}-z_{1}\right)}{z_{2}}&\frac{-z_{1}}{z_{2}}&1
\end{array}
%%% Local Variables:
%%% TeX-master: "sepdesc"
%%% End:

%% file: exchi.tex
\begin{tikzpicture}[execute at begin picture={\bgroup\tikzset{every path/.style={}}\clip (-2.48,-.432012) rectangle ++(4.96,4.46413);\egroup},x={(1.55426771761071cm,0cm)},y={(0cm,-1.55426394417233cm)},baseline=(current  bounding  box.center),every path/.style={draw=black,fill=none},line join=round]
\begin{scope}[every path/.append style={fill=white,line width=.0310853166178533cm}]
\path (-.5,.866025403784439) -- (0,0) -- (.5,.866025403784439) -- cycle;
\path (-.5,.866025403784439) -- (.5,.866025403784439) -- (0,1.73205080756888) -- cycle;
\path[draw=green,line width=.108798608162487cm] (.25,.433012701892219) -- (-.04,.866025403784439);
\path[draw=green,line width=.108798608162487cm] (-.25,1.29903810567666) -- (-.04,.866025403784439);
\path[draw=blue,line width=.108798608162487cm] (.04,.866025403784439) -- (-.25,.433012701892219);
\path[draw=blue,line width=.108798608162487cm] (.04,.866025403784439) -- (.25,1.29903810567666);
\node[black,,scale=1.66461870488604] at (.25,.433012701892219) {1};
\node[black,,scale=1.66461870488604] at (-.25,.433012701892219) {2};
\node[black,,scale=1.38718062209258] at (0,.866025403784439) {12};
\path (0,1.73205080756888) -- (.5,.866025403784439) -- (1,1.73205080756888) -- cycle;
\path (0,1.73205080756888) -- (1,1.73205080756888) -- (.5,2.59807621135332) -- cycle;
\path[draw=red,line width=.108798608162487cm] (.75,1.29903810567666) -- (.46,1.73205080756888);
\path[draw=red,line width=.108798608162487cm] (.46,1.73205080756888) -- (.75,2.1650635094611);
\path[draw=blue,line width=.108798608162487cm] (.54,1.73205080756888) -- (.25,1.29903810567666);
\path[draw=blue,line width=.108798608162487cm] (.25,2.1650635094611) -- (.54,1.73205080756888);
\node[black,,scale=1.66461870488604] at (.75,1.29903810567666) {0};
\node[black,,scale=1.66461870488604] at (.25,1.29903810567666) {2};
\node[black,,scale=1.38718062209258] at (.5,1.73205080756888) {02};
\path (.5,2.59807621135332) -- (1,1.73205080756888) -- (1.5,2.59807621135332) -- cycle;
\path (.5,2.59807621135332) -- (1.5,2.59807621135332) -- (1,3.46410161513775) -- cycle;
\path[draw=red,line width=.108798608162487cm] (1,2.59807621135332) -- (.75,2.1650635094611);
\path[draw=red,line width=.108798608162487cm] (1,2.59807621135332) -- (1.25,3.03108891324554);
\node[black,,scale=1.66461870488604] at (.75,2.1650635094611) {0};
\node[black,,scale=1.66461870488604] at (1,2.59807621135332) {0};
\path (1,3.46410161513775) -- (1.5,2.59807621135332) -- (2,3.46410161513775) -- cycle;
\path[draw=red,line width=.108798608162487cm] (1.5,3.46410161513775) -- (1.25,3.03108891324554);
\node[black,,scale=1.66461870488604] at (1.25,3.03108891324554) {0};
\node[black,,scale=1.66461870488604] at (1.5,3.46410161513775) {0};
\path (-1,1.73205080756888) -- (-.5,.866025403784439) -- (0,1.73205080756888) -- cycle;
\path (-1,1.73205080756888) -- (0,1.73205080756888) -- (-.5,2.59807621135332) -- cycle;
\path[draw=green,line width=.108798608162487cm] (-.25,1.29903810567666) -- (-.5,1.73205080756888);
\path[draw=green,line width=.108798608162487cm] (-.5,1.73205080756888) -- (-.25,2.1650635094611);
\node[black,,scale=1.66461870488604] at (-.25,1.29903810567666) {1};
\node[black,,scale=1.66461870488604] at (-.5,1.73205080756888) {1};
\path (-.5,2.59807621135332) -- (0,1.73205080756888) -- (.5,2.59807621135332) -- cycle;
\path (-.5,2.59807621135332) -- (.5,2.59807621135332) -- (0,3.46410161513775) -- cycle;
\path[draw=green,line width=.108798608162487cm] (-.04,2.59807621135332) -- (-.25,2.1650635094611);
\path[draw=green,line width=.108798608162487cm] (-.25,3.03108891324554) -- (-.04,2.59807621135332);
\path[draw=blue,line width=.108798608162487cm] (.25,2.1650635094611) -- (.04,2.59807621135332);
\path[draw=blue,line width=.108798608162487cm] (.04,2.59807621135332) -- (.25,3.03108891324554);
\node[black,,scale=1.66461870488604] at (.25,2.1650635094611) {2};
\node[black,,scale=1.66461870488604] at (-.25,2.1650635094611) {1};
\node[black,,scale=1.38718062209258] at (0,2.59807621135332) {12};
\path (0,3.46410161513775) -- (.5,2.59807621135332) -- (1,3.46410161513775) -- cycle;
\path[draw=blue,line width=.108798608162487cm] (.5,3.46410161513775) -- (.25,3.03108891324554);
\node[black,,scale=1.66461870488604] at (.25,3.03108891324554) {2};
\node[black,,scale=1.66461870488604] at (.5,3.46410161513775) {2};
\path (-1.5,2.59807621135332) -- (-1,1.73205080756888) -- (-.5,2.59807621135332) -- cycle;
\path (-1.5,2.59807621135332) -- (-.5,2.59807621135332) -- (-1,3.46410161513775) -- cycle;
\path[draw=blue,line width=.108798608162487cm] (-1,2.59807621135332) -- (-1.25,2.1650635094611);
\path[draw=blue,line width=.108798608162487cm] (-1.25,3.03108891324554) -- (-1,2.59807621135332);
\node[black,,scale=1.66461870488604] at (-1.25,2.1650635094611) {2};
\node[black,,scale=1.66461870488604] at (-1,2.59807621135332) {2};
\path (-1,3.46410161513775) -- (-.5,2.59807621135332) -- (0,3.46410161513775) -- cycle;
\path[draw=green,line width=.108798608162487cm] (-.25,3.03108891324554) -- (-.5,3.46410161513775);
\node[black,,scale=1.66461870488604] at (-.25,3.03108891324554) {1};
\node[black,,scale=1.66461870488604] at (-.5,3.46410161513775) {1};
\path (-2,3.46410161513775) -- (-1.5,2.59807621135332) -- (-1,3.46410161513775) -- cycle;
\path[draw=blue,line width=.108798608162487cm] (-1.25,3.03108891324554) -- (-1.5,3.46410161513775);
\node[black,,scale=1.66461870488604] at (-1.25,3.03108891324554) {2};
\node[black,,scale=1.66461870488604] at (-1.5,3.46410161513775) {2};
\end{scope}
\end{tikzpicture}\begin{tikzpicture}[execute at begin picture={\bgroup\tikzset{every path/.style={}}\clip (-2.48,-.432012) rectangle ++(4.96,4.46413);\egroup},x={(1.55426771761071cm,0cm)},y={(0cm,-1.55426394417233cm)},baseline=(current  bounding  box.center),every path/.style={draw=black,fill=none},line join=round]
\begin{scope}[every path/.append style={fill=white,line width=.0310853166178533cm}]
\path (-.5,.866025403784439) -- (0,0) -- (.5,.866025403784439) -- cycle;
\path (-.5,.866025403784439) -- (.5,.866025403784439) -- (0,1.73205080756888) -- cycle;
\path[draw=green,line width=.108798608162487cm] (.25,.433012701892219) -- (-.04,.866025403784439);
\path[draw=green,line width=.108798608162487cm] (-.25,1.29903810567666) -- (-.04,.866025403784439);
\path[draw=blue,line width=.108798608162487cm] (.04,.866025403784439) -- (-.25,.433012701892219);
\path[draw=blue,line width=.108798608162487cm] (.04,.866025403784439) -- (.25,1.29903810567666);
\node[black,,scale=1.66461870488604] at (.25,.433012701892219) {1};
\node[black,,scale=1.66461870488604] at (-.25,.433012701892219) {2};
\node[black,,scale=1.38718062209258] at (0,.866025403784439) {12};
\path (0,1.73205080756888) -- (.5,.866025403784439) -- (1,1.73205080756888) -- cycle;
\path (0,1.73205080756888) -- (1,1.73205080756888) -- (.5,2.59807621135332) -- cycle;
\path[draw=red,line width=.108798608162487cm] (.75,1.29903810567666) -- (.46,1.73205080756888);
\path[draw=red,line width=.108798608162487cm] (.46,1.73205080756888) -- (.75,2.1650635094611);
\path[draw=blue,line width=.108798608162487cm] (.54,1.73205080756888) -- (.25,1.29903810567666);
\path[draw=blue,line width=.108798608162487cm] (.25,2.1650635094611) -- (.54,1.73205080756888);
\node[black,,scale=1.66461870488604] at (.75,1.29903810567666) {0};
\node[black,,scale=1.66461870488604] at (.25,1.29903810567666) {2};
\node[black,,scale=1.38718062209258] at (.5,1.73205080756888) {02};
\path (.5,2.59807621135332) -- (1,1.73205080756888) -- (1.5,2.59807621135332) -- cycle;
\path (.5,2.59807621135332) -- (1.5,2.59807621135332) -- (1,3.46410161513775) -- cycle;
\path[draw=red,line width=.108798608162487cm] (1,2.59807621135332) -- (.75,2.1650635094611);
\path[draw=red,line width=.108798608162487cm] (1,2.59807621135332) -- (1.25,3.03108891324554);
\node[black,,scale=1.66461870488604] at (.75,2.1650635094611) {0};
\node[black,,scale=1.66461870488604] at (1,2.59807621135332) {0};
\path (1,3.46410161513775) -- (1.5,2.59807621135332) -- (2,3.46410161513775) -- cycle;
\path[draw=red,line width=.108798608162487cm] (1.5,3.46410161513775) -- (1.25,3.03108891324554);
\node[black,,scale=1.66461870488604] at (1.25,3.03108891324554) {0};
\node[black,,scale=1.66461870488604] at (1.5,3.46410161513775) {0};
\path (-1,1.73205080756888) -- (-.5,.866025403784439) -- (0,1.73205080756888) -- cycle;
\path (-1,1.73205080756888) -- (0,1.73205080756888) -- (-.5,2.59807621135332) -- cycle;
\path[draw=green,line width=.108798608162487cm] (-.25,1.29903810567666) -- (-.5,1.73205080756888);
\path[draw=green,line width=.108798608162487cm] (-.75,2.1650635094611) -- (-.5,1.73205080756888);
\node[black,,scale=1.66461870488604] at (-.25,1.29903810567666) {1};
\node[black,,scale=1.66461870488604] at (-.5,1.73205080756888) {1};
\path (-.5,2.59807621135332) -- (0,1.73205080756888) -- (.5,2.59807621135332) -- cycle;
\path (-.5,2.59807621135332) -- (.5,2.59807621135332) -- (0,3.46410161513775) -- cycle;
\path[draw=blue,line width=.108798608162487cm] (.25,2.1650635094611) -- (0,2.59807621135332);
\path[draw=blue,line width=.108798608162487cm] (0,2.59807621135332) -- (.25,3.03108891324554);
\node[black,,scale=1.66461870488604] at (.25,2.1650635094611) {2};
\node[black,,scale=1.66461870488604] at (0,2.59807621135332) {2};
\path (0,3.46410161513775) -- (.5,2.59807621135332) -- (1,3.46410161513775) -- cycle;
\path[draw=blue,line width=.108798608162487cm] (.5,3.46410161513775) -- (.25,3.03108891324554);
\node[black,,scale=1.66461870488604] at (.25,3.03108891324554) {2};
\node[black,,scale=1.66461870488604] at (.5,3.46410161513775) {2};
\path (-1.5,2.59807621135332) -- (-1,1.73205080756888) -- (-.5,2.59807621135332) -- cycle;
\path (-1.5,2.59807621135332) -- (-.5,2.59807621135332) -- (-1,3.46410161513775) -- cycle;
\path[draw=green,line width=.108798608162487cm] (-.75,2.1650635094611) -- (-1.04,2.59807621135332);
\path[draw=green,line width=.108798608162487cm] (-1.04,2.59807621135332) -- (-.75,3.03108891324553);
\path[draw=blue,line width=.108798608162487cm] (-.96,2.59807621135332) -- (-1.25,2.1650635094611);
\path[draw=blue,line width=.108798608162487cm] (-1.25,3.03108891324554) -- (-.96,2.59807621135332);
\node[black,,scale=1.66461870488604] at (-.75,2.1650635094611) {1};
\node[black,,scale=1.66461870488604] at (-1.25,2.1650635094611) {2};
\node[black,,scale=1.38718062209258] at (-1,2.59807621135332) {12};
\path (-1,3.46410161513775) -- (-.5,2.59807621135332) -- (0,3.46410161513775) -- cycle;
\path[draw=green,line width=.108798608162487cm] (-.5,3.46410161513775) -- (-.75,3.03108891324553);
\node[black,,scale=1.66461870488604] at (-.75,3.03108891324553) {1};
\node[black,,scale=1.66461870488604] at (-.5,3.46410161513775) {1};
\path (-2,3.46410161513775) -- (-1.5,2.59807621135332) -- (-1,3.46410161513775) -- cycle;
\path[draw=blue,line width=.108798608162487cm] (-1.25,3.03108891324554) -- (-1.5,3.46410161513775);
\node[black,,scale=1.66461870488604] at (-1.25,3.03108891324554) {2};
\node[black,,scale=1.66461870488604] at (-1.5,3.46410161513775) {2};
\end{scope}
\end{tikzpicture}\begin{tikzpicture}[execute at begin picture={\bgroup\tikzset{every path/.style={}}\clip (-2.48,-.432012) rectangle ++(4.96,4.46413);\egroup},x={(1.55426771761071cm,0cm)},y={(0cm,-1.55426394417233cm)},baseline=(current  bounding  box.center),every path/.style={draw=black,fill=none},line join=round]
\begin{scope}[every path/.append style={fill=white,line width=.0310853166178533cm}]
\path (-.5,.866025403784439) -- (0,0) -- (.5,.866025403784439) -- cycle;
\path (-.5,.866025403784439) -- (.5,.866025403784439) -- (0,1.73205080756888) -- cycle;
\path[draw=green,line width=.108798608162487cm] (.25,.433012701892219) -- (-.04,.866025403784439);
\path[draw=green,line width=.108798608162487cm] (-.04,.866025403784439) -- (.25,1.29903810567666);
\path[draw=blue,line width=.108798608162487cm] (.04,.866025403784439) -- (-.25,.433012701892219);
\path[draw=blue,line width=.108798608162487cm] (-.25,1.29903810567666) -- (.04,.866025403784439);
\node[black,,scale=1.66461870488604] at (.25,.433012701892219) {1};
\node[black,,scale=1.66461870488604] at (-.25,.433012701892219) {2};
\node[black,,scale=1.38718062209258] at (0,.866025403784439) {12};
\path (0,1.73205080756888) -- (.5,.866025403784439) -- (1,1.73205080756888) -- cycle;
\path (0,1.73205080756888) -- (1,1.73205080756888) -- (.5,2.59807621135332) -- cycle;
\path[draw=red,line width=.108798608162487cm] (.75,1.29903810567666) -- (.46,1.73205080756888);
\path[draw=red,line width=.108798608162487cm] (.46,1.73205080756888) -- (.75,2.1650635094611);
\path[draw=green,line width=.108798608162487cm] (.54,1.73205080756888) -- (.25,1.29903810567666);
\path[draw=green,line width=.108798608162487cm] (.25,2.1650635094611) -- (.54,1.73205080756888);
\node[black,,scale=1.66461870488604] at (.75,1.29903810567666) {0};
\node[black,,scale=1.66461870488604] at (.25,1.29903810567666) {1};
\node[black,,scale=1.38718062209258] at (.5,1.73205080756888) {01};
\path (.5,2.59807621135332) -- (1,1.73205080756888) -- (1.5,2.59807621135332) -- cycle;
\path (.5,2.59807621135332) -- (1.5,2.59807621135332) -- (1,3.46410161513775) -- cycle;
\path[draw=red,line width=.108798608162487cm] (1,2.59807621135332) -- (.75,2.1650635094611);
\path[draw=red,line width=.108798608162487cm] (1,2.59807621135332) -- (1.25,3.03108891324554);
\node[black,,scale=1.66461870488604] at (.75,2.1650635094611) {0};
\node[black,,scale=1.66461870488604] at (1,2.59807621135332) {0};
\path (1,3.46410161513775) -- (1.5,2.59807621135332) -- (2,3.46410161513775) -- cycle;
\path[draw=red,line width=.108798608162487cm] (1.5,3.46410161513775) -- (1.25,3.03108891324554);
\node[black,,scale=1.66461870488604] at (1.25,3.03108891324554) {0};
\node[black,,scale=1.66461870488604] at (1.5,3.46410161513775) {0};
\path (-1,1.73205080756888) -- (-.5,.866025403784439) -- (0,1.73205080756888) -- cycle;
\path (-1,1.73205080756888) -- (0,1.73205080756888) -- (-.5,2.59807621135332) -- cycle;
\path[draw=blue,line width=.108798608162487cm] (-.25,1.29903810567666) -- (-.5,1.73205080756888);
\path[draw=blue,line width=.108798608162487cm] (-.5,1.73205080756888) -- (-.25,2.1650635094611);
\node[black,,scale=1.66461870488604] at (-.25,1.29903810567666) {2};
\node[black,,scale=1.66461870488604] at (-.5,1.73205080756888) {2};
\path (-.5,2.59807621135332) -- (0,1.73205080756888) -- (.5,2.59807621135332) -- cycle;
\path (-.5,2.59807621135332) -- (.5,2.59807621135332) -- (0,3.46410161513775) -- cycle;
\path[draw=green,line width=.108798608162487cm] (.25,2.1650635094611) -- (-.04,2.59807621135332);
\path[draw=green,line width=.108798608162487cm] (-.25,3.03108891324554) -- (-.04,2.59807621135332);
\path[draw=blue,line width=.108798608162487cm] (.04,2.59807621135332) -- (-.25,2.1650635094611);
\path[draw=blue,line width=.108798608162487cm] (.04,2.59807621135332) -- (.25,3.03108891324554);
\node[black,,scale=1.66461870488604] at (.25,2.1650635094611) {1};
\node[black,,scale=1.66461870488604] at (-.25,2.1650635094611) {2};
\node[black,,scale=1.38718062209258] at (0,2.59807621135332) {12};
\path (0,3.46410161513775) -- (.5,2.59807621135332) -- (1,3.46410161513775) -- cycle;
\path[draw=blue,line width=.108798608162487cm] (.5,3.46410161513775) -- (.25,3.03108891324554);
\node[black,,scale=1.66461870488604] at (.25,3.03108891324554) {2};
\node[black,,scale=1.66461870488604] at (.5,3.46410161513775) {2};
\path (-1.5,2.59807621135332) -- (-1,1.73205080756888) -- (-.5,2.59807621135332) -- cycle;
\path (-1.5,2.59807621135332) -- (-.5,2.59807621135332) -- (-1,3.46410161513775) -- cycle;
\path[draw=blue,line width=.108798608162487cm] (-1,2.59807621135332) -- (-1.25,2.1650635094611);
\path[draw=blue,line width=.108798608162487cm] (-1.25,3.03108891324554) -- (-1,2.59807621135332);
\node[black,,scale=1.66461870488604] at (-1.25,2.1650635094611) {2};
\node[black,,scale=1.66461870488604] at (-1,2.59807621135332) {2};
\path (-1,3.46410161513775) -- (-.5,2.59807621135332) -- (0,3.46410161513775) -- cycle;
\path[draw=green,line width=.108798608162487cm] (-.25,3.03108891324554) -- (-.5,3.46410161513775);
\node[black,,scale=1.66461870488604] at (-.25,3.03108891324554) {1};
\node[black,,scale=1.66461870488604] at (-.5,3.46410161513775) {1};
\path (-2,3.46410161513775) -- (-1.5,2.59807621135332) -- (-1,3.46410161513775) -- cycle;
\path[draw=blue,line width=.108798608162487cm] (-1.25,3.03108891324554) -- (-1.5,3.46410161513775);
\node[black,,scale=1.66461870488604] at (-1.25,3.03108891324554) {2};
\node[black,,scale=1.66461870488604] at (-1.5,3.46410161513775) {2};
\end{scope}
\end{tikzpicture}

%% file: ex5.tex
\begin{tikzpicture}[execute at begin picture={\bgroup\tikzset{every path/.style={}}\clip (-3.1,-.533215) rectangle ++(6.2,5.50989);\egroup},x={(1.56304426682143cm,0cm)},y={(0cm,-1.56304631412272cm)},baseline=(current  bounding  box.center),every path/.style={draw=black,fill=none},line join=round]
\begin{scope}[every path/.append style={fill=white,line width=.0312609058094483cm}]
\path (-.5,.866025403784439) -- (0,0) -- (.5,.866025403784439) -- cycle;
\path (-.5,.866025403784439) -- (.5,.866025403784439) -- (0,1.73205080756888) -- cycle;
\node[black,,scale=1.04626344130997] at (0,.866025403784439) {even};
\path (0,1.73205080756888) -- (.5,.866025403784439) -- (1,1.73205080756888) -- cycle;
\path (0,1.73205080756888) -- (1,1.73205080756888) -- (.5,2.59807621135332) -- cycle;
\path[draw=blue,line width=.109413170333069cm] (.27,2.19970452561247) -- (.75,2.1650635094611);
\path[draw=magenta,line width=.109413170333069cm] (.75,1.29903810567666) -- (.56,1.73205080756888);
\path[draw=magenta,line width=.109413170333069cm] (.23,2.13042249330972) -- (.56,1.73205080756888);
\node[black,,scale=1.67402150609596] at (.75,1.29903810567666) {4};
\node[black,,scale=1.39501628054907] at (.5,1.73205080756888) {\(\nearrow\)4};
\path (.5,2.59807621135332) -- (1,1.73205080756888) -- (1.5,2.59807621135332) -- cycle;
\path (.5,2.59807621135332) -- (1.5,2.59807621135332) -- (1,3.46410161513775) -- cycle;
\path[draw=blue,line width=.109413170333069cm] (1.25,2.1650635094611) -- (.75,2.1650635094611);
\path[draw=yellow,line width=.109413170333069cm] (.75,3.03108891324553) -- (1.25,3.03108891324554);
\node[black,,scale=1.67402150609596] at (1.25,2.1650635094611) {2};
\node[black,,scale=1.67402150609596] at (.75,2.1650635094611) {2};
\node[black,,scale=1.19572894384102] at (1,2.59807621135332) {odd};
\path (1,3.46410161513775) -- (1.5,2.59807621135332) -- (2,3.46410161513775) -- cycle;
\path (1,3.46410161513775) -- (2,3.46410161513775) -- (1.5,4.33012701892219) -- cycle;
\path[draw=red,line width=.109413170333069cm] (1.25,3.89711431702997) -- (1.75,3.89711431702997);
\path[draw=yellow,line width=.109413170333069cm] (1.75,3.03108891324554) -- (1.25,3.03108891324554);
\node[black,,scale=1.67402150609596] at (1.75,3.03108891324554) {3};
\node[black,,scale=1.67402150609596] at (1.25,3.03108891324554) {3};
\node[black,,scale=1.19572894384102] at (1.5,3.46410161513775) {odd};
\path (1.5,4.33012701892219) -- (2,3.46410161513775) -- (2.5,4.33012701892219) -- cycle;
\path[draw=red,line width=.109413170333069cm] (2.06,4.33012701892219) -- (1.75,3.89711431702997);
\node[black,,scale=1.67402150609596] at (1.75,3.89711431702997) {0};
\node[black,,scale=1.39501628054907] at (2,4.33012701892219) {\(\searrow\)0};
\path (-1,1.73205080756888) -- (-.5,.866025403784439) -- (0,1.73205080756888) -- cycle;
\path (-1,1.73205080756888) -- (0,1.73205080756888) -- (-.5,2.59807621135332) -- cycle;
\path[draw=blue,line width=.109413170333069cm] (-.44,1.73205080756888) -- (-.75,1.29903810567666);
\path[draw=blue,line width=.109413170333069cm] (-.44,1.73205080756888) -- (-.25,2.1650635094611);
\node[black,,scale=1.67402150609596] at (-.75,1.29903810567666) {2};
\node[black,,scale=1.39501628054907] at (-.5,1.73205080756888) {\(\searrow\)2};
\path (-.5,2.59807621135332) -- (0,1.73205080756888) -- (.5,2.59807621135332) -- cycle;
\path (-.5,2.59807621135332) -- (.5,2.59807621135332) -- (0,3.46410161513775) -- cycle;
\path[draw=blue,line width=.109413170333069cm] (.27,2.19970452561247) -- (-.25,2.1650635094611);
\path[draw=magenta,line width=.109413170333069cm] (.23,2.13042249330972) -- (.0600000000000001,2.59807621135332);
\path[draw=magenta,line width=.109413170333069cm] (-.25,3.03108891324554) -- (.0600000000000001,2.59807621135332);
\node[black,,scale=1.39501628054907] at (.25,2.1650635094611) {24};
\node[black,,scale=1.67402150609596] at (-.25,2.1650635094611) {2};
\node[black,,scale=1.39501628054907] at (0,2.59807621135332) {\(\nearrow\)4};
\path (0,3.46410161513775) -- (.5,2.59807621135332) -- (1,3.46410161513775) -- cycle;
\path (0,3.46410161513775) -- (1,3.46410161513775) -- (.5,4.33012701892219) -- cycle;
\path[draw=red,line width=.109413170333069cm] (.27,3.93175533318135) -- (.75,3.89711431702997);
\path[draw=yellow,line width=.109413170333069cm] (.75,3.03108891324553) -- (.56,3.46410161513775);
\path[draw=yellow,line width=.109413170333069cm] (.23,3.8624733008786) -- (.56,3.46410161513775);
\node[black,,scale=1.67402150609596] at (.75,3.03108891324553) {3};
\node[black,,scale=1.39501628054907] at (.5,3.46410161513775) {\(\nearrow\)3};
\path (.5,4.33012701892219) -- (1,3.46410161513775) -- (1.5,4.33012701892219) -- cycle;
\path[draw=red,line width=.109413170333069cm] (1.25,3.89711431702997) -- (.75,3.89711431702997);
\node[black,,scale=1.67402150609596] at (1.25,3.89711431702997) {0};
\node[black,,scale=1.67402150609596] at (.75,3.89711431702997) {0};
\node[black,,scale=1.19572894384102] at (1,4.33012701892219) {odd};
\path (-1.5,2.59807621135332) -- (-1,1.73205080756888) -- (-.5,2.59807621135332) -- cycle;
\path (-1.5,2.59807621135332) -- (-.5,2.59807621135332) -- (-1,3.46410161513775) -- cycle;
\node[black,,scale=1.04626344130997] at (-1,2.59807621135332) {even};
\path (-1,3.46410161513775) -- (-.5,2.59807621135332) -- (0,3.46410161513775) -- cycle;
\path (-1,3.46410161513775) -- (0,3.46410161513775) -- (-.5,4.33012701892219) -- cycle;
\path[draw=red,line width=.109413170333069cm] (-.72,3.94907584125704) -- (-.276666666666667,3.94330233856514);
\path[draw=green,line width=.109413170333069cm] (-.74,3.91443482510566) -- (-.25,3.89711431702997);
\path[draw=yellow,line width=.109413170333069cm] (-.76,3.87979380895429) -- (-.223333333333333,3.8509262954948);
\path[draw=magenta,line width=.109413170333069cm] (-.25,3.03108891324554) -- (-.44,3.46410161513775);
\path[draw=magenta,line width=.109413170333069cm] (-.78,3.84515279280291) -- (-.44,3.46410161513775);
\node[black,,scale=1.67402150609596] at (-.25,3.03108891324554) {4};
\node[black,,scale=1.39501628054907] at (-.5,3.46410161513775) {\(\nearrow\)4};
\path (-.5,4.33012701892219) -- (0,3.46410161513775) -- (.5,4.33012701892219) -- cycle;
\path[draw=red,line width=.109413170333069cm] (.27,3.93175533318135) -- (-.276666666666667,3.94330233856514);
\path[draw=green,line width=.109413170333069cm] (.0600000000000001,4.33012701892219) -- (-.25,3.89711431702997);
\path[draw=yellow,line width=.109413170333069cm] (.23,3.8624733008786) -- (-.223333333333333,3.8509262954948);
\node[black,,scale=1.39501628054907] at (.25,3.89711431702997) {03};
\node[black,,scale=1.19572894384102] at (-.25,3.89711431702997) {013};
\node[black,,scale=1.39501628054907] at (0,4.33012701892219) {\(\searrow\)1};
\path (-2,3.46410161513775) -- (-1.5,2.59807621135332) -- (-1,3.46410161513775) -- cycle;
\path (-2,3.46410161513775) -- (-1,3.46410161513775) -- (-1.5,4.33012701892219) -- cycle;
\path[draw=red,line width=.109413170333069cm] (-1.44,3.46410161513775) -- (-1.75,3.03108891324554);
\path[draw=red,line width=.109413170333069cm] (-1.44,3.46410161513775) -- (-1.27666666666667,3.94330233856514);
\path[draw=green,line width=.109413170333069cm] (-1.73,3.93175533318135) -- (-1.25,3.89711431702997);
\path[draw=yellow,line width=.109413170333069cm] (-1.77,3.8624733008786) -- (-1.22333333333333,3.8509262954948);
\node[black,,scale=1.67402150609596] at (-1.75,3.03108891324554) {0};
\node[black,,scale=1.39501628054907] at (-1.5,3.46410161513775) {\(\searrow\)0};
\path (-1.5,4.33012701892219) -- (-1,3.46410161513775) -- (-.5,4.33012701892219) -- cycle;
\path[draw=red,line width=.109413170333069cm] (-.72,3.94907584125704) -- (-1.27666666666667,3.94330233856514);
\path[draw=green,line width=.109413170333069cm] (-.74,3.91443482510566) -- (-1.25,3.89711431702997);
\path[draw=yellow,line width=.109413170333069cm] (-.76,3.87979380895429) -- (-1.22333333333333,3.8509262954948);
\path[draw=magenta,line width=.109413170333069cm] (-.78,3.84515279280291) -- (-.94,4.33012701892219);
\node[black,,scale=1.04626344130997] at (-.75,3.89711431702997) {0134};
\node[black,,scale=1.19572894384102] at (-1.25,3.89711431702997) {013};
\node[black,,scale=1.39501628054907] at (-1,4.33012701892219) {\(\nearrow\)4};
\path (-2.5,4.33012701892219) -- (-2,3.46410161513775) -- (-1.5,4.33012701892219) -- cycle;
\path[draw=green,line width=.109413170333069cm] (-1.73,3.93175533318135) -- (-2.25,3.89711431702997);
\path[draw=yellow,line width=.109413170333069cm] (-1.77,3.8624733008786) -- (-1.94,4.33012701892219);
\node[black,,scale=1.39501628054907] at (-1.75,3.89711431702997) {13};
\node[black,,scale=1.67402150609596] at (-2.25,3.89711431702997) {1};
\node[black,,scale=1.39501628054907] at (-2,4.33012701892219) {\(\nearrow\)3};
\end{scope}
\end{tikzpicture}\begin{tikzpicture}[execute at begin picture={\bgroup\tikzset{every path/.style={}}\clip (-3.1,-.533215) rectangle ++(6.2,5.50989);\egroup},x={(1.56304426682143cm,0cm)},y={(0cm,-1.56304631412272cm)},baseline=(current  bounding  box.center),every path/.style={draw=black,fill=none},line join=round]
\begin{scope}[every path/.append style={fill=white,line width=.0312609058094483cm}]
\path (-.5,.866025403784439) -- (0,0) -- (.5,.866025403784439) -- cycle;
\path (-.5,.866025403784439) -- (.5,.866025403784439) -- (0,1.73205080756888) -- cycle;
\node[black,,scale=1.04626344130997] at (0,.866025403784439) {even};
\path (0,1.73205080756888) -- (.5,.866025403784439) -- (1,1.73205080756888) -- cycle;
\path (0,1.73205080756888) -- (1,1.73205080756888) -- (.5,2.59807621135332) -- cycle;
\path[draw=blue,line width=.109413170333069cm] (.27,2.19970452561247) -- (.75,2.1650635094611);
\path[draw=magenta,line width=.109413170333069cm] (.75,1.29903810567666) -- (.56,1.73205080756888);
\path[draw=magenta,line width=.109413170333069cm] (.23,2.13042249330972) -- (.56,1.73205080756888);
\node[black,,scale=1.67402150609596] at (.75,1.29903810567666) {4};
\node[black,,scale=1.39501628054907] at (.5,1.73205080756888) {\(\nearrow\)4};
\path (.5,2.59807621135332) -- (1,1.73205080756888) -- (1.5,2.59807621135332) -- cycle;
\path (.5,2.59807621135332) -- (1.5,2.59807621135332) -- (1,3.46410161513775) -- cycle;
\path[draw=blue,line width=.109413170333069cm] (1.25,2.1650635094611) -- (.75,2.1650635094611);
\path[draw=yellow,line width=.109413170333069cm] (.75,3.03108891324553) -- (1.25,3.03108891324554);
\node[black,,scale=1.67402150609596] at (1.25,2.1650635094611) {2};
\node[black,,scale=1.67402150609596] at (.75,2.1650635094611) {2};
\node[black,,scale=1.19572894384102] at (1,2.59807621135332) {odd};
\path (1,3.46410161513775) -- (1.5,2.59807621135332) -- (2,3.46410161513775) -- cycle;
\path (1,3.46410161513775) -- (2,3.46410161513775) -- (1.5,4.33012701892219) -- cycle;
\path[draw=red,line width=.109413170333069cm] (1.25,3.89711431702997) -- (1.75,3.89711431702997);
\path[draw=yellow,line width=.109413170333069cm] (1.75,3.03108891324554) -- (1.25,3.03108891324554);
\node[black,,scale=1.67402150609596] at (1.75,3.03108891324554) {3};
\node[black,,scale=1.67402150609596] at (1.25,3.03108891324554) {3};
\node[black,,scale=1.19572894384102] at (1.5,3.46410161513775) {odd};
\path (1.5,4.33012701892219) -- (2,3.46410161513775) -- (2.5,4.33012701892219) -- cycle;
\path[draw=red,line width=.109413170333069cm] (2.06,4.33012701892219) -- (1.75,3.89711431702997);
\node[black,,scale=1.67402150609596] at (1.75,3.89711431702997) {0};
\node[black,,scale=1.39501628054907] at (2,4.33012701892219) {\(\searrow\)0};
\path (-1,1.73205080756888) -- (-.5,.866025403784439) -- (0,1.73205080756888) -- cycle;
\path (-1,1.73205080756888) -- (0,1.73205080756888) -- (-.5,2.59807621135332) -- cycle;
\path[draw=blue,line width=.109413170333069cm] (-.44,1.73205080756888) -- (-.75,1.29903810567666);
\path[draw=blue,line width=.109413170333069cm] (-.44,1.73205080756888) -- (-.25,2.1650635094611);
\node[black,,scale=1.67402150609596] at (-.75,1.29903810567666) {2};
\node[black,,scale=1.39501628054907] at (-.5,1.73205080756888) {\(\searrow\)2};
\path (-.5,2.59807621135332) -- (0,1.73205080756888) -- (.5,2.59807621135332) -- cycle;
\path (-.5,2.59807621135332) -- (.5,2.59807621135332) -- (0,3.46410161513775) -- cycle;
\path[draw=red,line width=.109413170333069cm] (-.223333333333333,3.0772769347807) -- (.23,3.06572992939691);
\path[draw=blue,line width=.109413170333069cm] (.27,2.19970452561247) -- (-.25,2.1650635094611);
\path[draw=yellow,line width=.109413170333069cm] (-.25,3.03108891324554) -- (.27,2.99644789709416);
\path[draw=magenta,line width=.109413170333069cm] (.23,2.13042249330972) -- (.0600000000000001,2.59807621135332);
\path[draw=magenta,line width=.109413170333069cm] (-.276666666666667,2.98490089171037) -- (.0600000000000001,2.59807621135332);
\node[black,,scale=1.39501628054907] at (.25,2.1650635094611) {24};
\node[black,,scale=1.67402150609596] at (-.25,2.1650635094611) {2};
\node[black,,scale=1.39501628054907] at (0,2.59807621135332) {\(\nearrow\)4};
\path (0,3.46410161513775) -- (.5,2.59807621135332) -- (1,3.46410161513775) -- cycle;
\path (0,3.46410161513775) -- (1,3.46410161513775) -- (.5,4.33012701892219) -- cycle;
\path[draw=red,line width=.109413170333069cm] (.56,3.46410161513775) -- (.23,3.06572992939691);
\path[draw=red,line width=.109413170333069cm] (.56,3.46410161513775) -- (.75,3.89711431702997);
\path[draw=yellow,line width=.109413170333069cm] (.75,3.03108891324553) -- (.27,2.99644789709416);
\node[black,,scale=1.67402150609596] at (.75,3.03108891324553) {3};
\node[black,,scale=1.39501628054907] at (.25,3.03108891324554) {03};
\node[black,,scale=1.39501628054907] at (.5,3.46410161513775) {\(\searrow\)0};
\path (.5,4.33012701892219) -- (1,3.46410161513775) -- (1.5,4.33012701892219) -- cycle;
\path[draw=red,line width=.109413170333069cm] (1.25,3.89711431702997) -- (.75,3.89711431702997);
\node[black,,scale=1.67402150609596] at (1.25,3.89711431702997) {0};
\node[black,,scale=1.67402150609596] at (.75,3.89711431702997) {0};
\node[black,,scale=1.19572894384102] at (1,4.33012701892219) {odd};
\path (-1.5,2.59807621135332) -- (-1,1.73205080756888) -- (-.5,2.59807621135332) -- cycle;
\path (-1.5,2.59807621135332) -- (-.5,2.59807621135332) -- (-1,3.46410161513775) -- cycle;
\path[draw=red,line width=.109413170333069cm] (-1.23,3.06572992939691) -- (-.77,3.06572992939691);
\path[draw=yellow,line width=.109413170333069cm] (-1.27,2.99644789709416) -- (-.73,2.99644789709416);
\node[black,,scale=1.04626344130997] at (-1,2.59807621135332) {even};
\path (-1,3.46410161513775) -- (-.5,2.59807621135332) -- (0,3.46410161513775) -- cycle;
\path (-1,3.46410161513775) -- (0,3.46410161513775) -- (-.5,4.33012701892219) -- cycle;
\path[draw=red,line width=.109413170333069cm] (-.223333333333333,3.0772769347807) -- (-.77,3.06572992939691);
\path[draw=green,line width=.109413170333069cm] (-.73,3.93175533318135) -- (-.25,3.89711431702997);
\path[draw=yellow,line width=.109413170333069cm] (-.25,3.03108891324554) -- (-.73,2.99644789709416);
\path[draw=magenta,line width=.109413170333069cm] (-.276666666666667,2.98490089171037) -- (-.44,3.46410161513775);
\path[draw=magenta,line width=.109413170333069cm] (-.77,3.8624733008786) -- (-.44,3.46410161513775);
\node[black,,scale=1.19572894384102] at (-.25,3.03108891324554) {034};
\node[black,,scale=1.39501628054907] at (-.75,3.03108891324553) {03};
\node[black,,scale=1.39501628054907] at (-.5,3.46410161513775) {\(\nearrow\)4};
\path (-.5,4.33012701892219) -- (0,3.46410161513775) -- (.5,4.33012701892219) -- cycle;
\path[draw=green,line width=.109413170333069cm] (.0600000000000001,4.33012701892219) -- (-.25,3.89711431702997);
\node[black,,scale=1.67402150609596] at (-.25,3.89711431702997) {1};
\node[black,,scale=1.39501628054907] at (0,4.33012701892219) {\(\searrow\)1};
\path (-2,3.46410161513775) -- (-1.5,2.59807621135332) -- (-1,3.46410161513775) -- cycle;
\path (-2,3.46410161513775) -- (-1,3.46410161513775) -- (-1.5,4.33012701892219) -- cycle;
\path[draw=red,line width=.109413170333069cm] (-1.23,3.06572992939691) -- (-1.75,3.03108891324554);
\path[draw=green,line width=.109413170333069cm] (-1.73,3.93175533318135) -- (-1.25,3.89711431702997);
\path[draw=yellow,line width=.109413170333069cm] (-1.27,2.99644789709416) -- (-1.44,3.46410161513775);
\path[draw=yellow,line width=.109413170333069cm] (-1.77,3.8624733008786) -- (-1.44,3.46410161513775);
\node[black,,scale=1.39501628054907] at (-1.25,3.03108891324554) {03};
\node[black,,scale=1.67402150609596] at (-1.75,3.03108891324554) {0};
\node[black,,scale=1.39501628054907] at (-1.5,3.46410161513775) {\(\nearrow\)3};
\path (-1.5,4.33012701892219) -- (-1,3.46410161513775) -- (-.5,4.33012701892219) -- cycle;
\path[draw=green,line width=.109413170333069cm] (-.73,3.93175533318135) -- (-1.25,3.89711431702997);
\path[draw=magenta,line width=.109413170333069cm] (-.77,3.8624733008786) -- (-.94,4.33012701892219);
\node[black,,scale=1.39501628054907] at (-.75,3.89711431702997) {14};
\node[black,,scale=1.67402150609596] at (-1.25,3.89711431702997) {1};
\node[black,,scale=1.39501628054907] at (-1,4.33012701892219) {\(\nearrow\)4};
\path (-2.5,4.33012701892219) -- (-2,3.46410161513775) -- (-1.5,4.33012701892219) -- cycle;
\path[draw=green,line width=.109413170333069cm] (-1.73,3.93175533318135) -- (-2.25,3.89711431702997);
\path[draw=yellow,line width=.109413170333069cm] (-1.77,3.8624733008786) -- (-1.94,4.33012701892219);
\node[black,,scale=1.39501628054907] at (-1.75,3.89711431702997) {13};
\node[black,,scale=1.67402150609596] at (-2.25,3.89711431702997) {1};
\node[black,,scale=1.39501628054907] at (-2,4.33012701892219) {\(\nearrow\)3};
\end{scope}
\end{tikzpicture}\begin{tikzpicture}[execute at begin picture={\bgroup\tikzset{every path/.style={}}\clip (-3.1,-.533215) rectangle ++(6.2,5.50989);\egroup},x={(1.56304426682143cm,0cm)},y={(0cm,-1.56304631412272cm)},baseline=(current  bounding  box.center),every path/.style={draw=black,fill=none},line join=round]
\begin{scope}[every path/.append style={fill=white,line width=.0312609058094483cm}]
\path (-.5,.866025403784439) -- (0,0) -- (.5,.866025403784439) -- cycle;
\path (-.5,.866025403784439) -- (.5,.866025403784439) -- (0,1.73205080756888) -- cycle;
\node[black,,scale=1.04626344130997] at (0,.866025403784439) {even};
\path (0,1.73205080756888) -- (.5,.866025403784439) -- (1,1.73205080756888) -- cycle;
\path (0,1.73205080756888) -- (1,1.73205080756888) -- (.5,2.59807621135332) -- cycle;
\path[draw=blue,line width=.109413170333069cm] (.276666666666667,2.21125153099627) -- (.73,2.19970452561247);
\path[draw=yellow,line width=.109413170333069cm] (.25,2.1650635094611) -- (.77,2.13042249330972);
\path[draw=magenta,line width=.109413170333069cm] (.75,1.29903810567666) -- (.56,1.73205080756888);
\path[draw=magenta,line width=.109413170333069cm] (.223333333333333,2.11887548792593) -- (.56,1.73205080756888);
\node[black,,scale=1.67402150609596] at (.75,1.29903810567666) {4};
\node[black,,scale=1.39501628054907] at (.5,1.73205080756888) {\(\nearrow\)4};
\path (.5,2.59807621135332) -- (1,1.73205080756888) -- (1.5,2.59807621135332) -- cycle;
\path (.5,2.59807621135332) -- (1.5,2.59807621135332) -- (1,3.46410161513775) -- cycle;
\path[draw=blue,line width=.109413170333069cm] (1.25,2.1650635094611) -- (.73,2.19970452561247);
\path[draw=yellow,line width=.109413170333069cm] (1.06,2.59807621135332) -- (.77,2.13042249330972);
\path[draw=yellow,line width=.109413170333069cm] (1.06,2.59807621135332) -- (1.25,3.03108891324554);
\node[black,,scale=1.67402150609596] at (1.25,2.1650635094611) {2};
\node[black,,scale=1.39501628054907] at (.75,2.1650635094611) {23};
\node[black,,scale=1.39501628054907] at (1,2.59807621135332) {\(\searrow\)3};
\path (1,3.46410161513775) -- (1.5,2.59807621135332) -- (2,3.46410161513775) -- cycle;
\path (1,3.46410161513775) -- (2,3.46410161513775) -- (1.5,4.33012701892219) -- cycle;
\path[draw=red,line width=.109413170333069cm] (1.25,3.89711431702997) -- (1.75,3.89711431702997);
\path[draw=yellow,line width=.109413170333069cm] (1.75,3.03108891324554) -- (1.25,3.03108891324554);
\node[black,,scale=1.67402150609596] at (1.75,3.03108891324554) {3};
\node[black,,scale=1.67402150609596] at (1.25,3.03108891324554) {3};
\node[black,,scale=1.19572894384102] at (1.5,3.46410161513775) {odd};
\path (1.5,4.33012701892219) -- (2,3.46410161513775) -- (2.5,4.33012701892219) -- cycle;
\path[draw=red,line width=.109413170333069cm] (2.06,4.33012701892219) -- (1.75,3.89711431702997);
\node[black,,scale=1.67402150609596] at (1.75,3.89711431702997) {0};
\node[black,,scale=1.39501628054907] at (2,4.33012701892219) {\(\searrow\)0};
\path (-1,1.73205080756888) -- (-.5,.866025403784439) -- (0,1.73205080756888) -- cycle;
\path (-1,1.73205080756888) -- (0,1.73205080756888) -- (-.5,2.59807621135332) -- cycle;
\path[draw=blue,line width=.109413170333069cm] (-.44,1.73205080756888) -- (-.75,1.29903810567666);
\path[draw=blue,line width=.109413170333069cm] (-.44,1.73205080756888) -- (-.27,2.19970452561247);
\path[draw=yellow,line width=.109413170333069cm] (-.75,2.1650635094611) -- (-.23,2.13042249330972);
\node[black,,scale=1.67402150609596] at (-.75,1.29903810567666) {2};
\node[black,,scale=1.39501628054907] at (-.5,1.73205080756888) {\(\searrow\)2};
\path (-.5,2.59807621135332) -- (0,1.73205080756888) -- (.5,2.59807621135332) -- cycle;
\path (-.5,2.59807621135332) -- (.5,2.59807621135332) -- (0,3.46410161513775) -- cycle;
\path[draw=red,line width=.109413170333069cm] (-.23,3.06572992939691) -- (.25,3.03108891324554);
\path[draw=blue,line width=.109413170333069cm] (.276666666666667,2.21125153099627) -- (-.27,2.19970452561247);
\path[draw=yellow,line width=.109413170333069cm] (.25,2.1650635094611) -- (-.23,2.13042249330972);
\path[draw=magenta,line width=.109413170333069cm] (.223333333333333,2.11887548792593) -- (.0600000000000001,2.59807621135332);
\path[draw=magenta,line width=.109413170333069cm] (-.27,2.99644789709416) -- (.0600000000000001,2.59807621135332);
\node[black,,scale=1.19572894384102] at (.25,2.1650635094611) {234};
\node[black,,scale=1.39501628054907] at (-.25,2.1650635094611) {23};
\node[black,,scale=1.39501628054907] at (0,2.59807621135332) {\(\nearrow\)4};
\path (0,3.46410161513775) -- (.5,2.59807621135332) -- (1,3.46410161513775) -- cycle;
\path (0,3.46410161513775) -- (1,3.46410161513775) -- (.5,4.33012701892219) -- cycle;
\path[draw=red,line width=.109413170333069cm] (.56,3.46410161513775) -- (.25,3.03108891324554);
\path[draw=red,line width=.109413170333069cm] (.56,3.46410161513775) -- (.75,3.89711431702997);
\node[black,,scale=1.67402150609596] at (.25,3.03108891324554) {0};
\node[black,,scale=1.39501628054907] at (.5,3.46410161513775) {\(\searrow\)0};
\path (.5,4.33012701892219) -- (1,3.46410161513775) -- (1.5,4.33012701892219) -- cycle;
\path[draw=red,line width=.109413170333069cm] (1.25,3.89711431702997) -- (.75,3.89711431702997);
\node[black,,scale=1.67402150609596] at (1.25,3.89711431702997) {0};
\node[black,,scale=1.67402150609596] at (.75,3.89711431702997) {0};
\node[black,,scale=1.19572894384102] at (1,4.33012701892219) {odd};
\path (-1.5,2.59807621135332) -- (-1,1.73205080756888) -- (-.5,2.59807621135332) -- cycle;
\path (-1.5,2.59807621135332) -- (-.5,2.59807621135332) -- (-1,3.46410161513775) -- cycle;
\path[draw=red,line width=.109413170333069cm] (-1.23,3.06572992939691) -- (-.75,3.03108891324553);
\path[draw=yellow,line width=.109413170333069cm] (-.75,2.1650635094611) -- (-.94,2.59807621135332);
\path[draw=yellow,line width=.109413170333069cm] (-1.27,2.99644789709416) -- (-.94,2.59807621135332);
\node[black,,scale=1.67402150609596] at (-.75,2.1650635094611) {3};
\node[black,,scale=1.39501628054907] at (-1,2.59807621135332) {\(\nearrow\)3};
\path (-1,3.46410161513775) -- (-.5,2.59807621135332) -- (0,3.46410161513775) -- cycle;
\path (-1,3.46410161513775) -- (0,3.46410161513775) -- (-.5,4.33012701892219) -- cycle;
\path[draw=red,line width=.109413170333069cm] (-.23,3.06572992939691) -- (-.75,3.03108891324553);
\path[draw=green,line width=.109413170333069cm] (-.73,3.93175533318135) -- (-.25,3.89711431702997);
\path[draw=magenta,line width=.109413170333069cm] (-.27,2.99644789709416) -- (-.44,3.46410161513775);
\path[draw=magenta,line width=.109413170333069cm] (-.77,3.8624733008786) -- (-.44,3.46410161513775);
\node[black,,scale=1.39501628054907] at (-.25,3.03108891324554) {04};
\node[black,,scale=1.67402150609596] at (-.75,3.03108891324553) {0};
\node[black,,scale=1.39501628054907] at (-.5,3.46410161513775) {\(\nearrow\)4};
\path (-.5,4.33012701892219) -- (0,3.46410161513775) -- (.5,4.33012701892219) -- cycle;
\path[draw=green,line width=.109413170333069cm] (.0600000000000001,4.33012701892219) -- (-.25,3.89711431702997);
\node[black,,scale=1.67402150609596] at (-.25,3.89711431702997) {1};
\node[black,,scale=1.39501628054907] at (0,4.33012701892219) {\(\searrow\)1};
\path (-2,3.46410161513775) -- (-1.5,2.59807621135332) -- (-1,3.46410161513775) -- cycle;
\path (-2,3.46410161513775) -- (-1,3.46410161513775) -- (-1.5,4.33012701892219) -- cycle;
\path[draw=red,line width=.109413170333069cm] (-1.23,3.06572992939691) -- (-1.75,3.03108891324554);
\path[draw=green,line width=.109413170333069cm] (-1.73,3.93175533318135) -- (-1.25,3.89711431702997);
\path[draw=yellow,line width=.109413170333069cm] (-1.27,2.99644789709416) -- (-1.44,3.46410161513775);
\path[draw=yellow,line width=.109413170333069cm] (-1.77,3.8624733008786) -- (-1.44,3.46410161513775);
\node[black,,scale=1.39501628054907] at (-1.25,3.03108891324554) {03};
\node[black,,scale=1.67402150609596] at (-1.75,3.03108891324554) {0};
\node[black,,scale=1.39501628054907] at (-1.5,3.46410161513775) {\(\nearrow\)3};
\path (-1.5,4.33012701892219) -- (-1,3.46410161513775) -- (-.5,4.33012701892219) -- cycle;
\path[draw=green,line width=.109413170333069cm] (-.73,3.93175533318135) -- (-1.25,3.89711431702997);
\path[draw=magenta,line width=.109413170333069cm] (-.77,3.8624733008786) -- (-.94,4.33012701892219);
\node[black,,scale=1.39501628054907] at (-.75,3.89711431702997) {14};
\node[black,,scale=1.67402150609596] at (-1.25,3.89711431702997) {1};
\node[black,,scale=1.39501628054907] at (-1,4.33012701892219) {\(\nearrow\)4};
\path (-2.5,4.33012701892219) -- (-2,3.46410161513775) -- (-1.5,4.33012701892219) -- cycle;
\path[draw=green,line width=.109413170333069cm] (-1.73,3.93175533318135) -- (-2.25,3.89711431702997);
\path[draw=yellow,line width=.109413170333069cm] (-1.77,3.8624733008786) -- (-1.94,4.33012701892219);
\node[black,,scale=1.39501628054907] at (-1.75,3.89711431702997) {13};
\node[black,,scale=1.67402150609596] at (-2.25,3.89711431702997) {1};
\node[black,,scale=1.39501628054907] at (-2,4.33012701892219) {\(\nearrow\)3};
\end{scope}
\end{tikzpicture}

%% file: ex2stepa.tex
\begin{tikzpicture}[execute at begin picture={\bgroup\tikzset{every path/.style={}}\clip (-2.48,-.429292) rectangle ++(4.96,4.43602);\egroup},x={(1.55865151133662cm,0cm)},y={(0cm,-1.55865187137441cm)},baseline=(current  bounding  box.center),every path/.style={draw=black,fill=none},line join=round]
\begin{scope}[every path/.append style={fill=white,line width=.0311730338271105cm}]
\path (-.5,.866025403784439) -- (0,0) -- (.5,.866025403784439) -- cycle;
\path (-.5,.866025403784439) -- (.5,.866025403784439) -- (0,1.73205080756888) -- cycle;
\node[black,,scale=1.04332247590111] at (0,.866025403784439) {even};
\path (0,1.73205080756888) -- (.5,.866025403784439) -- (1,1.73205080756888) -- cycle;
\path (0,1.73205080756888) -- (1,1.73205080756888) -- (.5,2.59807621135332) -- cycle;
\path[draw=red,line width=.109105618394887cm] (.27,2.19970452561247) -- (.75,2.1650635094611);
\path[draw=blue,line width=.109105618394887cm] (.75,1.29903810567666) -- (.56,1.73205080756888);
\path[draw=blue,line width=.109105618394887cm] (.23,2.13042249330972) -- (.56,1.73205080756888);
\node[black,,scale=1.66931596144177] at (.75,1.29903810567666) {2};
\node[black,,scale=1.39109499795053] at (.5,1.73205080756888) {\(\nearrow\)2};
\path (.5,2.59807621135332) -- (1,1.73205080756888) -- (1.5,2.59807621135332) -- cycle;
\path (.5,2.59807621135332) -- (1.5,2.59807621135332) -- (1,3.46410161513775) -- cycle;
\path[draw=red,line width=.109105618394887cm] (1.06,2.59807621135332) -- (.75,2.1650635094611);
\path[draw=red,line width=.109105618394887cm] (1.06,2.59807621135332) -- (1.23,3.06572992939691);
\path[draw=green,line width=.109105618394887cm] (.75,3.03108891324553) -- (1.27,2.99644789709416);
\node[black,,scale=1.66931596144177] at (.75,2.1650635094611) {0};
\node[black,,scale=1.39109499795053] at (1,2.59807621135332) {\(\searrow\)0};
\path (1,3.46410161513775) -- (1.5,2.59807621135332) -- (2,3.46410161513775) -- cycle;
\path[draw=red,line width=.109105618394887cm] (1.56,3.46410161513775) -- (1.23,3.06572992939691);
\path[draw=green,line width=.109105618394887cm] (1.75,3.03108891324554) -- (1.27,2.99644789709416);
\node[black,,scale=1.66931596144177] at (1.75,3.03108891324554) {1};
\node[black,,scale=1.39109499795053] at (1.25,3.03108891324554) {01};
\node[black,,scale=1.39109499795053] at (1.5,3.46410161513775) {\(\searrow\)0};
\path (-1,1.73205080756888) -- (-.5,.866025403784439) -- (0,1.73205080756888) -- cycle;
\path (-1,1.73205080756888) -- (0,1.73205080756888) -- (-.5,2.59807621135332) -- cycle;
\path[draw=red,line width=.109105618394887cm] (-.44,1.73205080756888) -- (-.75,1.29903810567666);
\path[draw=red,line width=.109105618394887cm] (-.44,1.73205080756888) -- (-.25,2.1650635094611);
\node[black,,scale=1.66931596144177] at (-.75,1.29903810567666) {0};
\node[black,,scale=1.39109499795053] at (-.5,1.73205080756888) {\(\searrow\)0};
\path (-.5,2.59807621135332) -- (0,1.73205080756888) -- (.5,2.59807621135332) -- cycle;
\path (-.5,2.59807621135332) -- (.5,2.59807621135332) -- (0,3.46410161513775) -- cycle;
\path[draw=red,line width=.109105618394887cm] (.27,2.19970452561247) -- (-.25,2.1650635094611);
\path[draw=green,line width=.109105618394887cm] (-.23,3.06572992939691) -- (.25,3.03108891324554);
\path[draw=blue,line width=.109105618394887cm] (.23,2.13042249330972) -- (.0600000000000001,2.59807621135332);
\path[draw=blue,line width=.109105618394887cm] (-.27,2.99644789709416) -- (.0600000000000001,2.59807621135332);
\node[black,,scale=1.39109499795053] at (.25,2.1650635094611) {02};
\node[black,,scale=1.66931596144177] at (-.25,2.1650635094611) {0};
\node[black,,scale=1.39109499795053] at (0,2.59807621135332) {\(\nearrow\)2};
\path (0,3.46410161513775) -- (.5,2.59807621135332) -- (1,3.46410161513775) -- cycle;
\path[draw=green,line width=.109105618394887cm] (.75,3.03108891324553) -- (.25,3.03108891324554);
\node[black,,scale=1.66931596144177] at (.75,3.03108891324553) {1};
\node[black,,scale=1.66931596144177] at (.25,3.03108891324554) {1};
\node[black,,scale=1.19236784249372] at (.5,3.46410161513775) {odd};
\path (-1.5,2.59807621135332) -- (-1,1.73205080756888) -- (-.5,2.59807621135332) -- cycle;
\path (-1.5,2.59807621135332) -- (-.5,2.59807621135332) -- (-1,3.46410161513775) -- cycle;
\path[draw=green,line width=.109105618394887cm] (-.94,2.59807621135332) -- (-1.25,2.1650635094611);
\path[draw=green,line width=.109105618394887cm] (-.94,2.59807621135332) -- (-.75,3.03108891324553);
\node[black,,scale=1.66931596144177] at (-1.25,2.1650635094611) {1};
\node[black,,scale=1.39109499795053] at (-1,2.59807621135332) {\(\searrow\)1};
\path (-1,3.46410161513775) -- (-.5,2.59807621135332) -- (0,3.46410161513775) -- cycle;
\path[draw=green,line width=.109105618394887cm] (-.23,3.06572992939691) -- (-.75,3.03108891324553);
\path[draw=blue,line width=.109105618394887cm] (-.27,2.99644789709416) -- (-.44,3.46410161513775);
\node[black,,scale=1.39109499795053] at (-.25,3.03108891324554) {12};
\node[black,,scale=1.66931596144177] at (-.75,3.03108891324553) {1};
\node[black,,scale=1.39109499795053] at (-.5,3.46410161513775) {\(\nearrow\)2};
\path (-2,3.46410161513775) -- (-1.5,2.59807621135332) -- (-1,3.46410161513775) -- cycle;
\path[draw=red,line width=.109105618394887cm] (-1.44,3.46410161513775) -- (-1.75,3.03108891324554);
\node[black,,scale=1.66931596144177] at (-1.75,3.03108891324554) {0};
\node[black,,scale=1.39109499795053] at (-1.5,3.46410161513775) {\(\searrow\)0};
\end{scope}
\end{tikzpicture}

%% file: ex10a.tex
\begin{tikzpicture}[execute at begin picture={\bgroup\tikzset{every path/.style={}}\clip (-1.86,-.328089) rectangle ++(3.72,3.39025);\egroup},x={(1.54552502393359cm,0cm)},y={(0cm,-1.54552912706501cm)},baseline=(current  bounding  box.center),every path/.style={draw=black,fill=none},line join=round]
\begin{scope}[every path/.append style={fill=white,line width=.0309105415100132cm}]
\path (-.5,.866025403784439) -- (0,0) -- (.5,.866025403784439) -- cycle;
\path (-.5,.866025403784439) -- (.5,.866025403784439) -- (0,1.73205080756888) -- cycle;
\node[black,,scale=1.65525949786121] at (.25,.433012701892219) {1};
\node[black,,scale=1.65525949786121] at (-.25,.433012701892219) {1};
\node[black,,scale=1.65525949786121] at (0,.866025403784439) {1};
\path (0,1.73205080756888) -- (.5,.866025403784439) -- (1,1.73205080756888) -- cycle;
\path (0,1.73205080756888) -- (1,1.73205080756888) -- (.5,2.59807621135332) -- cycle;
\node[black,,scale=1.65525949786121] at (.75,1.29903810567666) {0};
\node[black,,scale=1.65525949786121] at (.25,1.29903810567666) {1};
\node[black,,scale=1.37938129208091] at (.5,1.73205080756888) {10};
\path (.5,2.59807621135332) -- (1,1.73205080756888) -- (1.5,2.59807621135332) -- cycle;
\node[black,,scale=1.65525949786121] at (1.25,2.1650635094611) {0};
\node[black,,scale=1.65525949786121] at (.75,2.1650635094611) {1};
\node[black,,scale=1.37938129208091] at (1,2.59807621135332) {10};
\path[fill=LightGray] (-1,1.73205080756888) -- (-.5,.866025403784439) -- (0,1.73205080756888) -- (-.5,2.59807621135332) -- cycle;
\node[black,,scale=1.65525949786121] at (-.25,1.29903810567666) {1};
\node[black,,scale=1.65525949786121] at (-.75,1.29903810567666) {0};
\path (-.5,2.59807621135332) -- (0,1.73205080756888) -- (.5,2.59807621135332) -- cycle;
\node[black,,scale=1.65525949786121] at (.25,2.1650635094611) {0};
\node[black,,scale=1.65525949786121] at (-.25,2.1650635094611) {0};
\node[black,,scale=1.65525949786121] at (0,2.59807621135332) {0};
\path (-1.5,2.59807621135332) -- (-1,1.73205080756888) -- (-.5,2.59807621135332) -- cycle;
\node[black,,scale=1.65525949786121] at (-.75,2.1650635094611) {1};
\node[black,,scale=1.65525949786121] at (-1.25,2.1650635094611) {1};
\node[black,,scale=1.65525949786121] at (-1,2.59807621135332) {1};
\end{scope}
\end{tikzpicture}\begin{tikzpicture}[execute at begin picture={\bgroup\tikzset{every path/.style={}}\clip (-1.86,-.328089) rectangle ++(3.72,3.39025);\egroup},x={(1.54552502393359cm,0cm)},y={(0cm,-1.54552912706501cm)},baseline=(current  bounding  box.center),every path/.style={draw=black,fill=none},line join=round]
\begin{scope}[every path/.append style={fill=white,line width=.0309105415100132cm}]
\path (-.5,.866025403784439) -- (0,0) -- (.5,.866025403784439) -- cycle;
\path (-.5,.866025403784439) -- (.5,.866025403784439) -- (0,1.73205080756888) -- cycle;
\node[black,,scale=1.65525949786121] at (.25,.433012701892219) {1};
\node[black,,scale=1.65525949786121] at (-.25,.433012701892219) {1};
\node[black,,scale=1.65525949786121] at (0,.866025403784439) {1};
\path (0,1.73205080756888) -- (.5,.866025403784439) -- (1,1.73205080756888) -- cycle;
\path (0,1.73205080756888) -- (1,1.73205080756888) -- (.5,2.59807621135332) -- cycle;
\node[black,,scale=1.65525949786121] at (.75,1.29903810567666) {0};
\node[black,,scale=1.65525949786121] at (.25,1.29903810567666) {0};
\node[black,,scale=1.65525949786121] at (.5,1.73205080756888) {0};
\path (.5,2.59807621135332) -- (1,1.73205080756888) -- (1.5,2.59807621135332) -- cycle;
\node[black,,scale=1.65525949786121] at (1.25,2.1650635094611) {0};
\node[black,,scale=1.65525949786121] at (.75,2.1650635094611) {0};
\node[black,,scale=1.65525949786121] at (1,2.59807621135332) {0};
\path (-1,1.73205080756888) -- (-.5,.866025403784439) -- (0,1.73205080756888) -- cycle;
\path (-1,1.73205080756888) -- (0,1.73205080756888) -- (-.5,2.59807621135332) -- cycle;
\node[black,,scale=1.37938129208091] at (-.25,1.29903810567666) {10};
\node[black,,scale=1.65525949786121] at (-.75,1.29903810567666) {0};
\node[black,,scale=1.65525949786121] at (-.5,1.73205080756888) {1};
\path (-.5,2.59807621135332) -- (0,1.73205080756888) -- (.5,2.59807621135332) -- cycle;
\node[black,,scale=1.65525949786121] at (.25,2.1650635094611) {0};
\node[black,,scale=1.65525949786121] at (-.25,2.1650635094611) {1};
\node[black,,scale=1.37938129208091] at (0,2.59807621135332) {10};
\path (-1.5,2.59807621135332) -- (-1,1.73205080756888) -- (-.5,2.59807621135332) -- cycle;
\node[black,,scale=1.65525949786121] at (-.75,2.1650635094611) {1};
\node[black,,scale=1.65525949786121] at (-1.25,2.1650635094611) {1};
\node[black,,scale=1.65525949786121] at (-1,2.59807621135332) {1};
\end{scope}
\end{tikzpicture}

%% file: ex10b.tex
\begin{tikzpicture}[execute at begin picture={\bgroup\tikzset{every path/.style={}}\clip (-1.86,-.328089) rectangle ++(3.72,3.39025);\egroup},x={(1.54552502393359cm,0cm)},y={(0cm,-1.54552912706501cm)},baseline=(current  bounding  box.center),every path/.style={draw=black,fill=none},line join=round]
\begin{scope}[every path/.append style={fill=white,line width=.0309105415100132cm}]
\path (-.5,.866025403784439) -- (0,0) -- (.5,.866025403784439) -- cycle;
\path (-.5,.866025403784439) -- (.5,.866025403784439) -- (0,1.73205080756888) -- cycle;
\node[black,,scale=1.65525949786121] at (.25,.433012701892219) {2};
\node[black,,scale=1.65525949786121] at (-.25,.433012701892219) {2};
\node[black,,scale=1.65525949786121] at (0,.866025403784439) {2};
\path (0,1.73205080756888) -- (.5,.866025403784439) -- (1,1.73205080756888) -- cycle;
\path (0,1.73205080756888) -- (1,1.73205080756888) -- (.5,2.59807621135332) -- cycle;
\node[black,,scale=1.65525949786121] at (.75,1.29903810567666) {0};
\node[black,,scale=1.65525949786121] at (.25,1.29903810567666) {1};
\node[black,,scale=1.37938129208091] at (.5,1.73205080756888) {10};
\path (.5,2.59807621135332) -- (1,1.73205080756888) -- (1.5,2.59807621135332) -- cycle;
\node[black,,scale=1.65525949786121] at (1.25,2.1650635094611) {1};
\node[black,,scale=1.65525949786121] at (.75,2.1650635094611) {1};
\node[black,,scale=1.65525949786121] at (1,2.59807621135332) {1};
\path[fill=LightGray] (-1,1.73205080756888) -- (-.5,.866025403784439) -- (0,1.73205080756888) -- (-.5,2.59807621135332) -- cycle;
\node[black,,scale=1.37938129208091] at (-.25,1.29903810567666) {21};
\node[black,,scale=1.65525949786121] at (-.75,1.29903810567666) {0};
\path (-.5,2.59807621135332) -- (0,1.73205080756888) -- (.5,2.59807621135332) -- cycle;
\node[black,,scale=1.65525949786121] at (.25,2.1650635094611) {0};
\node[black,,scale=1.65525949786121] at (-.25,2.1650635094611) {0};
\node[black,,scale=1.65525949786121] at (0,2.59807621135332) {0};
\path (-1.5,2.59807621135332) -- (-1,1.73205080756888) -- (-.5,2.59807621135332) -- cycle;
\node[black,,scale=1.37938129208091] at (-.75,2.1650635094611) {21};
\node[black,,scale=1.65525949786121] at (-1.25,2.1650635094611) {1};
\node[black,,scale=1.65525949786121] at (-1,2.59807621135332) {2};
\end{scope}
\end{tikzpicture}\begin{tikzpicture}[execute at begin picture={\bgroup\tikzset{every path/.style={}}\clip (-1.86,-.328089) rectangle ++(3.72,3.39025);\egroup},x={(1.54552502393359cm,0cm)},y={(0cm,-1.54552912706501cm)},baseline=(current  bounding  box.center),every path/.style={draw=black,fill=none},line join=round]
\begin{scope}[every path/.append style={fill=white,line width=.0309105415100132cm}]
\path (-.5,.866025403784439) -- (0,0) -- (.5,.866025403784439) -- cycle;
\path (-.5,.866025403784439) -- (.5,.866025403784439) -- (0,1.73205080756888) -- cycle;
\node[black,,scale=1.65525949786121] at (.25,.433012701892219) {2};
\node[black,,scale=1.65525949786121] at (-.25,.433012701892219) {2};
\node[black,,scale=1.65525949786121] at (0,.866025403784439) {2};
\path (0,1.73205080756888) -- (.5,.866025403784439) -- (1,1.73205080756888) -- cycle;
\path (0,1.73205080756888) -- (1,1.73205080756888) -- (.5,2.59807621135332) -- cycle;
\node[black,,scale=1.65525949786121] at (.75,1.29903810567666) {0};
\node[black,,scale=1.65525949786121] at (.25,1.29903810567666) {0};
\node[black,,scale=1.65525949786121] at (.5,1.73205080756888) {0};
\path (.5,2.59807621135332) -- (1,1.73205080756888) -- (1.5,2.59807621135332) -- cycle;
\node[black,,scale=1.65525949786121] at (1.25,2.1650635094611) {1};
\node[black,,scale=1.37938129208091] at (.75,2.1650635094611) {10};
\node[black,,scale=1.65525949786121] at (1,2.59807621135332) {0};
\path (-1,1.73205080756888) -- (-.5,.866025403784439) -- (0,1.73205080756888) -- cycle;
\path (-1,1.73205080756888) -- (0,1.73205080756888) -- (-.5,2.59807621135332) -- cycle;
\node[black,,scale=1.37938129208091] at (-.25,1.29903810567666) {20};
\node[black,,scale=1.65525949786121] at (-.75,1.29903810567666) {0};
\node[black,,scale=1.65525949786121] at (-.5,1.73205080756888) {2};
\path (-.5,2.59807621135332) -- (0,1.73205080756888) -- (.5,2.59807621135332) -- cycle;
\node[black,,scale=1.65525949786121] at (.25,2.1650635094611) {1};
\node[black,,scale=1.65525949786121] at (-.25,2.1650635094611) {1};
\node[black,,scale=1.65525949786121] at (0,2.59807621135332) {1};
\path (-1.5,2.59807621135332) -- (-1,1.73205080756888) -- (-.5,2.59807621135332) -- cycle;
\node[black,,scale=1.37938129208091] at (-.75,2.1650635094611) {21};
\node[black,,scale=1.65525949786121] at (-1.25,2.1650635094611) {1};
\node[black,,scale=1.65525949786121] at (-1,2.59807621135332) {2};
\end{scope}
\end{tikzpicture}

%% file: ex10c.tex
\begin{tikzpicture}[execute at begin picture={\bgroup\tikzset{every path/.style={}}\clip (-1.86,-.328089) rectangle ++(3.72,3.39025);\egroup},x={(1.54552502393359cm,0cm)},y={(0cm,-1.54552912706501cm)},baseline=(current  bounding  box.center),every path/.style={draw=black,fill=none},line join=round]
\begin{scope}[every path/.append style={fill=white,line width=.0309105415100132cm}]
\path (-.5,.866025403784439) -- (0,0) -- (.5,.866025403784439) -- cycle;
\path (-.5,.866025403784439) -- (.5,.866025403784439) -- (0,1.73205080756888) -- cycle;
\node[black,,scale=1.65525949786121] at (-.25,.433012701892219) {1};
\node[black,,scale=1.65525949786121] at (0,.866025403784439) {1};
\path (0,1.73205080756888) -- (.5,.866025403784439) -- (1,1.73205080756888) -- cycle;
\path (0,1.73205080756888) -- (1,1.73205080756888) -- (.5,2.59807621135332) -- cycle;
\node[black,,scale=1.65525949786121] at (.75,1.29903810567666) {0};
\node[black,,scale=1.65525949786121] at (.25,1.29903810567666) {1};
\node[black,,scale=1.37938129208091] at (.5,1.73205080756888) {01};
\path (.5,2.59807621135332) -- (1,1.73205080756888) -- (1.5,2.59807621135332) -- cycle;
\node[black,,scale=1.65525949786121] at (.75,2.1650635094611) {1};
\node[black,,scale=1.65525949786121] at (1,2.59807621135332) {1};
\path[fill=LightGray] (-1,1.73205080756888) -- (-.5,.866025403784439) -- (0,1.73205080756888) -- (-.5,2.59807621135332) -- cycle;
\path (-.5,2.59807621135332) -- (0,1.73205080756888) -- (.5,2.59807621135332) -- cycle;
\node[black,,scale=1.65525949786121] at (.25,2.1650635094611) {0};
\node[black,,scale=1.65525949786121] at (0,2.59807621135332) {0};
\path (-1.5,2.59807621135332) -- (-1,1.73205080756888) -- (-.5,2.59807621135332) -- cycle;
\node[black,,scale=1.65525949786121] at (-1.25,2.1650635094611) {2};
\node[black,,scale=1.65525949786121] at (-1,2.59807621135332) {2};
\end{scope}
\end{tikzpicture}\begin{tikzpicture}[execute at begin picture={\bgroup\tikzset{every path/.style={}}\clip (-1.86,-.328089) rectangle ++(3.72,3.39025);\egroup},x={(1.54552502393359cm,0cm)},y={(0cm,-1.54552912706501cm)},baseline=(current  bounding  box.center),every path/.style={draw=black,fill=none},line join=round]
\begin{scope}[every path/.append style={fill=white,line width=.0309105415100132cm}]
\path (-.5,.866025403784439) -- (0,0) -- (.5,.866025403784439) -- cycle;
\path (-.5,.866025403784439) -- (.5,.866025403784439) -- (0,1.73205080756888) -- cycle;
\node[black,,scale=1.65525949786121] at (-.25,.433012701892219) {1};
\node[black,,scale=1.65525949786121] at (0,.866025403784439) {1};
\path (0,1.73205080756888) -- (.5,.866025403784439) -- (1,1.73205080756888) -- cycle;
\path (0,1.73205080756888) -- (1,1.73205080756888) -- (.5,2.59807621135332) -- cycle;
\node[black,,scale=1.65525949786121] at (.75,1.29903810567666) {0};
\node[black,,scale=1.65525949786121] at (.5,1.73205080756888) {0};
\path (.5,2.59807621135332) -- (1,1.73205080756888) -- (1.5,2.59807621135332) -- cycle;
\node[black,,scale=1.65525949786121] at (.75,2.1650635094611) {0};
\node[black,,scale=1.65525949786121] at (1,2.59807621135332) {0};
\path (-1,1.73205080756888) -- (-.5,.866025403784439) -- (0,1.73205080756888) -- cycle;
\path (-1,1.73205080756888) -- (0,1.73205080756888) -- (-.5,2.59807621135332) -- cycle;
\node[black,,scale=1.65525949786121] at (-.25,1.29903810567666) {1};
\node[black,,scale=1.65525949786121] at (-.5,1.73205080756888) {1};
\path (-.5,2.59807621135332) -- (0,1.73205080756888) -- (.5,2.59807621135332) -- cycle;
\node[black,,scale=1.65525949786121] at (-.25,2.1650635094611) {1};
\node[black,,scale=1.65525949786121] at (0,2.59807621135332) {1};
\path (-1.5,2.59807621135332) -- (-1,1.73205080756888) -- (-.5,2.59807621135332) -- cycle;
\node[black,,scale=1.65525949786121] at (-1.25,2.1650635094611) {2};
\node[black,,scale=1.65525949786121] at (-1,2.59807621135332) {2};
\end{scope}
\end{tikzpicture}

%% file: ex10d.tex
\begin{tikzpicture}[execute at begin picture={\bgroup\tikzset{every path/.style={}}\clip (-1.86,-.328089) rectangle ++(3.72,3.39025);\egroup},x={(1.54552502393359cm,0cm)},y={(0cm,-1.54552912706501cm)},baseline=(current  bounding  box.center),every path/.style={draw=black,fill=none},line join=round]
\begin{scope}[every path/.append style={fill=white,line width=.0309105415100132cm}]
\path (-.5,.866025403784439) -- (0,0) -- (.5,.866025403784439) -- cycle;
\path (-.5,.866025403784439) -- (.5,.866025403784439) -- (0,1.73205080756888) -- cycle;
\node[black,,scale=1.65525949786121] at (.25,.433012701892219) {1};
\node[black,,scale=1.65525949786121] at (0,.866025403784439) {1};
\path (0,1.73205080756888) -- (.5,.866025403784439) -- (1,1.73205080756888) -- cycle;
\path (0,1.73205080756888) -- (1,1.73205080756888) -- (.5,2.59807621135332) -- cycle;
\node[black,,scale=1.65525949786121] at (.75,1.29903810567666) {0};
\node[black,,scale=1.65525949786121] at (.25,1.29903810567666) {1};
\node[black,,scale=1.37938129208091] at (.5,1.73205080756888) {01};
\path (.5,2.59807621135332) -- (1,1.73205080756888) -- (1.5,2.59807621135332) -- cycle;
\node[black,,scale=1.65525949786121] at (.75,2.1650635094611) {1};
\node[black,,scale=1.65525949786121] at (1,2.59807621135332) {1};
\path[fill=LightGray] (-1,1.73205080756888) -- (-.5,.866025403784439) -- (0,1.73205080756888) -- (-.5,2.59807621135332) -- cycle;
\path (-.5,2.59807621135332) -- (0,1.73205080756888) -- (.5,2.59807621135332) -- cycle;
\node[black,,scale=1.65525949786121] at (.25,2.1650635094611) {0};
\node[black,,scale=1.65525949786121] at (0,2.59807621135332) {0};
\path (-1.5,2.59807621135332) -- (-1,1.73205080756888) -- (-.5,2.59807621135332) -- cycle;
\node[black,,scale=1.65525949786121] at (-1.25,2.1650635094611) {2};
\node[black,,scale=1.65525949786121] at (-1,2.59807621135332) {2};
\end{scope}
\end{tikzpicture}\begin{tikzpicture}[execute at begin picture={\bgroup\tikzset{every path/.style={}}\clip (-1.86,-.328089) rectangle ++(3.72,3.39025);\egroup},x={(1.54552502393359cm,0cm)},y={(0cm,-1.54552912706501cm)},baseline=(current  bounding  box.center),every path/.style={draw=black,fill=none},line join=round]
\begin{scope}[every path/.append style={fill=white,line width=.0309105415100132cm}]
\path (-.5,.866025403784439) -- (0,0) -- (.5,.866025403784439) -- cycle;
\path (-.5,.866025403784439) -- (.5,.866025403784439) -- (0,1.73205080756888) -- cycle;
\node[black,,scale=1.65525949786121] at (.25,.433012701892219) {1};
\node[black,,scale=1.65525949786121] at (0,.866025403784439) {1};
\path (0,1.73205080756888) -- (.5,.866025403784439) -- (1,1.73205080756888) -- cycle;
\path[fill=LightPink] (0,1.73205080756888) -- (1,1.73205080756888) -- (.5,2.59807621135332) -- cycle;
\node[black,,scale=1.65525949786121] at (.75,1.29903810567666) {0};
\node[black,,scale=1.65525949786121] at (.25,1.29903810567666) {1};
\node[black,,scale=1.37938129208091] at (.5,1.73205080756888) {01};
\path (.5,2.59807621135332) -- (1,1.73205080756888) -- (1.5,2.59807621135332) -- cycle;
\node[black,,scale=1.65525949786121] at (.75,2.1650635094611) {0};
\node[black,,scale=1.65525949786121] at (1,2.59807621135332) {0};
\path[fill=LightGray] (-1,1.73205080756888) -- (-.5,.866025403784439) -- (0,1.73205080756888) -- (-.5,2.59807621135332) -- cycle;
\path (-.5,2.59807621135332) -- (0,1.73205080756888) -- (.5,2.59807621135332) -- cycle;
\node[black,,scale=1.65525949786121] at (.25,2.1650635094611) {1};
\node[black,,scale=1.65525949786121] at (0,2.59807621135332) {1};
\path (-1.5,2.59807621135332) -- (-1,1.73205080756888) -- (-.5,2.59807621135332) -- cycle;
\node[black,,scale=1.65525949786121] at (-1.25,2.1650635094611) {2};
\node[black,,scale=1.65525949786121] at (-1,2.59807621135332) {2};
\end{scope}
\end{tikzpicture}\begin{tikzpicture}[execute at begin picture={\bgroup\tikzset{every path/.style={}}\clip (-1.86,-.328089) rectangle ++(3.72,3.39025);\egroup},x={(1.54552502393359cm,0cm)},y={(0cm,-1.54552912706501cm)},baseline=(current  bounding  box.center),every path/.style={draw=black,fill=none},line join=round]
\begin{scope}[every path/.append style={fill=white,line width=.0309105415100132cm}]
\path (-.5,.866025403784439) -- (0,0) -- (.5,.866025403784439) -- cycle;
\path (-.5,.866025403784439) -- (.5,.866025403784439) -- (0,1.73205080756888) -- cycle;
\node[black,,scale=1.65525949786121] at (.25,.433012701892219) {1};
\node[black,,scale=1.65525949786121] at (0,.866025403784439) {1};
\path (0,1.73205080756888) -- (.5,.866025403784439) -- (1,1.73205080756888) -- cycle;
\path (0,1.73205080756888) -- (1,1.73205080756888) -- (.5,2.59807621135332) -- cycle;
\node[black,,scale=1.65525949786121] at (.75,1.29903810567666) {0};
\node[black,,scale=1.65525949786121] at (.5,1.73205080756888) {0};
\path (.5,2.59807621135332) -- (1,1.73205080756888) -- (1.5,2.59807621135332) -- cycle;
\node[black,,scale=1.65525949786121] at (.75,2.1650635094611) {0};
\node[black,,scale=1.65525949786121] at (1,2.59807621135332) {0};
\path (-1,1.73205080756888) -- (-.5,.866025403784439) -- (0,1.73205080756888) -- cycle;
\path (-1,1.73205080756888) -- (0,1.73205080756888) -- (-.5,2.59807621135332) -- cycle;
\node[black,,scale=1.65525949786121] at (-.25,1.29903810567666) {1};
\node[black,,scale=1.65525949786121] at (-.5,1.73205080756888) {1};
\path (-.5,2.59807621135332) -- (0,1.73205080756888) -- (.5,2.59807621135332) -- cycle;
\node[black,,scale=1.65525949786121] at (-.25,2.1650635094611) {1};
\node[black,,scale=1.65525949786121] at (0,2.59807621135332) {1};
\path (-1.5,2.59807621135332) -- (-1,1.73205080756888) -- (-.5,2.59807621135332) -- cycle;
\node[black,,scale=1.65525949786121] at (-1.25,2.1650635094611) {2};
\node[black,,scale=1.65525949786121] at (-1,2.59807621135332) {2};
\end{scope}
\end{tikzpicture}

%% file: ex10e.tex
\begin{tikzpicture}[execute at begin picture={\bgroup\tikzset{every path/.style={}}\clip (-1.86,-.328089) rectangle ++(3.72,3.39025);\egroup},x={(1.54552502393359cm,0cm)},y={(0cm,-1.54552912706501cm)},baseline=(current  bounding  box.center),every path/.style={draw=black,fill=none},line join=round]
\begin{scope}[every path/.append style={fill=white,line width=.0309105415100132cm}]
\path (-.5,.866025403784439) -- (0,0) -- (.5,.866025403784439) -- cycle;
\path (-.5,.866025403784439) -- (.5,.866025403784439) -- (0,1.73205080756888) -- cycle;
\node[black,,scale=1.65525949786121] at (.25,.433012701892219) {1};
\node[black,,scale=1.65525949786121] at (-.25,.433012701892219) {1};
\node[black,,scale=1.65525949786121] at (0,.866025403784439) {1};
\path (0,1.73205080756888) -- (.5,.866025403784439) -- (1,1.73205080756888) -- cycle;
\path (0,1.73205080756888) -- (1,1.73205080756888) -- (.5,2.59807621135332) -- cycle;
\node[black,,scale=1.65525949786121] at (.75,1.29903810567666) {0};
\node[black,,scale=1.37938129208091] at (.25,1.29903810567666) {21};
\node[black,,scale=.919589150857369] at (.5,1.73205080756888) {(21)0};
\path (.5,2.59807621135332) -- (1,1.73205080756888) -- (1.5,2.59807621135332) -- cycle;
\node[black,,scale=1.65525949786121] at (1.25,2.1650635094611) {2};
\node[black,,scale=1.37938129208091] at (.75,2.1650635094611) {21};
\node[black,,scale=1.65525949786121] at (1,2.59807621135332) {1};
\path[fill=LightGray] (-1,1.73205080756888) -- (-.5,.866025403784439) -- (0,1.73205080756888) -- (-.5,2.59807621135332) -- cycle;
\node[black,,scale=1.65525949786121] at (-.25,1.29903810567666) {2};
\node[black,,scale=1.65525949786121] at (-.75,1.29903810567666) {0};
\path (-.5,2.59807621135332) -- (0,1.73205080756888) -- (.5,2.59807621135332) -- cycle;
\node[black,,scale=1.65525949786121] at (.25,2.1650635094611) {0};
\node[black,,scale=1.65525949786121] at (-.25,2.1650635094611) {0};
\node[black,,scale=1.65525949786121] at (0,2.59807621135332) {0};
\path (-1.5,2.59807621135332) -- (-1,1.73205080756888) -- (-.5,2.59807621135332) -- cycle;
\node[black,,scale=1.65525949786121] at (-.75,2.1650635094611) {2};
\node[black,,scale=1.65525949786121] at (-1.25,2.1650635094611) {2};
\node[black,,scale=1.65525949786121] at (-1,2.59807621135332) {2};
\end{scope}
\end{tikzpicture}\begin{tikzpicture}[execute at begin picture={\bgroup\tikzset{every path/.style={}}\clip (-1.86,-.328089) rectangle ++(3.72,3.39025);\egroup},x={(1.54552502393359cm,0cm)},y={(0cm,-1.54552912706501cm)},baseline=(current  bounding  box.center),every path/.style={draw=black,fill=none},line join=round]
\begin{scope}[every path/.append style={fill=white,line width=.0309105415100132cm}]
\path (-.5,.866025403784439) -- (0,0) -- (.5,.866025403784439) -- cycle;
\path (-.5,.866025403784439) -- (.5,.866025403784439) -- (0,1.73205080756888) -- cycle;
\node[black,,scale=1.65525949786121] at (.25,.433012701892219) {1};
\node[black,,scale=1.65525949786121] at (-.25,.433012701892219) {1};
\node[black,,scale=1.65525949786121] at (0,.866025403784439) {1};
\path (0,1.73205080756888) -- (.5,.866025403784439) -- (1,1.73205080756888) -- cycle;
\path[fill=LightPink] (0,1.73205080756888) -- (1,1.73205080756888) -- (.5,2.59807621135332) -- cycle;
\node[black,,scale=1.65525949786121] at (.75,1.29903810567666) {0};
\node[black,,scale=1.37938129208091] at (.25,1.29903810567666) {21};
\node[black,,scale=.919589150857369] at (.5,1.73205080756888) {(21)0};
\path (.5,2.59807621135332) -- (1,1.73205080756888) -- (1.5,2.59807621135332) -- cycle;
\node[black,,scale=1.65525949786121] at (1.25,2.1650635094611) {2};
\node[black,,scale=1.37938129208091] at (.75,2.1650635094611) {20};
\node[black,,scale=1.65525949786121] at (1,2.59807621135332) {0};
\path[fill=LightGray] (-1,1.73205080756888) -- (-.5,.866025403784439) -- (0,1.73205080756888) -- (-.5,2.59807621135332) -- cycle;
\node[black,,scale=1.65525949786121] at (-.25,1.29903810567666) {2};
\node[black,,scale=1.65525949786121] at (-.75,1.29903810567666) {0};
\path (-.5,2.59807621135332) -- (0,1.73205080756888) -- (.5,2.59807621135332) -- cycle;
\node[black,,scale=1.37938129208091] at (.25,2.1650635094611) {10};
\node[black,,scale=1.65525949786121] at (-.25,2.1650635094611) {0};
\node[black,,scale=1.65525949786121] at (0,2.59807621135332) {1};
\path (-1.5,2.59807621135332) -- (-1,1.73205080756888) -- (-.5,2.59807621135332) -- cycle;
\node[black,,scale=1.65525949786121] at (-.75,2.1650635094611) {2};
\node[black,,scale=1.65525949786121] at (-1.25,2.1650635094611) {2};
\node[black,,scale=1.65525949786121] at (-1,2.59807621135332) {2};
\end{scope}
\end{tikzpicture}\begin{tikzpicture}[execute at begin picture={\bgroup\tikzset{every path/.style={}}\clip (-1.86,-.328089) rectangle ++(3.72,3.39025);\egroup},x={(1.54552502393359cm,0cm)},y={(0cm,-1.54552912706501cm)},baseline=(current  bounding  box.center),every path/.style={draw=black,fill=none},line join=round]
\begin{scope}[every path/.append style={fill=white,line width=.0309105415100132cm}]
\path (-.5,.866025403784439) -- (0,0) -- (.5,.866025403784439) -- cycle;
\path (-.5,.866025403784439) -- (.5,.866025403784439) -- (0,1.73205080756888) -- cycle;
\node[black,,scale=1.65525949786121] at (.25,.433012701892219) {1};
\node[black,,scale=1.65525949786121] at (-.25,.433012701892219) {1};
\node[black,,scale=1.65525949786121] at (0,.866025403784439) {1};
\path (0,1.73205080756888) -- (.5,.866025403784439) -- (1,1.73205080756888) -- cycle;
\path (0,1.73205080756888) -- (1,1.73205080756888) -- (.5,2.59807621135332) -- cycle;
\node[black,,scale=1.65525949786121] at (.75,1.29903810567666) {0};
\node[black,,scale=1.65525949786121] at (.25,1.29903810567666) {0};
\node[black,,scale=1.65525949786121] at (.5,1.73205080756888) {0};
\path (.5,2.59807621135332) -- (1,1.73205080756888) -- (1.5,2.59807621135332) -- cycle;
\node[black,,scale=1.65525949786121] at (1.25,2.1650635094611) {2};
\node[black,,scale=1.37938129208091] at (.75,2.1650635094611) {20};
\node[black,,scale=1.65525949786121] at (1,2.59807621135332) {0};
\path (-1,1.73205080756888) -- (-.5,.866025403784439) -- (0,1.73205080756888) -- cycle;
\path (-1,1.73205080756888) -- (0,1.73205080756888) -- (-.5,2.59807621135332) -- cycle;
\node[black,,scale=1.37938129208091] at (-.25,1.29903810567666) {10};
\node[black,,scale=1.65525949786121] at (-.75,1.29903810567666) {0};
\node[black,,scale=1.65525949786121] at (-.5,1.73205080756888) {1};
\path (-.5,2.59807621135332) -- (0,1.73205080756888) -- (.5,2.59807621135332) -- cycle;
\node[black,,scale=1.65525949786121] at (.25,2.1650635094611) {2};
\node[black,,scale=1.37938129208091] at (-.25,2.1650635094611) {21};
\node[black,,scale=1.65525949786121] at (0,2.59807621135332) {1};
\path (-1.5,2.59807621135332) -- (-1,1.73205080756888) -- (-.5,2.59807621135332) -- cycle;
\node[black,,scale=1.65525949786121] at (-.75,2.1650635094611) {2};
\node[black,,scale=1.65525949786121] at (-1.25,2.1650635094611) {2};
\node[black,,scale=1.65525949786121] at (-1,2.59807621135332) {2};
\end{scope}
\end{tikzpicture}

%% file: ex21a.tex
\begin{tikzpicture}[execute at begin picture={\bgroup\tikzset{every path/.style={}}\clip (-2.48,-.432012) rectangle ++(4.96,4.46413);\egroup},x={(1.55426771761071cm,0cm)},y={(0cm,-1.55426394417233cm)},baseline=(current  bounding  box.center),every path/.style={draw=black,fill=none},line join=round]
\begin{scope}[every path/.append style={fill=white,line width=.0310853166178533cm}]
\path (-.5,.866025403784439) -- (0,0) -- (.5,.866025403784439) -- cycle;
\path (-.5,.866025403784439) -- (.5,.866025403784439) -- (0,1.73205080756888) -- cycle;
\node[black,,scale=1.66461870488604] at (.25,.433012701892219) {0};
\node[black,,scale=1.66461870488604] at (-.25,.433012701892219) {1};
\node[black,,scale=1.38718062209258] at (0,.866025403784439) {10};
\path (0,1.73205080756888) -- (.5,.866025403784439) -- (1,1.73205080756888) -- cycle;
\path (0,1.73205080756888) -- (1,1.73205080756888) -- (.5,2.59807621135332) -- cycle;
\node[black,,scale=1.66461870488604] at (.75,1.29903810567666) {2};
\node[black,,scale=.924788713374177] at (.25,1.29903810567666) {2(10)};
\node[black,,scale=1.38718062209258] at (.5,1.73205080756888) {10};
\path (.5,2.59807621135332) -- (1,1.73205080756888) -- (1.5,2.59807621135332) -- cycle;
\path (.5,2.59807621135332) -- (1.5,2.59807621135332) -- (1,3.46410161513775) -- cycle;
\node[black,,scale=1.66461870488604] at (1.25,2.1650635094611) {1};
\node[black,,scale=1.66461870488604] at (.75,2.1650635094611) {1};
\node[black,,scale=1.66461870488604] at (1,2.59807621135332) {1};
\path (1,3.46410161513775) -- (1.5,2.59807621135332) -- (2,3.46410161513775) -- cycle;
\node[black,,scale=1.66461870488604] at (1.75,3.03108891324554) {0};
\node[black,,scale=1.66461870488604] at (1.25,3.03108891324554) {1};
\node[black,,scale=1.38718062209258] at (1.5,3.46410161513775) {10};
\path (-1,1.73205080756888) -- (-.5,.866025403784439) -- (0,1.73205080756888) -- cycle;
\path (-1,1.73205080756888) -- (0,1.73205080756888) -- (-.5,2.59807621135332) -- cycle;
\node[black,,scale=1.66461870488604] at (-.25,1.29903810567666) {2};
\node[black,,scale=1.66461870488604] at (-.75,1.29903810567666) {2};
\node[black,,scale=1.66461870488604] at (-.5,1.73205080756888) {2};
\path (-.5,2.59807621135332) -- (0,1.73205080756888) -- (.5,2.59807621135332) -- cycle;
\path (-.5,2.59807621135332) -- (.5,2.59807621135332) -- (0,3.46410161513775) -- cycle;
\node[black,,scale=1.66461870488604] at (.25,2.1650635094611) {0};
\node[black,,scale=1.66461870488604] at (-.25,2.1650635094611) {0};
\node[black,,scale=1.66461870488604] at (0,2.59807621135332) {0};
\path (0,3.46410161513775) -- (.5,2.59807621135332) -- (1,3.46410161513775) -- cycle;
\node[black,,scale=1.66461870488604] at (.75,3.03108891324553) {1};
\node[black,,scale=1.38718062209258] at (.25,3.03108891324554) {10};
\node[black,,scale=1.66461870488604] at (.5,3.46410161513775) {0};
\path (-1.5,2.59807621135332) -- (-1,1.73205080756888) -- (-.5,2.59807621135332) -- cycle;
\path (-1.5,2.59807621135332) -- (-.5,2.59807621135332) -- (-1,3.46410161513775) -- cycle;
\node[black,,scale=1.38718062209258] at (-.75,2.1650635094611) {20};
\node[black,,scale=1.66461870488604] at (-1.25,2.1650635094611) {0};
\node[black,,scale=1.66461870488604] at (-1,2.59807621135332) {2};
\path (-1,3.46410161513775) -- (-.5,2.59807621135332) -- (0,3.46410161513775) -- cycle;
\node[black,,scale=1.66461870488604] at (-.25,3.03108891324554) {1};
\node[black,,scale=1.66461870488604] at (-.75,3.03108891324553) {1};
\node[black,,scale=1.66461870488604] at (-.5,3.46410161513775) {1};
\path (-2,3.46410161513775) -- (-1.5,2.59807621135332) -- (-1,3.46410161513775) -- cycle;
\node[black,,scale=1.38718062209258] at (-1.25,3.03108891324554) {21};
\node[black,,scale=1.66461870488604] at (-1.75,3.03108891324554) {1};
\node[black,,scale=1.66461870488604] at (-1.5,3.46410161513775) {2};
\end{scope}
\end{tikzpicture}\begin{tikzpicture}[execute at begin picture={\bgroup\tikzset{every path/.style={}}\clip (-2.48,-.432012) rectangle ++(4.96,4.46413);\egroup},x={(1.55426771761071cm,0cm)},y={(0cm,-1.55426394417233cm)},baseline=(current  bounding  box.center),every path/.style={draw=black,fill=none},line join=round]
\begin{scope}[every path/.append style={fill=white,line width=.0310853166178533cm}]
\path (-.5,.866025403784439) -- (0,0) -- (.5,.866025403784439) -- cycle;
\path (-.5,.866025403784439) -- (.5,.866025403784439) -- (0,1.73205080756888) -- cycle;
\node[black,,scale=1.66461870488604] at (.25,.433012701892219) {0};
\node[black,,scale=1.66461870488604] at (-.25,.433012701892219) {1};
\node[black,,scale=1.38718062209258] at (0,.866025403784439) {10};
\path (0,1.73205080756888) -- (.5,.866025403784439) -- (1,1.73205080756888) -- cycle;
\path (0,1.73205080756888) -- (1,1.73205080756888) -- (.5,2.59807621135332) -- cycle;
\node[black,,scale=1.66461870488604] at (.75,1.29903810567666) {2};
\node[black,,scale=.924788713374177] at (.25,1.29903810567666) {2(10)};
\node[black,,scale=1.38718062209258] at (.5,1.73205080756888) {10};
\path (.5,2.59807621135332) -- (1,1.73205080756888) -- (1.5,2.59807621135332) -- cycle;
\path (.5,2.59807621135332) -- (1.5,2.59807621135332) -- (1,3.46410161513775) -- cycle;
\node[black,,scale=1.66461870488604] at (1.25,2.1650635094611) {1};
\node[black,,scale=1.66461870488604] at (.75,2.1650635094611) {1};
\node[black,,scale=1.66461870488604] at (1,2.59807621135332) {1};
\path (1,3.46410161513775) -- (1.5,2.59807621135332) -- (2,3.46410161513775) -- cycle;
\node[black,,scale=1.66461870488604] at (1.75,3.03108891324554) {0};
\node[black,,scale=1.66461870488604] at (1.25,3.03108891324554) {0};
\node[black,,scale=1.66461870488604] at (1.5,3.46410161513775) {0};
\path (-1,1.73205080756888) -- (-.5,.866025403784439) -- (0,1.73205080756888) -- cycle;
\path (-1,1.73205080756888) -- (0,1.73205080756888) -- (-.5,2.59807621135332) -- cycle;
\node[black,,scale=1.66461870488604] at (-.25,1.29903810567666) {2};
\node[black,,scale=1.66461870488604] at (-.75,1.29903810567666) {2};
\node[black,,scale=1.66461870488604] at (-.5,1.73205080756888) {2};
\path (-.5,2.59807621135332) -- (0,1.73205080756888) -- (.5,2.59807621135332) -- cycle;
\path (-.5,2.59807621135332) -- (.5,2.59807621135332) -- (0,3.46410161513775) -- cycle;
\node[black,,scale=1.66461870488604] at (.25,2.1650635094611) {0};
\node[black,,scale=1.66461870488604] at (-.25,2.1650635094611) {0};
\node[black,,scale=1.66461870488604] at (0,2.59807621135332) {0};
\path (0,3.46410161513775) -- (.5,2.59807621135332) -- (1,3.46410161513775) -- cycle;
\node[black,,scale=1.38718062209258] at (.75,3.03108891324553) {10};
\node[black,,scale=1.66461870488604] at (.25,3.03108891324554) {0};
\node[black,,scale=1.66461870488604] at (.5,3.46410161513775) {1};
\path (-1.5,2.59807621135332) -- (-1,1.73205080756888) -- (-.5,2.59807621135332) -- cycle;
\path (-1.5,2.59807621135332) -- (-.5,2.59807621135332) -- (-1,3.46410161513775) -- cycle;
\node[black,,scale=1.38718062209258] at (-.75,2.1650635094611) {20};
\node[black,,scale=1.66461870488604] at (-1.25,2.1650635094611) {0};
\node[black,,scale=1.66461870488604] at (-1,2.59807621135332) {2};
\path (-1,3.46410161513775) -- (-.5,2.59807621135332) -- (0,3.46410161513775) -- cycle;
\node[black,,scale=1.66461870488604] at (-.25,3.03108891324554) {0};
\node[black,,scale=1.66461870488604] at (-.75,3.03108891324553) {1};
\node[black,,scale=1.38718062209258] at (-.5,3.46410161513775) {10};
\path (-2,3.46410161513775) -- (-1.5,2.59807621135332) -- (-1,3.46410161513775) -- cycle;
\node[black,,scale=1.38718062209258] at (-1.25,3.03108891324554) {21};
\node[black,,scale=1.66461870488604] at (-1.75,3.03108891324554) {1};
\node[black,,scale=1.66461870488604] at (-1.5,3.46410161513775) {2};
\end{scope}
\end{tikzpicture}\begin{tikzpicture}[execute at begin picture={\bgroup\tikzset{every path/.style={}}\clip (-2.48,-.432012) rectangle ++(4.96,4.46413);\egroup},x={(1.55426771761071cm,0cm)},y={(0cm,-1.55426394417233cm)},baseline=(current  bounding  box.center),every path/.style={draw=black,fill=none},line join=round]
\begin{scope}[every path/.append style={fill=white,line width=.0310853166178533cm}]
\path (-.5,.866025403784439) -- (0,0) -- (.5,.866025403784439) -- cycle;
\path (-.5,.866025403784439) -- (.5,.866025403784439) -- (0,1.73205080756888) -- cycle;
\node[black,,scale=1.66461870488604] at (.25,.433012701892219) {0};
\node[black,,scale=1.66461870488604] at (-.25,.433012701892219) {1};
\node[black,,scale=1.38718062209258] at (0,.866025403784439) {10};
\path (0,1.73205080756888) -- (.5,.866025403784439) -- (1,1.73205080756888) -- cycle;
\path (0,1.73205080756888) -- (1,1.73205080756888) -- (.5,2.59807621135332) -- cycle;
\node[black,,scale=1.66461870488604] at (.75,1.29903810567666) {2};
\node[black,,scale=.924788713374177] at (.25,1.29903810567666) {2(10)};
\node[black,,scale=1.38718062209258] at (.5,1.73205080756888) {10};
\path (.5,2.59807621135332) -- (1,1.73205080756888) -- (1.5,2.59807621135332) -- cycle;
\path (.5,2.59807621135332) -- (1.5,2.59807621135332) -- (1,3.46410161513775) -- cycle;
\node[black,,scale=1.66461870488604] at (1.25,2.1650635094611) {1};
\node[black,,scale=1.66461870488604] at (.75,2.1650635094611) {1};
\node[black,,scale=1.66461870488604] at (1,2.59807621135332) {1};
\path (1,3.46410161513775) -- (1.5,2.59807621135332) -- (2,3.46410161513775) -- cycle;
\node[black,,scale=1.66461870488604] at (1.75,3.03108891324554) {0};
\node[black,,scale=1.66461870488604] at (1.25,3.03108891324554) {0};
\node[black,,scale=1.66461870488604] at (1.5,3.46410161513775) {0};
\path (-1,1.73205080756888) -- (-.5,.866025403784439) -- (0,1.73205080756888) -- cycle;
\path (-1,1.73205080756888) -- (0,1.73205080756888) -- (-.5,2.59807621135332) -- cycle;
\node[black,,scale=1.66461870488604] at (-.25,1.29903810567666) {2};
\node[black,,scale=1.66461870488604] at (-.75,1.29903810567666) {2};
\node[black,,scale=1.66461870488604] at (-.5,1.73205080756888) {2};
\path (-.5,2.59807621135332) -- (0,1.73205080756888) -- (.5,2.59807621135332) -- cycle;
\path (-.5,2.59807621135332) -- (.5,2.59807621135332) -- (0,3.46410161513775) -- cycle;
\node[black,,scale=1.66461870488604] at (.25,2.1650635094611) {0};
\node[black,,scale=1.66461870488604] at (-.25,2.1650635094611) {0};
\node[black,,scale=1.66461870488604] at (0,2.59807621135332) {0};
\path (0,3.46410161513775) -- (.5,2.59807621135332) -- (1,3.46410161513775) -- cycle;
\node[black,,scale=1.38718062209258] at (.75,3.03108891324553) {10};
\node[black,,scale=1.38718062209258] at (.25,3.03108891324554) {10};
\node[black,,scale=1.38718062209258] at (.5,3.46410161513775) {10};
\path (-1.5,2.59807621135332) -- (-1,1.73205080756888) -- (-.5,2.59807621135332) -- cycle;
\path (-1.5,2.59807621135332) -- (-.5,2.59807621135332) -- (-1,3.46410161513775) -- cycle;
\node[black,,scale=1.38718062209258] at (-.75,2.1650635094611) {20};
\node[black,,scale=1.66461870488604] at (-1.25,2.1650635094611) {0};
\node[black,,scale=1.66461870488604] at (-1,2.59807621135332) {2};
\path (-1,3.46410161513775) -- (-.5,2.59807621135332) -- (0,3.46410161513775) -- cycle;
\node[black,,scale=1.66461870488604] at (-.25,3.03108891324554) {1};
\node[black,,scale=1.66461870488604] at (-.75,3.03108891324553) {1};
\node[black,,scale=1.66461870488604] at (-.5,3.46410161513775) {1};
\path (-2,3.46410161513775) -- (-1.5,2.59807621135332) -- (-1,3.46410161513775) -- cycle;
\node[black,,scale=1.38718062209258] at (-1.25,3.03108891324554) {21};
\node[black,,scale=1.66461870488604] at (-1.75,3.03108891324554) {1};
\node[black,,scale=1.66461870488604] at (-1.5,3.46410161513775) {2};
\end{scope}
\end{tikzpicture}\begin{tikzpicture}[execute at begin picture={\bgroup\tikzset{every path/.style={}}\clip (-2.48,-.432012) rectangle ++(4.96,4.46413);\egroup},x={(1.55426771761071cm,0cm)},y={(0cm,-1.55426394417233cm)},baseline=(current  bounding  box.center),every path/.style={draw=black,fill=none},line join=round]
\begin{scope}[every path/.append style={fill=white,line width=.0310853166178533cm}]
\path (-.5,.866025403784439) -- (0,0) -- (.5,.866025403784439) -- cycle;
\path (-.5,.866025403784439) -- (.5,.866025403784439) -- (0,1.73205080756888) -- cycle;
\node[black,,scale=1.66461870488604] at (.25,.433012701892219) {0};
\node[black,,scale=1.66461870488604] at (-.25,.433012701892219) {1};
\node[black,,scale=1.38718062209258] at (0,.866025403784439) {10};
\path (0,1.73205080756888) -- (.5,.866025403784439) -- (1,1.73205080756888) -- cycle;
\path (0,1.73205080756888) -- (1,1.73205080756888) -- (.5,2.59807621135332) -- cycle;
\node[black,,scale=1.66461870488604] at (.75,1.29903810567666) {2};
\node[black,,scale=.924788713374177] at (.25,1.29903810567666) {2(10)};
\node[black,,scale=1.38718062209258] at (.5,1.73205080756888) {10};
\path (.5,2.59807621135332) -- (1,1.73205080756888) -- (1.5,2.59807621135332) -- cycle;
\path (.5,2.59807621135332) -- (1.5,2.59807621135332) -- (1,3.46410161513775) -- cycle;
\node[black,,scale=1.66461870488604] at (1.25,2.1650635094611) {1};
\node[black,,scale=1.38718062209258] at (.75,2.1650635094611) {10};
\node[black,,scale=1.66461870488604] at (1,2.59807621135332) {0};
\path (1,3.46410161513775) -- (1.5,2.59807621135332) -- (2,3.46410161513775) -- cycle;
\node[black,,scale=1.66461870488604] at (1.75,3.03108891324554) {0};
\node[black,,scale=1.66461870488604] at (1.25,3.03108891324554) {0};
\node[black,,scale=1.66461870488604] at (1.5,3.46410161513775) {0};
\path (-1,1.73205080756888) -- (-.5,.866025403784439) -- (0,1.73205080756888) -- cycle;
\path (-1,1.73205080756888) -- (0,1.73205080756888) -- (-.5,2.59807621135332) -- cycle;
\node[black,,scale=1.66461870488604] at (-.25,1.29903810567666) {2};
\node[black,,scale=1.66461870488604] at (-.75,1.29903810567666) {2};
\node[black,,scale=1.66461870488604] at (-.5,1.73205080756888) {2};
\path (-.5,2.59807621135332) -- (0,1.73205080756888) -- (.5,2.59807621135332) -- cycle;
\path (-.5,2.59807621135332) -- (.5,2.59807621135332) -- (0,3.46410161513775) -- cycle;
\node[black,,scale=1.38718062209258] at (.25,2.1650635094611) {10};
\node[black,,scale=1.66461870488604] at (-.25,2.1650635094611) {0};
\node[black,,scale=1.66461870488604] at (0,2.59807621135332) {1};
\path (0,3.46410161513775) -- (.5,2.59807621135332) -- (1,3.46410161513775) -- cycle;
\node[black,,scale=1.66461870488604] at (.75,3.03108891324553) {0};
\node[black,,scale=1.66461870488604] at (.25,3.03108891324554) {1};
\node[black,,scale=1.38718062209258] at (.5,3.46410161513775) {10};
\path (-1.5,2.59807621135332) -- (-1,1.73205080756888) -- (-.5,2.59807621135332) -- cycle;
\path (-1.5,2.59807621135332) -- (-.5,2.59807621135332) -- (-1,3.46410161513775) -- cycle;
\node[black,,scale=1.38718062209258] at (-.75,2.1650635094611) {20};
\node[black,,scale=1.66461870488604] at (-1.25,2.1650635094611) {0};
\node[black,,scale=1.66461870488604] at (-1,2.59807621135332) {2};
\path (-1,3.46410161513775) -- (-.5,2.59807621135332) -- (0,3.46410161513775) -- cycle;
\node[black,,scale=1.66461870488604] at (-.25,3.03108891324554) {1};
\node[black,,scale=1.66461870488604] at (-.75,3.03108891324553) {1};
\node[black,,scale=1.66461870488604] at (-.5,3.46410161513775) {1};
\path (-2,3.46410161513775) -- (-1.5,2.59807621135332) -- (-1,3.46410161513775) -- cycle;
\node[black,,scale=1.38718062209258] at (-1.25,3.03108891324554) {21};
\node[black,,scale=1.66461870488604] at (-1.75,3.03108891324554) {1};
\node[black,,scale=1.66461870488604] at (-1.5,3.46410161513775) {2};
\end{scope}
\end{tikzpicture}

%% file: ex21b.tex
\begin{tikzpicture}[execute at begin picture={\bgroup\tikzset{every path/.style={}}\clip (-2.48,-.429292) rectangle ++(4.96,4.43602);\egroup},x={(1.55865151133662cm,0cm)},y={(0cm,-1.55865187137441cm)},baseline=(current  bounding  box.center),every path/.style={draw=black,fill=none},line join=round]
\begin{scope}[every path/.append style={fill=white,line width=.0311730338271105cm}]
\path (-.5,.866025403784439) -- (0,0) -- (.5,.866025403784439) -- cycle;
\path (-.5,.866025403784439) -- (.5,.866025403784439) -- (0,1.73205080756888) -- cycle;
\node[black,,scale=1.66931596144177] at (-.25,.433012701892219) {2};
\node[black,,scale=1.39109499795053] at (0,.866025403784439) {\(\searrow\)2};
\path (0,1.73205080756888) -- (.5,.866025403784439) -- (1,1.73205080756888) -- cycle;
\path (0,1.73205080756888) -- (1,1.73205080756888) -- (.5,2.59807621135332) -- cycle;
\node[black,,scale=1.66931596144177] at (.75,1.29903810567666) {3};
\node[black,,scale=1.39109499795053] at (.25,1.29903810567666) {23};
\node[black,,scale=1.39109499795053] at (.5,1.73205080756888) {\(\searrow\)2};
\path (.5,2.59807621135332) -- (1,1.73205080756888) -- (1.5,2.59807621135332) -- cycle;
\path (.5,2.59807621135332) -- (1.5,2.59807621135332) -- (1,3.46410161513775) -- cycle;
\node[black,,scale=1.66931596144177] at (1.25,2.1650635094611) {2};
\node[black,,scale=1.66931596144177] at (.75,2.1650635094611) {2};
\node[black,,scale=1.19236784249372] at (1,2.59807621135332) {odd};
\path (1,3.46410161513775) -- (1.5,2.59807621135332) -- (2,3.46410161513775) -- cycle;
\node[black,,scale=1.66931596144177] at (1.25,3.03108891324554) {1};
\node[black,,scale=1.39109499795053] at (1.5,3.46410161513775) {\(\searrow\)1};
\path (-1,1.73205080756888) -- (-.5,.866025403784439) -- (0,1.73205080756888) -- cycle;
\path (-1,1.73205080756888) -- (0,1.73205080756888) -- (-.5,2.59807621135332) -- cycle;
\node[black,,scale=1.66931596144177] at (-.25,1.29903810567666) {3};
\node[black,,scale=1.39109499795053] at (-.5,1.73205080756888) {\(\nearrow\)3};
\path (-.5,2.59807621135332) -- (0,1.73205080756888) -- (.5,2.59807621135332) -- cycle;
\path (-.5,2.59807621135332) -- (.5,2.59807621135332) -- (0,3.46410161513775) -- cycle;
\node[black,,scale=1.66931596144177] at (-.25,2.1650635094611) {0};
\node[black,,scale=1.39109499795053] at (0,2.59807621135332) {\(\searrow\)0};
\path (0,3.46410161513775) -- (.5,2.59807621135332) -- (1,3.46410161513775) -- cycle;
\node[black,,scale=1.66931596144177] at (.75,3.03108891324553) {1};
\node[black,,scale=1.39109499795053] at (.25,3.03108891324554) {01};
\node[black,,scale=1.39109499795053] at (.5,3.46410161513775) {\(\searrow\)0};
\path (-1.5,2.59807621135332) -- (-1,1.73205080756888) -- (-.5,2.59807621135332) -- cycle;
\path (-1.5,2.59807621135332) -- (-.5,2.59807621135332) -- (-1,3.46410161513775) -- cycle;
\node[black,,scale=1.39109499795053] at (-.75,2.1650635094611) {03};
\node[black,,scale=1.66931596144177] at (-1.25,2.1650635094611) {0};
\node[black,,scale=1.39109499795053] at (-1,2.59807621135332) {\(\nearrow\)3};
\path (-1,3.46410161513775) -- (-.5,2.59807621135332) -- (0,3.46410161513775) -- cycle;
\node[black,,scale=1.66931596144177] at (-.25,3.03108891324554) {1};
\node[black,,scale=1.66931596144177] at (-.75,3.03108891324553) {1};
\node[black,,scale=1.19236784249372] at (-.5,3.46410161513775) {odd};
\path (-2,3.46410161513775) -- (-1.5,2.59807621135332) -- (-1,3.46410161513775) -- cycle;
\node[black,,scale=1.39109499795053] at (-1.25,3.03108891324554) {13};
\node[black,,scale=1.66931596144177] at (-1.75,3.03108891324554) {1};
\node[black,,scale=1.39109499795053] at (-1.5,3.46410161513775) {\(\nearrow\)3};
\end{scope}
\end{tikzpicture}\begin{tikzpicture}[execute at begin picture={\bgroup\tikzset{every path/.style={}}\clip (-2.48,-.429292) rectangle ++(4.96,4.43602);\egroup},x={(1.55865151133662cm,0cm)},y={(0cm,-1.55865187137441cm)},baseline=(current  bounding  box.center),every path/.style={draw=black,fill=none},line join=round]
\begin{scope}[every path/.append style={fill=white,line width=.0311730338271105cm}]
\path (-.5,.866025403784439) -- (0,0) -- (.5,.866025403784439) -- cycle;
\path (-.5,.866025403784439) -- (.5,.866025403784439) -- (0,1.73205080756888) -- cycle;
\node[black,,scale=1.66931596144177] at (-.25,.433012701892219) {2};
\node[black,,scale=1.39109499795053] at (0,.866025403784439) {\(\searrow\)2};
\path (0,1.73205080756888) -- (.5,.866025403784439) -- (1,1.73205080756888) -- cycle;
\path (0,1.73205080756888) -- (1,1.73205080756888) -- (.5,2.59807621135332) -- cycle;
\node[black,,scale=1.66931596144177] at (.75,1.29903810567666) {3};
\node[black,,scale=1.39109499795053] at (.25,1.29903810567666) {23};
\node[black,,scale=1.39109499795053] at (.5,1.73205080756888) {\(\searrow\)2};
\path (.5,2.59807621135332) -- (1,1.73205080756888) -- (1.5,2.59807621135332) -- cycle;
\path (.5,2.59807621135332) -- (1.5,2.59807621135332) -- (1,3.46410161513775) -- cycle;
\node[black,,scale=1.66931596144177] at (1.25,2.1650635094611) {2};
\node[black,,scale=1.66931596144177] at (.75,2.1650635094611) {2};
\node[black,,scale=1.19236784249372] at (1,2.59807621135332) {odd};
\path (1,3.46410161513775) -- (1.5,2.59807621135332) -- (2,3.46410161513775) -- cycle;
\node[black,,scale=1.66931596144177] at (1.25,3.03108891324554) {0};
\node[black,,scale=1.39109499795053] at (1.5,3.46410161513775) {\(\searrow\)0};
\path (-1,1.73205080756888) -- (-.5,.866025403784439) -- (0,1.73205080756888) -- cycle;
\path (-1,1.73205080756888) -- (0,1.73205080756888) -- (-.5,2.59807621135332) -- cycle;
\node[black,,scale=1.66931596144177] at (-.25,1.29903810567666) {3};
\node[black,,scale=1.39109499795053] at (-.5,1.73205080756888) {\(\nearrow\)3};
\path (-.5,2.59807621135332) -- (0,1.73205080756888) -- (.5,2.59807621135332) -- cycle;
\path (-.5,2.59807621135332) -- (.5,2.59807621135332) -- (0,3.46410161513775) -- cycle;
\node[black,,scale=1.66931596144177] at (-.25,2.1650635094611) {0};
\node[black,,scale=1.39109499795053] at (0,2.59807621135332) {\(\searrow\)0};
\path (0,3.46410161513775) -- (.5,2.59807621135332) -- (1,3.46410161513775) -- cycle;
\node[black,,scale=1.66931596144177] at (.75,3.03108891324553) {0};
\node[black,,scale=1.66931596144177] at (.25,3.03108891324554) {0};
\node[black,,scale=1.19236784249372] at (.5,3.46410161513775) {odd};
\path (-1.5,2.59807621135332) -- (-1,1.73205080756888) -- (-.5,2.59807621135332) -- cycle;
\path (-1.5,2.59807621135332) -- (-.5,2.59807621135332) -- (-1,3.46410161513775) -- cycle;
\node[black,,scale=1.39109499795053] at (-.75,2.1650635094611) {03};
\node[black,,scale=1.66931596144177] at (-1.25,2.1650635094611) {0};
\node[black,,scale=1.39109499795053] at (-1,2.59807621135332) {\(\nearrow\)3};
\path (-1,3.46410161513775) -- (-.5,2.59807621135332) -- (0,3.46410161513775) -- cycle;
\node[black,,scale=1.66931596144177] at (-.75,3.03108891324553) {1};
\node[black,,scale=1.39109499795053] at (-.5,3.46410161513775) {\(\searrow\)1};
\path (-2,3.46410161513775) -- (-1.5,2.59807621135332) -- (-1,3.46410161513775) -- cycle;
\node[black,,scale=1.39109499795053] at (-1.25,3.03108891324554) {13};
\node[black,,scale=1.66931596144177] at (-1.75,3.03108891324554) {1};
\node[black,,scale=1.39109499795053] at (-1.5,3.46410161513775) {\(\nearrow\)3};
\end{scope}
\end{tikzpicture}\begin{tikzpicture}[execute at begin picture={\bgroup\tikzset{every path/.style={}}\clip (-2.48,-.429292) rectangle ++(4.96,4.43602);\egroup},x={(1.55865151133662cm,0cm)},y={(0cm,-1.55865187137441cm)},baseline=(current  bounding  box.center),every path/.style={draw=black,fill=none},line join=round]
\begin{scope}[every path/.append style={fill=white,line width=.0311730338271105cm}]
\path (-.5,.866025403784439) -- (0,0) -- (.5,.866025403784439) -- cycle;
\path (-.5,.866025403784439) -- (.5,.866025403784439) -- (0,1.73205080756888) -- cycle;
\node[black,,scale=1.66931596144177] at (-.25,.433012701892219) {2};
\node[black,,scale=1.39109499795053] at (0,.866025403784439) {\(\searrow\)2};
\path (0,1.73205080756888) -- (.5,.866025403784439) -- (1,1.73205080756888) -- cycle;
\path (0,1.73205080756888) -- (1,1.73205080756888) -- (.5,2.59807621135332) -- cycle;
\node[black,,scale=1.66931596144177] at (.75,1.29903810567666) {3};
\node[black,,scale=1.39109499795053] at (.25,1.29903810567666) {23};
\node[black,,scale=1.39109499795053] at (.5,1.73205080756888) {\(\searrow\)2};
\path (.5,2.59807621135332) -- (1,1.73205080756888) -- (1.5,2.59807621135332) -- cycle;
\path (.5,2.59807621135332) -- (1.5,2.59807621135332) -- (1,3.46410161513775) -- cycle;
\node[black,,scale=1.66931596144177] at (1.25,2.1650635094611) {2};
\node[black,,scale=1.66931596144177] at (.75,2.1650635094611) {2};
\node[black,,scale=1.19236784249372] at (1,2.59807621135332) {odd};
\path (1,3.46410161513775) -- (1.5,2.59807621135332) -- (2,3.46410161513775) -- cycle;
\node[black,,scale=1.66931596144177] at (1.25,3.03108891324554) {0};
\node[black,,scale=1.39109499795053] at (1.5,3.46410161513775) {\(\searrow\)0};
\path (-1,1.73205080756888) -- (-.5,.866025403784439) -- (0,1.73205080756888) -- cycle;
\path (-1,1.73205080756888) -- (0,1.73205080756888) -- (-.5,2.59807621135332) -- cycle;
\node[black,,scale=1.66931596144177] at (-.25,1.29903810567666) {3};
\node[black,,scale=1.39109499795053] at (-.5,1.73205080756888) {\(\nearrow\)3};
\path (-.5,2.59807621135332) -- (0,1.73205080756888) -- (.5,2.59807621135332) -- cycle;
\path (-.5,2.59807621135332) -- (.5,2.59807621135332) -- (0,3.46410161513775) -- cycle;
\node[black,,scale=1.66931596144177] at (-.25,2.1650635094611) {0};
\node[black,,scale=1.39109499795053] at (0,2.59807621135332) {\(\searrow\)0};
\path (0,3.46410161513775) -- (.5,2.59807621135332) -- (1,3.46410161513775) -- cycle;
\node[black,,scale=1.66931596144177] at (.75,3.03108891324553) {0};
\node[black,,scale=1.39109499795053] at (.25,3.03108891324554) {01};
\node[black,,scale=1.39109499795053] at (.5,3.46410161513775) {\(\searrow\)1};
\path (-1.5,2.59807621135332) -- (-1,1.73205080756888) -- (-.5,2.59807621135332) -- cycle;
\path (-1.5,2.59807621135332) -- (-.5,2.59807621135332) -- (-1,3.46410161513775) -- cycle;
\node[black,,scale=1.39109499795053] at (-.75,2.1650635094611) {03};
\node[black,,scale=1.66931596144177] at (-1.25,2.1650635094611) {0};
\node[black,,scale=1.39109499795053] at (-1,2.59807621135332) {\(\nearrow\)3};
\path (-1,3.46410161513775) -- (-.5,2.59807621135332) -- (0,3.46410161513775) -- cycle;
\node[black,,scale=1.66931596144177] at (-.25,3.03108891324554) {1};
\node[black,,scale=1.66931596144177] at (-.75,3.03108891324553) {1};
\node[black,,scale=1.19236784249372] at (-.5,3.46410161513775) {odd};
\path (-2,3.46410161513775) -- (-1.5,2.59807621135332) -- (-1,3.46410161513775) -- cycle;
\node[black,,scale=1.39109499795053] at (-1.25,3.03108891324554) {13};
\node[black,,scale=1.66931596144177] at (-1.75,3.03108891324554) {1};
\node[black,,scale=1.39109499795053] at (-1.5,3.46410161513775) {\(\nearrow\)3};
\end{scope}
\end{tikzpicture}\begin{tikzpicture}[execute at begin picture={\bgroup\tikzset{every path/.style={}}\clip (-2.48,-.429292) rectangle ++(4.96,4.43602);\egroup},x={(1.55865151133662cm,0cm)},y={(0cm,-1.55865187137441cm)},baseline=(current  bounding  box.center),every path/.style={draw=black,fill=none},line join=round]
\begin{scope}[every path/.append style={fill=white,line width=.0311730338271105cm}]
\path (-.5,.866025403784439) -- (0,0) -- (.5,.866025403784439) -- cycle;
\path (-.5,.866025403784439) -- (.5,.866025403784439) -- (0,1.73205080756888) -- cycle;
\node[black,,scale=1.66931596144177] at (-.25,.433012701892219) {2};
\node[black,,scale=1.39109499795053] at (0,.866025403784439) {\(\searrow\)2};
\path (0,1.73205080756888) -- (.5,.866025403784439) -- (1,1.73205080756888) -- cycle;
\path (0,1.73205080756888) -- (1,1.73205080756888) -- (.5,2.59807621135332) -- cycle;
\node[black,,scale=1.66931596144177] at (.75,1.29903810567666) {3};
\node[black,,scale=1.39109499795053] at (.25,1.29903810567666) {23};
\node[black,,scale=1.39109499795053] at (.5,1.73205080756888) {\(\searrow\)2};
\path (.5,2.59807621135332) -- (1,1.73205080756888) -- (1.5,2.59807621135332) -- cycle;
\path (.5,2.59807621135332) -- (1.5,2.59807621135332) -- (1,3.46410161513775) -- cycle;
\node[black,,scale=1.66931596144177] at (1.25,2.1650635094611) {2};
\node[black,,scale=1.39109499795053] at (.75,2.1650635094611) {02};
\node[black,,scale=1.39109499795053] at (1,2.59807621135332) {\(\searrow\)0};
\path (1,3.46410161513775) -- (1.5,2.59807621135332) -- (2,3.46410161513775) -- cycle;
\node[black,,scale=1.66931596144177] at (1.25,3.03108891324554) {0};
\node[black,,scale=1.39109499795053] at (1.5,3.46410161513775) {\(\searrow\)0};
\path (-1,1.73205080756888) -- (-.5,.866025403784439) -- (0,1.73205080756888) -- cycle;
\path (-1,1.73205080756888) -- (0,1.73205080756888) -- (-.5,2.59807621135332) -- cycle;
\node[black,,scale=1.66931596144177] at (-.25,1.29903810567666) {3};
\node[black,,scale=1.39109499795053] at (-.5,1.73205080756888) {\(\nearrow\)3};
\path (-.5,2.59807621135332) -- (0,1.73205080756888) -- (.5,2.59807621135332) -- cycle;
\path (-.5,2.59807621135332) -- (.5,2.59807621135332) -- (0,3.46410161513775) -- cycle;
\node[black,,scale=1.66931596144177] at (.25,2.1650635094611) {0};
\node[black,,scale=1.66931596144177] at (-.25,2.1650635094611) {0};
\node[black,,scale=1.19236784249372] at (0,2.59807621135332) {odd};
\path (0,3.46410161513775) -- (.5,2.59807621135332) -- (1,3.46410161513775) -- cycle;
\node[black,,scale=1.66931596144177] at (.25,3.03108891324554) {1};
\node[black,,scale=1.39109499795053] at (.5,3.46410161513775) {\(\searrow\)1};
\path (-1.5,2.59807621135332) -- (-1,1.73205080756888) -- (-.5,2.59807621135332) -- cycle;
\path (-1.5,2.59807621135332) -- (-.5,2.59807621135332) -- (-1,3.46410161513775) -- cycle;
\node[black,,scale=1.39109499795053] at (-.75,2.1650635094611) {03};
\node[black,,scale=1.66931596144177] at (-1.25,2.1650635094611) {0};
\node[black,,scale=1.39109499795053] at (-1,2.59807621135332) {\(\nearrow\)3};
\path (-1,3.46410161513775) -- (-.5,2.59807621135332) -- (0,3.46410161513775) -- cycle;
\node[black,,scale=1.66931596144177] at (-.25,3.03108891324554) {1};
\node[black,,scale=1.66931596144177] at (-.75,3.03108891324553) {1};
\node[black,,scale=1.19236784249372] at (-.5,3.46410161513775) {odd};
\path (-2,3.46410161513775) -- (-1.5,2.59807621135332) -- (-1,3.46410161513775) -- cycle;
\node[black,,scale=1.39109499795053] at (-1.25,3.03108891324554) {13};
\node[black,,scale=1.66931596144177] at (-1.75,3.03108891324554) {1};
\node[black,,scale=1.39109499795053] at (-1.5,3.46410161513775) {\(\nearrow\)3};
\end{scope}
\end{tikzpicture}